\documentclass[11pt,reqno]{amsart}
\usepackage{amsmath,amssymb,amsthm,mathtools,graphicx,booktabs,array,microtype,xcolor,float}
\microtypesetup{expansion=false}
\usepackage[margin=0.7in]{geometry}
\usepackage{xr}
\makeatletter
\@ifundefined{XR@addURL}{\usepackage{xr-hyper}}{}
\makeatother
\usepackage{subfiles}
\usepackage[hidelinks]{hyperref}
\usepackage{aliascnt}
\usepackage[nameinlink,capitalize,noabbrev]{cleveref}
\usepackage{etoolbox}

\makeatletter
\patchcmd{\XR@test}
 {\expandafter\protected@xdef\csname r@\XR@prefix#2\endcsname{\XR@addURL{#3}}}
 {\in@{@cref}{#2}\ifin@\expandafter\gdef\csname r@\XR@prefix#2\endcsname{#3}\else\expandafter\protected@xdef\csname r@\XR@prefix#2\endcsname{\XR@addURL{#3}}\fi}
 {}{} 
\makeatother

\ifSubfilesClassLoaded
  {\externaldocument[M-][nocite]{main}[main.pdf]}
  {\externaldocument[M-][nocite]{main}[]}
\hypersetup{pdftitle={Regularized Unbalanced Optimal Transport},pdfauthor={Minh-Nhat Phung and Minh-Binh Tran}}
\newtheorem{theorem}{Theorem}[section]

\newaliascnt{lemma}{theorem}
\newtheorem{lemma}[lemma]{Lemma}
\aliascntresetthe{lemma}

\newaliascnt{proposition}{theorem}
\newtheorem{proposition}[proposition]{Proposition}
\aliascntresetthe{proposition}

\newaliascnt{corollary}{theorem}
\newtheorem{corollary}[corollary]{Corollary}
\aliascntresetthe{corollary}

\theoremstyle{remark}
\newaliascnt{remark}{theorem}
\newtheorem{remark}[remark]{Remark}
\aliascntresetthe{remark}

\theoremstyle{definition}
\newaliascnt{algorithm}{theorem}
\newtheorem{algorithm}[algorithm]{Algorithm}
\aliascntresetthe{algorithm}

\crefname{lemma}{Lemma}{Lemmas}
\Crefname{lemma}{Lemma}{Lemmas}
\crefname{proposition}{Proposition}{Propositions}
\Crefname{proposition}{Proposition}{Propositions}
\crefname{corollary}{Corollary}{Corollaries}
\Crefname{corollary}{Corollary}{Corollaries}
\crefname{remark}{Remark}{Remarks}
\Crefname{remark}{Remark}{Remarks}
\crefname{algorithm}{Algorithm}{Algorithms}
\Crefname{algorithm}{Algorithm}{Algorithms}
\newcommand{\eps}{\varepsilon}
\newcommand{\Hh}{\mathcal H}
\newcommand{\Zz}{\mathcal Z}
\newcommand{\Bb}{\mathcal B}
\newcommand{\G}{\mathcal G}
\newcommand{\Pp}{\mathcal P}
\newcommand{\T}{\mathsf T}
\newcommand{\Q}{\mathsf Q}
\newcommand{\N}{\mathsf N}
\newcommand{\Lsf}{\mathsf L}
\newcommand{\norm}[1]{\lVert#1\rVert}
\newcommand{\ip}[2]{\langle#1,#2\rangle}
\newcommand{\dd}{\,\mathrm d}
\newcommand{\one}{\mathbf 1}
\newcommand{\Gap}{\operatorname{Gap}}
\newcommand{\esssup}{\operatorname*{ess\,sup}}
\newcommand{\essinf}{\operatorname*{ess\,inf}}
\newcommand{\ind}{\mathbf 1\mkern-4.5mu\mathrm l}
\numberwithin{equation}{section}
\allowdisplaybreaks
\graphicspath{{figures/}}
\title[Regularized Unbalanced Optimal Transport]
{Regularized Unbalanced Optimal Transport with Pearson Marginal Penalties}
\author[M.-N. Phung]{Minh-Nhat Phung}
\address{Department of Mathematics, Texas A\&M University, College Station, TX 77843, USA}
\email{pmnt1114@tamu.edu}
\author[M.-B. Tran]{Minh-Binh Tran}
\address{Department of Mathematics, Texas A\&M University, College Station, TX 77843, USA}
\email{minhbinh@tamu.edu}
\thanks{M.-N. Phung and M.-B. Tran are funded in part by a Humboldt Fellowship, NSF CAREER DMS-2303146, and NSF grants DMS-2204795, DMS-2305523, and DMS-2306379.}
\begin{document}
\begin{abstract}
We study quadratically regularized unbalanced optimal transport with Pearson marginal penalties and a lower bound $\delta\ge0$ on the plan, including the sparse case $\delta=0$. We organize the analysis around the active interface, where the plan switches from its lower bound to larger values. We prove a Newton remainder estimate that depends linearly on the marginal measure of the boundary layer around this interface, and a maximum-norm inverse estimate uniform over all Newton derivatives and cell volumes. Together they give local superlinear convergence of semismooth Newton iteration in function space, of q-order at least $1+\alpha$ when the boundary-layer measure is $O(t^\alpha)$, and a common convergence neighborhood along each refinement sequence whose data converge uniformly. Transverse interfaces give second-order supercloseness of the discrete potentials to the cell averages of the continuum potentials. A directional curvature condition on the input data verifies the interface condition before the solution is known and yields an explicit Newton convergence radius on all sufficiently fine meshes, uniformly in the lower bound. A globalized inexact semismooth Newton iteration has mesh-independent numbers of outer iterations and of full-array operations, that is, passes over the plan array. Its output satisfies an a posteriori error bound relative to the continuum optimal plan. Experiments test the rates at transverse interfaces and at interfaces with tangencies, the curvature criterion, and the algorithmic work. All estimates hold for a fixed regularization parameter and a fixed lower bound on the densities.
\end{abstract}
\maketitle

\section{Introduction}

For many years, optimal transport, with its primal and dual formulations, has been employed in various fields. In the classical setting, optimal transport compares measures of equal total mass by minimizing the cost of transporting one into the other. In mathematics, it has prominent applications to variational problems and strong connections with PDEs \cite{caffarelli1996boundary,mccann1997convexity,jordan1998variational,villani2009optimal}. Beyond these mathematical developments, optimal transport has also been applied to economic matching and equilibrium models \cite{galichon2016optimal,chiappori2010hedonic,carlier2010matching}, as well as image processing, generative modeling, and document comparison \cite{tartavel2016wasserstein,arjovsky2017wasserstein,kusner2015word}. In applications, however, the measures being compared may have different total masses, and these differences may carry meaningful information. This motivates unbalanced optimal transport (UOT), which allows measures of different total mass by relaxing the marginal constraints and penalizing deviations of the transported marginals from the prescribed measures \cite{chizat2018unbalanced,liero2018optimal,chizat2018scaling}. This flexibility is useful, for example, in modeling cellular growth and death \cite{schiebinger2019optimal} and in domain adaptation \cite{fatras2021unbalanced}.

These applications have also motivated the development of numerical solvers for UOT, with computational background available in \cite{sejourne2023unbalanced,peyre2019computational}. In this work, we use quadratic regularization because it allows the transport plan to vanish on parts of the domain, so that mass is transported only between selected pairs of locations \cite{blondel2018smooth,lorenz2021quadratic,wiesel2025sparsity}. By contrast, entropic regularization typically produces positive transport between all pairs, making quadratic regularization a natural choice when sparse transport plans are desirable. In addition to this sparsity property, quadratic regularization has a structure that is useful for computation: on a fixed mesh, when combined with the quadratic marginal penalties in \eqref{eq:primal}, it leads to optimality equations that are affine whenever the set where the plan attains its lower bound remains unchanged. We exploit this structure to develop and analyze our numerical solver. To describe the transition between the two regions of the plan, we call the boundary between pairs of locations where the plan is positive and those where it vanishes the \emph{active interface}. If a positive lower bound is imposed, this interface instead separates the region where the plan equals the bound from the region where it exceeds it. The behavior of the solution near this interface plays a central role in our analysis, and the precise slack defining the transition is given in \eqref{eq:optimal-slack}.

The quadratic marginal penalties used in \eqref{eq:primal} are Pearson ($\chi^2$) divergences, with a factor of $1/2$. Related quadratic marginal penalties have been studied in unbalanced transport \cite{chapel2021unbalanced}. Pearson divergence also appears in generative modeling \cite{nowozin2016fgan} and distributionally robust learning \cite{duchi2021learning}, while related uses include Pearson-like divergences for change-point detection \cite{si2024multivariate} and a $\chi^2$-based sampling criterion in variational inference \cite{zimmermann2024visa}. A general framework for transport with convex marginal penalties, without additional regularization of the plan, is developed in \cite{liero2018optimal}. For our purposes, these quadratic penalties are well aligned with quadratic regularization and retain the affine structure of the optimality equations described above. They also coincide with the second-order expansion of the Kullback--Leibler (KL) penalty at unit relative density, that is, at $k_X=f$ and $k_Y=g$. We adopt them as a model rather than as a surrogate for KL. As additional material, we nevertheless consider the difference between the Pearson and KL minimizers. For a strictly positive lower bound $\delta$, their quantitative comparison and a KL solver with a posteriori error bounds are given in \cref{S-sec:kl} of the supplementary materials.

To solve regularized transport problems, first-order methods have been studied extensively, including algorithms for quadratic regularization \cite{lorenz2019orlicz}, gradient descent for balanced quadratically regularized transport \cite{gonzalez2025linear}, and gradient-based methods for unbalanced transport \cite{nguyen2023uot,genans2026adaptive}. For the Pearson model considered here, projected gradient (PG), defined in \eqref{S-eq:PG}, and accelerated proximal gradient methods such as V-FISTA, defined in \eqref{S-eq:VFISTA}, have linear convergence guarantees following from standard strongly convex optimization theory \cite{beck2017first,nesterov2004introductory}. When formulated in the quadrature-weighted metric defined by \eqref{eq:weighted-norm}, these estimates have constants independent of the mesh, as recorded in the supplementary materials (\cref{S-sec:global}). The convergence factors in these worst-case bounds nevertheless approach one as the regularization parameter decreases.

The structure of the problem also motivates semismooth Newton methods, which extend Newton's method to equations involving nonsmooth functions such as the positive part \cite{qi1993nonsmooth,ulbrich2011semismooth}. These methods are particularly suited to the present model: with quadratic regularization and Pearson marginal penalties, the discrete optimality equations are affine whenever the set where the plan attains its lower bound remains fixed. On each fixed mesh, exact full-step Newton iteration therefore reaches the solution in one step once sufficiently close to it. However, this local property does not guarantee a convergence neighborhood that remains uniform as the mesh is refined.

The central question in our work is whether this fast local behavior persists as the mesh is refined. Although mesh independence of Newton methods has an established theory \cite{hintermuller2004mesh,weiser2005asymptotic}, its application here requires understanding how the active interface affects convergence. At fixed regularization, our contribution is to quantify how the behavior of the optimal slack near zero controls the convergence of semismooth Newton methods and the discretization error. This leads to a common convergence neighborhood along refinement sequences whose data converge uniformly. We further give conditions on the input data that make the convergence radius explicit on sufficiently fine meshes, uniformly over an admissible range of lower bounds $\delta$.

\subsection{The model}
To make this question precise, we first state the model and its dual optimality equation, and we introduce the optimal slack, through which we describe the active interface quantitatively.

Let $X\subset\mathbb R^{d_X}$ and $Y\subset\mathbb R^{d_Y}$ be bounded measurable sets, with $a=|X|>0$ and $b=|Y|>0$. Here and throughout, we denote by $|A|$ the Lebesgue measure of set $A$ in the corresponding dimension.  We also use Lebesgue measure for all the integrals. Given densities $f,g$, a nonnegative cost $C$, a regularization parameter $\eps>0$, and a lower bound $\delta\ge0$ on the plan, we consider
\begin{equation}\label{eq:primal}
	\begin{aligned}
		&\min_{k\in L^2(X\times Y),\,k\geq\delta}\Pp(k),\\
&\Pp(k):=\int_{X\times Y} Ck\dd x\dd y+\frac\eps2\norm{k}_{L^2(X\times Y)}^2
+\frac12\int_X\frac{(k_X-f)^2}{f}\dd x
+\frac12\int_Y\frac{(k_Y-g)^2}{g}\dd y,
	\end{aligned}
\end{equation}
where
\[
k_X(x)=\int_Yk(x,y)\dd y,
\qquad
k_Y(y)=\int_Xk(x,y)\dd x
\]
are the marginals of the plan, under the standing assumptions
\begin{equation}\label{eq:data}
0<c\leq f,g\leq E,\qquad
0\leq C\in L^\infty(X\times Y),\qquad
\delta\geq0,\quad \delta\max\{a,b\}<c.
\end{equation}
Here, the density and cost bounds hold almost everywhere.

The case $\delta=0$ allows sparse plans and is covered by every Pearson result below, while a positive lower bound enforces strictly positive plans when these are required.
The condition $\delta\max\{a,b\}<c$ implies $\delta b<f$ and $\delta a<g$ almost everywhere. Thus the constant lower bound $\delta$ contributes less mass to each marginal than the prescribed density.

For the potentials $z=(u,v)\in L^2(X)\times L^2(Y)$, we will derive, in \cref{sec:duality}, the dual optimality equation
\begin{equation}\label{eq:residual-intro}
F(z)=\T^*\!\left[
\delta+\eps^{-1}(\T z-C-\eps\delta)_+
\right]
+(f(u-1),g(v-1))=0,
\qquad
\T(u,v)(x,y)=u(x)+v(y),
\end{equation}
where $\T^*k=(k_X,k_Y)$ and $r_+=\max\{r,0\}$. We call $F$ the \emph{residual}. We define the optimal slack by
\begin{equation}\label{eq:optimal-slack}
s^*=\T z^*-C-\eps\delta,
\end{equation}
where $z^*$ is the solution of \eqref{eq:residual-intro}. The optimal plan is then recovered through
\[
k^*=\delta+\eps^{-1}(s^*)_+.
\]
Thus $k^*=\delta$ where $s^*\leq0$, while $k^*>\delta$ where $s^*>0$. The level $s^*=0$ is the threshold between these two branches and provides the precise quantity used below to study the active interface.

\subsection{Obstructions}

We now examine the difficulties in obtaining local superlinear convergence of semismooth Newton iteration for \eqref{eq:residual-intro}, with convergence estimates that remain uniform as the mesh is refined.

The diagonal terms $f(u-1)$ and $g(v-1)$ in \eqref{eq:residual-intro} make the Newton derivatives of $F$, which will be defined in \eqref{eq:H-cont} and \eqref{eq:Hh}, uniformly coercive in the corresponding $L^2$ and quadrature-weighted inner products, independently of the slope selection in \eqref{eq:slopes}. Consequently, the linear equation defining each Newton step in \eqref{eq:Newton} has a unique solution. However, this property alone does not establish local superlinear convergence with estimates that remain uniform under mesh refinement. We clarify this with three main obstructions.

\begin{enumerate}
\item[(a)] \emph{The choice of norm.}
Because of the positive part $(\T z-C-\eps\delta)_+$ in \eqref{eq:residual-intro}, the residual $F$ is not Newton differentiable from $L^2$ to $L^2$. Consider a perturbation of the potential $u$ that takes a fixed constant value on a set $A\subset X$ and vanishes elsewhere. Its $L^2$ norm is proportional to $|A|^{1/2}$ and tends to zero as $|A|\to0$, but its value on $A$ stays the same, so it can move the slack across zero on all of $A\times Y$. The Newton remainder, that is, the error in the Newton linear approximation, is then of the same order as the perturbation itself (\cref{rem:L2}). The role of different domain and range norms is well established in semismooth Newton theory in function spaces \cite{ulbrich2002semismooth,ulbrich2011semismooth}. In our analysis, we measure proximity to the solution in the maximum norm, that is, the maximum of the essential-supremum norms of the two potential components. However, converting the $L^2$ inverse estimate to a maximum-norm estimate through norm equivalence on a cell space introduces a factor involving the reciprocal square root of the smallest cell volume (\cref{sec:discretization,rem:M-operator}). Therefore, this argument does not provide an inverse bound independent of the mesh.

\item[(b)] \emph{Uniformity under mesh refinement.}
On each fixed mesh, the residual is piecewise affine, and exact full-step Newton iteration reaches the solution in one step once sufficiently close to it. More precisely, this holds in the maximum-norm ball around the discrete solution whose radius is half the smallest nonzero value of $|s_{ij}^*|$, where $s_{ij}^*=u_i^*+v_j^*-C_{ij}-\varepsilon\delta$ is the discrete slack. This radius may tend to zero as the mesh is refined, so the fixed-mesh termination property does not by itself establish a common convergence neighborhood along refinement. The distinction between convergence on a fixed discretization and convergence estimates uniform under refinement is central to mesh-independence theory \cite{hintermuller2004mesh,weiser2005asymptotic}. We therefore need an estimate that remains valid even when the neighborhood in which one-step termination is guaranteed shrinks (see \cref{rem:two-scale}).

\item[(c)] \emph{Uniformity with respect to the lower bound.}
A separate question arises when the lower bound $\delta$ varies. Even if each problem admits local superlinear convergence, the family need not share a common superlinear error estimate, meaning a bound on the ratio of successive errors that tends to zero independently of $\delta$. The example in \cref{rem:lower-bound-nonuniform} demonstrates this difficulty. Thus convergence estimates uniform under mesh refinement for fixed data and a fixed lower bound must be distinguished from estimates that also hold uniformly over a range of lower bounds.
\end{enumerate}

These difficulties motivate a condition on the measure of the sets where the optimal slack is close to zero, uniformly in each family of marginal sections. To state the interface condition used in our main results, we call the small-slack set $\{|s^*|\le t\}$ the \emph{boundary layer} of width $t$ around the active interface, following \cite{gonzalez2025linear}, and we define the \emph{marginal boundary-layer function}
\begin{equation}\label{eq:layer-intro}
m_*(t)=\max\left\{
\operatorname*{ess\,sup}_{x\in X}
|\{y\in Y:|s^*(x,y)|\leq t\}|,
\operatorname*{ess\,sup}_{y\in Y}
|\{x\in X:|s^*(x,y)|\leq t\}|
\right\}.
\end{equation}
This quantity measures the size of the small-slack set uniformly in both families of marginal sections. The condition $m_*(t)\to0$ as $t\downarrow0$ requires these sets to become uniformly small as the slack approaches zero. For balanced quadratically regularized transport, a uniform estimate of this marginal boundary layer is established in \cite[Lemma~3.2(i) and (43)--(44)]{gonzalez2025linear}, and small-level-set conditions of this type are also used in sparse optimal control \cite[Section~3]{WW2011sparsity}. In control-constrained optimal control, the analogous assumption bounds the measure of almost active sets, and it yields semismooth Newton convergence of order $3/2$ \cite{hinze2012semismooth}.

\subsection{Main results and contributions}\label{sec:contributions}

Our contributions connect the interface condition \eqref{eq:layer-intro} with Newton convergence, discretization error, and computable guarantees for the numerical solution. They can be summarized as follows.

\begin{enumerate}
\item \textbf{Mesh-independent local convergence controlled by the active interface.}
We show that, in the maximum norm, the Newton remainder is bounded linearly by the measure of the marginal boundary layer (\cref{lem:remainder}). We also show that every Newton derivative has an inverse bounded by $c^{-1}$, independently of the slopes in \eqref{eq:slopes} and of the cell volumes. This bound follows from a comparison principle (\cref{lem:inverse}). Consequently, the semismooth Newton method converges q-superlinearly in function space. Its local q-order is at least $1+\alpha$ when $m_*(t)=O(t^\alpha)$ (\cref{thm:local}). In the terminology of \cite{ulbrich2011semismooth}, the remainder bound then says that the residual is semismooth of order $\alpha$ at $z^*$ in the maximum norm. To pass to the discrete problems, we prove that the optimal slack is stable in the maximum norm with respect to the data (\cref{lem:uniform-stability}). The continuum boundary-layer bound therefore carries over to every sufficiently fine mesh. This is why $m_*$ includes the points where $s^*=0$, since a zero continuum slack may become a small nonzero discrete slack after data approximation. Hence, along every refinement sequence whose data converge uniformly, the discrete iterations share a common convergence neighborhood and a common contraction modulus. No quasi-uniformity assumption on the partitions is needed (\cref{thm:mesh-local}). These results resolve obstructions~(a) and~(b). They identify a geometric condition on the active interface under which mesh independence \cite{hintermuller2004mesh,weiser2005asymptotic} holds for the present transport model.

\item \textbf{Supercloseness of the discrete potentials.}
	With exact cell averages of Lipschitz data, the discrete potentials satisfy $\norm{\Lsf z_h^*-\Q_h^{\Hh}z^*}_{L^\infty(X)\times L^\infty(Y)}=O(h^{1+\alpha}+h^2)$ whenever $m_*(t)=O(t^\alpha)$ (\cref{thm:superclose}). As in finite element superconvergence theory \cite{wahlbin1995superconvergence}, we call a discrete solution \emph{superclose} to a projection of the exact solution if their distance is of higher order than the best approximation error of the discrete space. The proof compares the discrete residual at the cell averages $\Q_h^{\Hh}z^*$ with the cell averages of $F(z^*)=0$. The linear terms cancel exactly, the positive-part term survives only on cells crossing the interface and contributes $O(h\,m_*(O(h)))$, and the products of densities and potentials leave covariance terms of order $h^2$. The proof compares the discrete residual at the cell averages with the cell averages of the continuum residual, and the linear terms cancel. The remaining nonlinear term, from the positive part, vanishes on every cell on which the optimal slack keeps one sign, so it is confined to cells in the boundary layer. Products of densities and potentials contribute covariance terms of order $h^2$ (see \eqref{eq:residual-decomposition} and the proof of \cref{thm:superclose}). For transverse interfaces, \eqref{eq:transverse} gives $\alpha=1$, and the discrete potentials are thus second-order superclose. In contrast, the piecewise-constant plan is in general only first-order accurate (see the discussion after \eqref{eq:plan-decomposition}).

\item \textbf{An interface condition verifiable from the input data, uniformly in the lower bound.}
The interface condition involves the unknown optimal slack. We show that it can be verified from the input data alone. Suppose that the cost satisfies the directional curvature conditions \eqref{eq:fc-cost-curvature} and that the densities satisfy the margin conditions \eqref{eq:fc-margins}. Then the optimal slack is strongly concave along the chosen directions, and $m_{*,\delta}(t)\le A_\eps\sqrt t$ with $A_\eps$ independent of $\delta\in[0,\bar\delta]$ (\cref{thm:curvature}). This bound yields an explicit convergence radius, an explicit mesh threshold, and a prescribed contraction factor for the full-step Newton iteration. These guarantees hold on every sufficiently fine mesh and uniformly in the lower bound $\delta$ (\cref{cor:explicit-radius}). Such uniformity cannot be taken for granted. \Cref{rem:lower-bound-nonuniform} gives a family of problems in which each problem has superlinear local convergence, but the family admits no common superlinear error bound. The curvature conditions exclude this degeneracy and thereby resolve obstruction~(c).

\item \textbf{A globalized method with mesh-independent work bounds and an a posteriori error bound.}
We embed the local theory in a globalized inexact semismooth Newton method (\cref{alg:globalized}), related to the framework of \cite{alphonse2026globalized}. The method reaches the local convergence neighborhood from the zero initial guess without knowledge of the solution. The M-matrix structure behind \cref{lem:inverse} gives mesh-independent maximum-norm contraction factors for two components of the method. The first is the diagonally scaled gradient step, which globalizes the method. The second is the Jacobi iteration, which solves the Newton equations approximately (\cref{lem:inf-gradient,lem:jacobi}). As a result, the numbers of outer iterations and of full-array operations are bounded in terms of $a,b,c,E,\eps,\varpi$, and the tolerance only. Moreover, the iteration eventually takes only Newton steps with step size one (\cref{thm:algorithm}). Here a \emph{full-array operation} is one evaluation of the discrete residual, one evaluation of the slopes \eqref{eq:slopes} together with the diagonal of the Newton derivative, or one application of the off-diagonal part of the Newton derivative (\cref{thm:algorithm}). Each such operation costs $O(N_XN_Y)$. Finally, the returned potential satisfies a computable a posteriori error bound relative to the continuum optimal plan. This bound separates two sources of error (\cref{cor:computational}). The \emph{algebraic error} is caused by solving the discrete problem only approximately. The \emph{data-approximation error} is caused by replacing $f,g,C$ by cellwise data. We use these terms as in \cite{ern2013adaptive,papez2018estimating}.
\end{enumerate}

Throughout, a \emph{discretization} consists of finite partitions $\{X_i\}_{i=1}^{N_X}$ and $\{Y_j\}_{j=1}^{N_Y}$, giving product cells $X_i\times Y_j$, with piecewise-constant data and potentials. The mesh size $h$ is the largest product-cell diameter, as in \eqref{eq:h}. The following theorem states the first three contributions precisely. The computational guarantees are summarized in \cref{cor:computational}. Precise constants for the local theory are given in \crefrange{sec:local}{sec:curvature}.

\begin{theorem}\label{thm:principal}
	Assume \eqref{eq:data}, and let $f_h,g_h,C_h$ be cellwise data on finite product-cell partitions satisfying the same bounds. Let $\alpha>0$ be a fixed order.
\begin{enumerate}
\item[(i)] \emph{Interface-controlled Newton convergence.}
If $m_*(t)\to0$ as $t\downarrow0$, the semismooth Newton iteration for \eqref{eq:residual-intro} converges q-superlinearly from a maximum-norm neighborhood of $z^*$, with local q-order at least $1+\alpha$ when $m_*(t)\le At^\alpha$ (\cref{thm:local}). Along any refinement sequence whose data converge uniformly, the discrete iterations share a common maximum-norm convergence neighborhood and a common bound on the ratio of successive errors that tends to zero with the error, as in \eqref{eq:uniform-newton}. Armijo backtracking \cite{armijo1966minimization}, \cite[Section~3.1]{nocedal2006numerical}, with the Armijo condition \eqref{eq:Armijo}, accepts the step size $\tau=1$ in a possibly smaller common neighborhood
(\cref{thm:mesh-local}).

\item[(ii)] \emph{Supercloseness.}
For Lipschitz data, exact cell averages, and $m_*(t)\le At^\alpha$, the discrete potentials satisfy, for all sufficiently small $h$,
\begin{equation}\label{eq:principal-superclose}
	\norm{\Lsf z_h^*-\Q_h^{\Hh}z^*}_{L^\infty(X)\times L^\infty(Y)}
\leq\Xi_1h^{1+\alpha}+\Xi_2h^2,
\end{equation}
where $\Lsf z_h^*,\Q_h^{\Hh}z^*$ denote the lift and the cell averages of the two potential components (\cref{sec:discretization}). The constants $\Xi_1,\Xi_2$ are explicit (\cref{thm:superclose}). A transverse interface satisfying the conditions in \cref{sec:suff_geo} thus gives second-order supercloseness, whereas the piecewise-constant plan is in general only first-order accurate.

\item[(iii)] \emph{A data-verifiable criterion.}
Suppose $X,Y$ are convex, the data admit $C^{1,1}$ extensions, and the directional curvature conditions \eqref{eq:fc-cost-curvature} and \eqref{eq:fc-margins} hold. Fix $\bar\delta\ge0$ with $\bar\delta\max\{a,b\}<c$. Then the function $m_{*,\delta}$ defined in \eqref{eq:fc-layer-definition} satisfies $m_{*,\delta}(t)\le A_\eps\sqrt t$ for all $t>0$, with $A_\eps$ independent of the lower bound $\delta\in[0,\bar\delta]$ (\cref{thm:curvature}). Consequently, on every mesh with cell-averaged data finer than an explicit threshold, the full-step Newton iteration converges from an explicit maximum-norm radius and contracts the error by a prescribed factor at every step. When this factor is small enough relative to the Armijo parameter and the condition number $\kappa=L/c$ of \cref{prop:global-constants}, the Armijo condition accepts the step size $\tau=1$ at every step. These statements hold uniformly in $\delta\in[0,\bar\delta]$ (\cref{cor:explicit-radius}).
\end{enumerate}
\end{theorem}

To reach the local convergence neighborhood without knowing the solution, the algorithm combines inexact Newton trials with diagonally scaled gradient steps \eqref{eq:inf-scaled}. The gradient steps ensure residual reduction when a Newton trial is rejected, thereby providing a globalization strategy. The following corollary gives the resulting computational guarantees and an a posteriori error bound, meaning a computable upper bound on the error of the computed plan relative to the continuum optimal plan.

\begin{corollary}\label{cor:computational}
Let the cellwise data satisfy \eqref{eq:data}, and fix $\varpi\in(0,1)$ and $\mathrm{tol}>0$. \Cref{alg:globalized}, started from zero with parameters computed from $a,b,c,E,\eps,\varpi$, returns after at most
\[
M_{\mathrm{tol}}
=\left\lceil
\frac{\log_+(E/\mathrm{tol})}{\log(1/\varpi)}
\right\rceil
\]
outer iterations, where $\log_+(x):=\max\{0,\log x\}$ for $x>0$. Its number of full-array operations, each of cost $O(N_XN_Y)$, is bounded by \eqref{eq:total-work} in terms of $a,b,c,E,\eps,\varpi$, and $\mathrm{tol}$ only, and hence independently of the mesh. Without the stopping test, the iteration eventually takes only Newton steps with step size one and converges superlinearly. Under the hypotheses of \cref{thm:principal}(iii), the number of rejected Newton trials is bounded explicitly, uniformly over all meshes that meet the threshold of \cref{thm:principal}(iii) with a fixed margin. The returned potential $z$ satisfies
\begin{equation}\label{eq:principal-aposteriori}
	\norm{\Lsf K(z)-k^*}_{L^2(X\times Y)}
\leq D_h+
\frac{\sqrt{a+b}\,\mathrm{tol}}{2\sqrt{\mu_h\eps}},
\end{equation}
where $K(z)$ is the reconstructed plan from \eqref{eq:finite-reconstruct}, $\Lsf$ is the lift from \eqref{eq:lift-K}, $\mu_h$ is defined in \eqref{eq:finite-constants}, and
\begin{equation}\label{eq:Dh-intro}
D_h=\frac1{2\eps}\left[
	\norm{C-C_h}_{L^2(X\times Y)}+
\Big(
	\norm{C-C_h}_{L^2(X\times Y)}^2+
\frac{4\eps}{c}
\big(\norm{f-f_h}_{L^2(X)}^2+\norm{g-g_h}_{L^2(Y)}^2\big)
\Big)^{1/2}
\right].
\end{equation}
\end{corollary}

The conclusions above require different assumptions. The termination and work bounds require only \eqref{eq:data} for the cellwise data, while the a posteriori error bound additionally requires bounds on the $L^2$ data-approximation error. A common Newton convergence neighborhood along refinement requires uniform data convergence and $m_*(t)\to0$. Supercloseness further requires Lipschitz data, exact cell averages, and a power bound on $m_*$. The explicit convergence radius and the explicit bound on rejected Newton trials, uniform over the admissible lower bounds $\delta$ and sufficiently fine meshes, require the additional domain, regularity, and curvature assumptions in part~(iii).

\subsection{Relation to earlier work}
We now place these results and techniques in the context of earlier work on quadratically regularized transport, semismooth Newton methods, and a posteriori error estimation.

In the continuum setting, balanced quadratically regularized transport has been studied for duality and existence \cite{lorenz2021quadratic}, existence and multiplicity of dual potentials \cite{nutz2025potentials}, and stability under perturbations of the marginals and cost \cite{gonzalez2026stability}. The geometry of the support as the regularization parameter tends to zero is investigated in \cite{gonzalez2026geometry}. Our analysis fixes the regularization parameter and studies how the geometry near the active interface affects Newton convergence and discretization error. The Pearson marginal penalties also change the structure of the dual problem: their positive diagonal contribution removes the degeneracy of the Newton derivative of the balanced problem (\cref{sec:newton-derivatives}).

The dual obtained from Pearson divergence is closely related to the Tikhonov-regularized balanced dual introduced by Lorenz, Manns, and Meyer \cite[Section~2.3]{lorenz2021quadratic}. We give the precise identification in \cref{rem:regularized-balanced}. Their implemented semismooth quasi-Newton method combines the Newton matrix of the regularized dual with the residual of the balanced equation \cite[Section~3.2]{lorenz2021quadratic}. Our continuum convergence results and convergence estimates uniform under mesh refinement concern semismooth Newton methods for the dual optimality equation \eqref{eq:residual-intro}. They do not establish convergence of that balanced quasi-Newton iteration.

The localization of changes in the linearized operator to small-slack sets already appears in the gradient-descent analysis of \cite{gonzalez2025linear}. Our analysis uses this localization to estimate the Newton remainder and combines it with a maximum-norm inverse bound (\cref{lem:remainder,lem:inverse}). Related geometric estimates for balanced quadratically regularized transport are developed in \cite{gonzalez2025linear,gonzalez2026scalar,gonzalez2026stability}. For the present unbalanced model, we derive conditions on the input data that yield explicit bounds uniform over an admissible range of lower bounds $\delta$ (\cref{thm:curvature}). This uniformity is a separate issue from convergence for each fixed lower bound, as illustrated by \cref{rem:lower-bound-nonuniform}.

Globalized inexact semismooth Newton methods for contractive fixed-point equations are studied in \cite{alphonse2026globalized}. A globalized inexact semismooth Newton method for strongly convex optimal control problems is analyzed in \cite{wachsmuth2025globalized}. As here, Newton's method is applied to the dual problem, and the dual objective serves as the merit function. The transition to the step size one there also relies on a second-order expansion of the dual objective, which plays the role of \eqref{eq:objective-expansion} in our analysis. That analysis concerns convergence in a Hilbert space and does not address uniformity under mesh refinement. In contrast, we derive maximum-norm rates and convergence neighborhoods that are uniform in the mesh. For the present transport model, we exploit the marginal structure to obtain explicit bounds on the work of the globalized iteration and its inner linear solves (\cref{lem:inf-gradient,lem:jacobi}, \eqref{eq:total-work}). The local estimates additionally describe how the interface affects the transition to Newton steps with step size one (\cref{thm:mesh-local,cor:explicit-radius}).

The a posteriori error bound \eqref{eq:principal-aposteriori} belongs to the general framework of estimates based on primal--dual gaps \cite{repin2008aposteriori,neittaanmaki2004reliable,bartels2020gap}. For the reconstructed plan, we establish the sharp gap bound \eqref{eq:sharp} and combine it with stability under data approximation (\cref{prop:stability,cor:total}).
The resulting separation into algebraic and data-approximation errors is related to the treatment of discretization and iterative-solver errors in \cite{ern2013adaptive,papez2018estimating}. Here it provides a stopping and verification criterion for the computed plan relative to the original continuum problem. This bound is a guaranteed upper bound, but it need not be bounded above by a constant multiple of the actual plan error (\cref{S-rem:no-plan-efficiency}). Its sharp constant and further properties are treated in \cref{S-sec:aposteriori-limits} of the supplementary materials.

\subsection{Scope and organization}
Before outlining the organization of the paper, we summarize the scope and limitations of the present analysis.

All estimates are for fixed $\varepsilon>0$ and a fixed density lower bound $c>0$. Several constants deteriorate as either parameter tends to zero. This dependence is recorded in \cref{rem:eps-c}. We consider finite product-cell partitions with piecewise-constant data and potentials. By \cref{prop:reduction}, each discrete problem is exactly the continuum problem for the corresponding piecewise-constant data, so the analysis of discretization error reduces to the effect of data approximation. Higher-order and non-product discretizations are outside the scope of this work. The supercloseness result concerns discrete potentials relative to cell averages of the continuum potentials. The reconstructed piecewise-constant plan is in general only first-order accurate under the stated Lipschitz assumptions. Moreover, each full-array operation costs $O(N_XN_Y)$ arithmetic operations. Thus the bounds on iteration and operation counts are independent of the mesh, but the total arithmetic cost is not.

The remainder of the paper is organized as follows. \Cref{sec:duality} establishes duality, global convexity and smoothness bounds, exact cell reduction, and $L^2$ stability with respect to the data. \Cref{sec:local} proves the maximum-norm inverse estimate, the Newton remainder estimates, and local convergence in function space. \Cref{sec:mesh} derives convergence estimates uniform under mesh refinement and proves supercloseness of the discrete potentials to cell averages of the continuum solution. \Cref{sec:curvature} gives the curvature criterion on the input data and the explicit Newton convergence radius. \Cref{sec:practical} analyzes the globalized inexact semismooth Newton iteration, and \cref{sec:numerics} presents numerical experiments. The supplementary materials contain further results on the a posteriori error bound and on the globalized iteration, the analysis of weighted global solvers, additional experiments, the KL extension, and the comparison of the Pearson and KL models. Principal notation is collected in \cref{tab:notation}.

\begin{table}[htb]
\caption{Principal notation.}\label{tab:notation}
\centering\small
\begin{tabular}{@{}p{0.32\textwidth}p{0.63\textwidth}@{}}
\toprule
$X,Y,a,b$
& domains and their volumes; $c,E$: density bounds;
  $\delta$: lower bound on the plan \eqref{eq:data}\\

$C,\eps$
& cost and regularization parameter; $M=\max\{\delta,2/\eps\}$;
  $\theta_X^*,\theta_Y^*$: bounds on $1-u^*,1-v^*$
  \eqref{eq:theta-star}\\

$k^*,z^*=(u^*,v^*)$
& optimal plan and optimal potentials;
  $s^*=\T z^*-C-\eps\delta$: optimal slack
  \eqref{eq:optimal-slack}\\

$L,c,\kappa=L/c$
& Lipschitz and strong-convexity constants of $\G$, and the
  condition number (\cref{prop:global-constants})\\

$h,f_h,g_h,C_h$
& mesh size and cellwise data
  (\cref{sec:discretization})\\

$\Q_h^{\Hh}$
& componentwise cell averaging of the potentials
  (\cref{sec:discretization})\\

$\mu_h,L_h,E_h^{\max}$
& cellwise convexity, smoothness, and density bound
  \eqref{eq:finite-constants}\\

$D_h,D_{z,h}$
& $L^2$ data-stability radii
  (\cref{prop:stability})\\

$\zeta_h,\lambda_h,r_{\infty,h}$
& maximum-norm data, slack, and residual radii
  (\cref{lem:uniform-stability,sec:uniform-stability})\\

$L_u,L_v,L_s,L_{\rm cov}$
& potential and slack Lipschitz constants, and the covariance constant
  (\cref{lem:Lip,thm:superclose})\\

$m_*,\nu_*,\nu_h,\widehat\nu$
& inclusive, punctured, discrete, and refinement-uniform marginal
  boundary-layer functions
  (\eqref{eq:layer-intro},
   \eqref{eq:punctured-layer},
   \eqref{eq:discrete-layer},
   \eqref{eq:uniform-layer})\\

$\omega,\widehat\omega,r_0$
& local Newton contraction moduli and common radius
  (\cref{thm:local,thm:mesh-local})\\

$t_\eps,r_{0,h},\bar\omega,L_\lambda$
& explicit boundary-layer width, explicit convergence radius, prescribed contraction factor, and data
  constant; $r_{\infty,h}\le r_{0,h}$ is a computable test for entering the ball of radius $r_{0,h}$
  (\cref{cor:explicit-radius})\\

$L_\infty,L_{\infty,h},\tau_\infty,\mathsf D_h,\varpi$
& maximum-norm Lipschitz constants, plain gradient step,
  diagonal scaling of the gradient step, and residual-reduction factor
  (\cref{sec:safeguard,sec:algorithm})\\

$\eta_m,\bar\eta,q_{{\rm J},h},q_{\rm J}$
& maximum-norm forcing terms and their bound, and contraction factors of the
  scaled step and the Jacobi iterations
  (\cref{sec:safeguard,sec:inexact-inner})\\

$N_{\rm G},N_{\rm J},\bar N_{\rm G},\bar N_{\rm J}$
& gradient and Jacobi work counts and their mesh-uniform bounds
  (\cref{sec:safeguard,sec:algorithm})\\

$\kappa_C$
& quadratic-cost curvature parameter in the examples of
  \cref{sec:curvature,sec:num-curvature}\\

$\mathcal E_h,\Theta$
& total $L^2$ data error and its stability constant
  (\cref{sec:stability})\\
\bottomrule
\end{tabular}
\end{table}
\subsection*{AI Use Declaration} AI tools (ChatGPT and Claude) were used for polishing the presentation, language editing, code review,  and manuscript checking. Especially, the numerical results were obtained by executing
code developed with the assistance of GPT-5.6 Sol (OpenAI). The authors reviewed   and validated
the  results.  The authors assume
responsibility for all content.
\section{Model, duality, and exact reduction}\label{sec:duality}
In this section, we establish duality, the global constants of the dual functional, and the a posteriori error bound \eqref{eq:sharp} for the plan in the continuum. Then, in \cref{sec:discretization}, we transfer these statements to cell discretizations.

We set $\Zz=L^2(X\times Y)$ and $\Hh=L^2(X)\times L^2(Y)$, with the usual Hilbert norms.
We observe that the adjoint of $\T$ in \eqref{eq:residual-intro} is $\T^*k=(k_X,k_Y)$, and
\begin{equation}\label{eq:Tnorm}
	\norm{k_X}_{L^2(X)}\leq\sqrt b\norm{k}_{\Zz},\quad
	\norm{k_Y}_{L^2(Y)}\leq\sqrt a\norm{k}_{\Zz},\quad
 \norm{\T}_{\Hh\to\Zz}=\sqrt{a+b}.
\end{equation}
First, by the Cauchy--Schwarz inequality in $L^2(Y)$, we obtain
\[\norm{k_X}_{L^2(X)}^2=\int_{X}\Bigg(\int_{Y}kdy\Bigg)^2\dd x\le \int_{X}\Bigg(b\int_{Y}k^2\dd y\Bigg)\dd x=b\norm{k}_{\Zz}^2.\]
This is the first estimate in \eqref{eq:Tnorm}. The second estimate in \eqref{eq:Tnorm} is obtained in a similar way with the Cauchy--Schwarz inequality in $L^2(X)$.

For $z=(u,v)\in\Hh$, we apply the triangle inequality in $\Zz$ and the Cauchy--Schwarz inequality in $\mathbb R^2$ to see that
\[\norm{\T z}_{\Zz}\le\norm{u}_{\Zz}+\norm{v}_{\Zz}=\sqrt b\norm{u}_{L^2(X)}+\sqrt a\norm{v}_{L^2(Y)} \leq\sqrt{a+b}\,\norm{z}_{\Hh},\]
with equality for $u\equiv b$, $v\equiv a$. This yields the third identity in \eqref{eq:Tnorm}.
\subsection{Duality and the gap decomposition}
We now define the dual problem and the gap.

We denote the scalar reconstruction $p$ and its primitive $\phi$, the convex conjugate of $q\mapsto\eps q^2/2$ on $[\delta,\infty)$, by
\begin{align}
 p(r)&:=\max\{\delta,r/\eps\}
       =\delta+\eps^{-1}(r-\eps\delta)_+,\label{eq:p}\\
 \phi(r)&:=\sup_{q\geq\delta}(rq-\eps q^2/2)
 =\begin{cases}r^2/(2\eps),&r\geq\eps\delta,\\
 \delta r-\eps\delta^2/2,&r<\eps\delta.
 \end{cases}\label{eq:phi}
\end{align}
We observe that $\phi'=p$ is monotone and $1/\eps$-Lipschitz.
Furthermore, we see that $q\mapsto-rq+\eps q^2/2$ is $\eps$-strongly convex on $[\delta,\infty)$ with minimizer $p(r)$ and minimum value $-\phi(r)$. For $r\in\mathbb R$ and $q\geq\delta$, we have a simple lower bound for the scalar Fenchel--Young gap
\begin{equation}\label{eq:FY}
	\Upsilon(r,q):=\phi(r)-rq+\tfrac\eps2q^2\geq\tfrac\eps2|q-p(r)|^2,
\end{equation}
and $\Upsilon(r,q)=0$ exactly when $q=p(r)$.

Now, we define
\begin{equation}\label{eq:G}
	\begin{aligned}
		\G(u,v)&:=\int_{X\times Y}\phi(u+v-C)\dd x\dd y
		+\frac12\int_X f(1-u)^2\dd x+\frac12\int_Yg(1-v)^2\dd y,\\
		\Lambda&:=\frac12\int_Xf\dd x+\frac12\int_Yg\dd y.	
	\end{aligned}
\end{equation}

The following identity underlies both duality and the a posteriori error bound.
\begin{lemma}\label{lem:gap-decomp}
For every $k\in\Zz$ with $k\geq\delta$ and every $z=(u,v)\in\Hh$,
\begin{equation}\label{eq:gap-decomp}
 \Pp(k)+\G(z)-\Lambda
 =\int_{X\times Y}\Upsilon(\T z-C,k)\dd x\dd y
 +\frac12\int_X\frac{|k_X+f(u-1)|^2}{f}\dd x
 +\frac12\int_Y\frac{|k_Y+g(v-1)|^2}{g}\dd y.
\end{equation}
In particular, $\Pp(k)+\G(z)\geq\Lambda$.
\end{lemma}
\begin{proof}
	By the definitions of $\Pp$ in \eqref{eq:primal}, $\G,\,\Lambda$ in \eqref{eq:G}, and $\T$ in \eqref{eq:optimal-slack}, we extend
	\begin{align*}
		\Pp(k)+\G(z)-\Lambda &= \int_{X\times Y}\Big[Ck+\frac\eps2k^2+\phi(\T z-C)-(\T z)k\Big]\dd x\dd y\\
		&\quad+ \int_{X}\Big[uk_X+\frac{(k_X-f)^2}{2f}+\frac{f(1-u)^2}2-\frac f2\Big]\dd x\\
		&\quad+\int_{Y}\Big[vk_Y+\frac{(k_Y-g)^2}{2g}+\frac{g(1-v)^2}2-\frac g2\Big]\dd y.
	\end{align*}
Then, we compute
\[
Ck+\frac\eps2k^2+\phi(\T z-C)-(\T z)k
=\phi(\T z-C)-(\T z-C)k+\frac\eps2k^2
=\Upsilon(\T z-C,k).
\]
Moreover, we compute
\[
 uk_X+\frac{(k_X-f)^2}{2f}+\frac{f(1-u)^2}2-\frac f2
 =\frac{k_X^2}{2f}+(u-1)k_X+\frac{f(u-1)^2}2
 =\frac{|k_X+f(u-1)|^2}{2f}.
\]
Similarly, we also have
\[
 vk_Y+\frac{(k_Y-g)^2}{2g}+\frac{g(1-v)^2}2-\frac g2
 =\frac{|k_Y+g(v-1)|^2}{2g}.
\]
We combine all the identities and obtain \eqref{eq:gap-decomp}.
\end{proof}

\begin{proposition}\label{prop:duality}
Under \eqref{eq:data}, the primal minimizer $k^*\in\Zz$ and the dual minimizer $z^*=(u^*,v^*)\in\Hh$ exist and are unique.
They satisfy
\begin{align}
 k^*&=p(\T z^*-C),& k_X^*&=f(1-u^*),&k_Y^*&=g(1-v^*),\label{eq:optimality}\\
 \min_{k\geq\delta}\Pp(k)&=\Lambda-\min_{z\in\Hh}\G(z).\label{eq:duality}
\end{align}
The functional $\G$ is differentiable on $\Hh$ with gradient $F$ from
\eqref{eq:residual-intro}, and $F(z^*)=0$. The minimizer $z^*$ belongs to the set
\begin{equation}\label{eq:caps}
\mathcal A_\delta=\{u\leq1-\delta b/f,\quad v\leq1-\delta a/g\},
\end{equation}
which we call the \emph{cap set}. In addition,
\begin{equation}\label{eq:bounded-solutions}
 \delta\leq k^*\leq M:=\max\{\delta,2/\eps\},\quad
 0\leq1-u^*\leq bM/c,\quad 0\leq1-v^*\leq aM/c.
\end{equation}
Moreover, independently of $\eps$ and $\delta$,
\begin{equation}\label{eq:fc-box}
 -1<u^*\le1,\qquad -1<v^*\le1\qquad\text{almost everywhere}.
\end{equation}
In particular, with
\begin{equation}\label{eq:theta-star}
 \theta_X^*=\min\{2,bM/c\},\qquad\theta_Y^*=\min\{2,aM/c\},
\end{equation}
we have $0\leq1-u^*\leq\theta_X^*$ and $0\leq1-v^*\leq\theta_Y^*$.
\end{proposition}
\begin{proof}
	First, we note that the primal functional is continuous and $\eps$-strongly convex on $\Zz$. We also note that $\Pp(k)\geq\eps\norm{k}_{\Zz}^2/2$ on the closed convex feasible set because $C\geq0$ and $k\geq\delta\geq0$.
The direct method of the calculus of variations gives a unique minimizer $k^*$, characterized by
\begin{equation}\label{eq:VI}
	\ip{A(k^*)}{\ell-k^*}_{\Zz}\geq0\quad(\ell\geq\delta),\qquad
 A(k)=C+\eps k+k_X/f+k_Y/g-2.
\end{equation}
We now define $u^*=1-k_X^*/f$ and $v^*=1-k_Y^*/g$. Then, $z^*=(u^*,v^*)\in\Hh$ by \eqref{eq:Tnorm} and $f,g\geq c$.

We have $A(k^*)=\eps k^*-r^*$ with $r^*=\T z^*-C$. For a measurable set $S$ of finite measure and $q\geq\delta$, the test function $\ell=k^*+(q-k^*)\ind_S$ is admissible in \eqref{eq:VI}, which gives
\[
\int_S(\eps k^*-r^*)(q-k^*)\dd x\dd y\geq0.
\]
Since $S$ is arbitrary, $(\eps k^*-r^*)(q-k^*)\geq0$ almost everywhere for every $q\geq\delta$. Hence,
\[k^*=\max\{\delta,r^*/\eps\}=p(\T z^*-C).\]
Thus, we get the identities \eqref{eq:optimality}, which give the equation $F(z^*)=0$.

Now, at the pair $(k^*,z^*)$, we obtain all the identities in \eqref{eq:duality}. By \eqref{eq:gap-decomp}, we further obtain $\Pp(k^*)+\G(z^*)=\Lambda$. Also by \cref{lem:gap-decomp}, $\G(z)\geq\Lambda-\Pp(k^*)$ for all $z\in\Hh$.
Hence, we conclude that $z^*$ minimizes $\G$ and \eqref{eq:duality} holds.
By its definition in \eqref{eq:G}, $\G$ is strictly convex and thus $z^*$ is its unique minimizer.
We will further quantify the convexity of $\G$ in \cref{prop:global-constants}.

Since $\phi'=p$ is $1/\eps$-Lipschitz, we have
\[
	|\phi(r+s)-\phi(r)-p(r)s|=\Bigg|\int_{r}^{r+s}[p(r')-p(r)]\dd r'\Bigg|\leq s^2/(2\eps).
\]
Now, we consider the case $r=\T z-C$. By the definitions of $\G$ in \eqref{eq:G}, $F$ in \eqref{eq:residual-intro}, and $p$ in \eqref{eq:p}, for $e=(e_u,e_v)\in\Hh$, we observe that
\[
\G(z+e)-\G(z)-\ip{F(z)}{e}_{\Hh}
=\int_{X\times Y}\big[\phi(r+\T e)-\phi(r)-p(r)\T e\big]\dd x\dd y
+\frac12\int_Xfe_u^2\dd x+\frac12\int_Yge_v^2\dd y .
\]
By \eqref{eq:Tnorm}, and $f,g\leq E$, we estimate
\[
\big|\G(z+e)-\G(z)-\ip{F(z)}{e}_{\Hh}\big|
\leq \int_{X\times Y}\frac{(\T e)^2}{2\varepsilon}\dd x\dd y+\frac E2\norm{e}_{\Hh}^2
\leq\Big(\frac{a+b}{2\eps}+\frac E2\Big)\norm{e}_{\Hh}^2 .
\]
Hence $\G$ is Fr\'echet differentiable on $\Hh$ with gradient $\nabla\G=F$.

Since $k^*\geq\delta$, the identities \eqref{eq:optimality} give $(u^*,v^*)$ in the cap set \eqref{eq:caps}, which leads to $u^*,v^*\leq1$. As $C\geq0$, it follows that
\begin{equation}\label{eq:reconstruction}k^*=\max\{\delta,(u^*+v^*-C)/\eps\}\leq M.\end{equation}
Moreover, $1-u^*=k_X^*/f\in[0,bM/c]$, and similarly for $v^*$.

Finally, we prove \eqref{eq:fc-box}.
Suppose that $u^*(x)\leq-1$ at a point $x$ for which the identities \eqref{eq:optimality} hold on the section $\{x\}\times Y$.
Since $v^*\leq1$ and $C\geq0$, we have $u^*(x)+v^*(y)-C(x,y)\leq0$ for almost every $y$.
By \eqref{eq:reconstruction}, we get $k^*(x,y)=\delta$ for almost every $y$. Hence, $k_X^*(x)=\delta b$.
The identities in \eqref{eq:optimality} now yields
\[
u^*(x)=1-\frac{\delta b}{f(x)}>0,
\]
since $\delta b<c\leq f(x)$ by \eqref{eq:data}.
This contradicts $u^*(x)\leq-1$.
Since almost every $x$ is of this kind, the set $\{u^*\leq-1\}$ is a null set.
The same argument applies to $v^*$.
\end{proof}

\begin{remark}\label{rem:regularized-balanced}
	The optimality system \eqref{eq:optimality} is closely related to the Tikhonov-regularized dual that Lorenz, Manns, and Meyer \cite[Section~2.3]{lorenz2021quadratic} introduced for balanced quadratically regularized transport, whose dual solutions may be non-unique beyond constant shifts. Their regularization adds $\frac\varsigma2(\norm{\alpha_1}_{L^2(X)}^2+\norm{\alpha_2}_{L^2(Y)}^2)$ to the dual objective in the normalization of \cite[(D)]{lorenz2021quadratic}, in which the linear terms carry their regularization parameter as a factor. With this regularization parameter identified with $\eps$ and marginals $\mu_1,\mu_2$, the shift
$$u=\alpha_1+1-\eps\mu_1/\varsigma,\qquad v=\alpha_2+1-\eps\mu_2/\varsigma$$
transforms its optimality system into \eqref{eq:residual-intro} with $\delta=0$, constant densities $f=g=\varsigma/\eps$, and cost $C+2-\eps(\mu_1(x)+\mu_2(y))/\varsigma$, which is nonnegative when, for instance, $\eps\max\{\mu_1,\mu_2\}\le\varsigma$. Conversely, the Pearson dual with $\delta=0$ is their regularized dual with $\mu_1=f$, $\mu_2=g$, and the constant weight $\varsigma$ replaced by $\eps f$ and $\eps g$. Consequently, our results in \cref{sec:local,sec:mesh} also apply to the semismooth Newton method for their regularized dual whenever the shifted cost is nonnegative.
\end{remark}

\begin{proposition}\label{prop:global-constants}
The functional $\G$ is $c$-strongly convex and has an $L$-Lipschitz gradient, where
\begin{equation}\label{eq:L}
 L=E+(a+b)/\eps.
\end{equation}
In particular, $c\norm{z-z^*}_{\Hh}\leq\norm{F(z)}_{\Hh}$.
\end{proposition}
\begin{proof}
	For $z,\,\widehat{z},\,e=(e_u,e_v)=z-\widehat z\in \Hh$, we compute
\[
F(z)-F(\widehat z)
=\T^*[p(\T z-C)-p(\T\widehat z-C)]+(fe_u,ge_v).
\]
We pair $F(z)-F(\widehat z)$ with $e$ to obtain
\[
	\ip{F(z)-F(\widehat z)}{e}_{\Hh}
	=\ip{p(\T z-C)-p(\T\widehat z-C)}{\T e}_{\Zz}
+\int_Xfe_u^2\dd x+\int_Yge_v^2\dd y
\geq c\norm{e}_{\Hh}^2,
\]
by monotonicity of $p$ and $f,g\geq c$.
Since $F=\nabla\G$, this proves that $\G$ is $c$-strongly convex.

By \eqref{eq:Tnorm}, $f,g\leq E$, and the fact that $p$ is $1/\varepsilon$-Lipschitz, we also have
\[
\norm{F(z)-F(\widehat z)}_{\Hh}
\leq [(a+b)/\eps+E]\norm{e}_{\Hh}.
\]
Finally, we take $\widehat z=z^*$ in the estimate and note that $F(z^*)=0$. We then obtain the estimate $c\norm{z-z^*}_{\Hh}\leq\norm{F(z)}_{\Hh}$ using the Cauchy--Schwarz inequality.
\end{proof}

\subsection{An a posteriori error bound for the reconstructed plan}
For an arbitrary potential $z=(u,v)\in\Hh$, we denote its reconstructed plan $k_z$ and the duality gap of the pair $(k_z,z)$ by
\begin{equation}\label{eq:Gap}k_z:=p(\T z-C),\quad \Gap(z):=\Pp(k_z)+\G(z)-\Lambda.\end{equation}
By \eqref{eq:duality}, we observe that
\[
\Gap(z)=\big(\Pp(k_z)-\min\Pp\big)+\big(\G(z)-\min\G\big)\geq0,
\]
with $\Gap(z^*)=0$.
\begin{proposition}\label{prop:aposteriori}
The gap has the residual-square representation
\begin{equation}\label{eq:gap-squares}
	\begin{aligned}
		\Gap(z)&=\frac12\int_X\frac{|k_{z,X}+f(u-1)|^2}{f}\dd x
		+\frac12\int_Y\frac{|k_{z,Y}+g(v-1)|^2}{g}\dd y\\
		&= \frac12\int_X\frac{F_u(z)^2}{f}\dd x
		+\frac12\int_Y\frac{F_v(z)^2}{g}\dd y,
	\end{aligned}
\end{equation}
where $(F_u,F_v)=F$ in \eqref{eq:residual-intro}.
Consequently,
\begin{align}
 \frac1{2E}\norm{F(z)}_{\Hh}^2
 &\leq\Gap(z)\leq\frac1{2c}\norm{F(z)}_{\Hh}^2,\label{eq:gap-residual}\\
 \norm{k_z-k^*}_{\Zz}&\leq\sqrt{\frac{\Gap(z)}{2\eps}}.\label{eq:sharp}
\end{align}
\end{proposition}
\begin{proof}
	We first prove \eqref{eq:gap-squares}. By the definition of $k_z$ in \eqref{eq:Gap} and $\Upsilon$ in \eqref{eq:FY}, we have \[\Upsilon(\T z-C,k_z)=0.\] 
	By definitions of $F$ in \eqref{eq:residual-intro}, $p$ in \eqref{eq:p}, and $k_z$ in \eqref{eq:Gap}, we also see that \begin{equation}\label{eq:residual-compute}(F_u,F_v)=F=\T^*k_z+(f(u-1),g(v-1))=(k_{z,X}+f(u-z),k_{z,Y}+g(v-1)).\end{equation}
	Hence, \eqref{eq:gap-decomp} with $k=k_z$ gives \eqref{eq:gap-squares}.
Since $c\leq f,g\leq E$, we obtain
\[
 \frac1{2E}\norm{F(z)}_{\Hh}^2
 \le\frac12\int_X\frac{F_u(z)^2}{f}\dd x+\frac12\int_Y\frac{F_v(z)^2}{g}\dd y
 \le\frac1{2c}\norm{F(z)}_{\Hh}^2,
\]
which is \eqref{eq:gap-residual}.

We now turn to the plan estimate \eqref{eq:sharp}. Let $d=k_z-k^*$ and $e=(e_u,e_v)=z-z^*$. By \eqref{eq:optimality}, we have $k^*_X=f(1-u^*)$, and hence
\[
 k_{z,X}+f(u-1)=d_X+k^*_X+f(u-1)=d_X+fe_u.
\]
Similarly, we have $k_{z,Y}+g(v-1)=d_Y+ge_v$. We expand
\[
	\frac{(d_X+fe_u)^2}{f}=\frac{(d_X-fe_u)^2}{f}+4e_ud_X,\quad\frac{(d_Y+ge_v)^2}{g}=\frac{(d_Y-ge_v)^2}{g}+4e_vd_Y.
\]
Since $\T^*d=(d_X,d_Y)$, we have \[\int_Xe_ud_X\dd x+\int_Ye_vd_Y\dd y=\ip{d}{\T e}_{\Zz}.\] Substituting these identities into \eqref{eq:gap-squares}, we obtain the exact identity
\begin{equation}\label{eq:reflected-squares}
 \Gap(z)=\frac12\int_X\frac{|d_X-fe_u|^2}{f}\dd x
 +\frac12\int_Y\frac{|d_Y-ge_v|^2}{g}\dd y+2\ip{d}{\T e}_{\Zz}.
\end{equation}
Next, let $r,r^*\in\mathbb R$. Since $p$ is nondecreasing and is $1/\eps$-Lipschitz, we obtain
\[
 (p(r)-p(r^*))(r-r^*)=|p(r)-p(r^*)|\,|r-r^*|\geq\eps|p(r)-p(r^*)|^2.
\]
We apply this inequality pointwise with $r=\T z-C$ and $r^*=\T z^*-C$. We note that $p(r)-p(r^*)=d$ and $r-r^*=\T e$, which leads to
\begin{equation}\label{eq:cocoercivity}
 \ip{d}{\T e}_{\Zz}\geq\eps\norm{d}_{\Zz}^2.
\end{equation}
Finally, by \eqref{eq:reflected-squares} and \eqref{eq:cocoercivity}, we obtain
\[
 \Gap(z)\geq2\ip{d}{\T e}_{\Zz}\geq2\eps\norm{d}_{\Zz}^2,
\]
which is \eqref{eq:sharp}.
\end{proof}
\Cref{S-sec:aposteriori-limits} of the supplementary materials shows that the coefficient $1/\sqrt{2\eps}$ in \eqref{eq:sharp} is attained by admissible constant data. It also shows that the bound is reliable but not efficient with respect to the plan error (\cref{S-rem:no-plan-efficiency}).
\subsection{Exact cell reduction}\label{sec:discretization}
We now pass from the continuum to cell arrays.
Cellwise data give finite problems whose solutions are exactly the continuum solutions for those data. We prove this exact reduction in \cref{prop:reduction}.
An $L^2$ stability estimate in \cref{prop:stability} compares them with the solution of the original problem \eqref{eq:primal}. Together with \cref{prop:aposteriori}, it yields the a posteriori error bound of \cref{cor:total} with respect to this problem.

Let $\{X_i\}_{i=1}^{N_X}$ and $\{Y_j\}_{j=1}^{N_Y}$ be finite measurable partitions, with positive cell volumes $w_i=|X_i|$ and $\omega_j=|Y_j|$.
Throughout the paper, we set $\N_X:=\{1,\dots,N_X\},\,\N_Y:=\{1,\dots,N_Y\},\,\N:=\N_X\times\N_Y$ and the mesh size
\begin{equation}\label{eq:h}
	h:=\max_{(i,j)\in \N}\operatorname{diam}(X_i\times Y_j).
\end{equation}
For example, if $X=Y=[0,1]$ are partitioned into $n$ equal intervals, then $h=\sqrt2/n$. For $n\times n$ square partitions of $X=Y=[0,1]^2$, we have $h=2/n$.

For an array $K=(K_{ij})_{(i,j)\in\N}\in \mathbb R^{N_XN_Y}$, we define $(K_X)_i:=\sum_{j=1}^{N_Y}\omega_jK_{ij}$ and $(K_Y)_j:=\sum_{i=1}^{N_X}w_iK_{ij}$.
The \emph{lift} of an array $K$ is defined by the piecewise function
\begin{equation}\label{eq:lift-K}
	\Lsf K(x,y)=K_{i,j}\quad\text{if }(x,y)\in X_i\times Y_j,\,(i,j)\in\N.
\end{equation}
Similarly, lifts for array potentials $u=(u_i)_{i\in\N_X}\in\mathbb R^{N_X},\, v =(v_j)_{j\in\N_Y}\in\mathbb R^{N_Y}$ and $z=(u,v)$ are given by
\begin{equation}	
	\label{eq:lift-z}
	\begin{aligned}
		\Lsf^X u(x)&=u_i\text{ if }x\in X_i,\,i=1,\dots,N_X\\
		\Lsf^Y v(y)&=v_j\text{ if }y\in Y_j,\,j=1,\dots,N_Y\\
		\Lsf z&=(\Lsf^Xu,\Lsf^Y v).
	\end{aligned}
\end{equation}
We also denote by $\Lsf^{-1}$ the inverse of the lift of $K$, which maps a piecewise constant function $k$ over the given partition $X_i\times Y_j$ to an array with entry $(i,j)$ is the value of $k$ on cell $X_i\times Y_j$. In a similar way, we also define $(\Lsf^X)^{-1},\, (\Lsf^Y)^{-1}$. For convenience, we also denote by $\Lsf^{-1}$ the inverse of $\Lsf$ of $z$.

The norms are given by
\begin{equation}\label{eq:weighted-norm}
	\norm{K}_{2,h}^2=\sum_{(i,j)\in\N}w_i\omega_jK_{ij}^2=\norm{\Lsf K}_{\Zz}^2,\qquad
	\norm{z}_h^2=\sum_{i=1}^{N_X}w_i u_i^2+\sum_{j=1}^{N_Y}\omega_jv_j^2=\norm{\Lsf z}_{\Hh}^2,
\end{equation}
We call the inner products associated with \eqref{eq:weighted-norm} the \emph{weighted inner products}, that is,
\begin{equation}
	\label{eq:weighted-inner-prod}
	\ip{K}{\widetilde K}_{2,h}=\sum_{(i,j)\in\N}w_i\omega_jK_{ij}\widetilde K_{ij}=\ip{\Lsf K}{\Lsf \widetilde{K}}_{\Zz},\qquad
	\ip{z}{\widetilde z}_h=\sum_{i=1}^{N_X}w_iu_i\widetilde u_i+\sum_{j=1}^{N_Y}\omega_jv_j\widetilde v_j=\ip{\Lsf z}{\Lsf \widetilde{z}}_{\Hh}.
\end{equation}
We also denote the usual Euclidean norm for arrays by
\begin{align*}
	\norm{K}_2^2=\sum_{(i,j)\in\N}K_{ij}^2,\qquad \norm{z}^2_2=\norm{u}_2^2+\norm{v}_2^2=\sum_{i=1}^{N_X}u_i^2+\sum_{j=1}^{N_Y}v_j^2.
\end{align*}

Let $\Q_h$ denote averaging over product cells $X_i\times Y_j$, that is,
\begin{equation}\label{eq:aver}(\Q_hk)|_{X_i\times Y_j}:=(w_i\omega_j)^{-1}\int_{X_i\times Y_j}k\dd x\dd y.\end{equation}
In a similar way, we denote $\Q_h^X$ and $\Q_h^Y$ the averaging over the cells $X_i$ and $Y_j$. We further denote $\Q_h^{\Hh}z=(\Q_h^Xu,\Q_h^Yv)$.
We have the following orthogonality.
\begin{lemma}	
	\label{lem:aver-ortho}
	All three $\Q_h,\,\Q_h^X,\,\Q_h^Y$ are orthogonal projections in $L^2$ onto the functions that are constant on each cell. 
\end{lemma}
\begin{proof}
	We show this for the case of $\Q_h$ and $\Q_h^X,\Q_h^Y$ are similar. We first have $\Q_hk$ is cellwise constant, and \[\int_{X_i\times Y_j}(k-\Q_hk)\dd x\dd y=0,\quad\text{for all }(i,j)\in \N.\]
Thus, $k-\Q_hk$ is orthogonal to every cellwise constant function.
\end{proof}

Let $f_h,g_h,C_h$ be cellwise data satisfying \eqref{eq:data}, with cell values $f_i,g_j,C_{ij}$. Evaluating the continuum functionals at lifts gives the finite objectives
\begin{align}
	\Pp_h(K):={}&\sum_{(i,j)\in\N}w_i\omega_j(C_{ij}K_{ij}+\eps K_{ij}^2/2)\notag\\
	&+\frac12\sum_{i=1}^{N_X}w_i\frac{((K_X)_i-f_i)^2}{f_i}
	+\frac12\sum_{j=1}^{N_Y}\omega_j\frac{((K_Y)_j-g_j)^2}{g_j},
 \qquad K_{ij}\geq\delta,\label{eq:finite-primal}\\
 \G_h(z):={}&\sum_{(i,j)\in\N}w_i\omega_j\phi(u_i+v_j-C_{ij})
 +\frac12\sum_{i=1}^{N_X}w_if_i(1-u_i)^2+\frac12\sum_{j=1}^{N_Y}\omega_jg_j(1-v_j)^2,
 \label{eq:finite-dual}
\end{align}
and the discrete duality gap
\[\Gap_h(z):=\Pp_h(K(z))+\G_h(z)-\Lambda_h,\quad \text{where }\Lambda_h:=\frac12\sum_{i=1}^{N_X}w_if_i+\frac12\sum_{j=1}^{N_Y}\omega_jg_j,\]
and
\begin{equation}\label{eq:finite-reconstruct}
	K_{ij}(z)=p(u_i+v_j-C_{ij}),\quad\text{for }(i,j)\in\N.
\end{equation}
We write $\widehat\Pp_h,\,\widehat{F}_h$ and $\widehat\G_h$ for the continuum functionals \eqref{eq:primal}, \eqref{eq:residual-intro}, and \eqref{eq:G} with data $f_h,g_h,C_h$. 

\begin{proposition}\label{prop:reduction}
The minimizers of $\widehat\Pp_h$ and $\widehat\G_h$ are cellwise constant, and they are the lifts of the minimizers of $\Pp_h$ and $\G_h$.
For every array potential $z=(u,v)=( (u_i)_{i\in\N_X},(v_j)_{j\in\N_Y})$, and $K(z)$ in \eqref{eq:finite-reconstruct}, we have
\[\Lsf K(z) = k_{\Lsf z},\]
where $\Lsf K(z),\,\Lsf z$ are from \eqref{eq:lift-K},\eqref{eq:lift-z} and $k_{\Lsf z}$ is from \eqref{eq:Gap}.
Moreover, the gradient of $\G_h$ at $z$ in the weighted inner product is
\begin{equation}\label{eq:Fh}
	F_h(z)=(F_{h,u}(z),F_{h,v}(z))=\Big(K_X+(f_i(u_i-1))_{i\in\N_X},\ K_Y+(g_j(v_j-1))_{j\in\N_Y}\Big).
\end{equation}

Finally,
\begin{equation}\label{eq:finite-gap}
	\Gap_h(z)=\Gap(\Lsf z)=\frac12\sum_{i=1}^{N_X}w_i\frac{F_{h,u,i}(z)^2}{f_i}
	+\frac12\sum_{j=1}^{N_Y}\omega_j\frac{F_{h,v,j}(z)^2}{g_j},
\end{equation}
where $\Gap(\Lsf z)$ is from \eqref{eq:Gap}.
\end{proposition}
\begin{proof}
	We first apply \cref{prop:duality} to the data $f_h,g_h,C_h$, which shows that both continuum problems $\widehat{\Pp}_h,\,\widehat{\G}_h$ have unique minimizers.
Next, we note that if $k\geq\delta$, then $\Q_hk\geq\delta$, since $\Q_hk$ is the cellwise average of $k$.
By Fubini's theorem, we obtain the identities \[(\Q_h k)_X=\Q_h^Xk_X\quad \text{and}\quad (\Q_h k)_Y=\Q_h^Yk_Y.\]
Since $C_h$ is constant on product cells, we also obtain \[\int_{X\times Y} C_h\Q_hk\dd x\dd y=\int_{X\times Y} C_hk\dd x\dd y.\]

Since $\Q_h^Xk_X-f_h$ and $1/f_h$ are constant on each cell $X_i$, by \cref{lem:aver-ortho}, we have
\[
\int_X\frac{(\Q_h^Xk_X-f_h)(k_X-\Q_h^Xk_X)}{f_h}\dd x=0.
\]
Hence, we obtain
\[
\int_X\frac{(k_X-f_h)^2}{f_h}\dd x=\int_X\frac{(\Q_h^Xk_X-f_h)^2}{f_h}\dd x+\int_X\frac{(k_X-\Q_h^Xk_X)^2}{f_h}\dd x,
\]
and the same identity holds on $Y$ with $f_h$ replaced by $g_h$.
We combine these identities with $\norm{k}_2^2=\norm{\Q_hk}_2^2+\norm{k-\Q_hk}_2^2$, and derive that
\begin{align}\label{eq:pythagoras}
 \widehat\Pp_h(k)-\widehat\Pp_h(\Q_hk)
 ={}\frac\eps2\norm{k-\Q_hk}_2^2
 +\frac12\int_X\frac{|k_X-\Q_h^Xk_X|^2}{f_h}\dd x
 +\frac12\int_Y\frac{|k_Y-\Q_h^Yk_Y|^2}{g_h}\dd y.
\end{align}
It follows that $\widehat\Pp_h(\Q_hk)<\widehat\Pp_h(k)$ unless $k=\Q_hk$, and therefore the minimizer of $\widehat\Pp_h$ is cellwise constant.
Since $\widehat\Pp_h$ coincides with $\Pp_h$ on cellwise constant functions, the minimizer of $\widehat\Pp_h$ coincides with the lift of the minimizer of $\Pp_h$.

We also observe that, on each cell $X_i\times,Y_j,\,(i,j)\in\N$,
\[\Lsf K(z)=K(z)=p(u_i+v_j-C_{ij})=p(\Lsf^X u+\Lsf^Y v-C_h)=p(\T\Lsf z-C_h)=k_{\Lsf z}.\]
Furthermore, by \eqref{eq:residual-compute}, 
\begin{equation}\label{eq:lift-F}
	\begin{aligned}
		\nabla\widehat{\G}_h(\Lsf z)=\widehat{F}_h(\Lsf z)&=\T^*k_{\Lsf z}+\big(f_h(\Lsf^Xu-1),\ g_h(\Lsf^Yv-1)\big)\\
		&=\T^*\Lsf K(z)+\Lsf \Big((f_i(u_i-1))_{i\in\N_X},\ (g_j(v_j-1))_{j\in\N_Y}\Big)\\
		&=\Lsf \Big(K_X+(f_i(u_i-1))_{i\in\N_X},\ K_Y+(g_j(v_j-1))_{j\in\N_Y}\Big)=\Lsf F_h(z).	
	\end{aligned}
\end{equation}
Recall that we also have \eqref{eq:weighted-inner-prod}, we conclude that $F_h$ is also the gradient of $\G_h$ in the weighted inner product.

It remains to compare the gaps. We note that $\widehat\Pp_h(\Lsf K(z))=\Pp_h(K(z))$ and $\widehat\G_h(\Lsf z)=\G_h(z)$, and that $\Lambda_h$ is $\Lambda$ for the data $f_h,g_h$. Hence, $\Gap(\Lsf z)=\Gap_h(z)$, and \eqref{eq:gap-squares} then gives \eqref{eq:finite-gap}.
\end{proof}

For cellwise data we use the constants
\begin{equation}\label{eq:finite-constants}
	\mu_h:=\min\{\min_{i\in\N_X} f_i,\min_{j\in\N_Y}g_j\},\quad
	E_h^{\max}:=\max\{\max_{i\in\N_X} f_i,\max_{j\in\N_Y}g_j\},\quad
 L_h:=E_h^{\max}+(a+b)/\eps.
\end{equation}
Then $\mu_h\geq c$ and $L_h\leq L$, independently of the number and size of
cells, and $\G_h$ is $\mu_h$-strongly convex with $L_h$-Lipschitz gradient
in the weighted inner product
(by \cref{prop:global-constants} applied to the cellwise data).
\subsection{Data stability and the total a posteriori error bound}\label{sec:stability}
Let $k^*,z^*=(u^*,v^*)$ be the minimizers of $\Pp$ in \eqref{eq:primal} and $\G$ in \eqref{eq:G} with the data $f,\,g,\,C$ in \eqref{eq:data}. Let $k_h^*,z_h^*=(u_h^*,v_h^*)$ be the minimizers of $\widehat{\Pp}_h,\widehat{\G}_h$ with cellwise-data as in \cref{sec:discretization}. 
We write $A_h$ for the operator $A$ of \eqref{eq:VI} with data $f_h,g_h,C_h$.
\begin{proposition}\label{prop:stability}
	Let $\theta_X\geq\norm{1-u^*}_{L^\infty(X)}$ and $\theta_Y\geq\norm{1-v^*}_{L^\infty(Y)}$ be known upper bounds. Set
\[
	\Delta_{\theta,h}=\norm{(\theta_X(f-f_h),\theta_Y(g-g_h))}_{\Hh},\quad \beta_h=\frac{\Delta_{\theta,h}}{\sqrt c},\quad
	J_h=\frac{\beta_h}2+\Big(\frac{\beta_h^2}4+\frac{\norm{C-C_h}_\Zz^2}{4\eps}\Big)^{1/2},
\]
and define
\[
	D_h=\frac1{2\eps}\Big[\norm{C-C_h}_\Zz+\Big(\norm{C-C_h}_\Zz^2+\frac{\eps}{c}\Delta_{\theta,h}^2\Big)^{1/2}\Big].
\]
Then
\begin{equation}\label{eq:stability}
 \norm{k_h^*-k^*}_\Zz\leq D_h,\qquad
 \norm{z_h^*-z^*}_{\Hh}\leq D_{z,h},
\end{equation}
where
\begin{equation}\label{eq:Dz}
	D_{z,h}=\frac1{\sqrt c}\Bigg(\beta_h+\min\Big\{J_h,\ \sqrt{\tfrac{a+b}{c}}\,D_h\Big\}\Bigg).
\end{equation}
Moreover, the marginals $d_X,\, d_Y$ of $d=k_h^*-k^*$ satisfy
\begin{equation}\label{eq:marginal-stability}
 \Big(\int_X\frac{|d_X|^2}{f_h}\dd x+\int_Y\frac{|d_Y|^2}{g_h}\dd y\Big)^{1/2}\le J_h.
\end{equation}
Particularly, The choice $\theta_X=\theta_Y=2$ recovers the constant $D_h$ as in \eqref{eq:Dh-intro}. By \cref{prop:duality}, one may always take $\theta_X=\theta_X^*$ and $\theta_Y=\theta_Y^*$ from \eqref{eq:theta-star}.
\end{proposition}
\begin{proof}	
We first compare the two plans $k_h^*$ and $k^*$. We denote
\[
 \mathcal J(d)=\mathcal J_X+\mathcal J_Y,\qquad
 \mathcal J_X=\int_X\frac{|d_X|^2}{f_h}\dd x,\qquad
 \mathcal J_Y=\int_Y\frac{|d_Y|^2}{g_h}\dd y.
\]
The variational inequality \eqref{eq:VI} for $A$ with data $f,\,g,\,C$, tested with $\ell=k_h^*$, gives $\ip{A(k^*)}{d}_{\Zz}\ge0$. The variational inequality \eqref{eq:VI} for $A_h$ with cellwise-data $f_h,\,g_h,\,C_h$, tested with $\ell=k^*$, gives $\ip{A_h(k_h^*)}{-d}_{\Zz}\ge0$. Adding these two inequalities, we obtain
\[
	\ip{A_h(k_h^*)-A_h(k^*)}{d}_{\Zz}\le\ip{A(k^*)-A_h(k^*)}{d}_{\Zz}.
\]
With the formula of $A_h$ in \eqref{eq:VI} with cellwise-data $f_h,\,g_h,\,C_h$, we also observe that
\[
		\ip{A_h(k_h^*)-A_h(k^*)}{d}_{\Zz}= \ip{\varepsilon d+d_X/f_h+d_Y/g_h}{d}_{\Zz}
		=\eps\norm{d}_\Zz^2+\int_X\frac{d_X^2}{f_h}\dd x+\int_Y\frac{d_Y^2}{g_h}\dd y
		=\eps\norm{d}_\Zz^2+\mathcal J(d).
\]
Thus, we obtain
\begin{equation}\label{eq:stability-core}
 \eps\norm{d}_\Zz^2+\mathcal J(d)
 \leq\ip{A(k^*)-A_h(k^*)}{d}_{\Zz}.
\end{equation}

Now, we compute
\[
 A(k^*)-A_h(k^*)=C-C_h+k_X^*\Big(\frac1f-\frac1{f_h}\Big)+k_Y^*\Big(\frac1g-\frac1{g_h}\Big).
\]
By the Cauchy--Schwarz inequality, we first have \[\ip{C-C_h}{d}_{\Zz}\le\norm{C-C_h}_\Zz\norm{d}_\Zz.\]
By \eqref{eq:optimality}, the Cauchy--schwarz inequality, $c\le f,\,g$, and the bounds $\theta_X\ge\norm{1-u^*}_{L^\infty(X)},\,\theta_Y\ge\norm{1-v^*}_{L^\infty(Y)}$, we also have
\[
	\begin{aligned}
		\ip{k_X^*\Big(\frac1f-\frac1{f_h}\Big)}{d}_{\Zz}=\int_{X}(1-u^*)\frac{f_h-f}{f_h}d_X\dd x\le\frac{\theta_X}{\sqrt c}\norm{f-f_h}_{L^2(X)}\sqrt{\mathcal J_X},\\
		\ip{k_Y^*\Big(\frac1g-\frac1{g_h}\Big)}{d}_{\Zz}=\int_{Y}(1-v^*)\frac{g_h-g}{g_h}d_Y\dd y\le\frac{\theta_Y}{\sqrt{c}}\norm{g-g_h}_{L^2(Y)}\sqrt{\mathcal J_Y}.
	\end{aligned}
\]
By the Cauchy--Schwarz inequality in $\mathbb R^2$, we have
\[
	\theta_X\norm{f-f_h}_{L^2(X)}\sqrt{\mathcal J_X}+\theta_Y\norm{g-g_h}_{L^2(Y)}\sqrt{\mathcal J_Y}
 \le\Delta_{\theta,h}\sqrt{\mathcal J(d)}.
\]
Combining these estimates with \eqref{eq:stability-core}, we get
\begin{equation}\label{eq:energy-inequality}
 \eps\norm{d}_\Zz^2+\mathcal J(d)
 \le\norm{C-C_h}_\Zz\norm{d}_\Zz+\beta_h\sqrt{\mathcal J(d)}.
\end{equation}

Since $\mathcal J-\beta_h\sqrt{\mathcal J}\ge-\beta_h^2/4$ and $\beta_h^2=\Delta_{\theta,h}^2/c$, it follows that
\[
 \eps\norm{d}_\Zz^2-\norm{C-C_h}_\Zz\norm{d}_\Zz-\frac{\Delta_{\theta,h}^2}{4c}\le0.
\]
We solve the quadratic inequality, which yields
\[
 \norm{d}_\Zz\le\frac1{2\eps}\Big[\norm{C-C_h}_\Zz+\Big(\norm{C-C_h}_\Zz^2+\frac{\eps}{c}\Delta_{\theta,h}^2\Big)^{1/2}\Big]=D_h.
\]

Next, we prove the estimate \eqref{eq:marginal-stability}. We rewrite \eqref{eq:energy-inequality} as
\[
 \mathcal J-\beta_h\sqrt{\mathcal J}
 \le\norm{C-C_h}_\Zz\norm{d}_\Zz-\eps\norm{d}_\Zz^2
 \le\frac{\norm{C-C_h}_\Zz^2}{4\eps}.
\]
Then, we solve the quadratic equation and get that
\[
 \sqrt{\mathcal J(d)}\le\frac{\beta_h}2+\Big(\frac{\beta_h^2}4+\frac{\norm{C-C_h}_\Zz^2}{4\eps}\Big)^{1/2}=J_h.
\]

We now turn to $z_h^*-z^*$. Applying \eqref{eq:optimality} to both $k_{h,X}^*$ and $k_X^*$, and since $k_{h,X}^*=k_X^*+d_X$, we have
\[
 u_h^*-u^*=\frac{k_X^*}{f}-\frac{k_{h,X}^*}{f_h}
 =(1-u^*)\frac{f_h-f}{f_h}-\frac{d_X}{f_h}.
\]
By the bounds $\theta_X\ge\norm{1-u^*}_{L^\infty(X)}$ and $f_h\ge c$, we estimate \[\Bigg\|(1-u^*)\frac{f_h-f}{f_h}\Bigg\|_{L^2(X)}\le\frac{\theta_X\norm{f-f_h}_{L^2(X)}}{c}.\]
Moreover, by $f_h\ge c$ and \eqref{eq:Tnorm}, we have the two bounds
\[
	\Big\|\frac{d_X}{f_h}\Big\|_{L^2(X)}^2\le\frac{1}{c}\int_{X}\frac{d_X^2}{f_h}\dd x=\frac{\mathcal J_X}{c},\qquad
	\Big\|\frac{d_X}{f_h}\Big\|_{L^2(X)}\le\frac{\norm{d_X}_{L^2(X)}}{c}\le\frac{\sqrt b\,\norm{d}_\Zz}{c}.
\]
Hence,
\[
	\norm{u_h^*-u^*}_{L^2(X)}\le \frac{1}{\sqrt{c}}\Bigg(\frac{\theta_X\norm{f-f_h}_{L^2(X)}}{\sqrt{c}}+\min\Big\{\sqrt{\mathcal J_X(d)},\ \sqrt{\frac{b}{c}}\,\norm{d}_\Zz\Big\}\Bigg).
\]
Analogously, we have
\[
	\norm{v_h^*-v^*}_{L^2(Y)}\le \frac{1}{\sqrt{c}}\Bigg(\frac{\theta_Y\norm{g-g_h}_{L^2(Y)}}{\sqrt{c}}+\min\Big\{\sqrt{\mathcal J_Y(d)},\ \sqrt{\frac{a}{c}}\,\norm{d}_\Zz\Big\}\Bigg).
\]
We apply the triangle inequality in $\mathbb R^2$, which reads $\sqrt{(x+y)^2+(z+w)^2}\le \sqrt{x^2+z^2}+\sqrt{y^2+w^2}$, and obtain
\[
	\norm{z_h^*-z^*}_{\Hh}\le \frac{1}{\sqrt{c}}\Bigg(\beta_h+\min\Big\{\sqrt{\mathcal J(d)},\ \sqrt{\frac{a+b}{c}}\,\norm{d}_\Zz\Big\}\Bigg)\le D_{z,h}.\qedhere
\]
\end{proof}
For $D_h$ in \eqref{eq:Dh-intro}, using $\sqrt{x^2+y^2}\leq x+y$ for $x,y\ge0$, we observe that
\begin{equation}\label{eq:Dh-simple}
	D_h\le\frac{\norm{C-C_h}_\Zz}{\eps}+\frac{\norm{(f-f_h,g-g_h)}_\Hh}{\sqrt{\eps c}}.
\end{equation}
Hence, the bound $M=\max\{\delta,2/\varepsilon\}$ of $k^*$ does not appear in $D_h$.
Moreover, when $C_h=C$, we have $J_h=\beta_h$. Hence, the estimate \eqref{eq:marginal-stability} does not depend on $\eps$. Consequently, when $C=C_h$, the potential radius $D_{z,h}$ has the upper bound that does not depend on $\eps$, that is,
\[
	D_{z,h}\leq\frac4c\norm{(f-f_h,g-g_h)}_{\Hh}.
\]

\begin{corollary}\label{cor:total}
For every array potential $z_h$, whether or not it lies in the cap set \eqref{eq:caps} (for instance an unprojected Newton iterate),
\begin{align}
 \norm{\Lsf K(z_h)-k^*}_\Zz&\leq D_h+\sqrt{\Gap_h(z_h)/(2\eps)},\label{eq:total-plan}\\
 \norm{\Lsf z_h-z^*}_{\Hh}&\leq D_{z,h}+\norm{F_h(z_h)}_h/c.
 \label{eq:total-dual}
\end{align}
\end{corollary}
\begin{proof}
	We recall $\Gap_h(z_h)$ and $F_h(z_h)$ from \cref{prop:reduction} along with the minimizers $k_h^*,z_h^*$ of $\widehat{\Pp}_h,\,\widehat{\G}_h$. We use three estimates for this problem.
First, the plan bound \eqref{eq:sharp} of \cref{prop:aposteriori} gives
\[
	\norm{\Lsf K(z_h)-k_h^*}_\Zz=\norm{k_{\Lsf z_h}-k_h^*}_\Zz\le\sqrt{\frac{\Gap(\Lsf z_h)}{2\eps}}=\sqrt{\frac{\Gap_h(z_h)}{2\eps}}.
\]
Second, by \cref{prop:global-constants} for the cellwise-data $f_h,\,g_h,\,C_h$, with $\mu_h\ge c$ in place of $c$, we have
\[
 \norm{z_h-z_h^*}_{h}\le\frac{\norm{F_h(z_h)}_h}{\mu_h}\le\frac{\norm{F_h(z_h)}_h}{c}.
\]
Third, the stability estimate \eqref{eq:stability} of \cref{prop:stability} gives
\[
 \norm{k_h^*-k^*}_\Zz\le D_h,\qquad \norm{z_h^*-z^*}_{\Hh}\le D_{z,h}.
\]
We now combine these estimates with the triangle inequality and obtain that
\[
	\norm{\Lsf K(z_h)-k^*}_\Zz\le\norm{K(z_h)-\Lsf^{-1}k_h^*}_h+\norm{k_h^*-k^*}_\Zz
 \le\sqrt{\frac{\Gap_h(z_h)}{2\eps}}+D_h,
\]
and
\[
	\norm{\Lsf z_h-z^*}_{\Hh}\le\norm{z_h-\Lsf^{-1}z_h^*}_{h}+\norm{z_h^*-z^*}_{\Hh}
 \le\frac{\norm{F_h(z_h)}_h}{c}+D_{z,h},
\]
which is \eqref{eq:total-dual}.
\end{proof}
At fixed $\eps$, the optimal plans for a lower bound $\delta>0$ and for $\delta=0$ differ by $O(\delta)$ in $L^2$ (\cref{S-prop:lower-bound}). Adding this bias to \eqref{eq:total-plan} bounds the distance to the optimal plan for $\delta=0$.

\begin{remark}\label{rem:eps-c}
The results above hold at every fixed $\eps>0$, $c>0$ satisfying
\eqref{eq:data}, but the constants are not uniform as $\eps\downarrow0$
or $c\downarrow0$.

By \eqref{eq:Dh-simple}, at fixed $c$ and data errors, the cost error contributes $O(\eps^{-1})$ to $D_h$ and the density errors contribute $O(\eps^{-1/2})$ as $\eps\downarrow0$.
The potential radius satisfies $D_{z,h}=O(\eps^{-1/2})$, and it does not depend on $\eps$ when $C_h=C$.
Neither radius depends on the lower bound $\delta$.
The Lipschitz constant $L=E+(a+b)/\eps$ also grows like $\eps^{-1}$, and the condition number is $\kappa=L/c$.

In the local theory, if we further assume a boundary-layer bound $\nu_*(t)\leq At^\alpha$, then
\[
 \omega(r)\leq\frac{2A(2r)^\alpha}{c\eps}.
\]
Thus, for fixed $A>0$ and $\alpha>0$, a radius proportional to
$(c\eps)^{1/\alpha}$, within the range of this bound, ensures
$\omega(r_0)<1$ in \cref{thm:local}. This sufficient radius shrinks
as $\eps\downarrow0$ or $c\downarrow0$.
\end{remark}

We end this section by recording the size of the data radii.
We denote the total $L^2$ difference between the cellwise data and the original data by $\mathcal E_h=\norm{f-f_h}_{L^2(X)}+\norm{g-g_h}_{L^2(Y)}+\norm{C-C_h}_\Zz$.
For $\theta_X=\theta_Y=2$, by \eqref{eq:Dh-simple}, \eqref{eq:Dz}, and $J_h\le\beta_h+\norm{C-C_h}_\Zz/(2\sqrt\eps)$, we observe that
\[D_h+D_{z,h}\leq\Theta\,\mathcal E_h,\]
with $\Theta$ depending only on $c$ and $\eps$.
Cell averages $\Q_h C,\, \Q_h^Xf,\, \Q_h^Y g$ inherit the bounds \eqref{eq:data}, and for them $\mathcal E_h\to0$ as $h\to0$ for all bounded measurable data.
Indeed, for an $L^2(R)$ datum $w$ and $R=X\times Y,\, X$ or $Y$, we choose restrictions $w_\nu$ of continuous compactly supported functions with $\norm{w-w_\nu}_{L^2(R)}\to0$ as $\nu\to\infty$. Since $\Q_h^R$ is a contraction in $L^2$, we have
\[
	\norm{\Q_h^Rw-w}_{L^2(R)}\leq2\norm{w-w_\nu}_{L^2(R)}+\norm{\Q_hw_\nu-w_\nu}_{L^2(R)},
\]
and for fixed $\nu$ the last term tends to zero as $h\to0$ by uniform continuity.

For data with Lipschitz constants $L_f,L_g,L_C$, we consider both cell averages and cell representative values, that is, values of the data at one point of each cell. In both cases, we have $|f-f_h|\le L_fh$ pointwise, and similarly for $g$ and $C$. Integrating these bounds, we obtain
\begin{equation}\label{eq:Lipschitz-errors}
	\norm{f-f_h}_{L^2(X)}\leq\sqrt a\,L_fh,\quad
	\norm{g-g_h}_{L^2(Y)}\leq\sqrt b\,L_gh,\quad
 \norm{C-C_h}_\Zz\leq\sqrt{ab}\,L_Ch.
\end{equation}
Data in the Sobolev space $H^1$, the computability of the data errors, and discontinuous densities are discussed in \cref{S-sec:data-access}.
\section{Newton derivatives and function-space convergence}\label{sec:local}
In this section, we show the local superlinear convergence of the semismooth Newton method for the continuum equation $F(z)=0$ with fixed data.
We use two ingredients, a maximum-norm inverse estimate in \cref{lem:inverse} and a remainder estimate controlled by the boundary layer in \cref{lem:remainder}.
Both hold verbatim on cell spaces, and we exploit this fact in \cref{sec:mesh}.

\subsection{Newton derivatives}\label{sec:newton-derivatives}
At a potential $z$, let $s=\T z-C-\eps\delta$ denote its slack, and select a measurable function $\vartheta$ with
\begin{equation}\label{eq:slopes}
 \vartheta=1\ \text{ where } s>0,\quad
 \vartheta=0\ \text{ where } s<0,\quad
 \vartheta\in[0,1]\ \text{ where } s=0.
\end{equation}
We call $\vartheta$ a \emph{slope}. For a measurable slope $0\leq\vartheta\leq1$, we define
\begin{equation}\label{eq:H-cont}
	H(\vartheta)d=(fd_u,gd_v)+\eps^{-1}\T^*(\vartheta\,\T d),\qquad\text{for }d=(d_u,d_v)\in\Hh.
\end{equation}
By \eqref{eq:Tnorm} and $f,g\le E$, it is a bounded operator on $\Hh$. In the Newton step \eqref{eq:Newton}, we use $H(\vartheta(z))$, with $\vartheta(z)$ selected by \eqref{eq:slopes}, as the generalized derivative of $F$ at $z$. Whether this choice controls the Newton remainder depends on the norm, as we show in \cref{rem:L2,lem:remainder}.

On a cell space, the slack $s$ of $\Lsf z$, where $z=( (u_i)_{i\in\N_X},(v_j)_{j\in\N_Y})$, is cellwise constant. In particular, on $X_i\times Y_j$, we have $s_{ij}=u_i+v_j-C_{ij}-\eps\delta$. Also on this cell space, \eqref{eq:slopes} selects cellwise constant slopes $\vartheta_{ij}$. For such slopes, $H(\vartheta)$ acts on the cell space by
\begin{equation}\label{eq:Hh}
	\begin{aligned}
		[H_h(\vartheta)d]_{u,i}
		&=f_i d_{u,i}+\eps^{-1}\sum_{j=1}^{N_Y}\omega_j\vartheta_{ij}(d_{u,i}+d_{v,j}),\\
		[H_h(\vartheta)d]_{v,j}
		&=g_j d_{v,j}+\eps^{-1}\sum_{i=1}^{N_X}w_i\vartheta_{ij}(d_{u,i}+d_{v,j}).	
	\end{aligned}
\end{equation}
Every element of the Clarke generalized Jacobian \cite{clarke1983optimization} of $F_h$ at $z$ equals $H_h(\vartheta)$ for some $\vartheta$ as in \eqref{eq:slopes}, and every such selection is a Newton derivative of $F_h$.
We will observe that \cref{lem:remainder} allows any slope in $[0,1]$ at zero slack.
The quadratic form of $H_h(\vartheta)$ in the weighted inner product is
\begin{equation}\label{eq:H-form}
 \ip{H_h(\vartheta)d}{d}_h
 =\sum_{i=1}^{N_X}w_if_i d_{u,i}^2+\sum_{j=1}^{N_Y}\omega_jg_jd_{v,j}^2
 +\eps^{-1}\sum_{(i,j)\in\N}w_i\omega_j\vartheta_{ij}(d_{u,i}+d_{v,j})^2.
\end{equation}
By \eqref{eq:finite-constants}, we note that
\begin{align*}
	\ip{H_h(\vartheta)d}{d}_h\ge \sum_{i=1}^{N_X}w_if_i d_{u,i}^2+\sum_{j=1}^{N_Y}\omega_jg_jd_{v,j}^2\ge \mu_h\Bigg(\sum_{i=1}^{N_X}w_i d_{u,i}^2+\sum_{j=1}^{N_Y}\omega_jd_{v,j}^2\Bigg)= \mu_h\norm{d}_h^2.
\end{align*}
By $\vartheta\in[0,1]$, \eqref{eq:finite-constants} and \eqref{eq:Tnorm}, we also note that
\begin{align*}
	\ip{H_h(\vartheta)d}{d}_h&\le E_h^{\rm max}\Bigg(\sum_{i=1}^{N_X}w_i d_{u,i}^2+\sum_{j=1}^{N_Y}\omega_jd_{v,j}^2\Bigg)+\eps^{-1}\sum_{(i,j)\in\N}w_i\omega_j\vartheta_{ij}(d_{u,i}+d_{v,j})^2\\
	&\le E_h^{\rm max}\norm{d}_h^2+\eps^{-1}\sum_{(i,j)\in\N}w_i\omega_j(d_{u,i}+d_{v,j})^2 = E_h^{\rm max}\norm{d}_h^2+\eps^{-1}\norm{\T d}_{2,h}^2\le L_h\norm{d}_h^2.
\end{align*}
Hence, we have $\mu_h I\leq H_h(\vartheta)\leq L_hI$ in the weighted inner product, for every choice of slopes.

Without the terms $f_id_{u,i}$ and $g_jd_{v,j}$ in \eqref{eq:Hh}, one obtains the Newton derivative of the balanced transport problem, which has a nontrivial nullspace. Indeed, in the balanced transport problem, equation \eqref{eq:H-form} becomes
\begin{align*}
	 \ip{H_h(\vartheta)d}{d}_h
 = \eps^{-1}\sum_{(i,j)\in\N}w_i\omega_j\vartheta_{ij}(d_{u,i}+d_{v,j})^2.
\end{align*}
Let $\mathcal K$ be a connected component of the bipartite graph whose edges are the pairs $(i,j)$ with $\vartheta_{ij}>0$. For a constant $\xi\neq0$, we define $d$ by $d_{u,i}=\xi$ for source vertices $i$ in $\mathcal K$, $d_{v,j}=-\xi$ for target vertices $j$ in $\mathcal K$, and $d=0$ elsewhere. Then $\vartheta_{ij}(d_{u,i}+d_{v,j})=0$ for all $i,j$, so \eqref{eq:H-form} vanishes in this case.

The diagonal terms $f_id_{u,i}$ and $g_jd_{v,j}$ in \eqref{eq:Hh}, which come from the Pearson marginal penalties, ensure that the nullspace is trivial. Indeed, one see that, with $f_iu_{u,i},\,g_jv_{v,j}$ in \eqref{eq:Hh}, $\ip{H_h(\vartheta)d}{d}_h=0$ only when $d=0$.

A Newton direction $d^m$ at $z^m$, exact or inexact, is globalized by Armijo backtracking \cite{armijo1966minimization,nocedal2006numerical}. We backtrack on the step size $\tau$ from $\tau=1$ by a factor $\beta\in(0,1)$ until meeting the Armijo condition
\begin{equation}\label{eq:Armijo}
 \G_h(z^m+\tau d^m)
 \leq\G_h(z^m)+\sigma\tau\ip{F_h(z^m)}{d^m}_h,
 \qquad 0<\sigma<1/2.
\end{equation}
Weighted projected-gradient, accelerated, and Newton--CG solvers built on these objects are recorded in \cref{S-sec:global} of the supplementary materials. Here, as in \cite[Section~7.1]{nocedal2006numerical}, a \emph{Newton--CG} method is a line-search Newton method whose linear equations are solved approximately by the conjugate gradient method. Their mesh-uniform global rates are standard consequences of strong convexity and a Lipschitz gradient \cite{beck2017first,nesterov2004introductory}. These solvers supply initial approximations for the local theory below.
\subsection{The maximum norm and an inverse estimate}
We denote
\[
 \Bb=L^\infty(X)\times L^\infty(Y),\qquad
 \norm{(u,v)}_\Bb=\max\{\norm{u}_{L^\infty(X)},\norm{v}_{L^\infty(Y)}\},
\]
and for the arrays $K=(K_{ij})_{(i,j)\in\N},\, ( (u_i)_{i\in\N_X},(v_j)_{j\in\N_Y})$,
\begin{align*}
	\norm{K}_\infty=\max_{(i,j)\in\N}|K_{ij}|,\qquad \norm{( (u_i)_{i\in\N_X},(v_j)_{j\in\N_Y})}_\infty=\max\{\max_{i\in\N_X}|u_i|,\max_{j\in\N_Y}|v_j|\}.
\end{align*}
By \eqref{eq:bounded-solutions}, we have the minimizer of $\G$ satisfies $z^*\in\Bb$. The following example shows why we measure proximity to $z^*$ in $\Bb$ rather than in $\Hh$.

\begin{remark}\label{rem:L2}
Marginal integration does not make the operators $H(\vartheta(z))$ Newton derivatives of $F$ with respect to the norm of $\Hh$.

For example, let us consider $X=Y=[0,1]$, $f=g=1$, $C=0$, $\eps=2$, and $\delta=0$, with solution $u^*=v^*=1/2$, slack $s^*=1$, and plan $k^*=1/2$.

Let $A_n\subset[0,1]$ be measurable, with $0<|A_n|\downarrow0$, and set $e_n=(-2\ind_{A_n},0)$.
Then, we have $\norm{e_n}_{\Hh}=2|A_n|^{1/2}\to0$.
We compute the slack at $z^*+e_n$
\begin{align*}
	s=\T (z^*+e_n)-C-\varepsilon\delta = 1-2\ind_{A_n\times Y}.
\end{align*}
By \eqref{eq:slopes}, we have the slope $\vartheta(z^*+e_n) = 1-\ind_{A_n\times Y}$.
By \eqref{eq:residual-intro} and \eqref{eq:H-cont}, we compute the Newton remainder
\begin{align*}
	\mathcal R_n&:= F(z^*+e_n)-F(z^*)-H(\vartheta(z^*+e_n))e_n \\
	&= \T^*[p(\T(z^*+e_n))-p(\T z^*)]+(-2\ind_{A_n},0)-(-2\ind_{A_n},0)+2\varepsilon^{-1}\T^*( (1-\ind_{A_n\times Y})\ind_{A_n\times Y})\\
	&= \T^*[(1/2-\ind_{A_n\times Y})_+-1/2]= \left(-\tfrac12\ind_{A_n},-\tfrac12|A_n|\right).
\end{align*}
Hence, we obtain
\[
 \frac{\norm{\mathcal R_n}_{\Hh}}{\norm{e_n}_{\Hh}}
 =\frac14\sqrt{1+|A_n|}
 \longrightarrow\frac14.
\]
Meanwhile, if $H(\vartheta)$ is the Newton derivatives of $F$ at $z^*$ in $\Hh$, then $\norm{\mathcal R_n}_{\Hh}/\norm{e_n}_{\Hh}\longrightarrow 0$. 

This is the norm-gap phenomenon familiar from semismooth Newton methods in function spaces \cite{ulbrich2002semismooth,hintermuller2002primal}.
Here, $\norm{e_n}_\Bb=2$ for every $n$, so these perturbations
are excluded by requiring smallness in $\Bb$.
Under the interface condition in \cref{thm:local}, the theorem uses this maximum-norm proximity to obtain superlinear error reduction in both $\Bb$ and $\Hh$.
\end{remark}

The operator $H(\vartheta)$ in \eqref{eq:H-cont} is also bounded on $\Bb$. However, smallness in $\Bb$ is useful only together with an inverse bound for $H(\vartheta)$ in the maximum norm. Deriving such a bound from the inverse bound \eqref{eq:inverse-H} in $\Hh$ by norm equivalence would introduce a mesh-dependent factor (\cref{rem:M-operator}). The next lemma gives a bound directly in the maximum norm.

\begin{lemma}\label{lem:inverse}
For every measurable $0\leq\vartheta\leq1$, $H(\vartheta)$ is invertible
on $\Bb$ and
\begin{equation}\label{eq:inverse-infty}
 \norm{H(\vartheta)^{-1}}_{\Bb\to\Bb}\leq c^{-1}.
\end{equation}
For $z,\widehat z\in\Bb$, we have the \emph{secant estimate}
\begin{equation}\label{eq:secant}
 \norm{z-\widehat z}_\Bb
 \leq c^{-1}\norm{F(z)-F(\widehat z)}_\Bb.
\end{equation}
In particular, $\norm{z-z^*}_\Bb\leq c^{-1}\norm{F(z)}_\Bb$. Both
estimates hold on every cell space, with $c$ replaced by $\mu_h\geq c$.

Moreover, $H(\vartheta)$ is self-adjoint and positive on $\Hh$, with
\begin{equation}\label{eq:inverse-H}
 \norm{H(\vartheta)^{-1}}_{\Hh\to\Hh}\leq c^{-1}.
\end{equation}
\end{lemma}
\begin{proof}
Let $S(d_u,d_v)=(d_u,-d_v)$, an isometry on both spaces $\Bb$ and $\Hh$, and write $\widetilde H=SHS$.
We define
\[
 \varrho_X(x)=\eps^{-1}\int_Y\vartheta(x,y)\dd y,\qquad
 \varrho_Y(y)=\eps^{-1}\int_X\vartheta(x,y)\dd x.
\]
By \eqref{eq:H-cont}, we compute
\[
 \widetilde Hd=\Big((f+\varrho_X)d_u-\eps^{-1}\int_Y\vartheta\,d_v\dd y,\ (g+\varrho_Y)d_v-\eps^{-1}\int_X\vartheta\,d_u\dd x\Big).
\]
Thus, we get $\widetilde H=D(I-P)$, where $D=\operatorname{diag}(f+\varrho_X,g+\varrho_Y)$ is boundedly invertible and
\[
 (Pd)_u(x)=\frac{\int_Y\vartheta(x,y)d_v(y)\dd y}
                  {\eps(f(x)+\varrho_X(x))},\qquad
 (Pd)_v(y)=\frac{\int_X\vartheta(x,y)d_u(x)\dd x}
                  {\eps(g(y)+\varrho_Y(y))}.
\]
The operator $P$ maps nonnegative functions to nonnegative functions.
Moreover, we have
\[
 |(Pd)_u(x)|\le\frac{\varrho_X(x)}{f(x)+\varrho_X(x)}\norm{d}_\Bb,\qquad
 |(Pd)_v(y)|\le\frac{\varrho_Y(y)}{g(y)+\varrho_Y(y)}\norm{d}_\Bb.
\]
Since $f,g\ge c$, $\varrho_X\leq b/\eps$, $\varrho_Y\leq a/\eps$, and
$t\mapsto t/(c+t)$ is increasing, we obtain
\[
 \norm{P}_{\Bb\to\Bb}
 \leq\max\left\{\frac{b/\eps}{c+b/\eps},
                 \frac{a/\eps}{c+a/\eps}\right\}<1.
\]

Since $\norm{P}_{\Bb\to\Bb}<1$, the Neumann series $\sum_{k\geq0}P^k$ converges in operator norm on $\Bb$. Hence, $\widetilde H^{-1}=\sum_{k\geq0}P^kD^{-1}$ exists and maps nonnegative functions to nonnegative functions.
Moreover, as $\widetilde H\one=(f,g)$, we get $\widetilde H^{-1}(f,g)=\one$.
Since $(f,g)-c\one\geq0$, we obtain
\[
 \one-c\widetilde H^{-1}\one=\widetilde H^{-1}\big((f,g)-c\one\big)\ge0,
 \qquad\text{that is,}\qquad \widetilde H^{-1}\one\leq c^{-1}\one.
\]
For arbitrary $r\in\Bb$, we have $-\norm{r}_\Bb\one\leq r\leq\norm{r}_\Bb\one$. Thus, we obtain
\[
	\norm{\widetilde H^{-1}r}_{\Bb}
	\leq\norm{r}_\Bb\norm{\widetilde H^{-1}\one}_{\Bb}
	\leq c^{-1}\norm{r}_\Bb.
\]
Hence, it follows that $\norm{\widetilde H^{-1}}_{\Bb\to\Bb}\leq c^{-1}$.
Since $H^{-1}=S\widetilde H^{-1}S$ and $S$ is an isometry, this proves \eqref{eq:inverse-infty}.

To prove \eqref{eq:secant}, we choose the measurable slope of the secant of $r\mapsto r_+$ between the two slacks $s$ of $z=(u,v)$ and $\widehat{s}$ of $\widehat{z}=(\widehat{u},\widehat{v})$, that is,
\[
	\vartheta=\frac{s_+-\widehat{s}_+}{s-\widehat{s}}
	=\frac{(\T z-C-\eps\delta)_+-(\T\widehat z-C-\eps\delta)_+}
	{\T(z-\widehat z)}.
\]
Here, we have the convention that $\vartheta=0$ if $s=\widehat{s}$.
By definition, $\vartheta\in [0,1]$. By \eqref{eq:residual-intro} and \eqref{eq:H-cont}, we also observe that
\begin{align*}
	F(z)-F(\widehat z)&=\varepsilon^{-1}\T^*[s_+-\widehat{s}_+]+(f(u-\widehat{u}),g(v-\widehat{v}))\\
	&= \varepsilon^{-1}\T^*[\vartheta\T(z-\widehat{z})]+(f(u-\widehat{u}),g(v-\widehat{v}))= H(\vartheta)(z-\widehat z).
\end{align*}
We apply \eqref{eq:inverse-infty} to obtain \eqref{eq:secant}.

Finally, we write the continuum form of \eqref{eq:H-form}, which yields the lower bound
\begin{equation}\label{eq:H-form-cont}\ip{H(\vartheta)d}{d}_{\Hh}=\int_Xfd_u^2\dd x+\int_Ygd_v^2\dd y
+\eps^{-1}\int_{X\times Y}\vartheta(\T d)^2\dd x\dd y\geq c\norm{d}_{\Hh}^2.\end{equation}
Moreover, $H(\vartheta)$ is bounded and self-adjoint on $\Hh$.
The Lax--Milgram theorem therefore gives invertibility on $\Hh$, and the displayed coercivity estimate gives \eqref{eq:inverse-H}.
\end{proof}

\begin{remark}
\label{rem:M-operator}
The inverse bound follows from a comparison principle for the signed block operator.
After the sign flip $S$, we have $\widetilde H=D(I-P)$, where $P$ has a nonnegative kernel and $\norm{P}_{\Bb\to\Bb}<1$.
Thus $\widetilde H$ is an integral-operator analogue of a nonsingular $M$-matrix, and \eqref{eq:inverse-infty} is the corresponding maximum-principle bound.

On a cell space, deriving a maximum-norm bound from the bound \eqref{eq:inverse-H} in the weighted inner product would introduce a factor depending on the smallest cell volume, of order $\min\{w_i,\omega_j\}^{-1/2}$.
The estimate in \cref{lem:inverse} avoids this factor.

A related positivity argument shows that balanced dual gradient descent is nonexpansive in the maximum norm \cite[Lemma~5.2]{gonzalez2025linear}.
Here the positive diagonal terms $f,g\geq c$ give the uniform inverse bound.
\end{remark}
\subsection{The boundary layer and the Newton remainder}\label{sec:boundary-layer}
The second ingredient is a bound on the Newton remainder in terms of the boundary layer around the active interface.
For a fixed root $z^*$, we define its slack $s^*=\T z^*-C-\eps\delta$ and the punctured marginal boundary-layer function
\begin{equation}\label{eq:punctured-layer}
 \nu_*(t)=\max\left\{
 \esssup_x|\{y:0<|s^*(x,y)|\leq t\}|,
 \esssup_y|\{x:0<|s^*(x,y)|\leq t\}|\right\}.
\end{equation}
The zero-slack set is excluded because the linearization error of the positive part vanishes where $s^*=0$: the map $t\mapsto t_+$ is positively homogeneous, so a slope taken at $t$ reproduces it exactly. The inclusive function $m_*$ in \eqref{eq:layer-intro} satisfies $\nu_*\leq m_*$.
In this section, we work under the condition $\nu_*(t)\to0$ as $t\downarrow0$. Since $\nu_*\le m_*$, this condition is weaker than $m_*(t)\to0$ with $m_*$ from \eqref{eq:layer-intro}.

For a partition with exact array root $z_h^*$ and slack
$s_h^*=s_{ij,h}^*=u_{i,h}^*+v_{j,h}^*-C_{ij}-\eps\delta$, we define the discrete
boundary-layer function
\begin{equation}\label{eq:discrete-layer}
 \nu_h(t)=\max\bigg\{
	 \max_{i\in\N_X}\sum_{j=1}^{N_Y}\omega_j\ind_{\{0<|s^*_{ij,h}|\leq t\}},\
 \max_{j\in\N_Y}\sum_{i=1}^{N_X}w_i\ind_{\{0<|s^*_{ij,h}|\leq t\}}\bigg\}.
\end{equation}
It is the punctured boundary-layer function \eqref{eq:punctured-layer} of the lifted slack $\Lsf s_h^*$, so \cref{lem:remainder} holds on the cell space with $\nu_h$.
We denote by $m_h$ the inclusive discrete boundary-layer function, defined by \eqref{eq:discrete-layer} with the condition $0<|s^*_{ij,h}|\leq t$ replaced by $|s^*_{ij,h}|\leq t$.

The mesh-independence argument of \cref{sec:mesh} uses instead the function $m_*$, because a zero continuum slack can become a small nonzero slack after data approximation. This is why the hypotheses of \cref{sec:local,sec:mesh} are not identical.

\begin{lemma}\label{lem:remainder}
Let $z=z^*+e\in\Bb$, $r=\norm{e}_\Bb$, and let $\vartheta(z)$ be any
measurable selection of \eqref{eq:slopes} at the slack $\T z-C-\eps\delta$.
Then
\begin{align}
 \norm{F(z)-F(z^*)-H(\vartheta(z))e}_\Bb
 &\leq\frac{2}{\eps}\nu_*(2r)\norm{e}_\Bb,\label{eq:rem-infty}\\
 \norm{F(z)-F(z^*)-H(\vartheta(z))e}_{\Hh}
 &\leq\frac{2}{\eps}\nu_*(2r)\norm{e}_{\Hh},\label{eq:rem-H}\\
 \left|\G(z)-\G(z^*)-\tfrac12\ip{H(\vartheta(z))e}{e}_{\Hh}\right|
 &\leq\frac{\nu_*(2r)}{\eps}\norm{e}_{\Hh}^2.\label{eq:objective-expansion}
\end{align}
The same three inequalities hold on finite weighted cell spaces, with
$\nu_*$ replaced by the discrete boundary-layer function \eqref{eq:discrete-layer} of the array root.
\end{lemma}
\begin{proof}
	For real $s,t$, let $\theta$ satisfy \eqref{eq:slopes} with $s$ replaced by $s+t$. We observe that
\begin{equation}\label{eq:scalar-rem}
 |(s+t)_+-s_+-\theta t|
 \leq |t|\ind_{\{0<|s|\leq|t|\}}.
\end{equation}
Indeed, we compute
\[
	|(s+t)_+-s_+-\theta t|=
	\begin{cases}
		0,& (s+t)s>0\ \text{or}\ s=0,\\
		|s|,& (s+t)s<0,\\
		(1-\theta)|t|,& t=-s<0,\\
		\theta|t|,& t=-s>0.
\end{cases}
\]
In the case $(s+t)s<0$, we have $0<|s|\le|t|$, and in the last two cases $|s|=|t|>0$. Hence, \eqref{eq:scalar-rem} holds in all cases.

By \eqref{eq:residual-intro} and \eqref{eq:H-cont}, we also have
\[F(z)-F(z^*)-H(\vartheta(z))e
=\eps^{-1}\T^*[(s^*+\T e)_+-s^*_+-\vartheta(z)\T e].\]
We then apply \eqref{eq:scalar-rem} with $s=s^*$ and $t=\T e$. Since $|\T e|\le 2r$, for $\Omega_r=\{(x,y)\in X\times Y:0<|s^*|\leq2r\}$, we obtain
\[
	\norm{F(z)-F(z^*)-H(\vartheta(z))e}_{\Bb}
	\le\frac1\eps\max\Bigg\{\int_Y|\T e|\ind_{\Omega_r}(x,y)\dd y,\,\int_X|\T e|\ind_{\Omega_r}(x,y)\dd x\Bigg\}
	\le\frac{2r}{\eps}\,\nu_*(2r).
\]

For \eqref{eq:rem-H}, we note that $|\T e|\le\T(|e_u|,|e_v|)$, where $e=(e_u,e_v)$. 
For any $d=(d_u,d_v)\in\Hh$, we estimate
\[
 \ip{\T^*\ind_{\Omega_r}\T d}{d}_{\Hh}=\int_{\Omega_r}|d_u(x)+d_v(y)|^2\dd x\dd y
 \leq2\int_{\Omega_r}\big(d_u(x)^2+d_v(y)^2\big)\dd x\dd y
 \leq2\nu_*(2r)\norm{d}_\Hh^2.
\]
Hence, the positive self-adjoint operator $\T^*\ind_{\Omega_r}\T$ has norm at
most $2\nu_*(2r)$, and we obtain
\[
 \norm{F(z)-F(z^*)-H(\vartheta(z))e}_{\Hh}
 \le\frac1\eps\norm{\T^*\ind_{\Omega_r}\T|e|}_{\Hh}
 \le\frac{2}{\eps}\nu_*(2r)\norm{e}_{\Hh}.
\]

For \eqref{eq:objective-expansion}, with $r^*=\T z^*-C$ and by \eqref{eq:residual-intro}, \eqref{eq:p}, \eqref{eq:G}, \eqref{eq:H-form-cont}, we compute
\[
		\G(z)-\G(z^*)-\tfrac12\ip{H(\vartheta(z))e}{e}
		=\int_{X\times Y}\Big[\phi(r^*+\T e)-\phi(r^*)-p(r^*)\T e-\frac{\vartheta(z)}{2\eps}(\T e)^2\Big]\dd x\dd y+\ip{F(z^*)}{e}_{\Hh},
\]
and recall that $F(z^*)=0$. 

We now have the identities $\phi(r)=\delta r-\eps\delta^2/2+(r-\eps\delta)_+^2/(2\eps)$, which follows from \eqref{eq:phi}, and $s^*=\T z^*-C-\varepsilon\delta=r^{*}-\varepsilon\delta$ from \eqref{eq:optimal-slack}. We observe that
\[
	\begin{aligned}
		\phi(r^*+\T e)-\phi(r^*)-p(r^*)\T e-\frac{\vartheta(z)}{2\eps}(\T e)^2 &= \frac{1}{\varepsilon}\Bigg(\frac12(s^*+\T e)_+^2-\frac12(s^*)_+^2-(s^*)_+\T e-\frac{\vartheta(z)}{2}(\T e)^2\Bigg)\\
		&= \frac{1}{\varepsilon}\int_0^{\T e}\big[(s^*+\lambda)_+-(s^*)_+-\vartheta(z)\lambda\big]\dd\lambda.
	\end{aligned}
\]
By \eqref{eq:scalar-rem} and $|\T e|\le 2r$, we obtain
\[
	\left|\int_0^{\T e}
	\big[(s^*+\lambda)_+-(s^*)_+-\vartheta(z)\lambda\big]\dd\lambda\right|
	\leq\tfrac12|\T e|^2\ind_{\Omega_r}.
\]
Therefore, we arrive at
\[
 \left|\G(z)-\G(z^*)-\tfrac12\ip{H(\vartheta(z))e}{e}\right|
 \le\frac1{2\eps}\int_{\Omega_r}(\T e)^2\dd x\dd y
 \le\frac{\nu_*(2r)}{\eps}\norm{e}_{\Hh}^2.
\]
Since the argument only integrates pointwise bounds, the same inequalities hold for weighted sums on cell spaces.
\end{proof}

Estimate \eqref{eq:scalar-rem} localizes the linearization error $(s^*+\T e)_+-s^*_+-\vartheta(z)\T e$ to the set $\Omega_r$. A localization of the same type appears in the analysis of gradient descent for balanced quadratically regularized transport \cite[(43)--(44)]{gonzalez2025linear}. There, the authors compare two slopes of the positive part. The first is the average slope $\bar\vartheta=\int_0^1\ind_{\{s^*+\lambda\T e>0\}}\dd\lambda$ along the segment from $s^*$ to $s^*+\T e$. The second is the slope $\ind_{\{s^*>0\}}$ at $z^*$. The two slopes agree where $|s^*|>|\T e|$, so their difference vanishes outside the boundary layer $\{|s^*|\le2r\}$. The integral of this difference against $\psi\in L^2(Y)$ is at most $\norm{\psi}_{L^2(Y)}\,m_*(2r)^{1/2}$ for every $x$.

\Cref{lem:remainder} differs in three respects that matter for Newton's method.
First, we take the slope at the current iterate rather than averaging it along the segment. This allows us to exclude the zero-slack set, so the punctured function $\nu_*$ enters instead of $m_*$.
Second, we measure the perturbation in the maximum norm, so the bound \eqref{eq:rem-infty} is linear in the measure of the boundary layer $\nu^*$ rather than in its square root.
Third, we include the objective expansion \eqref{eq:objective-expansion}, which we use to show that the Armijo condition accepts the step size $\tau=1$.
\subsection{Local superlinear convergence}
We now combine the inverse estimate of \cref{lem:inverse} with the remainder estimate of \cref{lem:remainder}. This gives local superlinear convergence of the semismooth Newton method in both $\Bb$ and $\Hh$, and conditions under which the Armijo condition accepts the step size $\tau=1$.

\begin{theorem}\label{thm:local}
	Let $\nu_*$ be from \eqref{eq:punctured-layer} and $z^*$ be the minimizer of $\G$ from \eqref{eq:G}. Assume $\nu_*(t)\to0$ as $t\downarrow0$. Set
\[
 \omega(r)=\frac{2\nu_*(2r)}{c\eps},
\]
and choose $r_0>0$ with $\omega(r_0)<1$. For every $z^0$ with
$\norm{z^0-z^*}_\Bb\leq r_0$, the full-step Newton iteration
\begin{equation}\label{eq:Newton}
 z^{m+1}=z^m-H(\vartheta(z^m))^{-1}F(z^m)
\end{equation}
is well defined and converges to $z^*$, with
\begin{align}
 \norm{z^{m+1}-z^*}_\Bb
 &\leq\omega(\norm{z^m-z^*}_\Bb)
       \norm{z^m-z^*}_\Bb,\label{eq:local-infty}\\
 \norm{z^{m+1}-z^*}_{\Hh}
 &\leq\omega(\norm{z^m-z^*}_\Bb)
       \norm{z^m-z^*}_{\Hh}.\label{eq:local-H}
\end{align}
Thus, the convergence is $q$-superlinear in both norms.
If $\nu_*(t)\leq At^\alpha$ for $0<t\leq2r_0$, with
$A>0$ and $\alpha>0$, then
\begin{equation}\label{eq:local-order}
 \norm{z^{m+1}-z^*}_\Bb
 \leq\frac{2^{1+\alpha}A}{c\eps}
       \norm{z^m-z^*}_\Bb^{1+\alpha}.
\end{equation}

Let $F,\,\vartheta,\,H$ from \eqref{eq:residual-intro}, \eqref{eq:slopes}, \eqref{eq:H-cont}. The continuum Armijo backtracking, in the $\Hh$ inner product, backtracks
on the step size $\tau$ from $\tau=1$ by a factor $\beta\in(0,1)$ until meeting the Armijo condition
\begin{equation}\label{eq:Armijo-cont}
 \G(z+\tau d)\leq\G(z)+\sigma\tau\ip{F(z)}{d}_{\Hh},\qquad d=-H(\vartheta(z))^{-1}F(z),\qquad 0<\sigma<1/2.
\end{equation}
This is the continuum counterpart of \eqref{eq:Armijo}.
For $0<\sigma<1/2$ and $\kappa=L/c$, this condition accepts
the step size $\tau=1$ whenever
$r=\norm{z-z^*}_\Bb\leq r_0$ satisfies
\begin{equation}\label{eq:full-step}
 \omega(r)\Big[
 \tfrac12+2\sigma\sqrt\kappa+\kappa\big(\sigma+\tfrac12\big)\omega(r)
 \Big]\leq\tfrac12-\sigma.
\end{equation}
\end{theorem}

\begin{proof}
We denote
\[
\mathcal R(z)=F(z)-F(z^*)-H(\vartheta(z))(z-z^*).
\]
Let $z^+=z-H(\vartheta(z))^{-1}F(z)$. Since $F(z^*)=0$, it follows that
\[
 z^+-z^*=-H(\vartheta(z))^{-1}\mathcal R(z).
\]
Let $r=\norm{z-z^*}_\Bb$. By \eqref{eq:inverse-infty} and \eqref{eq:rem-infty}, we obtain
\[
 \norm{z^+-z^*}_\Bb\le\frac1c\norm{\mathcal R(z)}_\Bb
 \le\frac{2\nu_*(2r)}{c\eps}\norm{z-z^*}_\Bb=\omega(r)\norm{z-z^*}_\Bb.
\]
In the same way, the inverse bound \eqref{eq:inverse-H} and the remainder bound \eqref{eq:rem-H} give
\[
 \norm{z^+-z^*}_{\Hh}\le\frac1c\norm{\mathcal R(z)}_{\Hh}\le\omega(r)\norm{z-z^*}_{\Hh}.
\]
These are \eqref{eq:local-infty} and \eqref{eq:local-H}.

Since $\omega$ is nondecreasing, we have, as long as $\norm{z^m-z^*}_\Bb\le r_0$,
\[
 \norm{z^{m+1}-z^*}_\Bb
 \leq\omega(r_0)\norm{z^m-z^*}_\Bb.
\]
By induction, the iterates remain in the ball and converge to $z^*$ in the maximum norm.
Since $\omega(\norm{z^m-z^*}_\Bb)\to0$, the estimates \eqref{eq:local-infty} and \eqref{eq:local-H} give $q$-superlinear convergence in both norms.
Under the power bound, we have $\omega(r)\le2A(2r)^\alpha/(c\eps)$ for $2r\le2r_0$, and \eqref{eq:local-infty} gives \eqref{eq:local-order}.

For the acceptance of the step size $\tau=1$, we write $e=z-z^*$ and also denote $H=H(\vartheta(z))$, $d=-H^{-1}F(z)$, and $e^+=z^+-z^*=e+d$. By \eqref{eq:local-H}, we have $\norm{e^+}_{\Hh}\leq\omega\norm{e}_{\Hh}$ with $\omega=\omega(r)$.
By \eqref{eq:objective-expansion} and
$\nu_*(2r)/\eps=c\,\omega/2$, we have
\[
 \G(z)-\G(z^*)\geq
 \frac12\ip{He}{e}-\frac c2\omega\norm{e}_{\Hh}^2.
\]
Since $\G$ has an $L$-Lipschitz gradient from \cref{prop:global-constants} and $F(z^*)=0$, we also obtain
\[
 \G(z+d)-\G(z^*)\leq
 \frac L2\norm{e^+}_{\Hh}^2
 \leq\frac L2\omega^2\norm{e}_{\Hh}^2.
\]
Combining the two estimates and using $L=c\kappa$, we get
\[
 \G(z+d)-\G(z)\le-\frac12\ip{He}{e}+\frac c2\big(\omega+\kappa\omega^2\big)\norm{e}_{\Hh}^2.
\]
We now bound $\ip{Hd}{d}$ in the energy norm $|w|_H=\ip{Hw}{w}^{1/2}$.
Since $cI\leq H\leq LI$, we have
\[
 |e^+|_H\leq\sqrt L\,\norm{e^+}_{\Hh}\leq\sqrt L\,\omega\norm{e}_{\Hh}\leq\sqrt\kappa\,\omega|e|_H.
\]
Since $d=e^+-e$, the triangle inequality gives
\[
 \ip{Hd}{d}\leq\big(|e|_H+|e^+|_H\big)^2
 \leq\big(1+2\sqrt\kappa\,\omega+\kappa\omega^2\big)\ip{He}{e}.
\]
As $\ip{F(z)}{d}=-\ip{Hd}{d}$, we have
\[
 -\sigma\ip{F(z)}{d}\le\sigma\big(1+2\sqrt\kappa\,\omega+\kappa\omega^2\big)\ip{He}{e}.
\]
Adding this to the bound for $\G(z+d)-\G(z)$, we obtain
\begin{align*}
 \G(z+d)-\G(z)-\sigma\ip{F(z)}{d}
 \leq
 -(\tfrac12-\sigma)\ip{He}{e}
 +\sigma\big(2\sqrt\kappa\,\omega+\kappa\omega^2\big)\ip{He}{e}
 +\tfrac c2\big(\omega+\kappa\omega^2\big)\norm{e}_{\Hh}^2.
\end{align*}
Because $c\norm{e}_{\Hh}^2\leq\ip{He}{e}$, we obtain
\[
 \G(z+d)-\G(z)-\sigma\ip{F(z)}{d}
 \leq\Big(\omega\Big[\tfrac12+2\sigma\sqrt\kappa+\kappa\big(\sigma+\tfrac12\big)\omega\Big]-\big(\tfrac12-\sigma\big)\Big)\ip{He}{e},
\]
and the right-hand side is nonpositive under \eqref{eq:full-step}.
\end{proof}

\begin{corollary}\label{cor:local-inexact}
Under the assumptions of \cref{thm:local}, let
$H_m=H(\vartheta(z^m))$ and $z^{m+1}=z^m+d^m$.
Suppose that the inexact step has the linear residual
\[
 r^m=H_md^m+F(z^m),\qquad
 \norm{r^m}_\infty\leq\eta_m\norm{F(z^m)}_\infty,
\]
with forcing terms $\eta_m\geq0$.
Define $L_\infty=E+2\max\{a,b\}/\eps$ and $e^m=z^m-z^*$.
Then
\begin{equation}\label{eq:inexact-local}
 \norm{e^{m+1}}_\infty
 \leq\left[
 \omega(\norm{e^m}_\infty)+\frac{L_\infty}{c}\eta_m
 \right]\norm{e^m}_\infty.
\end{equation}
Consequently, if
$\omega(r_0)+L_\infty\sup_m\eta_m/c<1$,
the iteration converges from every
$\norm{e^0}_\infty\leq r_0$.
The convergence is $q$-superlinear in the maximum norm if,
in addition, $\eta_m\to0$.
\end{corollary}

\begin{proof}
We denote $\mathcal R(z^m)=F(z^m)-F(z^*)-H_me^m$.
Since $F(z^*)=0$ and $H_md^m=r^m-F(z^m)$,
we observe that the next error satisfies
\[
 e^{m+1}=e^m+d^m=H_m^{-1}\big(H_me^m-F(z^m)+r^m\big)=H_m^{-1}(r^m-\mathcal R(z^m)).
\]
Next, we bound the residual. For $z\in\Bb$ and $e=z-z^*$, the $u$-component of $F(z)-F(z^*)$ satisfies
\[
 \big|[F(z)-F(z^*)]_u(x)\big|
 \le\int_Y\big|p(\T z-C)-p(\T z^*-C)\big|\dd y+f(x)|e_u(x)|
 \le\Big(\frac{2b}{\eps}+E\Big)\norm{e}_\infty,
\]
since $p$ is $1/\eps$-Lipschitz and $|\T e|\leq2\norm{e}_\infty$. The $v$-component is bounded in the same way with $a$ in place of $b$. Hence,
\[
 \norm{F(z)-F(z^*)}_\infty
 \leq L_\infty\norm{z-z^*}_\infty,
\]
and, since $F(z^*)=0$, we have
$\norm{r^m}_\infty\leq\eta_m\norm{F(z^m)}_\infty\le\eta_mL_\infty\norm{e^m}_\infty$.
By \cref{lem:inverse} and \eqref{eq:rem-infty}, we obtain
\[
 \norm{e^{m+1}}_\infty\le\frac1c\big(\norm{r^m}_\infty+\norm{\mathcal R(z^m)}_\infty\big)
 \le\Big[\frac{L_\infty}{c}\eta_m+\omega(\norm{e^m}_\infty)\Big]\norm{e^m}_\infty,
\]
which is \eqref{eq:inexact-local}.

We set $q=\omega(r_0)+L_\infty\sup_m\eta_m/c<1$.
Since $\omega$ is nondecreasing, by \eqref{eq:inexact-local}, we have the estimate $\norm{e^{m+1}}_\infty\leq q\norm{e^m}_\infty$ as long as $\norm{e^m}_\infty\leq r_0$.
By induction, the iterates remain in the ball and converge to $z^*$.
If $\eta_m\to0$, the factor in \eqref{eq:inexact-local} tends to zero, which proves the superlinear convergence.
\end{proof}
\subsection{A sufficient geometric condition}\label{sec:suff_geo}
The interface condition can be verified through the geometry of the active interface.
On one-dimensional intervals, we suppose that each fiber $y\mapsto s^*(x,y)$ can be split into at most $J$ subintervals on each of which it is monotone, with $|\partial_ys^*|\geq\gamma>0$ almost everywhere on $\{|s^*|\leq t_0\}$.
We also assume the same in the $x$-direction.
We call the active interface \emph{transverse} in this case. In other words, the slack crosses zero with a derivative bounded away from zero in both variables, which is a nondegeneracy condition on the interface.
We will see in \cref{lem:Lip} that $s^*$ is Lipschitz for Lipschitz data, so each fiber is absolutely continuous.

Let $0<t\le t_0$, and let $y_1<y_2$ lie in the same subinterval with $|s^*(x,y_1)|,|s^*(x,y_2)|\le t$. Since the fiber is monotone there, we have $|s^*(x,y)|\le t$ for all $y\in[y_1,y_2]$. Hence, we obtain
\[
 2t\ge|s^*(x,y_2)-s^*(x,y_1)|=\int_{y_1}^{y_2}|\partial_ys^*(x,y)|\dd y\ge\gamma(y_2-y_1).
\]
Thus, on each subinterval, $\{y:|s^*(x,y)|\leq t\}$ is an interval of length at most $2t/\gamma$. We sum over the at most $J$ subintervals and argue in the same way in the $x$-direction, and we obtain the boundary-layer estimate
\begin{equation}\label{eq:transverse}
 m_*(t)\leq(2J/\gamma)t,\qquad 0<t\leq t_0.
\end{equation}
Since $\nu_*\le m_*$, \cref{thm:local} applies with $\alpha=1$. Thus, a transverse active interface gives local $q$-order at least $2$ for the continuum Newton iteration.

In higher dimensions, suppose that $\nabla_ys^*\neq0$ almost everywhere
on the boundary layer $\{|s^*|\le t_0\}$, and that the level-set integrals
\[\int_{\{s^*(x,\cdot)=r\}}|\nabla_y s^*|^{-1}\dd\mathcal H^{d_Y-1}\]
are bounded by $\Gamma$ uniformly in $x$ and $|r|\leq t_0$. Then, for $0<t\le t_0$, the coarea formula gives
\[
 |\{y:|s^*(x,y)|\leq t\}|
 =\int_{-t}^{t}\int_{\{s^*(x,\cdot)=r\}}|\nabla_y s^*|^{-1}\dd\mathcal H^{d_Y-1}\dd r
 \leq2\Gamma t.
\]
With the analogous bound in $x$, this is \eqref{eq:transverse} with $2J/\gamma$ replaced by
$2\Gamma$. Transversality is only one sufficient condition. Our results require only the boundary-layer estimate itself.
\section{Mesh independence and supercloseness}\label{sec:mesh}
In this section, we connect the function-space results of \cref{sec:local} to the discrete problems.
On a fixed mesh, the residual is piecewise affine, and the Newton iteration terminates in one step near the solution. However, the smallest nonzero discrete slack need not be bounded below uniformly in the mesh.
We therefore control all sufficiently fine meshes through the continuum interface, and we then handle the finitely many remaining refinement levels.
We also use the interface condition to compare the discrete potentials with the cell averages of the continuum potentials.
Throughout, we identify discrete roots with their lifts.

\subsection{Uniform stability of the roots}\label{sec:uniform-stability}
For cellwise data $f_h,g_h,C_h$ satisfying \eqref{eq:data}, we denote
\[
	\eta_{f,\infty}=\norm{f_h-f}_{L^\infty(X)},\quad
	\eta_{g,\infty}=\norm{g_h-g}_{L^\infty(Y)},\quad
	\eta_{C,\infty}=\norm{C_h-C}_{L^\infty(X\times Y)},
\]
and define
\begin{equation}\label{eq:zeta}
 \zeta_h=\frac1c\max\left\{
 \frac b\eps\eta_{C,\infty}+\theta_X^*\eta_{f,\infty},\
 \frac a\eps\eta_{C,\infty}+\theta_Y^*\eta_{g,\infty}\right\},
\end{equation}
with $\theta_X^*,\theta_Y^*$ from \eqref{eq:theta-star}.
\begin{lemma}\label{lem:uniform-stability}
	Let $z^*=(u^*,v^*)$ be the root of $F$, and let $z_h^*=(u_h^*,v_h^*)$ be the array root of $F_h$. By \cref{prop:reduction}, $\Lsf z_h^*$ is the root of the continuum residual $\widehat F_h$ with data $f_h,g_h,C_h$. Let $s^*=\T z^*-C-\eps\delta$ and $s_h^*=\T \Lsf z_h^*-C_h-\eps\delta$ be the corresponding slacks. Then
\begin{equation}\label{eq:uniform-stability}
 \norm{\Lsf z_h^*-z^*}_\Bb\leq\zeta_h,\qquad
 \norm{s_h^*-s^*}_{L^\infty(X\times Y)}\leq\lambda_h:=2\zeta_h+\eta_{C,\infty}.
\end{equation}
In particular, the roots and slacks converge in $L^\infty$ when the data do.
For Lipschitz data and cell averages or cell representative values,
$\lambda_h=O(h)$.
\end{lemma}
\begin{proof}
Since $\widehat F_h(z_h^*)=0=F(z^*)$, the estimate \eqref{eq:secant} for $\widehat F_h$ gives
\[
 \norm{z_h^*-z^*}_\Bb
 \leq c^{-1}\norm{\widehat F_h(z_h^*)-\widehat F_h(z^*)}_\Bb
 =c^{-1}\norm{\widehat F_h(z^*)-F(z^*)}_\Bb.
\]
By \eqref{eq:residual-intro}, \eqref{eq:p}, $p$ is $1/\eps$-Lipschitz, and the bounds $0\le 1-u^*\le\theta_X^*,\,0\le 1-v^*\le\theta_Y^*$, we estimate that
\begin{align*}
\norm{\widehat{F}_h(z^*)-F(z^*)}_{\Bb}&\le \norm{\T^*[p(\T z^*-C_h)-p(\T z^*-C)]+ ((f_h-f)(u^*-1),(g_h-g)(v^*-1))}_{\Bb}\\
&\le \norm{\eps^{-1}\T^*|C-C_h|+ \theta_X^*|f_h-f|,\theta_Y^*|g_h-g|)}_{\Bb}\\
&\le  \max\left\{
 \frac b\eps\eta_{C,\infty}+\theta_X^*\eta_{f,\infty},\
 \frac a\eps\eta_{C,\infty}+\theta_Y^*\eta_{g,\infty}\right\}.
\end{align*}
Thus, we obtain the first estimate in \eqref{eq:uniform-stability}.

For the second estimate in \eqref{eq:uniform-stability}, we use
\[
	s_h^*-s^*=\T(z_h^*-z^*)-(C_h-C),\quad \text{and}\quad \norm{\T(z_h^*-z^*)}_{L^\infty(X\times Y)} \leq2\norm{z_h^*-z^*}_\Bb.
\]
Hence, we obtain \[\norm{s_h^*-s^*}_{L^\infty(X\times Y)}\le2\zeta_h+\eta_{C,\infty}=\lambda_h.\]
Thus, $L^\infty$ convergence of the data gives $L^\infty$ convergence of the roots and slacks.
For Lipschitz data, cell averages and cell representative values satisfy $\eta_{f,\infty}\le L_fh$, $\eta_{g,\infty}\le L_gh$, and $\eta_{C,\infty}\le L_Ch$. Inserting these bounds into \eqref{eq:zeta}, we obtain $\lambda_h=O(h)$.
\end{proof}

\medskip
\noindent\textit{A computable residual radius.}
For any array potential $z_h$, we set
\[
 r_{\infty,h}=\frac{\norm{F_h(z_h)}_\infty}{\mu_h}.
\]
The secant estimate \eqref{eq:secant} on the cell space, with $\mu_h$ in place of $c$, gives
$\norm{z_h-z_h^*}_\infty\le r_{\infty,h}$. Thus, the distance to the discrete root $z_h^*$ is bounded by a quantity that depends only on the residual $F_h(z_h)$.

We can use $r_{\infty,h}$ to improve data constant in \cref{prop:stability}. That proposition requires upper bounds $\theta_X\ge\norm{1-u^*}_{L^\infty(X)}$ and $\theta_Y\ge\norm{1-v^*}_{L^\infty(Y)}$, and the a priori choice is $\theta_X^*,\theta_Y^*$ from \eqref{eq:theta-star}.
Suppose that the maximum-norm data errors $\eta_{f,\infty},\eta_{g,\infty},\eta_{C,\infty}$ are known. Then $r_{\infty,h}$ and \cref{lem:uniform-stability} give computable upper bounds $B_{X,h}\ge\norm{k_X^*}_{L^\infty(X)}$ and $B_{Y,h}\ge\norm{k_Y^*}_{L^\infty(Y)}$ for the marginals of the continuum plan $k^*$.
Since $1-u^*=k_X^*/f$ and $1-v^*=k_Y^*/g$ by \eqref{eq:optimality}, and $f,g\ge c$, the numbers $\theta_X=\min\{\theta_X^*,B_{X,h}/c\}$ and $\theta_Y=\min\{\theta_Y^*,B_{Y,h}/c\}$ are admissible in \cref{prop:stability}. They improve on $\theta_X^*,\theta_Y^*$ whenever $B_{X,h}<c\,\theta_X^*$ or $B_{Y,h}<c\,\theta_Y^*$. We give the details in \cref{S-sec:marginal-bounds} of the supplementary materials.

\subsection{Mesh-independent local convergence}
We now transfer the local theory of \cref{sec:local} to a sequence of refined partitions.

\begin{theorem}
\label{thm:mesh-local}
Consider partitions indexed by $\ell\in\mathbb N$, with mesh sizes $h_\ell\to0$ and cellwise data satisfying \eqref{eq:data} and converging uniformly to $f,g,C$.
Let $m_*,\,\nu_h$ be from \eqref{eq:layer-intro}, \eqref{eq:discrete-layer}, and assume $m_*(t)\to0$ as $t\downarrow0$, and set
\begin{equation}\label{eq:uniform-layer}
 \widehat\nu(t)=\sup_{\ell\geq1}\nu_{h_\ell}(t),\qquad
 \widehat\omega(r)=\frac{2\widehat\nu(2r)}{c\eps}.
\end{equation}
Then $\widehat\nu(t)\to0$ as $t\downarrow0$. Choose $r_0>0$ with
$\widehat\omega(r_0)<1$. On every mesh, let $z_{h_\ell}^*$ be the minimer of $\G_{h_\ell}$ in \eqref{eq:finite-dual}. The full-step Newton iteration \eqref{eq:Newton} started from $z_{h_\ell}^0$ with
$\norm{z_{h_\ell}^0-z_{h_\ell}^*}_\infty\leq r_0$ converges.
Writing $e_{h_\ell}^m=z_{h_\ell}^m-z_{h_\ell}^*$, we have
\begin{equation}\label{eq:uniform-newton}
 \norm{e_{h_\ell}^{m+1}}_\infty
 \leq\widehat\omega(\norm{e_{h_\ell}^m}_\infty)
                    \norm{e_{h_\ell}^m}_\infty,
 \qquad\norm{e_{h_\ell}^{m+1}}_h
 \leq\widehat\omega(\norm{e_{h_\ell}^m}_\infty)
 \norm{e_{h_\ell}^m}_h.
\end{equation}
The Armijo condition \eqref{eq:Armijo} accepts the step size $\tau=1$ on every mesh whenever \eqref{eq:full-step} holds with $\widehat\omega$
in place of $\omega$.

If $m_*(t)\leq At^\alpha$ for $0<t\leq t_A$, with
$A>0$, $\alpha>0$, and $t_A>0$, then in addition, for every $h=h_\ell$,
\begin{equation}\label{eq:two-scale-newton}
\begin{aligned}
 \norm{e_h^{m+1}}_\infty
 &\leq\frac{2A}{c\eps}(2\norm{e_h^m}_\infty+\lambda_h)^\alpha
                       \norm{e_h^m}_\infty\\[2pt]
 &\qquad\text{whenever } 2\norm{e_h^m}_\infty+\lambda_h\leq t_A.
\end{aligned}
\end{equation}
\end{theorem}
\begin{proof}
	We recall that $s^*,\,s_h^*$ are the slacks as in \eqref{eq:optimal-slack} corresponding to $z^*,\,\Lsf z_h^*$.
	By \cref{lem:uniform-stability}, we have $\norm{s_h^*-s^*}_\Bb\leq\lambda_h$. Hence, if $0<|s_h^*|\le t$, then, for almost everywhere
\[
	|s^*|\le|s_h^*|+|s_h^*-s^*|\le t+\lambda_h.
\]
Thus, up to a null set,
\[
	\{(x,y):0<|s_h^*(x,y)|\leq t\}
	\subseteq\{(x,y):|s^*(x,y)|\leq t+\lambda_h\}.
\]
Consequently, we obtain
\begin{equation}\label{eq:layer-transfer}
 \nu_h(t)\leq m_*(t+\lambda_h).
\end{equation}
The same argument, applied to $\{|s_h^*|\leq t\}$, gives $m_h(t)\leq m_*(t+\lambda_h)$.

Now, we fix $\chi>0$ and choose $t_0>0$ satisfying $m_*(2t_0)<\chi$.
Since $\lambda_{h_\ell}\to0$, we can also choose $\ell_0$ such that
$\lambda_{h_\ell}\leq t_0$ for $\ell\geq\ell_0$.
Then, by \eqref{eq:layer-transfer} and the monotonicity of $m_*$, we have
\[
 \nu_{h_\ell}(t)\leq m_*(t_0+\lambda_{h_\ell})\le m_*(2t_0)<\chi
 \qquad(\ell\geq\ell_0,\ 0<t\leq t_0).
\]
Each of the finitely many refinement levels $\ell<\ell_0$ has only finitely many nonzero slack values.
We choose $t_1>0$ smaller than all of these values, or $t_1=1$ if there are none. Then $\nu_{h_\ell}(t)=0$ for every $\ell<\ell_0$ and $0<t\leq t_1$.
Thus, $\widehat\nu(t)<\chi$ for
$0<t\leq\min\{t_0,t_1\}$, which proves that $\widehat\nu(t)\to0$.

On each mesh, \cref{lem:remainder} holds with
$\nu_{h_\ell}\leq\widehat\nu$, and \cref{lem:inverse} holds with
$\mu_{h_\ell}\geq c$. Hence, as in the proof of \cref{thm:local}, the Newton step with step size $\tau=1$ satisfies
\[
 \norm{e_{h_\ell}^{m+1}}_\infty
 \le\frac{2\nu_{h_\ell}(2\norm{e_{h_\ell}^m}_\infty)}{\mu_{h_\ell}\eps}\norm{e_{h_\ell}^m}_\infty
 \le\widehat\omega(\norm{e_{h_\ell}^m}_\infty)\norm{e_{h_\ell}^m}_\infty,
\]
and similarly $\norm{e_{h_\ell}^{m+1}}_h\le \widehat{\omega}(\norm{e_{h_\ell}^m}_\infty)\norm{e_{h_\ell}^m}_h$.
This is \eqref{eq:uniform-newton}. Each $\nu_{h_\ell}$ is nondecreasing, and hence so are $\widehat\nu$ and $\widehat\omega$. If $\norm{e_{h_\ell}^m}_\infty\le r_0$, then \eqref{eq:uniform-newton} with the maximum norm gives
\[
 \norm{e_{h_\ell}^{m+1}}_\infty\le\widehat\omega(r_0)\norm{e_{h_\ell}^m}_\infty\le r_0.
\]
By induction, all iterates $z_{h_\ell}^m$ remain in the ball of radius $r_0$ around $z_{h_\ell}^*$, and \[\norm{e_{h_\ell}^m}_\infty\le\widehat\omega(r_0)^m\norm{e_{h_\ell}^0}_\infty\to0,\qquad\norm{e_{h_\ell}^m}_h\le\widehat\omega(r_0)^m\norm{e_{h_\ell}^0}_h\to0.\]
Since $\widehat\omega(r)\to0$ as $r\downarrow0$, the ratio of successive errors tends to zero with a modulus that does not depend on $\ell$.

The term $\omega(r)[ 1/2+2\sigma\sqrt\kappa+\kappa\big(\sigma+1/2\big)\omega(r) ]$ of the condition \eqref{eq:full-step} is nondecreasing in both the modulus $\omega$ and the condition number $\kappa=L/c$.
Since $L_h/\mu_h\leq L/c$, the argument in the proof of \cref{thm:local} for accepting the step size $\tau=1$ applies on every mesh whenever
\eqref{eq:full-step} holds with $\widehat\omega$ in place of $\omega$.

Finally, we use \eqref{eq:layer-transfer} with $t=2\norm{e_h^m}_\infty$ and the power bound $m_*(t)\le At^\alpha$. Whenever $2\norm{e_h^m}_\infty+\lambda_h\le t_A$, we obtain
\[
 \norm{e_h^{m+1}}_\infty
 \le\frac{2\nu_h(2\norm{e_h^m}_\infty)}{\mu_h\eps}\norm{e_h^m}_\infty
 \leq\frac{2m_*(2\norm{e_h^m}_\infty+\lambda_h)}{c\eps}\norm{e_h^m}_\infty
 \le\frac{2A}{c\eps}(2\norm{e_h^m}_\infty+\lambda_h)^\alpha\norm{e_h^m}_\infty,
\]
which is \eqref{eq:two-scale-newton}.
\end{proof}

\begin{remark}\label{rem:two-scale}
The radius $r_0$ in \cref{thm:mesh-local} is common to the fixed
refinement sequence under consideration. It depends on the whole
sequence through $\widehat\nu$, including its finitely many coarse
levels. The theorem therefore gives existence of a common radius,
rather than a computable formula from continuum data alone.
Under the input-data curvature criterion of \cref{thm:curvature},
\cref{cor:explicit-radius} gives an explicit radius on every
sufficiently fine mesh.

The common modulus in \eqref{eq:uniform-newton} gives superlinear convergence uniformly over the sequence. A power bound $m_*(t)\le At^\alpha$ for the continuum interface does not, by itself, give a common discrete quadratic constant at arbitrarily small algebraic error.
For $\alpha=1$, \eqref{eq:two-scale-newton} reads
\[
 \norm{e_h^{m+1}}_\infty
 \leq\frac{2A}{c\eps}
      (2\norm{e_h^m}_\infty+\lambda_h)\norm{e_h^m}_\infty.
\]
Within the range of \eqref{eq:two-scale-newton}, the right-hand
side is bounded by $(6A/(c\eps))\norm{e_h^m}_\infty^2$
whenever $\lambda_h\leq\norm{e_h^m}_\infty$.
For Lipschitz data and cell averages or cell representative values,
$\lambda_h=O(h)$.

When $\norm{e_h^m}_\infty<\lambda_h$, the factor in \eqref{eq:two-scale-newton} is only bounded by a multiple of $\lambda_h^\alpha$, so the estimate no longer shows a quadratic rate. This is a limitation of the estimate, not of the iteration. On a fixed mesh, let $\sigma_h$ be the smallest nonzero value of $|s_{h,ij}^*|$. If $2\norm{e_h^m}_\infty<\sigma_h$, then $\nu_h(2\norm{e_h^m}_\infty)=0$, so \cref{lem:remainder} gives a vanishing remainder and $e_h^{m+1}=0$. Since the iterates converge, each fixed finite problem eventually reaches this neighborhood. Thus, a minimum nonzero slack $\sigma_h$ that shrinks with $h$ does not prevent a common superlinear neighborhood. However, the theorem gives no bound on the number of Newton steps that is uniform in $h$.
\end{remark}

\begin{remark}
\label{rem:lower-bound-nonuniform}
\Cref{thm:mesh-local} fixes $f,g,C,\eps,\delta$ and lets the mesh $h_\ell$ refine. It does not assert a common boundary-layer function or a common
superlinear error modulus across a family of lower bounds $\delta$.
The following example shows why this distinction is needed.

Take $X=Y=[0,1]$, $f=g=1$, $C=0$, $\eps=2$, and
$\delta=\tfrac12+\theta$ with $0<\theta<\tfrac12$.
These data satisfy \eqref{eq:data}, since $\delta\max\{a,b\}=\delta<1=c$. We claim that the exact solution is
\[
 k_\delta^*=\delta,\qquad
 u_\delta^*=v_\delta^*=\tfrac12-\theta,\qquad
 s_\delta^*=-4\theta.
\]
Indeed, the constant plan $k=\delta$ has marginals $k_X=k_Y=\delta$, so \eqref{eq:optimality} gives $u=v=1-\delta=\tfrac12-\theta$. The corresponding slack is $u+v-C-\eps\delta=1-2\theta-1-2\theta=-4\theta<0$, so $p(\T z-C)=\delta=k$. By the uniqueness in \cref{prop:duality}, this is the solution.
Hence, the punctured boundary-layer function is
\[
 \nu_\delta(t)=
 \begin{cases}
  0,&t<4\theta,\\
  1,&t\geq4\theta.
 \end{cases}
\]
For each fixed $\delta$, we have $\nu_\delta(t)\to0$ as
$t\downarrow0$. However,
\[
 \sup_{0<\theta<1/2}\nu_{1/2+\theta}(t)=1
 \qquad(t>0),
\]
so this convergence is not uniform over the family.

The Newton errors $e^m=z^m-z^*$ also show the lack of uniform superlinear reduction. We start at $u=v=\tfrac12+\theta$, whose maximum-norm
distance from the root is $2\theta$.
The current slack is $1+2\theta-2\delta=0$, so the slope $\vartheta=1$ is permitted, and the current plan is $\delta$. Hence $F(z)=(2\theta,2\theta)$. For the constant direction $d=(d_0,d_0)$, we have $H(1)d=(2d_0,2d_0)$, so the Newton step with step size $\tau=1$ is $d_0=-\theta$. It gives $u^+=v^+=\tfrac12$, and therefore
\[
 \frac{\norm{z^+-z_\delta^*}_\Bb}
      {\norm{z-z_\delta^*}_\Bb}
 =\frac{\theta}{2\theta}=\frac12.
\]
The initial error tends to zero as $\theta\downarrow0$, but the
error ratio remains $1/2$. Thus, no common error modulus tending
to zero can bound these steps over the family.

The inverse estimate in \cref{lem:inverse} and the global solver estimates in \cref{S-sec:global} remain uniform in the admissible lower bound $\delta$ at fixed $\eps,c$. To obtain a local superlinear estimate that is uniform in $\delta$ by the present argument, one needs an interface condition that is uniform over the family. The standing assumptions \eqref{eq:data} alone do not provide it. The results for a fixed $\delta$ in \cref{thm:local,thm:mesh-local,thm:superclose} are unaffected.
\end{remark}

In practice, the Newton iteration on a fine mesh is often started from an approximate root computed on a coarser mesh. Such a coarse potential can be used without interpolation when the partitions are nested, that is, when for all levels $\ell'\ge\ell$, every cell of the partitions of $X$ and $Y$ at level $\ell'$ is contained in a cell of the corresponding partition at level $\ell$. Then every potential in the cell space at level $\ell$ also belongs to the cell space at every finer level, with the same maximum norm. The following corollary bounds its distance to the discrete roots on these finer levels.

\begin{corollary}\label{cor:nested}
Under the assumptions of \cref{thm:mesh-local}, suppose that the partitions are nested.
Fix a level $\ell$, write $H=h_\ell$, and let $\widetilde z_H$ be a potential in the cell space at level $\ell$ with $\norm{F_H(\widetilde z_H)}_\infty\leq t_H$.
Then, for every level $\ell'\ge\ell$, with $h=h_{\ell'}$, we have
\begin{equation}\label{eq:nested}
 \norm{\widetilde z_H-z_h^*}_\infty
 \leq t_H/c+\zeta_H+\zeta_h,
\end{equation}
where $\zeta_H$ and $\zeta_h$ denote the quantity \eqref{eq:zeta} for the data at levels $\ell$ and $\ell'$, respectively. 
In particular, if $t_H/c+\zeta_H+\zeta_h\le r_0$, then $\widetilde z_H$ lies in the ball of radius $r_0$ around $z_h^*$. By \cref{thm:mesh-local}, the Newton iteration at level $\ell'$ started from $\widetilde z_H$ then converges and satisfies \eqref{eq:uniform-newton}. Since $\zeta_h\to0$, this holds for all sufficiently large $\ell'$ whenever $t_H/c+\zeta_H<r_0$.
\end{corollary}

\begin{proof}
Since $F_H(z_H^*)=0$, the secant estimate \eqref{eq:secant} on the cell space at level $\ell$, with $\mu_H\ge c$ in place of $c$, gives
\[
 \norm{\widetilde z_H-z_H^*}_\infty
 \leq\frac1{\mu_H}\norm{F_H(\widetilde z_H)}_\infty
 \leq\frac{t_H}{c}.
\]
Let $\ell'\ge\ell$ and $h=h_{\ell'}$. By the triangle inequality and \cref{lem:uniform-stability} at levels $\ell$ and $\ell'$, we obtain
\begin{align*}
 \norm{\widetilde z_H-z_h^*}_\infty
 &\leq\norm{\widetilde z_H-z_H^*}_\infty
      +\norm{\Lsf z_H^*-z^*}_\Bb
      +\norm{z^*-\Lsf z_h^*}_\Bb\\
 &\leq t_H/c+\zeta_H+\zeta_h,
\end{align*}
which is \eqref{eq:nested}.
Since the partitions are nested, $\widetilde z_H$ belongs to the cell space at level $\ell'$. Hence, if $t_H/c+\zeta_H+\zeta_h\le r_0$, it is an admissible starting point in \cref{thm:mesh-local}, which gives convergence and \eqref{eq:uniform-newton}. 

Finally, $\zeta_h\to0$ as $\ell'\to\infty$ by \cref{lem:uniform-stability}, since the data converge in the maximum norm. Thus, if $t_H/c+\zeta_H<r_0$, the condition holds for all sufficiently large $\ell'$.
\end{proof}
\subsection{Regularity of the potentials}
The supercloseness estimate of \cref{thm:superclose} requires Lipschitz bounds for the optimal potentials $z^*$ and the optimal slack $s^*$.
For the rest of this section, we assume that $f,g,C$ are Lipschitz on their domains, with constants $L_f,L_g,L_C$.
\begin{lemma}\label{lem:Lip}
The optimal potentials $u^*,v^*$ have representatives with Lipschitz constants
\begin{equation}\label{eq:potential-Lip}
 L_u=\frac{\theta_X^*}{c}L_f+\frac b{c\eps}L_C,\qquad
 L_v=\frac{\theta_Y^*}{c}L_g+\frac a{c\eps}L_C,
\end{equation}
with $\theta_X^*,\theta_Y^*$ from \eqref{eq:theta-star}.
Suppose, in addition, that $X$ is convex. Then $L_u$ can be replaced with
\[
 \max\Big\{L_C,\frac{\theta_X^*L_f}{c}\Big\}\le\max\Big\{L_C,\frac{2L_f}{c}\Big\}.
\]
This bound does not depend on $\eps$. If $Y$ is convex, the same holds for $v^*$, with $L_g$ and $\theta_Y^*$ in place of $L_f$ and $\theta_X^*$.

Moreover, $s^*$ is Lipschitz on $X\times Y$, with constant $L_s=L_u+L_v+L_C$.
\end{lemma}
\begin{proof}
For fixed $v^*$ and almost every $x\in X$, we define
\[
 \Phi_x(q)=\int_Y p(q+v^*(y)-C(x,y))\dd y+f(x)(q-1).
\]
The function $\Phi_x$ is continuous. Since $p$ is nondecreasing and $f(x)\ge c$, $\Phi_x$ also satisfies
\[
 \Phi_x(q)-\Phi_x(\widehat q)\geq c(q-\widehat q)
 \qquad\text{for }q>\widehat q.
\]
Consequently, $\Phi_x(q)\to\pm\infty$ as $q\to\pm\infty$, and $\Phi_x$ has a unique root.
By the first component of \eqref{eq:residual-intro}, this root
equals $u^*(x)$ for almost every $x$.

We choose a full-measure subset $X_0\subset X$ on which this identity and the bound $|u^*(x)-1|\leq\theta_X^*$ hold.
Let $x,x'\in X_0$. Since $\Phi_x(u^*(x))=\Phi_{x'}(u^*(x'))=0$, the monotonicity bound gives
\[
 c|u^*(x)-u^*(x')|\le|\Phi_x(u^*(x'))-\Phi_x(u^*(x))|=|\Phi_x(u^*(x'))-\Phi_{x'}(u^*(x'))|.
\]
With $q=u^*(x')$, we write
\[
 \Phi_x(q)-\Phi_{x'}(q)=\int_Y\big[p(q+v^*-C(x,\cdot))-p(q+v^*-C(x',\cdot))\big]\dd y+\big(f(x)-f(x')\big)(q-1).
\]
Since $p$ is $1/\eps$-Lipschitz and $|Y|=b$, we estimate \[\Bigg|\int_Y\big[p(q+v^*-C(x,\cdot))-p(q+v^*-C(x',\cdot))\big]\dd y\Bigg|\le (b/\eps)L_C|x-x'|.\]
Then, we also see that \[\Big|\big(f(x)-f(x')\big)(q-1)\Big|\le L_f|x-x'|\,|q-1|.\] Hence, we obtain
\begin{align*}
 c|u^*(x)-u^*(x')|
 &\leq\frac b\eps L_C|x-x'|
       +L_f|x-x'|\,|u^*(x')-1|\\
 &\leq\left(\frac b\eps L_C+\theta_X^*L_f\right)|x-x'|.
\end{align*}
Thus, $u^*$ is Lipschitz on $X_0$ with constant
\[
 L_u=\frac{\theta_X^*}{c}L_f+\frac b{c\eps}L_C.
\]

By the McShane extension theorem, the restriction of $u^*$ to $X_0$ extends to a Lipschitz function on $X$ with the same constant. This extension is the desired representative.

Now let $X$ be convex, and fix a unit direction $e$. We first derive a formula for $D_eu^*$.
For fixed $y$, the map $x\mapsto s^*(x,y)$ is Lipschitz. We claim that
\begin{equation}\label{eq:positive-part-chain}
 D_e(s^*)_+=\ind_{\{s^*>0\}}D_es^*
 \qquad\text{almost everywhere in }X\times Y.
\end{equation}
After a rotation, we may assume $e=e_1$ and write $x=(t,\hat x)$ with $t\in\mathbb R$ and $\hat x\in\mathbb R^{d_X-1}$, so that $D_e=\partial_t$.

We first fix $\hat x$ and $y$ and work on one line. Let $I=\{t:(t,\hat x)\in X\}$, which is an interval since $X$ is convex, and set $\xi(t)=s^*((t,\hat x),y)$. Then $\xi$ is Lipschitz on $I$, so $\xi$ and $\xi_+$ are differentiable almost everywhere on $I$. We show that
\[
 (\xi_+)'(t)=\ind_{\{\xi>0\}}(t)\,\xi'(t)\qquad\text{for almost every }t\in I.
\]
When $\xi(t)\neq 0$, then $\xi(t')$ has the same sign as $\xi(t)$ when $t'$ is sufficiently close $t$. If $\xi(t)>0$ then $(\xi_+)'(t)=\xi'(t)$ and if $\xi(t)<0$ then $(\xi_+)'(t)=0'=0$. Thus, we now only need to consider the case $\xi(t)=0$. 

The set $\underline I=\{t:\xi(t)=0,\, t\text{ is isolated}\}$ forms a countable set. Indeed, for  $t\in\underline I$ we can always find $\underline{q}_t,\,\overline{q}_t\in \mathbb Q$ such that $\{t\}=(\underline{q}_t,\overline{q}_t)\cap \overline I$. Hence, $t\mapsto (\underline{q}_t,\overline{q}_t)$ forms an injective map from $\underline I$ to $\mathbb Q^2$.

For $t\in I\setminus \underline I$ where $\xi(t)=0$ and $\xi'(t)$ exists, we choose $t_n\in\{\xi=0\}$ with $t_n\neq t$ and $t_n\to t$, and obtain
\[
 \xi'(t)=\lim_{n\to\infty}\frac{\xi(t_n)-\xi(t)}{t_n-t}=0.
\]
Since $\xi_+$ also vanishes on $\{\xi=0\}$, the same argument gives $(\xi_+)'=0$ almost everywhere on $\{\xi=0\}$. 

Let $B\subset X\times Y$ be the set where \eqref{eq:positive-part-chain} fails. We have shown that, for all $\hat x$ and $y$, the set $\{t\in I:((t,\hat x),y)\in B\}$ is null. Hence, by Tonelli's theorem,
\[
 |B|=\int_Y\int_{\mathbb R^{d_X-1}}\big|\{t\in I:((t,\hat x),y)\in B\}\big|\dd\hat x\dd y=0.
\]

Let $A_x=\{y:s^*(x,y)>0\}$ and $a_x=|A_x|$. By \eqref{eq:optimality}, we have
\begin{equation}\label{eq:marginal-equation}
 \delta b+\eps^{-1}\int_Y(s^*)_+\dd y=f(1-u^*).
\end{equation}
Since $D_es^*=D_eu^*-D_eC$, we then observe that
\[
 \eps^{-1}\int_{A_x}\big(D_eu^*(x)-D_eC(x,y)\big)\dd y
 =-f(x)D_eu^*(x)+(1-u^*(x))D_ef(x).
\]
Hence,
\begin{equation}\label{eq:fc-derivative}
 D_eu^*(x)=
 \frac{\eps^{-1}\int_{A_x}D_eC(x,y)\dd y+(1-u^*(x))D_ef(x)}
      {f(x)+a_x/\eps}.
\end{equation}
For almost every $x$, \eqref{eq:fc-derivative} and the bound $0\leq1-u^*\leq\theta_X^*$ give
\[
 |D_eu^*(x)|\leq\frac{(a_{x}/\eps)L_C+\theta_X^*L_f}{f(x)+a_{x}/\eps}
 =\frac{a_x/\eps}{f(x)+a_x/\eps}\,L_C+\frac{f(x)}{f(x)+a_x/\eps}\,\frac{\theta_X^*L_f}{f(x)}
 \leq\max\Big\{L_C,\frac{\theta_X^*L_f}{c}\Big\}.
\]
We claim that
\begin{equation}\label{eq:Lip-convex}
 |u^*(x')-u^*(x)|\le\max\Big\{L_C,\frac{\theta_X^*L_f}{c}\Big\}|x'-x|\qquad(x,x'\in X).
\end{equation}
Let $x,x'$ be interior points of $X$ with $x\neq x'$. After a rotation, we may assume $x'-x=|x'-x|\,e_1$, and we write points of $\mathbb R^{d_X}$ as $(t,\hat x)$ with $t\in\mathbb R$ and $\hat x\in\mathbb R^{d_X-1}$. Then $x=(t_0,\hat x_0)$ and $x'=(t_0',\hat x_0)$ for some $t_0<t_0'$.

We first fix $\hat x$ and work on one line. Let $I=\{t:(t,\hat x)\in X\}$, which is an interval since $X$ is convex, and set $u_{\hat x}(t)=u^*((t,\hat x))$. Then $u_{\hat x}$ is Lipschitz, and hence absolutely continuous, on $I$. Suppose that $|u_{\hat x}'|\le\max\{L_C,\theta_X^*L_f/c\}$ almost everywhere on $I$. Then, for $t_1<t_2$ in $I$,
\[
 |u_{\hat x}(t_2)-u_{\hat x}(t_1)|
 =\Big|\int_{t_1}^{t_2}u_{\hat x}'(t)\dd t\Big|
 \le\max\Big\{L_C,\frac{\theta_X^*L_f}{c}\Big\}(t_2-t_1).
\]

Let $\underline B\subset X$ be the null set on which $|\partial_tu^*|>\max\{L_C,\theta_X^*L_f/c\}$. By Tonelli's theorem, we have
\[
	0=|\underline B|=\int_{\mathbb R^{d_X-1}}\big|\{t\in I:(t,\hat x)\in \underline B\}\big|\dd\hat x.
\]
Thus, the assumption $|u_{\widehat{x}}'|\le\max\{L_C,\theta_X^*L_f/c\}$ holds for almost every $\hat x$. Since $x,x'$ are interior points and $X$ is convex, $(t,\hat x)\in X$ for all $t\in[t_0,t_0']$ and all $\hat x$ near $\hat x_0$. We choose $\hat x_n\to\hat x_0$ such that $|u_{\widehat{x}_n}'|\le\max\{L_C,\theta_X^*L_f/c\}$, and obtain
\[
 |u^*((t_0',\hat x_n))-u^*((t_0,\hat x_n))|\le\max\Big\{L_C,\frac{\theta_X^*L_f}{c}\Big\}|x'-x|.
\]
Letting $n\to\infty$ and using the continuity of $u^*$ gives \eqref{eq:Lip-convex} for interior points. Since the interior of $X$ is dense in $X$, \eqref{eq:Lip-convex} holds on $X$.
The other component follows by interchanging $X$ and $Y$.

Finally, by the definition of the slack,
\[
 |s^*(x,y)-s^*(x',y')|\le L_u|x-x'|+L_v|y-y'|+L_C|(x,y)-(x',y')|\le L_s|(x,y)-(x',y')|,
\]
with $L_s=L_u+L_v+L_C$.
\end{proof}
\subsection{Cell-average supercloseness}
We now compare the array root $z_h^*$ with the cell averages $\Q_h^{\Hh}z^*$. By the secant estimate \eqref{eq:secant} on the cell space, with $\mu_h\ge c$,
\[
 \norm{\Lsf z_h^*-\Q_h^\Hh z^*}_\Bb\le\frac1c\norm{F_h(\Q_h^\Hh z^*)}_\infty,
\]
so we try to bound the residual of the cell averages in this section. 
\begin{theorem}\label{thm:superclose}
	Recall cell averages $\Q_h,\,\Q_h^\Hh=(\Q_h^X,\Q_h^Y)$ from \eqref{eq:aver}.
	Assume that the data $f,\,g,\,C$ are Lipschitz and that $f_h=\Q_h^X f,\,g_h=\Q_h^Y g,\, C_h=\Q_h C$. Let $z_h^*=(u_h^*,v_h^*),\,z^*=(u^*,v^*)$ be the minimizers of $\G_h,\,\G$ in \eqref{eq:finite-dual}, \eqref{eq:G}. Let $s^*,\,m_*$ be from \eqref{eq:optimal-slack}, \eqref{eq:layer-intro}.
Let $L_u,L_v,L_s$ be any Lipschitz constants for $u^*,v^*,s^*$ on $X,Y,X\times Y$,
respectively, and set $L_{\rm cov}=\max\{L_fL_u,L_gL_v\}$. Then
\begin{equation}\label{eq:superclose}
 \norm{\Lsf z_h^*-\Q_h^{\Hh}z^*}_\Bb
 \leq\frac1c\left[
       \frac{L_sh}{4\eps}m_*(L_sh)+\frac{L_{\rm cov}}4h^2\right].
\end{equation}
If every cell $X_i$ and $Y_j$ is a product of intervals, then the two factors
$\tfrac14$ in \eqref{eq:superclose} can be replaced by $1/(4\sqrt3)$ and
$\tfrac1{12}$, respectively.

If $m_*(t)\leq At^\alpha$ for $0<t\leq t_A$, with
$A>0$, $\alpha>0$, and $t_A>0$, and $L_sh\leq t_A$, this becomes
\[
 \norm{\Lsf z_h^*-\Q_h^{\Hh}z^*}_\Bb
 \leq\frac{A L_s^{1+\alpha}}{4c\eps}h^{1+\alpha}
       +\frac{L_{\rm cov}}{4c} h^2.
\]
\end{theorem}
With $\Xi_1=AL_s^{1+\alpha}/(4c\varepsilon)$ and $\Xi_2 = L_{\rm cov}/(4c)$, this is the estimate \eqref{eq:principal-superclose} in \cref{thm:principal}. For transverse interfaces ($\alpha=1$), it is an $O(h^2)$ estimate.
\Cref{lem:Lip} supplies explicit choices of $L_u,L_v,L_s$. However, any
known upper bound may be used instead, including a smaller one confirmed
independently. For instance, $L_u=L_v=0$ when the potentials are constant,
as in \cref{sec:num-interface}.
\begin{proof}
	We denote $\bar z_h=\Q_h^{\Hh}z^*$ and $\tilde z_h=\Lsf^{-1}\bar z_h$. We first observe \[\T\bar z_h=\Q_h(\T z^*).\]
	Since $C_h=\Q_hC$, and by \eqref{eq:Gap}, \eqref{eq:p}, we obtain
\begin{align*}
	\Q_hk^*-k_{\bar{z}_h}&=\Q_hp(\T z^*-C)-p(\T \bar{z}_h-C_h)\\
	&=\Q_hp(\T z^*-C)-p(\Q_h(\T z^*-C))\\
	&=\varepsilon^{-1}( \Q_h (s^*)_+-(\Q_hs^*)_+)\ge 0.
\end{align*}
Here, the last estimate comes from Jensen's inequality.

On a cell $X_i\times Y_j$, if $s^*$ has one sign almost everywhere, then $\Q_h (s^*)_+-(\Q_hs^*)_+=0$.

We will focus on the cell $X_i\times Y_j$ such that $s^*$ takes both positive and negative values. On such cell, we have
\begin{align*}
	|s^*|\le\esssup_{X_i\times Y_j}s^*-\essinf_{X_i\times Y_j}s^*\le L_s \operatorname{diam}(X_i\times Y_j)\le L_s h.
\end{align*}
By $t_+=(t+|t|)/2$ and the Cauchy--Schwarz inequality, we observe that
\begin{align*}
	\Q_h (s^*)_+-(\Q_hs^*)_+ = \frac{1}{2}(\Q_h(|s^*|-|\Q_hs^*|))\le \frac{1}{2}\Q_h|s^*-\Q_hs^*|\le\frac{1}{2}(\Q_h|s^*-Q_h s^*|^2)^{1/2}.
\end{align*}
By \cref{lem:aver-ortho}, for $\overline s^*=(\esssup_{X_i\times Y_j}s^*+\essinf_{X_i\times Y_j}s^*)/2$, 
\begin{align}\label{eq:square-aver-est}
	\Q_h|s^*-\Q_h s^*|^2=\Q_h|s^*-\overline s^*|^2-|\Q_h s^*-\overline s^*|^2\le \frac{(\esssup_{X_i\times Y_j}s^*-\essinf_{X_i\times Y_j}s^*)^2}{4}\le \frac{L_s^2h^2}{4}.
\end{align}
Combine above estimates, on $X\times Y$, we have
\begin{align*}
	0\le \Q_h (s^*)_+-(\Q_hs^*)_+\le \frac{L_sh}{4}.
\end{align*}

Now, we consider the case $X_i\times Y_j$ is a product of intervals.
Let us denote by $\Q_h'$ a copy of $\Q_h$ but the average will be taken over $P'=(x',y')\in X_i\times Y_j$ and $\Q_h$ is the average over $P=(x,y)\in X_i\times Y_j$.
We then have the identity
\[\Q_h|s^*-Q_h s^*|^2=\Q_h (s^*)^2-(\Q_h s^*)^2=\frac{1}{2}(\Q_h(s^*)^2+\Q_h'(s^*)^2-2(\Q_h s^*)(\Q_h's^*))=\frac{1}{2}\Q_h\Q_h'|s^*(P)-s^*(P')|^2.\]
In a similar way, for $\overline P=(w_i\omega_j)^{-1}\int_{X_i\times Y_j}P\dd P$, we also have
\[\Q_h\norm{P-\overline{P}}^2_2=\frac{1}{2}\Q_h\Q_h'\norm{P-P'}^2_2.\]
Meanwhile, we can write $X_i\times Y_j = \prod_{n=1}^{d_X+d_Y}[\rho_n,\rho_n+\ell_n]$. Then, we compute
\[
	\Q_h\norm{P-\overline P_R}^2_2=\sum_{n=1}^{d_X+d_Y}\frac1{\ell_n}\int_{\rho_n}^{\rho_n+\ell_n}\Big(\tau-\rho_n-\frac{\ell_n}2\Big)^2\dd\tau=\sum_{n=1}^{d_X+d_Y}\frac{\ell_n^2}{12}=\frac{\operatorname{diam}(X_i\times Y_j)^2}{12}.
\]
Hence,
\begin{equation}\label{eq:box-variance}
		\Q_h|s^*-Q_h s^*|^2=\frac{1}{2}\Q_h\Q_h'|s^*(P)-s^*(P')|^2\leq \frac{L_s^2}{\Q_h\Q_h'\norm{P-P'}^2_2}=L_s^2\Q_h\norm{P-\overline{P}}^2_2=\frac{L_s^2}{12}\operatorname{diam}(X_i\times Y_j)^2.
\end{equation}
In that case, we obtain 
\begin{align*}
	0\le \Q_h (s^*)_+-(\Q_hs^*)_+\le \frac{L_sh}{4\sqrt{3}}.
\end{align*}

We now bound the marginals of $\Q_hk^*-k_{\bar z_h}$. If, on $X_i\times Y_j$, $\Q_h(s^*)_+\neq (\Q_h s^*)_+$, then $|s^*|\le L_sh$. Hence, for almost every $x\in X_i$,
\[\bigcup_{j:\Q_h(s^*)\neq(\Q_h s^*)_+\text{ on }X_i\times Y_j}Y_j\subset \{y:|s^*(x,y)|\le L_s h\}.\]
Therefore,
\[
	\big((\Q_hk^*-k_{\bar z_h})_X\big)_i
 \le\frac{L_sh}{4\eps}\sum_{j:\Q_h(s^*)\neq(\Q_h s^*)_+\text{ on }X_i\times Y_j}\omega_j
 \le\frac{L_sh}{4\eps}m_*(L_sh).
\]
By the same argument, the analogous bound holds for the $Y$-marginal.
For products of intervals, the factor $4$ is replaced by $4\sqrt3$.

By \eqref{eq:residual-intro}, \eqref{eq:Fh}, and the fact that $F(z^*)=0$, we compute
\begin{equation}\label{eq:residual-decomposition}
\begin{aligned}
	\Lsf F_h(\widetilde z_h) = \Lsf F_h(\widetilde z_h)-\Q_h^\Hh F(z^*)
	&=\bigl((k_{\bar z_h}-\Q_hk^*)_X,\,(k_{\bar z_h}-\Q_hk^*)_Y\bigr)\\
 &\quad+\bigl(f_h\Q_h^Xu^*-\Q_h^X(fu^*),\,g_h\Q_h^Yv^*-\Q_h^Y(gv^*)\bigr).
\end{aligned}
\end{equation}
From above estimate, we already have
\begin{align*}
	\norm{\bigl((k_{\bar z_h}-\Q_hk^*)_X,\,(k_{\bar z_h}-\Q_hk^*)_Y\bigr)}_\Bb\le \frac{L_sh}{4\varepsilon}m_*(L_sh).
\end{align*}

In a similar way to \eqref{eq:square-aver-est}, we have 
\[|f-f_h|^2=|f-\Q_h^Xf|^2\le \frac{L_f^2 h^2}{4},\qquad |u^*-\Q_h^Xu^*|^2\le \frac{L_u^2 h^2}{4}.\]
By Cauchy-Schwarz inequality, We estimate
\[
 |f_h\Q_h^Xu^*-\Q_h^X(fu^*)|
 =|\Q_h^X[(f-f_h)(u^*-\Q_h^Xu^*)]|\le(\Q_h^X|f-f_h|^2)^{1/2}(\Q_h^X|u^*-\Q_h^Xu^*|^2)^{1/2}\le \frac{L_fL_u h^2}{4}.
\]
If $X_i$ is a product of intervals, in a similar way to \eqref{eq:box-variance}, we can replace the bound by $L_fL_uh^2/12$.
In the same way, we bound $g_h\Q_h^Yv^*-\Q_h^Y(gv^*)$ by $L_gL_vh^2/4$, or by $L_gL_vh^2/12$.
Thus, we obtain the estimate
\[
	\norm{F_h(\widetilde z_h)}_\infty
 =\norm{\Lsf F_h(\widetilde z_h)}_\Bb
 \leq\frac{L_sh}{4\eps}m_*(L_sh)+\frac{L_{\rm cov}}4h^2,
\]
and the corresponding estimate for products of intervals.
Finally, the secant estimate \eqref{eq:secant}, the lift of $F_h$ in \eqref{eq:lift-F} on the cell space, and $F_h(z_h^*)=0$ give
\[
 \norm{z_h^*-\widetilde z_h}_\infty
 \leq\mu_h^{-1}\norm{F_h(\widetilde z_h)}_\infty
 \leq c^{-1}\norm{F_h(\widetilde z_h)}_\infty,
\]
which proves \eqref{eq:superclose}.
Under the power bound, we have $m_*(L_sh)\le A(L_sh)^\alpha$, which gives the remaining assertions.
\end{proof}

The supercloseness estimate compares potentials with cell averages. It
does not make a piecewise-constant plan a second-order approximation of a
nonconstant continuum plan. We now explain this by bounding the plan error from above and from below.

For the upper bound, we use
\[
 \T\Lsf z_h^*-C_h-\eps\delta=\T(\Lsf z_h^*-\bar z_h)+\Q_hs^*.
\]
Since $p$ is $1/\eps$-Lipschitz, this gives pointwise
\[
 |\Lsf K(z_h^*)-k^*|\le\frac1\eps\big(|\T(\Lsf z_h^*-\bar z_h)|+|\Q_hs^*-s^*|\big).
\]
Since $\norm{\T w}_2\le2\sqrt{ab}\norm{w}_\Bb$, we obtain
\begin{equation}\label{eq:plan-decomposition}
 \norm{\Lsf K(z_h^*)-k^*}_\Zz
 \leq\frac{2\sqrt{ab}}{\eps}\norm{\Lsf z_h^*-\Q_h^{\Hh}z^*}_\Bb
 +\frac1\eps\norm{\Q_hs^*-s^*}_\Hh.
\end{equation}
By \eqref{eq:square-aver-est}, we further obtain
\[
\frac1\eps\norm{\Q_hs^*-s^*}_\Hh=\frac{1}{\varepsilon}\Bigg(\sum_{(i,j)\in\N}(w_i\omega_j)(\Q_h|\Q_hs^*-s^*|^2)\Big|_{X_i\times Y_j}\Bigg)^{1/2}\le\frac{\sqrt{ab}\,L_sh}{2\eps},
\]
with $2$ replaced by $2\sqrt3$ for products of intervals by applying \eqref{eq:box-variance}.

For the lower bound, we note that $\Lsf K(z_h^*)$ is cellwise constant and that $\Q_h$ is the orthogonal projection onto cellwise constant functions by \cref{lem:aver-ortho}. Hence, we have
\[
 \norm{\Lsf K(z_h^*)-k^*}_\Zz\geq\norm{k^*-\Q_hk^*}_\Zz.
\]
The right-hand side vanishes only if $k^*$ is cellwise constant, and it can be of order $h$. For instance, for the transverse interface in \cref{sec:num-interface}, the plan $k^*=(1.3-x-y)_+$ is affine with nonzero gradient on $\{x+y<1.3\}$, so its deviation from its cell averages is of order $h$ there. Hence, the plan error is first order even when the potential supercloseness is quadratic.
\section{Input-data curvature criterion and an explicit convergence radius}
\label{sec:curvature}

The interface condition in \cref{thm:local} can also be verified through directional curvature of the input data.
For balanced quadratically regularized transport, sectional concavity and negligible sectional zero sets underlie the uniform marginal boundary-layer estimate of Gonz\'alez-Sanz, Nutz, and Riveros Valdevenito \cite[Lemma~3.2(i) and (43)--(44)]{gonzalez2025linear}.
Strong convexity of reparametrized scalar potentials is established in \cite[Theorem~2.2]{gonzalez2026scalar}, while concavity yields exterior nondegeneracy in \cite[Proposition~6.3]{gonzalez2026stability}.
Parameter-uniform small-level-set conditions also occur in sparse optimal control \cite[Section~3]{WW2011sparsity}.
Here, we give sufficient inequalities for the Pearson marginal equations, with constants uniform in the lower bound $\delta$.
The argument first bounds the directional derivatives of the unknown potentials by the data (\cref{lem:fc-potential}). It then compares the slack at the midpoint of a segment with its values at the endpoints (\cref{thm:curvature}).

Fix $\varepsilon>0$ and $\bar\delta\ge0$ with $\bar\delta\max\{a,b\}<c$, and consider $0\le\delta\le\bar\delta$.
The data $f,g,C$ and the domains are fixed and satisfy~\eqref{eq:data}.
In addition, we assume in this section that $X$ and $Y$ are bounded convex domains and that $f,g,C$ admit $C^{1,1}$ extensions to neighborhoods of their closures. 
Let $z_\delta^*=(u_\delta^*,v_\delta^*)$ and $s_\delta^*=u_\delta^*+v_\delta^*-C-\varepsilon\delta$.
We use the continuous representatives of the potentials from \cref{lem:Lip}, and define the marginal boundary-layer function of $s_\delta^*$ by
\begin{equation}\label{eq:fc-layer-definition}
 m_{*,\delta}(t)=\max\left\{
 \operatorname*{ess\,sup}_{x\in X}
 |\{y\in Y:|s_\delta^*(x,y)|\le t\}|,\quad
 \operatorname*{ess\,sup}_{y\in Y}
 |\{x\in X:|s_\delta^*(x,y)|\le t\}|
 \right\}.
\end{equation}

\begin{lemma}
\label{lem:fc-potential}
For every $\delta\in[0,\bar\delta]$, the continuous representatives satisfy $-1\le u_\delta^*\le1$ and $-1\le v_\delta^*\le1$ at every point of $X$ and $Y$.
For unit directions $e_X\in\mathbb R^{d_X}$ and
$e_Y\in\mathbb R^{d_Y}$, define
\begin{equation}\label{eq:fc-first-data}
\begin{aligned}
	L_{f,e}&=\|D_{e_X}f\|_{L^\infty(X)},&
	L_{C,X,e}&=\|D_{e_X}C\|_{L^\infty(X\times Y)},&
 U_{X,e}&=\max\{L_{C,X,e},2L_{f,e}/c\},\\
 L_{g,e}&=\|D_{e_Y}g\|_{L^\infty(Y)},&
 L_{C,Y,e}&=\|D_{e_Y}C\|_{L^\infty(X\times Y)},&
 U_{Y,e}&=\max\{L_{C,Y,e},2L_{g,e}/c\}.
\end{aligned}
\end{equation}
Then $|D_{e_X}u_\delta^*|\le U_{X,e}$ and
$|D_{e_Y}v_\delta^*|\le U_{Y,e}$ almost everywhere. In particular, we have
\[
 |u_\delta^*(x+re_X)-u_\delta^*(x)|\le U_{X,e}|r|,\qquad
 |v_\delta^*(y+re_Y)-v_\delta^*(y)|\le U_{Y,e}|r|,
\]
whenever both points lie in $X$, respectively in $Y$.
\end{lemma}
\begin{proof}
We fix $\delta\in[0,\bar\delta]$ and use the Lipschitz representatives $u^*_\delta,\,v^*_\delta,\,s^*_\delta$ of \cref{lem:Lip} throughout the proof.
The marginal equation \eqref{eq:marginal-equation} and its analogue on $Y$ hold almost everywhere, and both of their sides are continuous. Hence, they hold at every point of $X$ and $Y$.
By \eqref{eq:fc-box}, we have $-1<u_\delta^*\le1$ and $-1<v_\delta^*\le1$ almost everywhere. Since $u_\delta^*$ and $v_\delta^*$ are continuous, and $X$ and $Y$ are convex with nonempty interior, every point of $X$ and $Y$ is a limit of points at which these bounds hold. Hence, $-1\le u_\delta^*\le1$ and $-1\le v_\delta^*\le1$ at every point.

Since $X$ is convex, the derivative formula \eqref{eq:fc-derivative} in the proof of \cref{lem:Lip} applies to the problem with lower bound $\delta$. With $A_{x,\delta}=\{y:s_\delta^*(x,y)>0\}$ and $a_{x,\delta}=|A_{x,\delta}|$, it gives, for almost every $x$,
\[
 D_{e_X}u_\delta^*(x)=
 \frac{\eps^{-1}\int_{A_{x,\delta}}D_{e_X}C(x,y)\dd y+(1-u_\delta^*(x))D_{e_X}f(x)}
      {f(x)+a_{x,\delta}/\eps}.
\]
Since $0\leq1-u_\delta^*\le2$ and $f\geq c$, we obtain
\[
 |D_{e_X}u_\delta^*(x)|
 \leq
 \frac{(a_{x,\delta}/\eps)L_{C,X,e}+2L_{f,e}}
      {f(x)+a_{x,\delta}/\eps}
 =\frac{a_{x,\delta}/\eps}{f(x)+a_{x,\delta}/\eps}\,L_{C,X,e}+\frac{f(x)}{f(x)+a_{x,\delta}/\eps}\,\frac{2L_{f,e}}{f(x)}
 \leq U_{X,e}.
\]
This bound is independent of $\eps$ and $\delta$.
The bound for $v_\delta^*$ follows by interchanging $X$ and $Y$.

Finally, the argument after \eqref{eq:Lip-convex} in the proof of \cref{lem:Lip} uses only a bound on the derivative in the direction of the segment. We apply it with the segment from $x$ to $x+re_X$ and the bound $|D_{e_X}u_\delta^*|\le U_{X,e}$. This gives $|u_\delta^*(x+re_X)-u_\delta^*(x)|\le U_{X,e}|r|$ whenever both points lie in $X$. The bound for $v_\delta^*$ follows in the same way, with $e_Y$ and $U_{Y,e}$.
\end{proof}

\begin{theorem}\label{thm:curvature}
Fix directions as in \cref{lem:fc-potential}. Suppose that, almost
everywhere on $X\times Y$,
\begin{equation}\label{eq:fc-cost-curvature}
0<\underline{C}_{xx}\le D_{e_X}^2C\le\overline{C}_{xx},
 \qquad 0<\underline{C}_{yy}\le D_{e_Y}^2C\le\overline{C}_{yy}.
\end{equation}
Define
\begin{equation}\label{eq:fc-density-error}
\begin{aligned}
	B_{f,e}&=\|(D_{e_X}^2f)_+\|_{L^\infty(X)},&
 R_X&=2B_{f,e}+2L_{f,e}U_{X,e},\\
 B_{g,e}&=\|(D_{e_Y}^2g)_+\|_{L^\infty(Y)},&
 R_Y&=2B_{g,e}+2L_{g,e}U_{Y,e},
\end{aligned}
\end{equation}
and assume the margin conditions
\begin{equation}\label{eq:fc-margins}
	\gamma_X:=\frac{c\underline{C}_{xx}-(b/\varepsilon)(\overline{C}_{xx}-\underline{C}_{xx})-R_X}{c+b/\varepsilon}>0,
 \qquad
 \gamma_Y:=\frac{c\underline{C}_{yy}-(a/\varepsilon)(\overline{C}_{yy}-\underline{C}_{yy})-R_Y}{c+a/\varepsilon}>0.
\end{equation}
Then, for every $\delta\in[0,\bar\delta]$ and every $y\in Y$, the restriction of $s_\delta^*(\cdot,y)$ to every segment in $X$ parallel to $e_X$ is $\gamma_X$-strongly concave. Similarly, for every $x\in X$, the restriction of $s_\delta^*(x,\cdot)$ to every segment in $Y$ parallel to $e_Y$ is $\gamma_Y$-strongly concave. Here, recall that a function $\phi$ on an interval is $\gamma$-strongly concave if $\tau\mapsto\phi(\tau)+\frac\gamma2\tau^2$ is concave.

Let $\Hh^d$ be the $d$-dimentional Hausdorff  measure and let
\[
 S_X=\mathcal H^{d_X-1}(\operatorname{proj}_{e_X^\perp}X),\qquad
 S_Y=\mathcal H^{d_Y-1}(\operatorname{proj}_{e_Y^\perp}Y),
\]
with $S_X=1$ if $d_X=1$, and similarly for $S_Y$. Then, for every $t>0$,
\begin{equation}\label{eq:fc-uniform-layer}
 \sup_{0\le\delta\le\bar\delta}m_{*,\delta}(t)
 \le A_\varepsilon\sqrt t,\qquad
 A_\varepsilon=4\max\left\{
	 S_Y\gamma_Y^{-1/2},
 S_X\gamma_X^{-1/2}\right\}.
\end{equation}
In particular, $A_\varepsilon$ is independent of $\delta$ and
$\bar\delta$ within the stated admissible range.
\end{theorem}

\begin{proof}
We fix $\delta$ and suppress its subscript. We first record three midpoint inequalities. Let $x_0\in X$ and $r>0$ with $x_\pm=x_0\pm re_X\in X$, and let $y\in Y$. Then
\begin{equation}\label{eq:fc-midpoint}
 \frac{f(x_-)+f(x_+)}2-f(x_0)\le\frac{B_{f,e}}2r^2,\qquad
 \frac{\underline C_{xx}}2r^2\le\frac{C(x_-,y)+C(x_+,y)}2-C(x_0,y)\le\frac{\overline C_{xx}}2r^2.
\end{equation}
Indeed, for a function $\varphi$ of one variable with Lipschitz derivative,
\[
 \frac{\varphi(r)+\varphi(-r)}2-\varphi(0)=\frac12\int_{-r}^r(r-|\tau|)\,\varphi''(\tau)\dd\tau,
 \qquad \int_{-r}^r(r-|\tau|)\dd\tau=r^2.
\]
We apply this to $\tau\mapsto f(x_0+\tau e_X)$ and $\tau\mapsto C(x_0+\tau e_X,y)$. The bounds in \eqref{eq:fc-cost-curvature} and $D_{e_X}^2f\le B_{f,e}$ hold almost everywhere. To show the claim, let $N\subset X\times Y$ be the null set where $\underline C_{xx}\le D_{e_X}^2C\le\overline C_{xx}$ fails. After a rotation with $e_X=e_1$, writing $x=(t,\hat x)$, Tonelli's theorem gives
\[
 0=|N|=\int_Y\int_{\mathbb R^{d_X-1}}\big|\{t:((t,\hat x),y)\in N\}\big|\dd\hat x\dd y.
\]
Hence, the set $\{t:((t,\hat x),y)\in N\}$ is null for almost every $(\hat x,y)$. The same argument, without the variable $y$, applies to $D_{e_X}^2f\le B_{f,e}$. Thus, \eqref{eq:fc-midpoint} holds whenever the segment from $x_-$ to $x_+$ lies on such a line. Since $f$ and $C$ are continuous, both sides of \eqref{eq:fc-midpoint} are continuous in $x_0$ and $y$, so \eqref{eq:fc-midpoint} holds for all $x_0$, $r$, and $y$ as above.

We use the notation
$u_\pm=u^*(x_\pm)$, $\bar u=(u_-+u_+)/2$,
$f_\pm=f(x_\pm)$, and $\bar f=(f_-+f_+)/2$.
A direct expansion gives
\[
 \frac{f_-(u_--1)+f_+(u_+-1)}{2}=\bar f(\bar u-1)+\frac14(f_+-f_-)(u_+-u_-).
\]
Hence, we obtain
\begin{align*}
	\mathcal D_f:=
 f(x_0)(\bar u-1)
 -\frac{f_-(u_--1)+f_+(u_+-1)}{2}
 &=(\bar f-f(x_0))(1-\bar u)-\frac14(f_+-f_-)(u_+-u_-).
\end{align*}
By \eqref{eq:fc-midpoint}, we have $\bar f-f(x_0)\le\frac12B_{f,e}r^2$. By \eqref{eq:fc-box}, we have $0\leq1-\bar u\leq2$. By \cref{lem:fc-potential}, we also have $|u_+-u_-|\le2U_{X,e}r$, and $|f_+-f_-|\le2L_{f,e}r$. Hence, we obtain
\begin{equation}\label{eq:D_f-est}
 \mathcal D_f
 \leq B_{f,e}r^2+L_{f,e}U_{X,e}r^2
 =\frac12R_Xr^2.
\end{equation}

Now, we set
\begin{equation}\label{eq:beta-under-u}
	\beta_X=\frac{(b/\varepsilon)\overline{C}_{xx}+R_X}{c+b/\varepsilon},
 \qquad \underline u=\bar u-\beta_Xr^2/2.
\end{equation}
We show that $u^*(x_0)\ge\underline u$.
A direct computation gives $\underline{C}_{xx}-\beta_X=\gamma_X$. Hence, by~\eqref{eq:fc-margins}, we have $0\le\beta_X<\underline{C}_{xx}\le\overline{C}_{xx}$.
By the definition of the slack \eqref{eq:optimal-slack}, we have
\[
 \bar u+v^*(y)=\tfrac12\big(s^*(x_-,y)+s^*(x_+,y)\big)+\tfrac12\big(C(x_-,y)+C(x_+,y)\big)+\eps\delta.
\]
Hence, we obtain
\[
\underline u+v^*(y)-C(x_0,y)-\varepsilon\delta
= \tfrac12(s^*(x_-,y)+s^*(x_+,y))+\tfrac12(C(x_-,y)+C(x_+,y))-C(x_0,y)-\tfrac12\beta_Xr^2.
\]
By the upper bound in \eqref{eq:fc-midpoint}, we obtain
\[
\underline u+v^*(y)-C(x_0,y)-\varepsilon\delta
\le \tfrac12(s^*(x_-,y)+s^*(x_+,y))
+\tfrac12(\overline{C}_{xx}-\beta_X)r^2.
\]
We consider the strictly increasing function
\begin{equation}\label{eq:Phi_x}
 \Phi_{x_0}(q)=f(x_0)(q-1)+\delta b+
 \varepsilon^{-1}\int_Y(q+v^*(y)-C(x_0,y)-\varepsilon\delta)_+\dd y,
\end{equation}
which vanishes at $q=u^*(x_0)$ by the marginal equation \eqref{eq:marginal-equation} at $x_0$.
The map $r\to r_+$ is nondecreasing, convex, and $1$-Lipschitz. Hence, we obtain, for every $y\in Y$,
\begin{equation}\label{eq:under_u_pos_est}
 (\underline u+v^*(y)-C(x_0,y)-\varepsilon\delta)_+
 \le\frac{(s^*(x_-,y))_++(s^*(x_+,y))_+}{2}+\frac{\overline{C}_{xx}-\beta_X}2r^2.
\end{equation}
The marginal equations \eqref{eq:marginal-equation} at $x_\pm$ read
\[
 \eps^{-1}\int_Y(s^*(x_\pm,y))_+\dd y=-\delta b-f_\pm(u_\pm-1).
\]
We combine \eqref{eq:D_f-est}, \eqref{eq:beta-under-u}, \eqref{eq:Phi_x}, \eqref{eq:under_u_pos_est}, $f\ge c$, and obtain
\begin{align*}
 \Phi_{x_0}(\underline u)
 &\leq f(x_0)(\underline u-1)-\frac{f_-(u_--1)+f_+(u_+-1)}2+\frac{b}{2\eps}(\overline{C}_{xx}-\beta_X)r^2\\
 &=\mathcal D_f-\frac12f(x_0)\beta_Xr^2
       +\frac{b}{2\varepsilon}
        (\overline{C}_{xx}-\beta_X)r^2\\
 &\leq\frac12r^2
 \left\{R_X-c\beta_X
       +(b/\varepsilon)(\overline{C}_{xx}-\beta_X)\right\}
       =0.
\end{align*}
Since $\Phi_{x_0}$ is strictly increasing and vanishes at $u^*(x_0)$, we deduce
\[u^*(x_0)\ge\underline u=\bar u-\beta_Xr^2/2.\]

By the lower bound in \eqref{eq:fc-midpoint}, we obtain the midpoint inequality
\[
 s^*(x_0,y)\ge\tfrac12\{s^*(x_-,y)+s^*(x_+,y)\}
 +\tfrac12(\underline{C}_{xx}-\beta_X)r^2
 =\tfrac12\{s^*(x_-,y)+s^*(x_+,y)\}
 +\tfrac12\gamma_Xr^2.
\]
Along a line parallel to $e_X$, we write
$\phi_y(\tau)=s^*(x_0+\tau e_X,y)$.
Since the midpoint inequality holds at every center, we have
\[
 \phi_y(\tau)\geq
 \frac12\bigl(\phi_y(\tau-r)+\phi_y(\tau+r)\bigr)
 +\frac12\gamma_Xr^2
\]
whenever $\tau-r$ and $\tau+r$ belong to the segment.
For $\psi_y(\tau)=\phi_y(\tau)+\frac{\gamma_X}{2}\tau^2$, we have $\frac12((\tau-r)^2+(\tau+r)^2)=\tau^2+r^2$, and hence
\[
 \psi_y(\tau)\ge\frac12\bigl(\psi_y(\tau-r)+\psi_y(\tau+r)\bigr).
\]
Thus, $\psi_y$ is midpoint concave. Since $\psi_y$ is continuous, it is concave.
This proves that $\phi_y$ is $\gamma_X$-strongly concave. Interchanging the roles of $X$ and $Y$ proves the second assertion.

To obtain the boundary-layer estimate, we consider a parallel segment and write the slack on it as a $\gamma$-strongly concave function $\phi$ on an interval $I$, where $\gamma=\gamma_X$ or $\gamma_Y$ according to the direction.
Since $\phi$ is Lipschitz, it extends continuously to the close interval $\overline I$.
We let $\tau_m$ be a maximizer of $\phi$ on $\overline I$.
Since $\phi$ is concave, it is nondecreasing to the left of $\tau_m$ and nonincreasing to the right.

Let $\tau_1<\tau_2\leq\tau_m$ in $I$, and let $\phi'_-(\tau_2)$ be the left derivative. Since $\phi$ is nondecreasing on the left of $\tau_m$, we have $\phi'_-(\tau_2)\ge0$. The strong concavity gives
\[
 \phi(\tau_1)\le\phi(\tau_2)+\phi'_-(\tau_2)(\tau_1-\tau_2)-\frac\gamma2(\tau_2-\tau_1)^2
 \le\phi(\tau_2)-\frac\gamma2(\tau_2-\tau_1)^2.
\]
Hence, we have $\phi(\tau_2)-\phi(\tau_1)\geq\frac{\gamma}{2}(\tau_2-\tau_1)^2$.
Similarly, for $\tau_m\le\tau_1<\tau_2$ in $I$, we have $\phi(\tau_1)-\phi(\tau_2)\ge\frac\gamma2(\tau_2-\tau_1)^2$.
For $t>0$, we set
\[
 I_t^-=\{\tau\in I:\tau\le\tau_m,\ |\phi(\tau)|\le t\},\qquad
 I_t^+=\{\tau\in I:\tau\ge\tau_m,\ |\phi(\tau)|\le t\}.
\]
For $\tau_1<\tau_2$ both in $I_t^-$ or both in $I_t^+$, we have
\[
 \frac\gamma2(\tau_2-\tau_1)^2\le|\phi(\tau_2)-\phi(\tau_1)|\le2t,
 \qquad\text{that is,}\qquad
 \tau_2-\tau_1\le2\sqrt{t/\gamma}.
\]
Hence, $\operatorname{diam}I_t^\pm\le2\sqrt{t/\gamma}$, and we obtain
\begin{equation}\label{eq:1d-layer}
 |\{\tau\in I:|\phi(\tau)|\le t\}|
 \le\operatorname{diam}I_t^-+\operatorname{diam}I_t^+
 \le4\sqrt{t/\gamma}.
\end{equation}

We fix $x\in X$. After a rotation in $\mathbb R^{d_Y}$, we may assume $e_Y=e_1$, and we write $y=(\tau,\hat y)\in\mathbb R\times\mathbb R^{d_Y-1}$. We set
\[
 P_Y=\{\hat y:(\tau,\hat y)\in Y\text{ for some }\tau\},\qquad
 I_{\hat y}=\{\tau:(\tau,\hat y)\in Y\},\qquad
 \phi_{x,\hat y}(\tau)=s^*(x,(\tau,\hat y)),
\]
so that $|P_Y|=S_Y$. For $\hat y\in P_Y$, the set $I_{\hat y}$ is an interval because $Y$ is convex, and $\phi_{x,\hat y}$ is $\gamma_Y$-strongly concave on $I_{\hat y}$. Hence, \eqref{eq:1d-layer} with $\gamma=\gamma_Y$ gives
\[
 \big|\{\tau\in I_{\hat y}:|\phi_{x,\hat y}(\tau)|\le t\}\big|\le4\sqrt{t/\gamma_Y},
 \qquad \hat y\in P_Y.
\]
For $\hat y\notin P_Y$, the set $I_{\hat y}$ is empty. By Tonelli's theorem, we obtain
\[
 |\{y\in Y:|s^*(x,y)|\le t\}|
 =\int_{P_Y}\big|\{\tau\in I_{\hat y}:|s^*(x,(\tau,\hat y))|\le t\}\big|\dd\hat y
 \le4S_Y\sqrt{t/\gamma_Y}.
\]
If $d_Y=1$, there is no variable $\hat y$, and the bound holds with $S_Y=1$. Interchanging the roles of $X$ and $Y$ gives $|\{x\in X:|s^*(x,y)|\le t\}|\le4S_X\sqrt{t/\gamma_X}$ for every $y\in Y$.
\end{proof}

The bound \eqref{eq:fc-uniform-layer} verifies the interface condition of \cref{thm:local} for every $\delta\in[0,\bar\delta]$, since the punctured boundary-layer function of $s_\delta^*$ is bounded by $m_{*,\delta}$.

Under the hypotheses of \cref{thm:superclose}, we also have
\[
 \norm{\Lsf z_{h,\delta}^*-\Q_h^{\Hh}z_\delta^*}_\Bb
 =O(h^{3/2}),
\]
with a constant uniform in $\delta\in[0,\bar\delta]$.
Indeed, we apply \cref{lem:fc-potential} in all unit directions. This gives the Lipschitz bounds
\[
	L_u\le\max\Big\{\norm{\nabla_xC}_{L^\infty(X\times Y)},\frac{2\norm{\nabla f}_{L^\infty(X)}}{c}\Big\},\qquad
	L_v\le\max\Big\{\norm{\nabla_yC}_{L^\infty(X\times Y)},\frac{2\norm{\nabla g}_{L^\infty(Y)}}{c}\Big\},
\]
uniformly in $\delta$. Hence, $L_s$ and $L_{\rm cov}$ are uniform in $\delta$. Together with \eqref{eq:fc-uniform-layer}, \cref{thm:superclose} applies with $\alpha=1/2$ and constants independent of $\delta$.
This estimate compares the discrete potentials with cell averages
of the continuum potentials.

For $d_X=d_Y$ and $C(x,y)=\kappa_C|x-y|^2/2$, we have $D_{e_X}^2C=D_{e_Y}^2C=\kappa_C$ for all unit directions, so the curvature variations vanish.
The margin conditions \eqref{eq:fc-margins} then reduce to
\[
 R_X<c\kappa_C,\qquad R_Y<c\kappa_C,
\]
independently of $\varepsilon$ and $\delta$.
More explicitly, suppose that
$R_X,R_Y\leq(1-\theta)c\kappa_C$ for some $\theta\in(0,1]$. Then $\gamma_X\ge\theta\kappa_C/(1+b/(c\eps))$ and $\gamma_Y\ge\theta\kappa_C/(1+a/(c\eps))$, and hence
\begin{equation}\label{eq:fc-epsilon}
 A_\varepsilon\leq\frac4{\sqrt{\theta\kappa_C}}
 \max\left\{
 S_Y\sqrt{1+\frac a{c\varepsilon}},
 S_X\sqrt{1+\frac b{c\varepsilon}}
 \right\}.
\end{equation}

For example, we consider $X=Y=[0,1]$ and
\[
 C(x,y)=\frac{(x-y)^2}{2},\qquad
 f(x)=1+0.1x,\qquad g(y)=1+0.1y.
\]
Then, we have $c=1$, $\kappa_C=1$, $B_{f,e}=B_{g,e}=0$, $L_{f,e}=L_{g,e}=0.1$, and $U_{X,e}=U_{Y,e}=1$. Hence, $R_X=R_Y=0.2$, and
\[
 \gamma_X=\gamma_Y=\frac{0.8}{1+1/\varepsilon}.
\]
Thus, the margin conditions hold for every $\varepsilon>0$ and every
$\bar\delta<1$.

These estimates are uniform in the lower bound $\delta$, but not as
$\varepsilon\downarrow0$.
For general costs, the curvature-variation terms
\[
 \frac{\overline{C}_{xx}-\underline{C}_{xx}}{\varepsilon},
 \qquad
 \frac{\overline{C}_{yy}-\underline{C}_{yy}}{\varepsilon}
\]
can also prevent the margin conditions from holding for small
$\varepsilon$.

The curvature assumptions exclude the degeneracy in \cref{rem:lower-bound-nonuniform}, which depends on $\delta$, and they give the explicit boundary-layer bound \eqref{eq:fc-uniform-layer}, uniformly in $\delta$. Continuity of the solution in $\delta$ does not give such a bound. In \cref{rem:lower-bound-nonuniform}, the slack depends continuously on $\delta$, but the boundary-layer functions are not uniformly small.

Finally, \eqref{eq:fc-uniform-layer} holds for every $t>0$.
The next corollary uses this bound to give an explicit convergence radius for the Newton iteration on sufficiently fine meshes.

\begin{corollary}
\label{cor:explicit-radius}
Assume the hypotheses of \cref{thm:curvature}, fix
$\delta\in[0,\bar\delta]$, and set
\[
 t_\eps=\left(\frac{c\eps}{2A_\eps}\right)^2.
\]
Let $f_h,g_h,C_h$ be cellwise data on a product-cell partition
satisfying \eqref{eq:data} with the same constants $c,E$.
Let $\lambda_h$ be as in \cref{lem:uniform-stability},
and let $z_h^*$ be the array root with lower bound $\delta$.
Let $\bar\omega\in(0,1)$, and suppose that $\lambda_h$ satisfies the mesh condition
$\lambda_h<\bar\omega^2t_\eps$. Set
\begin{equation}\label{eq:explicit-radius}
 r_{0,h}=\tfrac12\left(\bar\omega^2t_\eps-\lambda_h\right).
\end{equation}
\begin{enumerate}
\item[(i)]
For every array potential $z_h^0$ with
$\norm{z_h^0-z_h^*}_\infty\leq r_{0,h}$, the iterates of the full-step Newton iteration \eqref{eq:Newton} on the cell space are well defined, remain in this ball,
and converge to $z_h^*$.
Writing $e_h^m=z_h^m-z_h^*$, the iterates satisfy
\begin{equation}\label{eq:explicit-rate}
\begin{aligned}
 \norm{e_h^{m+1}}_\infty
 \leq\frac{2A_\eps}{c\eps}
       \left(2\norm{e_h^m}_\infty+\lambda_h\right)^{1/2}
       \norm{e_h^m}_\infty
 \leq\bar\omega\,\norm{e_h^m}_\infty,
 \quad\text{and}\quad\norm{e_h^{m+1}}_h
 \leq\bar\omega\,\norm{e_h^m}_h
\end{aligned}
\end{equation}

\item[(ii)]
If $0<\sigma<1/2$ and, with $\kappa=L/c$,
\begin{equation}\label{eq:explicit-armijo}
 \bar\omega\left[
 \tfrac12+2\sigma\sqrt{\kappa}+\kappa
 \left(\sigma+\tfrac12\right)\bar\omega
 \right]\leq\tfrac12-\sigma,
\end{equation}
then the Armijo condition \eqref{eq:Armijo} accepts the step size $\tau=1$
at every iterate in (i).

\item[(iii)]
Every array potential $z_h$ with
$\norm{F_h(z_h)}_\infty\leq\mu_hr_{0,h}$, that is, with
$r_{\infty,h}\leq r_{0,h}$ in the notation of
\cref{sec:uniform-stability}, satisfies
$\norm{z_h-z_h^*}_\infty\leq r_{0,h}$.

\item[(iv)]
The quantities $A_\eps$, $t_\eps$, and condition
\eqref{eq:explicit-armijo} do not depend on
$\delta\in[0,\bar\delta]$.
If $M$ is replaced by $\bar M=\max\{\bar\delta,2/\eps\}$
in \eqref{eq:theta-star}, then $\lambda_h$, $r_{0,h}$, and the
mesh condition are uniform in $\delta$ as well.

For exact cell averages of data with Lipschitz constants
$L_f,L_g,L_C$, we have $\lambda_h\leq L_\lambda h$, where
\[
 L_\lambda=\frac2c\max\left\{
 \frac b\eps L_C+\bar\theta_XL_f,\,
 \frac a\eps L_C+\bar\theta_YL_g
 \right\}+L_C,\qquad
 \bar\theta_X=\min\Big\{2,\frac{b\bar M}{c}\Big\},\quad
 \bar\theta_Y=\min\Big\{2,\frac{a\bar M}{c}\Big\}.
\]
Thus, the mesh condition holds whenever
$h<\bar\omega^2t_\eps/L_\lambda$.

Under the stronger condition
$\lambda_h\leq\tfrac12\bar\omega^2t_\eps$, which for exact
cell averages holds whenever
$h\leq\bar\omega^2t_\eps/(2L_\lambda)$, we have the
mesh-independent lower bound
\[
 r_{0,h}\geq\tfrac14\bar\omega^2t_\eps.
\]
\end{enumerate}
\end{corollary}

\begin{proof}
We write $m_*=m_{*,\delta}$.
By \cref{thm:curvature}, we have $m_*(t)\leq A_\eps\sqrt t$ for every $t>0$.
The proof of \eqref{eq:layer-transfer} uses only \cref{lem:uniform-stability} on the given mesh. Hence, we deduce that
\[
 \nu_h(t)\leq m_*(t+\lambda_h)
 \leq A_\eps(t+\lambda_h)^{1/2}
 \qquad(t>0).
\]

(i) On the cell space, \cref{lem:remainder} holds with $\nu_h$, and \cref{lem:inverse} holds with $\mu_h\geq c$, including the inverse bound \eqref{eq:inverse-H} in the weighted inner product.
As in the proof of \cref{thm:mesh-local}, we obtain
\begin{align*}
 \norm{e_h^{m+1}}_\infty
 \leq
 \frac{2\nu_h(2\norm{e_h^m}_\infty)}{\mu_h\eps}
 \norm{e_h^m}_\infty
 \leq\frac{2A_\eps}{c\eps}
 \left(2\norm{e_h^m}_\infty+\lambda_h\right)^{1/2}
 \norm{e_h^m}_\infty.
\end{align*}
In a similar way, we also obtain \[\norm{e_h^{m+1}}_h\le \frac{2A_\eps}{c\eps}\left(2\norm{e_h^m}_\infty+\lambda_h\right)^{1/2} \norm{e_h^m}_h.\]
If $\norm{e_h^m}_\infty\leq r_{0,h}$, then \eqref{eq:explicit-radius} gives $2\norm{e_h^m}_\infty+\lambda_h\leq\bar\omega^2t_\eps$. By the definition of $t_\eps$, we also have
\[
 \frac{2A_\eps}{c\eps}\big(\bar\omega^2t_\eps\big)^{1/2}
 =\frac{2A_\eps}{c\eps}\,\bar\omega\,\frac{c\eps}{2A_\eps}
 =\bar\omega.
\]
Thus, we arrive at $\norm{e_h^{m+1}}_\infty
\leq\bar\omega\norm{e_h^m}_\infty\leq r_{0,h}$.
By induction, we obtain the invariance of the ball and \eqref{eq:explicit-rate}.
Since $\bar\omega<1$, the iterates converge to $z_h^*$.

(ii) For $r=\norm{e}_\infty\leq r_{0,h}$, we have $2r+\lambda_h\le\bar\omega^2t_\eps$. Hence, we further have
\[
 \frac{\nu_h(2r)}{\eps}\le\frac{A_\eps}{\eps}\big(\bar\omega^2t_\eps\big)^{1/2}
 =\frac c2\bar\omega.
\]
Thus, by \eqref{eq:objective-expansion}, for $e=z-z_h^*$, \[\Big|\G_h(z)-\G_h(z_h^*)-\frac12\ip{H_h(\vartheta(z))e}{e}_{h}\Big|\le\frac{c\bar\omega}{2}\norm{e}_h^2.\]
By part (i), we have $\norm{e-H_h(\vartheta(z))^{-1}F_h(z)}_h\leq\bar\omega\norm{e}_h$.
Since $\mu_h\ge c$ and $L_h\le L$, the spectral bounds $cI\leq H_h\leq LI$ hold, and $\G_h$ has an $L$-Lipschitz gradient.
The full-step argument in the proof for \eqref{eq:full-step} of \cref{thm:local} therefore applies
with $\bar\omega$ in place of $\omega$ and gives acceptance
under \eqref{eq:explicit-armijo}.
By (i), every iterate remains in this ball.

(iii) The secant estimate \eqref{eq:secant} on the cell space gives
\[
 \norm{z_h-z_h^*}_\infty
 \leq\mu_h^{-1}\norm{F_h(z_h)}_\infty
 \leq r_{0,h}.
\]

(iv) The constant $A_\eps$ in \eqref{eq:fc-uniform-layer} is independent of $\delta$. Hence, so are $t_\eps$ and \eqref{eq:explicit-armijo}.
The proof of \cref{lem:uniform-stability} uses $M$ only through
$\theta_X^*$ and $\theta_Y^*$. The bounds $0\leq1-u^*\leq\theta_X^*$ and $0\leq1-v^*\leq\theta_Y^*$ remain valid for any $M\geq\max\{\delta,2/\eps\}$.
We replace $M$ by $\bar M$ to obtain a common bound for every $\delta\in[0,\bar\delta]$.

For exact cell averages, we have $|f-f_h|\le L_f\operatorname{diam}(X_i)\le L_fh$ on each cell $X_i$, and similarly for $g$ and $C$. Hence,
\[
 \eta_{f,\infty}\leq L_fh,\qquad
 \eta_{g,\infty}\leq L_gh,\qquad
 \eta_{C,\infty}\leq L_Ch.
\]
Inserting these bounds into \eqref{eq:zeta} and $\lambda_h=2\zeta_h+\eta_{C,\infty}$, we obtain $\lambda_h\leq L_\lambda h$.
Hence, $h<\bar\omega^2t_\eps/L_\lambda$ implies $\lambda_h<\bar\omega^2t_\eps$, and $h\le\bar\omega^2t_\eps/(2L_\lambda)$ implies $\lambda_h\le\frac12\bar\omega^2t_\eps$.
Under the stronger condition
$\lambda_h\leq\tfrac12\bar\omega^2t_\eps$, we have
\[
 r_{0,h}
 =\tfrac12(\bar\omega^2t_\eps-\lambda_h)
 \geq\tfrac14\bar\omega^2t_\eps.\qedhere
\]
\end{proof}

Setting $\lambda_h=0$ in the same argument gives the continuum statement.
Under the hypotheses of \cref{thm:curvature}, \cref{thm:local} holds
with the explicit radius $\bar\omega^2t_\eps/2$ and local $q$-order at least $3/2$.
\Cref{cor:explicit-radius} covers every mesh finer than the explicit
threshold, independently of any refinement sequence. In contrast,
\cref{thm:mesh-local} also covers the finitely many coarse levels, but
only with a nonconstructive radius. 

For the quadratic cost $C(x,y)=\kappa_C|x-y|^2/2$ with fixed positive margins $c\kappa_C-R_X$ and $c\kappa_C-R_Y$, we have
\[
 \gamma_X=\frac{c\kappa_C-R_X}{c+b/\eps},\qquad
 \gamma_Y=\frac{c\kappa_C-R_Y}{c+a/\eps}.
\]
Hence, as $\eps\downarrow0$, we have
\[
 \frac{\gamma_X}{\eps}\to\frac{c\kappa_C-R_X}{b},\qquad
 \frac{\gamma_Y}{\eps}\to\frac{c\kappa_C-R_Y}{a}.
\]
By \eqref{eq:fc-uniform-layer}, $\eps A_\eps^2$ therefore converges to a positive limit. Since
\[
 \frac{t_\eps}{\eps^3}=\frac{c^2}{4\,\eps A_\eps^2},
\]
the ratio $t_\eps/\eps^3$ also converges to a positive limit. Thus, $t_\eps$ and $r_{0,h}$ decay like $\eps^3$ as $\eps\downarrow0$. The explicit radius is a sufficient condition derived from worst-case bounds, and it is not the largest radius of convergence. In \cref{sec:num-curvature}, the observed error ratios remain below $\bar\omega$ even at $16$ times the guaranteed radius.

\section{A globalized inexact semismooth Newton method}\label{sec:practical}
In this section, we combine the mesh-independence theory with a computable maximum-norm safeguard and inexact linear solves. This gives a complete iteration that is globally convergent and becomes the full-step Newton iteration near the solution.

\subsection{A maximum-norm safeguard}\label{sec:safeguard}
\Cref{thm:mesh-local,cor:explicit-radius} concern the full-step Newton iteration started in a maximum-norm ball.
The weighted global solvers (\cref{S-sec:global} of the supplementary materials) control the weighted norm $\norm{\cdot}_h$.
As explained in \cref{rem:M-operator}, passing from this norm to the maximum norm introduces a factor depending on the smallest cell volume.
Their mesh-uniform rates alone therefore do not establish mesh-uniform entry into the required maximum-norm ball.

We therefore use the following globalization. In each iteration, we first compute a \emph{Newton trial}, that is, the candidate iterate $z^m+d^m$ given by a Newton step with step size one. Its direction $d^m$ solves the Newton equation $H_h(\vartheta(z^m))d^m=-F_h(z^m)$ either exactly, as in \eqref{eq:Newton}, or inexactly, in the sense of the forcing condition \eqref{eq:inf-forcing} in \cref{sec:inexact-inner}. We accept the trial if it passes the \emph{residual test} $\norm{F_h(z^m+d^m)}_\infty\le\varpi\norm{F_h(z^m)}_\infty$ for a prescribed residual-reduction factor $\varpi\in(0,1)$. Otherwise, we reject it and apply gradient steps from $z^m$ until this reduction holds. We call this fallback the \emph{safeguard}.
The comparison principle behind \cref{lem:inverse}, that is, the $M$-matrix structure described in \cref{rem:M-operator}, gives a maximum-norm contraction for gradient steps, both plain and diagonally scaled.
The combination of Newton trials with this safeguard, first with exact linear solves, is analyzed in \cref{S-sec:globalized-further} of the supplementary materials. The inexact version used below is treated in \cref{sec:inexact-inner}.

We set
\begin{equation}\label{eq:inf-constants}
 L_{\infty,h}=E_h^{\max}+\frac{2\max\{a,b\}}{\eps},\qquad
 \tau_{\infty,h}=\left(E_h^{\max}+\frac{\max\{a,b\}}{\eps}\right)^{-1},
\end{equation}
so that $L_{\infty,h}\le L_\infty$ with $L_\infty$ as in \cref{cor:local-inexact}.
We also define the diagonal scaling $\mathsf D_h$ by
\begin{equation}\label{eq:diag-scaling}
 (\mathsf D_hz)_{u,i}=\Big(f_i+\frac b\eps\Big)u_i,\qquad
 (\mathsf D_hz)_{v,j}=\Big(g_j+\frac a\eps\Big)v_j,
\end{equation}
where $i=1,\dots,N_X,\,j=1,\dots,N_Y$, and put
\begin{equation}\label{eq:jacobi-factor}
 q_{{\rm J},h}=\max\left\{\frac{b}{\mu_h\eps+b},\ \frac{a}{\mu_h\eps+a}\right\}
 \le q_{\rm J}=\max\left\{\frac{b}{c\eps+b},\ \frac{a}{c\eps+a}\right\}<1 .
\end{equation}
The scaling $\mathsf D_h$ does not depend on $z$. By \eqref{eq:Hh}, its entries $f_i+b/\eps$ and $g_j+a/\eps$ are the diagonal entries of the Newton derivative $H_h(1)$ with all slopes equal to one, and they bound the diagonal entries of $H_h(\vartheta)$ for every slope $\vartheta$ with entries in $[0,1]$. The bound gives the contraction in \cref{lem:inf-gradient}(ii) below.

\begin{lemma}\label{lem:inf-gradient}
Let $f_h,g_h,C_h$ satisfy \eqref{eq:data}, and let $z_h^*$ be the array root of $F_h$. Then the following statements hold for every array potential $z$.
\begin{enumerate}
\item[(i)] For every step size $0<\tau\le\tau_{\infty,h}$, the plain gradient step satisfies
\begin{equation}\label{eq:inf-gradient}
 \norm{z-\tau F_h(z)-z_h^*}_\infty
 \le(1-\tau\mu_h)\norm{z-z_h^*}_\infty.
\end{equation}
\item[(ii)] The scaled gradient step satisfies
\begin{equation}\label{eq:inf-scaled}
 \norm{z-\mathsf D_h^{-1}F_h(z)-z_h^*}_\infty
 \le q_{{\rm J},h}\norm{z-z_h^*}_\infty.
\end{equation}
\item[(iii)] The residual satisfies
\begin{equation}\label{eq:inf-residual}
 \mu_h\norm{z-z_h^*}_\infty\le\norm{F_h(z)}_\infty
 \le L_{\infty,h}\norm{z-z_h^*}_\infty.
\end{equation}
\item[(iv)] Let $\varpi\in(0,1)$. We consider either the plain steps $z^{j+1}=z^j-\tau F_h(z^j)$, with $q_{\rm G}=1-\tau\mu_h$, or the scaled steps $z^{j+1}=z^j-\mathsf D_h^{-1}F_h(z^j)$, with $q_{\rm G}=q_{{\rm J},h}$.
Then $\norm{F_h(z^j)}_\infty\le\varpi\norm{F_h(z^0)}_\infty$ for every $j\ge N_{\rm G}$, where
\begin{equation}\label{eq:gradient-block}
 N_{\rm G}=\left\lceil
 \frac{\log\bigl(L_{\infty,h}/(\varpi\mu_h)\bigr)}{-\log q_{\rm G}}
 \right\rceil.
\end{equation}
\end{enumerate}
The same statements hold for $z,z^0\in\Bb$ in the continuum, with $F_h,z_h^*,\mu_h,E_h^{\max}$ replaced by $F,z^*,c,E$, respectively.
\end{lemma}
\begin{proof}
Let $e=z-z_h^*$. Since $F_h(z_h^*)=0$, the construction in the proof of \eqref{eq:secant}, carried out on the cell space, gives a slope array $\vartheta$ with entries in $[0,1]$ such that
\[
 F_h(z)=F_h(z)-F_h(z_h^*)=H_h(\vartheta)e.
\]
We use the notation
\[
	\varrho_{X,i}=\frac1\eps\sum_{j=1}^{N_Y}\omega_j\vartheta_{ij}\in[0,b/\eps],\qquad
	\varrho_{Y,j}=\frac1\eps\sum_{i=1}^{N_X}w_i\vartheta_{ij}\in[0,a/\eps].
\]

We first prove (i). Since $z-\tau F_h(z)-z_h^*=e-\tau H_h(\vartheta)e$, \eqref{eq:Hh} shows that, for $i=1,\dots,N_X$,
\[
	[z-\tau F_h(z)-z_h^*]_{u,i}=(1-\tau f_i-\tau\varrho_{X,i})\,e_{u,i}
	-\frac\tau\eps\sum_{j=1}^{N_Y}\omega_j\vartheta_{ij}e_{v,j}.
\]
Since $\tau(f_i+\varrho_{X,i})\le\tau(E_h^{\max}+b/\eps)\le\tau\,\tau_{\infty,h}^{-1}\le1$, we estimate
\[
\big|[z-\tau F_h(z)-z_h^*]_{u,i}\big|\le
 (1-\tau f_i-\tau\varrho_{X,i})\norm{e}_\infty+\tau\varrho_{X,i}\norm{e}_\infty
 =(1-\tau f_i)\norm{e}_\infty\le(1-\tau\mu_h)\norm{e}_\infty.
\]
In the same way, we estimate the $v$-component, for $j=1,\dots,N_Y$,
\[
\big|[z-\tau F_h(z)-z_h^*]_{v,j}\big|\le(1-\tau\mu_h)\norm{e}_\infty.
\]
This proves \eqref{eq:inf-gradient}.

We now prove (ii). We compute that, for $i=1,\dots,N_X$,
\[
	[e-\mathsf D_h^{-1}H_h(\vartheta)e]_{u,i}=
 \frac{b/\eps-\varrho_{X,i}}{f_i+b/\eps}\,e_{u,i}
 -\frac1{\eps(f_i+b/\eps)}\sum_{j=1}^{N_Y}\omega_j\vartheta_{ij}e_{v,j}.
\]
Since $\varrho_{X,i}\le b/\eps$, we estimate
\[
 \big|[e-\mathsf D_h^{-1}H_h(\vartheta)e]_{u,i}\big|\le \frac{(b/\eps-\varrho_{X,i})+\varrho_{X,i}}{f_i+b/\eps}\norm{e}_\infty
 =\frac{b}{f_i\eps+b}\norm{e}_\infty\le\frac{b}{\mu_h\eps+b}\norm{e}_\infty .
\]
In the same way, we estimate the $v$-component, for $j=1,\dots,N_Y$,
\[
\big|[e-\mathsf D_h^{-1}H_h(\vartheta)e]_{v,j}\big|\le\frac{a}{\mu_h\eps+a}\norm{e}_\infty .
\]
This proves \eqref{eq:inf-scaled}.

Next, we prove (iii). The lower bound is the secant estimate \eqref{eq:secant} on the cell space, with $\mu_h$ in place of $c$. For the upper bound, we estimate
\[
	|[H_h(\vartheta)e]_{u,i}|\le f_i|e_{u,i}|+\frac1\eps\sum_{j=1}^{N_Y}\omega_j
 \vartheta_{ij}\big(|e_{u,i}|+|e_{v,j}|\big)\le\Big(E_h^{\max}+\frac{2b}\eps\Big)\norm{e}_\infty,
\]
and the $v$-component we have \[|[H_h(\vartheta)e]_{v,j}|\le (E_h^{\max}+\frac{2a}{\eps})\norm{e}_\infty.\]
Hence, we obtain $\norm{F_h(z)}_\infty\le L_{\infty,h}\norm{e}_\infty$.

Finally, we prove (iv). By (i) or (ii) and induction, we have $\norm{z^j-z_h^*}_\infty\le q_{\rm G}^j\norm{z^0-z_h^*}_\infty$. By the definition \eqref{eq:gradient-block} of $N_{\rm G}$, we have $L_{\infty,h}q_{\rm G}^j/\mu_h\le\varpi$ for $j\ge N_{\rm G}$. Together with (iii), we obtain, for $j\ge N_{\rm G}$,
\[
 \norm{F_h(z^j)}_\infty\le L_{\infty,h}q_{\rm G}^j\norm{z^0-z_h^*}_\infty
 \le\frac{L_{\infty,h}}{\mu_h}q_{\rm G}^j\norm{F_h(z^0)}_\infty
 \le\varpi\norm{F_h(z^0)}_\infty.
\]

In the continuum, the same argument applies with the sums over $j$ and $i$ replaced by integrals over $Y$ and $X$, and with $\mu_h,E_h^{\max}$ replaced by $c,E$.
\end{proof}

We now choose a step size that does not depend on the mesh. We set
\[
 \tau_\infty
 :=\left(E+\frac{\max\{a,b\}}{\eps}\right)^{-1}.
\]
Since $E_h^{\max}\le E$, we have $\tau_\infty\le\tau_{\infty,h}$, so $\tau_\infty$ is admissible in (i). Moreover, since $\mu_h\ge c$ and $L_{\infty,h}\le L_\infty$, we have
\[
 1-\tau_\infty\mu_h\le1-\tau_\infty c,\qquad
 \frac{L_{\infty,h}}{\mu_h}\le\frac{L_\infty}{c}.
\]
Hence, \eqref{eq:gradient-block} gives a bound for the plain steps that does not depend on the mesh:
\[
 N_{\rm G}\leq
 \left\lceil
 \frac{\log\bigl(L_\infty/(\varpi c)\bigr)}
      {-\log(1-\tau_\infty c)}
 \right\rceil.
\]
For the scaled steps, we have $q_{{\rm J},h}\le q_{\rm J}$, and hence
\[
 N_{\rm G}\leq
 \left\lceil
 \frac{\log\bigl(L_\infty/(\varpi c)\bigr)}
      {-\log q_{\rm J}}
 \right\rceil.
\]
The guaranteed contraction factor of the scaled steps is never larger than that of the plain steps with step size $\tau_\infty$. Indeed, since $c\leq E$, we have
\[
 1-q_{\rm J}=\frac{c\eps}{c\eps+\max\{a,b\}}=\frac{c}{c+\max\{a,b\}/\eps}\ge\frac{c}{E+\max\{a,b\}/\eps}=\tau_\infty c.
\]
Thus, $q_{\rm J}\leq1-\tau_\infty c$, so the bound on $N_{\rm G}$ for the scaled steps is at most the bound for the plain steps. In addition, the upper bound $E$ does not enter $q_{\rm J}$.

For balanced dual gradient descent, an argument of the same type gives nonexpansiveness in the maximum norm \cite[Lemma~5.2]{gonzalez2025linear}. Here, the term $(fd_u,gd_v)$ of the Newton derivative $H(\vartheta)$ in \eqref{eq:H-cont}, which come from the Pearson marginal penalties and satisfy $f,g\ge c$, make the step a strict contraction. The contraction factors are $1-\tau\mu_h$ and $q_{{\rm J},h}$ in \cref{lem:inf-gradient}(i) and (ii).
\subsection{Inexact Newton trials and the inner Jacobi iteration}
\label{sec:inexact-inner}
A Newton trial may be computed inexactly. We measure the inexactness by the maximum-norm forcing condition of \cref{cor:local-inexact}, which is computable:
\begin{equation}\label{eq:inf-forcing}
 z_{\rm N}^m=z^m+\widehat d^{\,m},\qquad
 \norm{H_h(\vartheta(z^m))\widehat d^{\,m}+F_h(z^m)}_\infty
 \le\eta_m\norm{F_h(z^m)}_\infty .
\end{equation}
For each prescribed forcing term $\eta_m\in(0,1)$, the $M$-matrix structure of \cref{rem:M-operator} provides an inner solver, the Jacobi iteration, that meets this condition. Its number of steps is bounded independently of the mesh.
For a slope array $\vartheta$ with entries in $[0,1]$, we write $H_h(\vartheta)=\mathsf D_\vartheta+\mathsf O_\vartheta$, where
\begin{equation}\label{eq:jacobi-splitting}
\begin{aligned}
 (\mathsf D_\vartheta d)_{u,i}&=(f_i+\varrho_{X,i})\,d_{u,i},&
 (\mathsf O_\vartheta d)_{u,i}&=\frac1\eps\sum_{j=1}^{N_Y}\omega_j\vartheta_{ij}\,d_{v,j},\\
 (\mathsf D_\vartheta d)_{v,j}&=(g_j+\varrho_{Y,j})\,d_{v,j},&
 (\mathsf O_\vartheta d)_{v,j}&=\frac1\eps\sum_{i=1}^{N_X}w_i\vartheta_{ij}\,d_{u,i},
\end{aligned}
\end{equation}
with
\[
	\varrho_{X,i}=\frac1\eps\sum_{j=1}^{N_Y}\omega_j\vartheta_{ij},\qquad
	\varrho_{Y,j}=\frac1\eps\sum_{i=1}^{N_X}w_i\vartheta_{ij}.
\]
The coefficients of $\mathsf O_\vartheta$ are nonnegative, and $\mathsf O_\vartheta\one=(\varrho_X,\varrho_Y)$. Hence, $|(\mathsf O_\vartheta d)_{u,i}|\le\varrho_{X,i}\norm{d}_\infty$ and $|(\mathsf O_\vartheta d)_{v,j}|\le\varrho_{Y,j}\norm{d}_\infty$ for every array $d$.
For $\vartheta\equiv1$, we have $\varrho_{X,i}=b/\eps$ and $\varrho_{Y,j}=a/\eps$, so $\mathsf D_1=\mathsf D_h$ by \eqref{eq:diag-scaling}.

Let $q_{{\rm J},h}$ and $q_{\rm J}$ be as in \eqref{eq:jacobi-factor}. We have the following lemma.

\begin{lemma}\label{lem:jacobi}
Let $\vartheta$ have entries in $[0,1]$, let $F$ be an array, and let
$d=-H_h(\vartheta)^{-1}F$. The Jacobi iterates
$d^{k+1}=\mathsf D_\vartheta^{-1}(-F-\mathsf O_\vartheta d^k)$ with $d^0=0$ satisfy
\begin{equation}\label{eq:jacobi-rate}
 \norm{d^k-d}_\infty\le q_{{\rm J},h}^k\norm{d}_\infty
 \le\frac{q_{{\rm J},h}^k}{\mu_h}\norm{F}_\infty,
 \qquad
 \norm{H_h(\vartheta)d^k+F}_\infty
 \le\frac{L_{\infty,h}}{\mu_h}\,q_{{\rm J},h}^k\norm{F}_\infty .
\end{equation}
Consequently, for $\eta\in(0,1)$, the residual condition
$\norm{H_h(\vartheta)d^k+F}_\infty\le\eta\norm{F}_\infty$ holds for every
\begin{equation}\label{eq:jacobi-count}
 k\ge N_{\rm J}(\eta)=
 \left\lceil\frac{\log\bigl(\eta\mu_h/L_{\infty,h}\bigr)}{\log q_{{\rm J},h}}
 \right\rceil,
 \qquad
 N_{\rm J}(\eta)\le
 \left\lceil\frac{\log\bigl(\eta c/L_\infty\bigr)}{\log q_{\rm J}}\right\rceil .
\end{equation}
Each Jacobi iteration step costs one application of $\mathsf O_\vartheta$.
\end{lemma}
\begin{proof}
Since $H_h(\vartheta)d=-F$, we have $\mathsf D_\vartheta d+\mathsf O_\vartheta d=-F$. The Jacobi iterates satisfy $\mathsf D_\vartheta d^{k+1}=-F-\mathsf O_\vartheta d^k$. Subtracting these two identities, we see that the errors satisfy
\[
 d^{k+1}-d=-\mathsf D_\vartheta^{-1}\mathsf O_\vartheta(d^k-d).
\]
We now need to bound the iteration matrix $\mathsf D_\vartheta^{-1}\mathsf O_\vartheta$ in the maximum norm. By \eqref{eq:jacobi-splitting}, for every array $e$ and $i=1,\dots,N_X$,
\[
 \big|(\mathsf D_\vartheta^{-1}\mathsf O_\vartheta e)_{u,i}\big|
 =\frac1{f_i+\varrho_{X,i}}\Big|\frac1\eps\sum_{j=1}^{N_Y}\omega_j\vartheta_{ij}e_{v,j}\Big|
 \le\frac{\varrho_{X,i}}{f_i+\varrho_{X,i}}\norm{e}_\infty .
\]
Since $t\mapsto t/(f_i+t)$ is increasing, $\varrho_{X,i}\le b/\eps$, and $f_i\ge\mu_h$, we arrive at
\[\big|(\mathsf D_\vartheta^{-1}\mathsf O_\vartheta e)_{u,i}\big|\le\frac{b}{\mu_h\eps+b}\norm{e}_\infty.\]
In the same way, we bound $|(\mathsf D_\vartheta^{-1}\mathsf O_\vartheta e)_{u,i}|\le a/(\mu_h\eps+a)\norm{e}_\infty$. Hence, we obtain
\[
 \norm{\mathsf D_\vartheta^{-1}\mathsf O_\vartheta}_{\infty\to\infty}\le q_{{\rm J},h}.
\]

Since $d^0=0$, induction gives $\norm{d^k-d}_\infty\le q_{{\rm J},h}^k\norm{d}_\infty$, which is the first inequality in \eqref{eq:jacobi-rate}. The second inequality follows from \eqref{eq:inverse-infty} on the cell space, which gives $\norm{d}_\infty\le\norm{F}_\infty/\mu_h$.

For the residual condition $\norm{H_h(\vartheta)d^k+F}_\infty\le\eta\norm{F}_\infty$, we write $H_h(\vartheta)d^k+F=H_h(\vartheta)(d^k-d)$. By \eqref{eq:jacobi-splitting}, for $i=1,\dots,N_X$,
\[
 |(H_h(\vartheta)e)_{u,i}|
 \le(f_i+\varrho_{X,i})|e_{u,i}|+\frac1\eps\sum_{j=1}^{N_Y}\omega_j\vartheta_{ij}|e_{v,j}|
 \le(f_i+2\varrho_{X,i})\norm{e}_\infty
 \le\Big(E_h^{\max}+\frac{2b}\eps\Big)\norm{e}_\infty .
\]
In the same way, $|(H_h(\vartheta)e)_{v,j}|\le(E_h^{\max}+2a/\eps)\norm{e}_\infty$. Both factors are at most $L_{\infty,h}$ by \eqref{eq:inf-constants}, and we obtain
\[
 \norm{H_h(\vartheta)d^k+F}_\infty\le L_{\infty,h}\norm{d^k-d}_\infty
 \le\frac{L_{\infty,h}}{\mu_h}q_{{\rm J},h}^k\norm{F}_\infty.
\]
This is the last estimate in \eqref{eq:jacobi-rate}.

Since $\log q_{{\rm J},h}<0$, the condition $(L_{\infty,h}/\mu_h)q_{{\rm J},h}^k\le\eta$ is equivalent to
\[
 k\ge\frac{\log\bigl(\eta\mu_h/L_{\infty,h}\bigr)}{\log q_{{\rm J},h}},
\]
which holds for every $k\ge N_{\rm J}(\eta)$. For the uniform bound in \eqref{eq:jacobi-count}, we use $\mu_h\ge c$, $L_{\infty,h}\le L_\infty$, and $q_{{\rm J},h}\le q_{\rm J}$. They give
\[
 0<\log\frac{L_{\infty,h}}{\eta\mu_h}\le\log\frac{L_\infty}{\eta c},
 \qquad
 -\log q_{{\rm J},h}\ge-\log q_{\rm J}>0,
\]
and hence
\[
 \frac{\log\bigl(\eta\mu_h/L_{\infty,h}\bigr)}{\log q_{{\rm J},h}}
 =\frac{\log\bigl(L_{\infty,h}/(\eta\mu_h)\bigr)}{-\log q_{{\rm J},h}}
 \le\frac{\log\bigl(L_\infty/(\eta c)\bigr)}{-\log q_{\rm J}}
 =\frac{\log\bigl(\eta c/L_\infty\bigr)}{\log q_{\rm J}}.
\]
Finally, by the definition of the iterates, each Jacobi iteration step applies $\mathsf O_\vartheta$ once and then divides entrywise by the diagonal of $\mathsf D_\vartheta$. By \eqref{eq:jacobi-splitting}, both components of $\mathsf O_\vartheta d^k$ are computed in one pass over the plan array. The remaining operations act on arrays of length $N_X+N_Y$, and $\mathsf D_\vartheta$ is computed once for each slope array $\vartheta$.
\end{proof}

Since $H_h(\vartheta)=\mathsf D_\vartheta(I+\mathsf D_\vartheta^{-1}\mathsf O_\vartheta)$, the Jacobi iterates are the partial sums of the Neumann series in the proof of \cref{lem:inverse}. Indeed, with $S$ as in that proof of \cref{lem:inverse} and $P=\mathsf D_\vartheta^{-1}\mathsf O_\vartheta$, we have $S\mathsf O_\vartheta S=-\mathsf O_\vartheta$, and therefore
\[
 d^k=-\sum_{m=0}^{k-1}\big(-\mathsf D_\vartheta^{-1}\mathsf O_\vartheta\big)^m\mathsf D_\vartheta^{-1}F
 =-S\sum_{m=0}^{k-1}P^m\mathsf D_\vartheta^{-1}SF .
\]

\begin{proposition}\label{prop:inexact-trials}
Let $z^m$ be an array potential with $F_h(z^m)\ne0$. We write
$e=z^m-z_h^*$, $\omega_h(r)=2\nu_h(2r)/(\mu_h\eps)$, and
$r_{\infty,h}=\norm{F_h(z^m)}_\infty/\mu_h$. Let
$z_{\rm N}^m$ satisfy \eqref{eq:inf-forcing} for some $\eta_m\ge0$, and
let $\varpi\in(0,1)$.
\begin{enumerate}
\item[(i)] The trial passes the residual test
$\norm{F_h(z_{\rm N}^m)}_\infty\le\varpi\norm{F_h(z^m)}_\infty$ whenever
\begin{equation}\label{eq:direct-acceptance}
 \eta_m+\frac{2(1+\eta_m)}{\mu_h\eps}\,m_h\big((4+2\eta_m)r_{\infty,h}\big)
 \le\varpi,
\end{equation}
where $m_h$ is the inclusive discrete boundary-layer function of \cref{sec:boundary-layer}.
\item[(ii)] The trial in \cref{sec:safeguard} satisfies
\begin{equation}\label{eq:inexact-tail}
 \norm{z_{\rm N}^m-z_h^*}_\infty\le
 \left[\omega_h(\norm{e}_\infty)+\eta_m\frac{L_{\infty,h}}{\mu_h}\right]
 \norm{e}_\infty .
\end{equation}
\item[(iii)] The trial also passes the residual test whenever
\[
 \omega_h(\norm{e}_\infty)+\eta_m\frac{L_{\infty,h}}{\mu_h}
 \le\frac{\varpi\mu_h}{L_{\infty,h}}.
\]
\end{enumerate}
\end{proposition}
\begin{proof}
We first prove (i). We write $H=H_h(\vartheta(z^m))$, $d=\widehat d^{\,m}$, and
$r=Hd+F_h(z^m)$, and let $s$ be the slack at $z^m$. Then
\[
 F_h(z^m+d)=\mathcal R^m+r,
\]
where we denote by $\mathcal R^m=F_h(z^m+d)-F_h(z^m)-Hd$ the Newton remainder at $z^m$ in the direction $d$. By \eqref{eq:Fh} and \eqref{eq:Hh}, we have
\[
	\mathcal R^m=\eps^{-1}\Lsf^{-1}\T^*\Lsf\big[(s+\Lsf^{-1}\T\Lsf d)_+-s_+-\vartheta(z^m)\,\Lsf^{-1}\T\Lsf d\big].
\]

For real $s,t$, let $\theta$ satisfy \eqref{eq:slopes} at $s$. We observe that
\begin{equation}\label{eq:scalar-rem-inclusive}
 |(s+t)_+-s_+-\theta t|\leq|t|\,\ind_{\{|s|\leq|t|\}}.
\end{equation}
Indeed, we compute
\[
	|(s+t)_+-s_+-\theta t|=
	\begin{cases}
		0,& (s+t)s>0\ \text{or}\ t=-s\neq0,\\
		|s+t|,& (s+t)s<0,\\
		(1-\theta)|t|,& s=0,\ t>0,\\
		\theta|t|,& s=0,\ t\le0.
\end{cases}
\]
In the case $(s+t)s<0$, we have $|s+t|=|t|-|s|$, so $|s|\le|t|$. In the last two cases $s=0$. Hence, \eqref{eq:scalar-rem-inclusive} holds in all cases.
Below, we transfer this localization from the slack $s$ of the iterate to the root slack $s_h^*$. The resulting set contains the zero set of $s_h^*$, so the inclusive boundary-layer function $m_h$ appears in the bound.

We apply \eqref{eq:scalar-rem-inclusive} with $t=\Lsf^{-1}\T\Lsf d$. If $|s|\le|\Lsf^{-1}\T\Lsf d|$, then, since $s-s_h^*=\Lsf^{-1}\T\Lsf e$,
\[
 |s_h^*|\le|s|+|\Lsf^{-1}\T\Lsf e|\le|\Lsf^{-1}\T\Lsf d|+|\Lsf^{-1}\T\Lsf e|\le2\norm{d}_\infty+2\norm{e}_\infty.
\]
Hence, $\mathcal (R^m)_{ij}=0$ unless $(i,j)\in\Omega$, where
\[
 \Omega=\big\{(i,j):|s^*_{ij,h}|\leq2\norm{d}_\infty+2\norm{e}_\infty\big\}.
\]
Also, by $|(\Lsf^{-1}\T\Lsf d)_{ij}|\le2\norm{d}_\infty$ and by the definition of $m_h$ in \eqref{sec:boundary-layer}, for $i=1,\dots,N_X$, we obtain
\[
	\Big|\frac1\eps\sum_{j=1}^{N_Y}\omega_j\big[(s_{ij}+(\Lsf^{-1}\T\Lsf)_{ij}d)_+-(s_{ij})_+-\vartheta_{ij}(\Lsf^{-1}\T\Lsf d)_{ij}\big]\Big|
 \le\frac{2\norm{d}_\infty}{\eps}\sum_{j:\,(i,j)\in\Omega}\omega_j
 \le\frac{2\norm{d}_\infty}{\eps}\,m_h\big(2\norm{d}_\infty+2\norm{e}_\infty\big).
\]
In a similar way,
\[
	\Big|\frac1\eps\sum_{i=1}^{N_X}w_i\big[(s_{ij}+(\Lsf^{-1}\T\Lsf d)_{ij})_+-(s_{ij})_+-\vartheta_{ij}(\Lsf^{-1}\T\Lsf d)_{ij}\big]\Big|
 \le\frac{2\norm{d}_\infty}{\eps}\,m_h\big(2\norm{d}_\infty+2\norm{e}_\infty\big).
\]
Together with $\norm{r}_\infty\le\eta_m\norm{F_h(z^m)}_\infty$, this gives
\[
 \norm{F_h(z^m+d)}_\infty
 \leq\frac{2\norm{d}_\infty}{\eps}\,m_h\big(2\norm{d}_\infty+2\norm{e}_\infty\big)
 +\eta_m\norm{F_h(z^m)}_\infty.
\]
Since $Hd=r-F_h(z^m)$, \cref{lem:inverse} gives
\[
 \norm{d}_\infty\leq\frac{\norm{r}_\infty+\norm{F_h(z^m)}_\infty}{\mu_h}\le\frac{(1+\eta_m)\norm{F_h(z^m)}_\infty}{\mu_h}=(1+\eta_m)r_{\infty,h}.
\]
The secant estimate \eqref{eq:secant} on the cell space gives $\norm{e}_\infty\leq r_{\infty,h}$. Hence, we have
\[
 2\norm{d}_\infty+2\norm{e}_\infty\leq2(1+\eta_m)r_{\infty,h}+2r_{\infty,h}=(4+2\eta_m)r_{\infty,h}.
\]
Since $m_h$ is nondecreasing, we obtain
\[
 \norm{F_h(z^m+d)}_\infty
 \leq\left[\eta_m+\frac{2(1+\eta_m)}{\mu_h\eps}\,m_h\big((4+2\eta_m)r_{\infty,h}\big)\right]
 \norm{F_h(z^m)}_\infty
 \le \varpi\norm{F_h(z^m)}_\infty.
\]
Here, we use \eqref{eq:direct-acceptance} for the last estimate. Hence, we obtain $\norm{F_h(z_{\rm N}^m)}_\infty\le\varpi\norm{F_h(z^m)}_\infty$.

Next, we prove (ii). We denote by $\mathcal R(z^m)=F_h(z^m)-F_h(z_h^*)-He$ the Newton remainder on the cell space. As in the proof of \cref{cor:local-inexact}, we have $z_{\rm N}^m-z_h^*=H^{-1}(r-\mathcal R(z^m))$. By \cref{lem:inverse} with $\mu_h$, \cref{lem:remainder} on the cell space, and the upper bound in \eqref{eq:inf-residual}, we obtain
\[
 \norm{z_{\rm N}^m-z_h^*}_\infty\le\frac{\norm{r}_\infty+\norm{\mathcal R(z^m)}_\infty}{\mu_h}
 \le\frac{\eta_mL_{\infty,h}\norm{e}_\infty}{\mu_h}+\omega_h(\norm{e}_\infty)\norm{e}_\infty,
\]
which is \eqref{eq:inexact-tail}.

Finally, we prove (iii). By \eqref{eq:inf-residual}, we have
\[
 \norm{F_h(z_{\rm N}^m)}_\infty
 \leq L_{\infty,h}\norm{z_{\rm N}^m-z_h^*}_\infty,
 \qquad
 \norm{e}_\infty
 \leq\frac{\norm{F_h(z^m)}_\infty}{\mu_h}.
\]
Combining these estimates with \eqref{eq:inexact-tail} gives
\[
 \norm{F_h(z_{\rm N}^m)}_\infty
 \leq\frac{L_{\infty,h}}{\mu_h}
 \left[
 \omega_h(\norm{e}_\infty)
 +\eta_m\frac{L_{\infty,h}}{\mu_h}
 \right]\norm{F_h(z^m)}_\infty
 \le \varphi\norm{F_h(z^m)}_\infty.
\]
Here, we use the condition in (iii) for the last estimate.
\end{proof}
The criterion \eqref{eq:direct-acceptance} is stated in terms of the computable residual radius $r_{\infty,h}$, and it contains no factor $L_{\infty,h}/\mu_h$.
Criterion (iii) uses the punctured boundary-layer function $\nu_h$, through $\omega_h$. In the proof of \cref{thm:algorithm}(iii), it shows that rejections stop on a fixed mesh, even when the discrete root has zero slacks.
\subsection{A complete iteration with mesh-uniform work bounds}
\label{sec:algorithm}
The results of \cref{sec:safeguard,sec:inexact-inner} combine into one iteration. Its parameters are computable from $a,b,c,E,\eps$ and a prescribed residual-reduction factor $\varpi\in(0,1)$.
We use the diagonal scaling $\mathsf D_h$ from \cref{sec:safeguard}, and put
\[
 L_\infty=E+\frac{2\max\{a,b\}}{\eps},
 \qquad
 \bar\eta=\frac{\varpi}{8},
 \qquad
 \bar\omega_{\bar\eta}=\frac{\varpi-\bar\eta}{1+\bar\eta}.
\]
Here, $\bar\eta$ is an upper bound for the forcing terms. The analysis below uses only $0<\bar\eta<\varpi$.
Let $q_{\rm J}$ be as in \eqref{eq:jacobi-factor}. For $0<\eta<1$, we define
\begin{equation}\label{eq:uniform-counts}
\begin{aligned}
 \bar N_{\rm G}=\left\lceil\frac{\log\bigl(L_\infty/(\varpi c)\bigr)}
 {-\log q_{\rm J}}\right\rceil,\qquad
 \bar N_{\rm J}(\eta)=\left\lceil\frac{\log\bigl(\eta c/L_\infty\bigr)}
 {\log q_{\rm J}}\right\rceil .
\end{aligned}
\end{equation}

\begin{algorithm}\label{alg:globalized}
Let the cellwise data satisfy \eqref{eq:data}, and fix $\varpi\in(0,1)$ and
$\mathrm{tol}>0$. Set $z^0=0$ and $F^0=F_h(z^0)$. For $m=0,1,2,\dots$:
\begin{enumerate}
\item[(1)] If $\norm{F^m}_\infty\le\mathrm{tol}$, return $z^m$.
\item[(2)] Set $\eta_m=\min\{\bar\eta,\norm{F^m}_\infty\}$ and
$\vartheta=\vartheta(z^m)$.
Starting from $d^0=0$ and $w^0=0$, compute the Jacobi iterates
\[
 d^k=\mathsf D_\vartheta^{-1}(-F^m-w^{k-1}),
 \qquad w^k=\mathsf O_\vartheta d^k,
 \qquad k=1,2,\dots.
\]
Stop at the first $k\geq1$ with
\[
 \norm{w^k-w^{k-1}}_\infty
 \leq\eta_m\norm{F^m}_\infty,
\]
and set $\widehat d^{\,m}=d^k$.
\item[(3)] Evaluate $F_{\rm N}=F_h(z^m+\widehat d^{\,m})$. If the residual test
$\norm{F_{\rm N}}_\infty\le\varpi\norm{F^m}_\infty$ holds, set
$z^{m+1}=z^m+\widehat d^{\,m}$ and $F^{m+1}=F_{\rm N}$.
\item[(4)] Otherwise, set $z^{m,0}=z^m$ and compute the scaled gradient steps
\[
 z^{m,j+1}=z^{m,j}-\mathsf D_h^{-1}F_h(z^{m,j}),\qquad j=0,1,2,\dots.
\]
Stop at the first $j\ge1$ with
\[
 \norm{F_h(z^{m,j})}_\infty\le\varpi\norm{F^m}_\infty .
\]
Then set $z^{m+1}=z^{m,j}$ and $F^{m+1}=F_h(z^{m,j})$.
\end{enumerate}
\end{algorithm}

\Cref{alg:globalized} has three nested levels, which we name as follows. One execution of steps (1)--(4), which maps $z^m$ to $z^{m+1}$, is an \emph{outer iteration}. Each outer iteration reduces the residual by at least the factor $\varpi$ as in \cref{sec:safeguard}. Within step (2), the Jacobi steps $d^k$, $k=1,2,\dots$, form the \emph{inner iteration}. It solves the linear Newton equation $H_h(\vartheta(z^m))d=-F_h(z^m)$ approximately, with $z^m$ fixed. One execution of step (4) is a \emph{safeguard phase}. Its scaled gradient steps $z^{m,j}$ act on the nonlinear equation $F_h(z)=0$ directly and are used only when the Newton trial is rejected.
Since $\mathsf D_\vartheta d^k=-F^m-w^{k-1}$ and $w^k=\mathsf O_\vartheta d^k$, we have the linear-residual identity
\begin{equation}\label{eq:linear-residual}
 H_h(\vartheta)d^k+F^m=\mathsf D_\vartheta d^k+\mathsf O_\vartheta d^k+F^m=w^k-w^{k-1}.
\end{equation}
Thus, the stopping test in step (2) is exactly the forcing condition
\eqref{eq:inf-forcing}.
Each Jacobi iteration step requires one application of $\mathsf O_\vartheta$. The product $w^k$ is retained for the next step.

\Cref{alg:globalized} is the globalized Newton iteration of \cref{S-sec:globalized-further}. It uses the inexact Newton trials of \cref{sec:inexact-inner}, the scaled gradient steps of \cref{lem:inf-gradient}(ii) as the safeguard, and $z^0=0$.

\begin{theorem}
\label{thm:algorithm}
Let the cellwise data $f_h,\,g_h,\,C_h$ satisfy \eqref{eq:data}, fix
$\varpi\in(0,1)$ and $\mathrm{tol}>0$, and put
\[
 M_{\rm tol}
 =\left\lceil
 \frac{\log_+(E/\mathrm{tol})}{\log(1/\varpi)}
 \right\rceil.
\]
We assume that the slopes at zero slack in \eqref{eq:slopes} are chosen by a fixed rule that does not depend on $\mathrm{tol}$.
\begin{enumerate}
\item[(i)]
Step (2) stops after at most
$\bar N_{\rm J}(\min\{\bar\eta,\mathrm{tol}\})$ Jacobi iteration steps, and
step (4) stops after at most $\bar N_{\rm G}$ scaled gradient steps.
\Cref{alg:globalized} returns after at most
$M_{\rm tol}$ outer iterations.

Let $R$ be the number of outer iterations that use step (4).
We count each evaluation of $F_h$, each evaluation of the slopes
$\vartheta(z^m)$ together with $\mathsf D_\vartheta$, and each application of
$\mathsf O_\vartheta$ as one operation. Then the total count is at most
\begin{equation}\label{eq:total-work}
 1+M_{\rm tol}
 \left(\bar N_{\rm J}(\min\{\bar\eta,\mathrm{tol}\})+2\right)
 +R\,\bar N_{\rm G},
 \qquad R\leq M_{\rm tol}.
\end{equation}
Here, residuals and matrix-vector products are stored and reused.
Each such operation costs $O(N_XN_Y)$ arithmetic operations.
All quantities in \eqref{eq:total-work} except $R$ depend only
on $a,b,c,E,\eps,\varpi$, and $\mathrm{tol}$.

\item[(ii)]
The returned potential $z$ satisfies $\norm{z-z_h^*}_\infty\leq\mathrm{tol}/\mu_h$ and
\begin{equation}\label{eq:algorithm-aposteriori}
\begin{aligned}
 \norm{\Lsf K(z)-k^*}_\Zz
 \leq D_h+
 \frac{\sqrt{a+b}\,\mathrm{tol}}{2\sqrt{\mu_h\eps}},\qquad
 \norm{\Lsf z-z^*}_{\Hh}
 \leq D_{z,h}+
 \frac{\sqrt{a+b}\,\mathrm{tol}}{c},
\end{aligned}
\end{equation}
where $D_h,\,D_{z,h}$ are from \cref{prop:stability}.

\item[(iii)]
We consider the \emph{continued iteration}, that is, the sequence $(z^m)_{m\ge0}$ generated by \cref{alg:globalized} without the stopping test in step (1), where we set $z^{m+1}=z^m$ once $F_h(z^m)=0$.
This sequence does not depend on $\mathrm{tol}$, and \cref{alg:globalized} returns one of its terms.
There is an index $m_1$ such that every Newton trial with $m\ge m_1$ is accepted. In particular, the number $R_h$ of rejected Newton trials in the continued iteration is finite, and $R\leq R_h$.
Moreover, for all $m\ge m_1$, we have
\[
 \norm{z^{m+1}-z_h^*}_\infty
 \leq\frac{L_{\infty,h}^2}{\mu_h}\norm{z^m-z_h^*}_\infty^2.
\]
Hence, on each fixed mesh, the continued iteration converges $q$-quadratically.

\item[(iv)]
Let $\widehat t>0$ satisfy
\[
 m_*(\widehat t\,)\leq\frac{\bar\omega_{\bar\eta}c\eps}{2}.
\]
Such a $\widehat t$ exists whenever $m_*(t)\to0$ as $t\downarrow0$.
Then, on every mesh with $\lambda_h<\widehat t$,
\[
 R\leq
 \left\lceil
 \frac{\log_+\big(2(2+\bar\eta)E/(c(\widehat t-\lambda_h))\big)}{\log(1/\varpi)}
 \right\rceil.
\]
Under the hypotheses of \cref{thm:mesh-local}, $\lambda_h\to0$. Hence, this bound holds on all but finitely many meshes of the sequence, and it is uniform on the meshes with $\lambda_h\leq\widehat t/2$.

\item[(v)]
Assume the hypotheses of \cref{cor:explicit-radius} with
$\bar\omega=\bar\omega_{\bar\eta}$ and
$\lambda_h<\bar\omega^2t_\eps$. Then one may take $\widehat t=\bar\omega^2t_\eps$ in (iv), and we have
\[
 R\leq m_0
 :=\left\lceil
 \frac{\log_+\big((2+\bar\eta)E/(c\,r_{0,h})\big)}{\log(1/\varpi)}
 \right\rceil.
\]
This bound still depends on the mesh through $r_{0,h}$.
Under the stronger condition
$\lambda_h\leq\tfrac12\bar\omega^2t_\eps$ of
\cref{cor:explicit-radius}(iv), we have
$r_{0,h}\geq\tfrac14\bar\omega^2t_\eps$, and hence
\[
 R\leq
 \left\lceil
 \frac{\log_+\left(4(2+\bar\eta)E/(c\,\bar\omega^2t_\eps)\right)}
      {\log(1/\varpi)}
 \right\rceil.
\]
With the choice of $\lambda_h$ in \cref{cor:explicit-radius}(iv) that is uniform in $\delta$, the bound
\eqref{eq:total-work} is therefore explicit and uniform
over all meshes satisfying this stronger condition and
all $\delta\in[0,\bar\delta]$.
\end{enumerate}
\end{theorem}
\begin{proof}
(i) While $\norm{F^m}_\infty>\mathrm{tol}$, we have $\eta_m\geq\min\{\bar\eta,\mathrm{tol}\}$.
By \eqref{eq:linear-residual}, the stopping test in step (2) is the residual condition of \cref{lem:jacobi} with $\eta=\eta_m$. Since $\bar N_{\rm J}$ is nonincreasing in $\eta$, the inner iteration stops after at most
\[
 N_{\rm J}(\eta_m)
 \leq\bar N_{\rm J}(\eta_m)
 \leq\bar N_{\rm J}(\min\{\bar\eta,\mathrm{tol}\})
\]
Jacobi iteration steps.
Furthermore, we apply \cref{lem:inf-gradient}(iv) with $z^0=z^m$ to the scaled steps. Since $q_{{\rm J},h}\leq q_{\rm J}$, $\mu_h\geq c$, and $L_{\infty,h}\leq L_\infty$, step (4) stops after at most $\bar N_{\rm G}$ scaled gradient steps.

Steps (3) and (4) both enforce $\norm{F^{m+1}}_\infty\leq\varpi\norm{F^m}_\infty$.
At $z=0$, we have $K_{ij}(0)=p(-C_{ij})=\max\{\delta,-C_{ij}/\eps\}=\delta$, because $C_h\geq0$. Thus, it follows that
\[
 F_h(0)=(\delta b-f_h,\delta a-g_h).
\]
By \eqref{eq:data}, we have $\delta b<c\le f_i\le E$ and $\delta a<c\le g_j\le E$. Hence, the entries of $F_h(0)$ belong to $[-E,0)$. By induction, we arrive at the estimate
\[
 \norm{F^m}_\infty
 \leq\varpi^m\norm{F_h(0)}_\infty
 \leq\varpi^mE.
\]
Since $\varpi^{M_{\rm tol}}E\le\mathrm{tol}$ by the definition of $M_{\rm tol}$, step (1) returns for some $m\leq M_{\rm tol}$.

We now count operations. The test in step (1) uses the stored residual $F^m$ and does not use additional operation.
As shown above, steps (2)--(4) are executed in at most $M_{\rm tol}$ outer iterations, and step (4) is executed in $R\le M_{\rm tol}$ of them. The count $ 1+M_{\rm tol} \left(\bar N_{\rm J}(\min\{\bar\eta,\mathrm{tol}\})+2\right)  +R\,\bar N_{\rm G}$ in \eqref{eq:total-work} arise as follows.
\begin{itemize}
\item The leading term $1$ is the evaluation of the initial residual $F^0=F_h(0)$ by \eqref{eq:Fh}.
\item The term $M_{\rm tol}\bigl(\bar N_{\rm J}(\min\{\bar\eta,\mathrm{tol}\})+2\bigr)$ accounts for steps (2) and (3) of each outer iteration. Step (2) evaluates the slopes $\vartheta(z^m)$ by \eqref{eq:slopes}, which determine the diagonal part $\mathsf D_\vartheta$ through \eqref{eq:jacobi-splitting}. It then performs at most $\bar N_{\rm J}(\min\{\bar\eta,\mathrm{tol}\})$ Jacobi steps $d^{k}=\mathsf D_\vartheta^{-1}(-F^m-w^{k-1})$. Each of these steps applies the off-diagonal part $\mathsf O_\vartheta$ once, $w^k=\mathsf O_\vartheta d^k$, by \cref{lem:jacobi}. Step (3) evaluates $F_h$ once, at the Newton trial $z^m+\widehat d^{\,m}$.
\item The term $R\,\bar N_{\rm G}$ accounts for the safeguard phases. By \cref{lem:inf-gradient}(iv) and \eqref{eq:uniform-counts}, each phase takes at most $\bar N_{\rm G}$ scaled gradient steps $z^{m,j+1}=z^{m,j}-\mathsf D_h^{-1}F_h(z^{m,j})$, and each step requires one new evaluation of $F_h$. Indeed, the first step uses the stored residual $F_h(z^{m,0})=F^m$. For $j\ge1$, the residual $F_h(z^{m,j})$ is evaluated once and reused, both for the stopping test and for the next step or as $F^{m+1}$.
\end{itemize}
By \eqref{eq:Fh}, \eqref{eq:slopes}, and \eqref{eq:jacobi-splitting}, each counted operation consists of a fixed number of sums over the $N_XN_Y$ product cells $X_i\times Y_j$, and hence costs $O(N_XN_Y)$. The diagonal scalings by $\mathsf D_\vartheta^{-1}$ and $\mathsf D_h^{-1}$, the updates of $d^k$ and $z^{m,j}$, and the maximum norms $\norm{\cdot}_\infty$ in the stopping tests cost only $O(N_X+N_Y)$ per step.

(ii) At return, we have $\norm{F_h(z)}_\infty\leq\mathrm{tol}$.
By \eqref{eq:inf-residual}, we obtain $\norm{z-z_h^*}_\infty\le\norm{F_h(z)}_\infty/\mu_h\le\mathrm{tol}/\mu_h$.
We compute 
\[
	\norm{F_h(z)}_h^2=\sum_{i=1}^{N_X}w_iF_{h,u,i}(z)^2+\sum_{j=1}^{N_Y}\omega_jF_{h,v,j}(z)^2\le(a+b)\norm{F_h(z)}_\infty^2.
\]
By \eqref{eq:finite-gap} and $f_i,g_j\geq\mu_h$, we have
\[
 \Gap_h(z)
 \leq\frac{\norm{F_h(z)}_h^2}{2\mu_h}
 \leq\frac{a+b}{2\mu_h}\norm{F_h(z)}_\infty^2.
\]
Therefore, we further have
\[
 \sqrt{\frac{\Gap_h(z)}{2\eps}}
 \leq\frac{\sqrt{a+b}\,\mathrm{tol}}{2\sqrt{\mu_h\eps}},\qquad
 \norm{F_h(z)}_h\leq\sqrt{a+b}\,\mathrm{tol}.
\]
By applying \cref{cor:total}, we obtain both estimates in \eqref{eq:algorithm-aposteriori}.

(iii) Apart from step (1), \cref{alg:globalized} does not involve $\mathrm{tol}$. With the fixed rule for the slopes, the continued iteration therefore does not depend on $\mathrm{tol}$. \Cref{alg:globalized} returns the term $z^m$ with the first index $m$ such that $\norm{F^m}_\infty\le\mathrm{tol}$.

We write $e^m=z^m-z_h^*$. The residual reduction in (i) also holds for the continued iteration. Hence, by \eqref{eq:inf-residual}, we have
\[
 \norm{e^m}_\infty\le\frac{\norm{F^m}_\infty}{\mu_h}\le\frac{\varpi^mE}{\mu_h}\to0,
 \qquad
 \eta_m\le\norm{F^m}_\infty\le\varpi^mE\to0.
\]
On a fixed mesh, the nonzero values of $|s^*_{ij,h}|$ form a finite set. Hence, $\nu_h(t)=0$ for all $t$ below the smallest of these values, and $\omega_h(r)=0$ for all sufficiently small $r$.
We therefore choose $m_1$ such that, for every $m\ge m_1$,
\[
 \omega_h(\norm{e^m}_\infty)=0
 \qquad\text{and}\qquad
 \eta_m\le\frac{\varpi\mu_h^2}{L_{\infty,h}^2}.
\]
Let $m\ge m_1$ with $F^m\neq0$. Then we have
\[
 \omega_h(\norm{e^m}_\infty)+\eta_m\frac{L_{\infty,h}}{\mu_h}
 \le\frac{\varpi\mu_h}{L_{\infty,h}},
\]
so \cref{prop:inexact-trials}(iii) shows that the Newton trial is accepted.
Hence, every Newton trial with $m\ge m_1$ is accepted, and $R_h\le m_1<\infty$. Since the iterates of \cref{alg:globalized} are terms of the continued iteration, we also have $R\leq R_h$.

For $m\ge m_1$ with $F^m\neq0$, the accepted trial is $z^{m+1}$. Hence, \cref{prop:inexact-trials}(ii) with $\omega_h(\norm{e^m}_\infty)=0$ gives
\[
 \norm{e^{m+1}}_\infty\le\eta_m\frac{L_{\infty,h}}{\mu_h}\norm{e^m}_\infty.
\]
By \eqref{eq:inf-residual}, we have $\eta_m\le\norm{F^m}_\infty\le L_{\infty,h}\norm{e^m}_\infty$. Therefore,
\[
 \norm{e^{m+1}}_\infty\le\frac{L_{\infty,h}^2}{\mu_h}\norm{e^m}_\infty^2.
\]
If $F^m=0$, then $e^{m+1}=e^m=0$, and this estimate holds trivially.

(iv) By the residual reduction and $\mu_h\geq c$, we first note that $r_{\infty,h}(z^m)=\norm{F^m}_\infty/\mu_h\leq\varpi^mE/c$.
Since $\eta_m\leq\bar\eta$ and $\mu_h\ge c$, the left-hand side of \eqref{eq:direct-acceptance} is at most
\[
 \bar\eta+\frac{2(1+\bar\eta)}{c\eps}\,m_h\big((4+2\bar\eta)r_{\infty,h}\big).
\]
We now assume that $(4+2\bar\eta)r_{\infty,h}+\lambda_h\leq\widehat t$. By the inclusive transfer $m_h(t)\leq m_*(t+\lambda_h)$ in the proof of \cref{thm:mesh-local} and the monotonicity of $m_*$, we have
\[
 m_h\big((4+2\bar\eta)r_{\infty,h}\big)\le m_*(\widehat t\,)\le\frac{\bar\omega_{\bar\eta}c\eps}{2}.
\]
Hence, by the definition of $\bar\omega_{\bar\eta}$, the left-hand side of \eqref{eq:direct-acceptance} is at most
\[
 \bar\eta+(1+\bar\eta)\bar\omega_{\bar\eta}=\bar\eta+(\varpi-\bar\eta)=\varpi.
\]
Thus, by \cref{prop:inexact-trials}(i), the Newton trial at step $m$ is accepted whenever
\begin{equation}\label{eq:accept-index}
 \varpi^mE/c\leq\frac{\widehat t-\lambda_h}{2(2+\bar\eta)}.
\end{equation}
Let \[m_0=\left\lceil
 \frac{\log_+\big(2(2+\bar\eta)E/(c(\widehat t-\lambda_h))\big)}{\log(1/\varpi)}
 \right\rceil,\] which is the bound in (iv). Since $0<\varpi<1$, \eqref{eq:accept-index} is equivalent to
\[
 m\log(1/\varpi)\geq\log\frac{2(2+\bar\eta)E}{c(\widehat t-\lambda_h)}.
\]
If the right-hand side is positive, this holds for every $m\geq m_0$. Otherwise, it holds for every $m\geq0$, and $m_0=0$.
Hence, every Newton trial with $m\geq m_0$ is accepted. Since $R$ counts the rejected trials, only the indices $m<m_0$ can contribute, and $R\leq m_0$. This proves (iv).
If $\lambda_h\leq\widehat t/2$, then $\widehat t-\lambda_h\geq\widehat t/2$, and the bound no longer depends on the mesh.

(v) By \cref{thm:curvature}, we have $m_*(t)\leq A_\eps\sqrt t$.
Since $\sqrt{t_\eps}=c\eps/(2A_\eps)$, we obtain
\[
 m_*(\bar\omega^2t_\eps)\leq A_\eps\bar\omega\sqrt{t_\eps}
 =\frac{\bar\omega_{\bar\eta} c\eps}2.
\]
Hence, $\widehat t=\bar\omega^2t_\eps$ is admissible in (iv).
By \eqref{eq:explicit-radius}, we have $\widehat t-\lambda_h=2r_{0,h}$. Hence, the bound in (iv) gives $R\leq m_0$.
Under the stronger mesh condition,
\cref{cor:explicit-radius}(iv) gives
\[
 r_{0,h}\geq\tfrac14\bar\omega^2t_\eps.
\]
Substituting this lower bound into $m_0$ gives the stated
explicit rejection bound.
With the choice of $\lambda_h$ that is uniform in $\delta$, all constants
are independent of the mesh and of
$\delta\in[0,\bar\delta]$, as claimed.
\end{proof}

We end this subsection with several remarks. First, the bound \eqref{eq:total-work} uses only the residual reduction by $\varpi$ per outer iteration.
It therefore does not reflect the eventual $q$-quadratic convergence in part (iii).
Using $R\leq M_{\rm tol}$, for $\mathrm{tol},\eps\leq1/2$, the bound is at most a constant times
\[
 \eps^{-1}
 \bigl(\log(1/\mathrm{tol})+\log(1/\eps)\bigr)
 \log(1/\mathrm{tol})
\]
operations of cost $O(N_XN_Y)$.
The constant depends on $a,b,c,E,\varpi$.
Under part (iv), the gradient contribution $R\,\bar N_{\rm G}$ is bounded independently of $\mathrm{tol}$.

Second, step (2) may use another inner solver whose output satisfies the same forcing condition.
For example, weighted conjugate gradients may be capped at $\bar N_{\rm J}(\eta_m)$ iterations.
If the condition has not been met, restarting the Jacobi iteration from zero preserves the work bound up to a constant factor (\cref{S-sec:globalized-further}).
Storage is $O(N_XN_Y)$, and the sparsity of the plan for $\delta=0$ is not exploited.

Third, a related globalized iteration is developed in a general Banach-space setting by Alphonse, Christof, Hinterm\"uller, and Papadopoulos \cite[Algorithm~2 and Theorem~2.4]{alphonse2026globalized}.
Their framework treats fixed-point equations $x=H(x)$ with a Newton differentiable and globally contractive map $H$. It also accommodates locally contractive problems through a projection-based localization.
Inexact steps are admitted through a forcing sequence.

In contrast, the contraction estimates here follow from the marginal structure.
\Cref{lem:inf-gradient} gives a maximum-norm contraction factor $q_{{\rm J},h}$ for the scaled gradient step, and \cref{lem:jacobi} gives the same factor for the Jacobi iteration.
Both factors have mesh-independent bounds computed from $a,b,c,E,\eps$ through the $M$-matrix structure described in \cref{rem:M-operator}.
These estimates also control the work. Indeed, $\bar N_{\rm J}$ bounds the number of Jacobi iteration steps per Newton trial, and
$\bar N_{\rm G}$ bounds the number of scaled gradient steps per safeguard phase.
Together, they give the bound \eqref{eq:total-work} on the number of full-array operations.

Part (iii) holds on each fixed mesh. However, the index $m_1$ depends on the smallest nonzero discrete slack, which may tend to zero under refinement, as explained in obstruction~(b) of the introduction.
Parts (iv) and (v) avoid this dependence. They bound the number of rejected Newton trials through the maximum-norm estimates on the boundary layer in \cref{sec:local,sec:mesh}.
These estimates rely on the maximum norm. Indeed, the Newton remainder is in general not small relative to the error in the norm of $\Hh$ alone, as shown in \cref{rem:L2}.

Finally, \eqref{eq:algorithm-aposteriori} bounds the error of the computed output relative to the continuum solution.
It combines the algebraic stopping tolerance with the data-approximation error.

\begin{proof}[Proof of \cref{thm:principal}]
Part (i): the common Newton neighborhood, error modulus, and the acceptance of the step size $\tau=1$ are \cref{thm:mesh-local}, and the continuum order is
\cref{thm:local}.

Part (ii): the supercloseness estimate is
\cref{thm:superclose}, with the potential regularity from \cref{lem:Lip}.
The first-order plan error is the discussion after
\eqref{eq:plan-decomposition}.

Part (iii): the uniform boundary-layer bound is
\cref{thm:curvature}. The explicit radius, contraction, acceptance of the step size $\tau=1$, uniformity in $\delta$, and mesh threshold are
\cref{cor:explicit-radius}(i), (ii), and (iv).
\end{proof}

\begin{proof}[Proof of \cref{cor:computational}]
This is \cref{thm:algorithm}(i), (iii), (iv), (v), and (ii), where
\eqref{eq:algorithm-aposteriori} gives \eqref{eq:principal-aposteriori} with
$D_h$ from \cref{prop:stability,cor:total}.
\end{proof}
\section{Numerical experiments}\label{sec:numerics}
In this section, we test the interface-controlled theory numerically.
We first examine approximation orders and Newton convergence at straight and curved transverse interfaces (\cref{sec:num-interface}) and near a tangency, that is, a point where the active interface is tangent to a fiber of $s^*$, so that transversality in the sense of \cref{sec:suff_geo} fails (\cref{sec:num-tangent}).
We then test the curvature criterion, the explicit convergence radius, and the uniformity in the lower bound $\delta$ (\cref{sec:num-curvature}). Finally, we test the globalized inexact semismooth Newton method (\cref{sec:num-2d}).

The supplementary materials report further experiments (\cref{S-sec:supp-numerics}). They measure the effectivity of the a posteriori error bounds, the sensitivity to $\eps$, $\delta$, and the metric, and the work needed to reach a prescribed value of the a posteriori error bound. Here, as usual in a posteriori error analysis, the \emph{effectivity index} of a bound is its ratio to the actual error.
For smooth data, the total a posteriori error bound \eqref{eq:total-plan} has effectivity index about $2$. For a discontinuous density, it is about $3.5$.
For smooth data, the computable refinement \eqref{S-eq:refined-data-bound} of the data part lowers its effectivity index from $2.97$ to between $1.43$ and $1.52$.

Unless stated otherwise, $n$ denotes the number of cells in each of the uniform partitions of $X$ and $Y$. Following \eqref{eq:h}, the mesh size is the product-cell diameter, so $h=\sqrt2/n$ for $X=Y=[0,1]$, and $h=2/n$ for $n\times n$ square partitions of $X=Y=[0,1]^2$.
For an error quantity $e_h$ computed on a set of meshes, the \emph{observed order} is the slope of the least-squares line fitted to the points $(\log h,\log e_h)$.

In \cref{sec:num-interface,sec:num-tangent,sec:num-curvature}, we test the theory for Newton steps without inexactness. Unless the exact solution is known, the reference roots are computed with a line-search Newton--CG method \cite[Section~7.1]{nocedal2006numerical}, in our implementation. Its version for the present problem, in the weighted inner product, is also discussed in \cref{S-sec:global-armijo} of the supplementary materials. In the Newton tests, every Newton equation $H_h(\vartheta(z^m))d^m=-F_h(z^m)$ is solved to rounding accuracy, with the slope $\vartheta_{ij}=0$ at zero slack in \eqref{eq:slopes}. In \cref{sec:num-interface}, we use a dense direct solver. In \cref{sec:num-tangent,sec:num-curvature}, we use the conjugate gradient method in the weighted inner product \eqref{eq:weighted-norm}, in which $H_h(\vartheta)$ is symmetric positive definite by \eqref{eq:H-form} (see also \cref{S-sec:global}). We use the diagonal $\mathsf D_\vartheta$ of \eqref{eq:jacobi-splitting} as preconditioner and the relative residual tolerance $10^{-13}$. This method uses only the action \eqref{eq:Hh} of $H_h(\vartheta)$ on vectors, so the Newton derivative is never assembled as a matrix of size $N_X+N_Y$. Only \cref{sec:num-2d} uses the inexact inner solvers of \cref{sec:inexact-inner}.

All runs are deterministic and use one BLAS thread.
Solver code, raw data, and environment information accompany the paper.

The mathematical bounds are stated in exact arithmetic.
The numerical values of the a posteriori error bounds reported here and in the supplementary materials are binary64 evaluations of those formulas.
For duality gaps, we use the sum of nonnegative squares in
\eqref{eq:finite-gap}.
These values measure the effectivity of the bounds, but they are not machine-verified interval enclosures.
Such enclosures would additionally require directed rounding and verified evaluation of every data-error and quadrature term.

In the manufactured tests, the continuum plan errors $\norm{\Lsf K(z_h)-k^*}_\Zz$ are computed with analytic moments or split-cell quadrature.
The uncertainty of each discrete reference solution is estimated by the residual radius $r_{\infty,h}$ of \cref{sec:uniform-stability} or by the plan bound \eqref{eq:sharp} with the discrete duality gap \eqref{eq:finite-gap}.

\subsection{Transverse interfaces}\label{sec:num-interface}
We consider two examples with transverse interfaces in the sense of \cref{sec:suff_geo}. In the first, the interface is straight and the potentials are constant, so the exact solution is available in closed form. In the second, the interface is curved and the potentials are nonconstant, so that both terms of the residual decomposition \eqref{eq:residual-decomposition} are tested.

\medskip
\noindent\textit{A straight interface with an exact solution.}
For $X=Y=[0,1]$, we set $\eps=1$, $\delta=0$, and
\begin{equation}\label{eq:interface-data}
 u^*=v^*=0.8,\qquad C(x,y)=0.3+x+y,\qquad
 k^*(x,y)=(1.3-x-y)_+.
\end{equation}
Then $s^*=u^*+v^*-C=1.3-x-y$ and $k^*=s^*_+/\eps$.
The exact marginal and input densities are
\[
 R(x)=\tfrac12[(1.3-x)_+^2-(0.3-x)_+^2],
 \qquad f(x)=R(x)/0.2,\qquad g(y)=R(y)/0.2.
\]
Here, we have $c=f(1)=0.225$ and $E=f(0)=4$, and the optimality conditions \eqref{eq:optimality} hold exactly.
Each fiber of $s^*$ has derivative $-1$, so \eqref{eq:transverse} with $J=1$ and $\gamma=1$ gives $m_*(t)\leq2t$.
Thus, the hypotheses of \cref{thm:mesh-local,thm:superclose} hold for this example.

We evaluate the cell averages of $k^*$ and $(k^*)^2$ using truncated cubic and quartic antiderivatives.
The density averages follow by summing the exact plan averages, and $C_h$ is the exact affine average.
The cell averages $\Q_h^{\Hh}z^*$ of the potentials (\cref{sec:discretization}) are the constant pair $(0.8,0.8)$.

To test a fixed local perturbation on every mesh, we initialize
\begin{equation}\label{eq:local-initial}
 u_i^0=0.8+0.04\cos(2\pi x_i),\qquad
 v_j^0=0.8+0.04\sin(2\pi y_j).
\end{equation}

As described at the beginning of \cref{sec:numerics}, we solve the Newton equations with a dense direct solver. Up to rounding, the iteration is therefore the full-step Newton iteration \eqref{eq:Newton} on the cell space, with forcing terms $\eta_m=0$. This separates the behavior of the Newton iteration from the errors of inexact linear solves.

The final maximum residual is below $2\times10^{-13}$ on every mesh.
By the secant estimate \eqref{eq:secant} on the cell space and $\mu_h\ge c=0.225$, the maximum-norm distance to the exact array root is therefore below $2\times10^{-13}/0.225<9\times10^{-13}$.

The sufficient condition of \cref{thm:mesh-local} does not cover this perturbation.
Indeed, let $\omega_h$ be the discrete modulus of \cref{prop:inexact-trials}, and let $\lambda_h$ be as in \eqref{eq:uniform-stability}. By \eqref{eq:layer-transfer} and $m_*(t)\leq2t$, we have
\[
 \omega_h(r)=\frac{2\nu_h(2r)}{\mu_h\eps}\le\frac{2m_*(2r+\lambda_h)}{c\eps}\leq\frac{4(2r+\lambda_h)}{c\eps}.
\]
This upper bound is below one when
\[
 r<\frac{c\eps}{8}-\frac{\lambda_h}{2}
   =0.028125-\frac{\lambda_h}{2}.
\]
Convergence from the perturbation of amplitude $0.04$
therefore illustrates the conservatism of this sufficient estimate.

The computed continuum plan error $\norm{\Lsf K(z_h^*)-k^*}_\Zz$ halves with each refinement,
from $2.21\times10^{-2}$ at $n=16$ to $6.93\times10^{-4}$
at $n=512$, in agreement with the first-order plan error discussed after \eqref{eq:plan-decomposition}.
The potential error $e_{z,h}=\norm{z_h^*-\Lsf^{-1}\Q_h^{\Hh}z^*}_\infty$ satisfies $e_{z,h}/h^2<0.09$ (\cref{S-tab:interface}).
Since the exact potentials are constant, \cref{thm:superclose} applies with $L_u=L_v=L_{\rm cov}=0$ and $L_s=\sqrt2$.
The cells are intervals, and $m_*(L_sh)\le2\sqrt2\,h$. Hence, \eqref{eq:superclose} with the factor $1/(4\sqrt3)$ gives the conservative bound
\[
 e_{z,h}\leq\frac1c\cdot\frac{\sqrt2\,h}{4\sqrt3\,\eps}\cdot2\sqrt2\,h=\frac{h^2}{\sqrt3\,c\eps}\approx2.57\,h^2.
\]

Starting from \eqref{eq:local-initial}, the full-step Newton iteration \eqref{eq:Newton} on the cell space reduces the error superlinearly and reaches the stopping tolerance in a few steps.
At $n=512$, the maximum potential errors after the first and second steps are $3.9608\times10^{-3}$ and $1.3091\times10^{-5}$, respectively.
The third step reaches the stopping tolerance.
At every step with maximum residual above $10^{-8}$, the Newton step with step size one also satisfies the Armijo condition \eqref{eq:Armijo} with $\sigma=10^{-4}$.

\begin{figure}[tb]
\centering
\includegraphics[width=\textwidth]{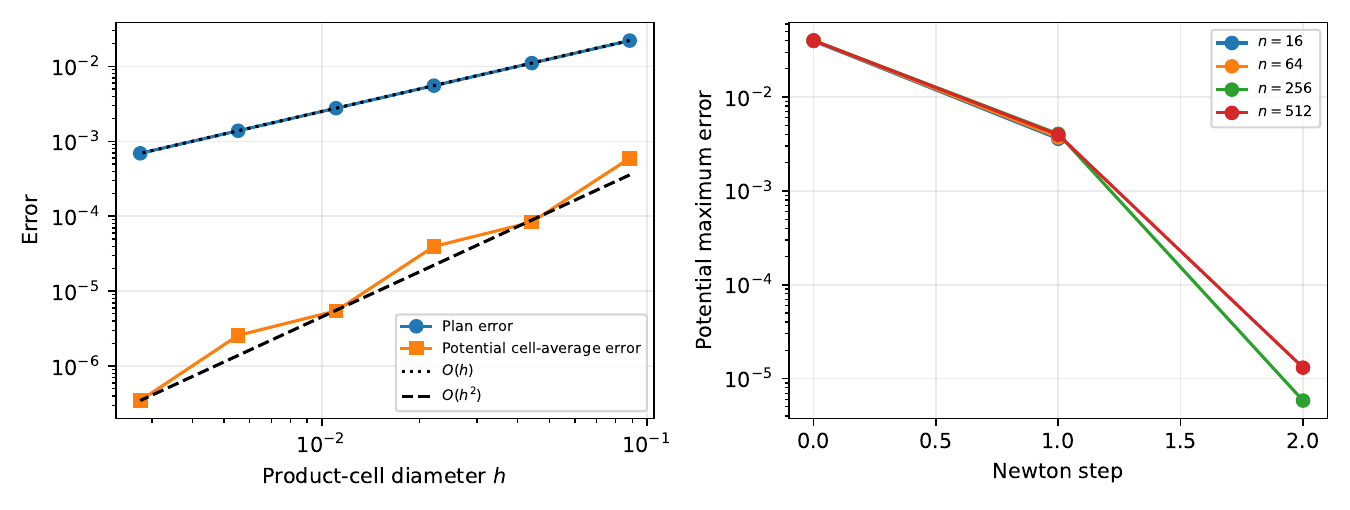}
\caption{Left: first-order continuum plan error and second-order
potential supercloseness. Right: maximum-norm Newton errors from the
perturbation \eqref{eq:local-initial} on four grids. Values at the final
numerical reference are omitted from the logarithmic plot. For $n=16$ and
$n=64$, this is already the second step, so these curves end after one
step.}\label{fig:interface}
\end{figure}

To test the boundary-layer bound, we evaluate $\nu_h(t)$ at
\[
 t\in\{0.1,0.05,0.025,0.0125,0.00625,0.003125\}.
\]
All thirty-six measured values, for six meshes and six values of $t$, satisfy
\[
 \nu_h(t)\leq2(t+2e_{z,h}+1/n),
\]
up to the recorded reference tolerance.

This uses a sharper slack-error estimate than the general data-dependent bound $\lambda_h$ in \eqref{eq:uniform-stability}.
Since the exact potentials are constant and $\norm{C_h-C}_{L^\infty(X\times Y)}=1/n$, we have
\[
	\norm{s_h^*-s^*}_{L^\infty(X\times Y)}
	\leq2\norm{\Lsf z_h^*-z^*}_\Bb+\norm{C_h-C}_{L^\infty(X\times Y)}
 =2e_{z,h}+1/n.
\]
Thus, $2e_{z,h}+1/n$ replaces $\lambda_h$ in the transferred boundary-layer bound \eqref{eq:layer-transfer} for this test.
It is a problem-specific a posteriori bound, rather than the quantity defined in \eqref{eq:uniform-stability}.

\medskip
\noindent\textit{A curved interface with nonconstant potentials.}
We call the first and second terms in the residual decomposition \eqref{eq:residual-decomposition} the \emph{interface defect} and the \emph{covariance term}. In \eqref{eq:interface-data}, the potentials are constant, so the covariance term vanishes.
The following example tests both the interface defect and the covariance term.

On $X=Y=[0,1]$, take $\eps=0.1$, $\delta=0$, and
\begin{equation}\label{eq:curved-data}
 \begin{gathered}
 T(x)=x+\tfrac12x(1-x),\qquad
 u^*(x)=0.3+0.1x,\qquad
 v^*(y)=0.25+0.1y^2,\\
 s^*(x,y)=10\bigl(0.04-(y-T(x))^2\bigr).
 \end{gathered}
\end{equation}
We take
\[
 C=u^*+v^*-s^*,\qquad
 k^*=s^*_+/\eps,\qquad
 f=\frac{k_X^*}{1-u^*},\qquad
 g=\frac{k_Y^*}{1-v^*}.
\]
The optimality conditions \eqref{eq:optimality} hold exactly.
The data satisfy $C\geq0.15$, $c=0.491$, and $E=2.080$.

The support of $k^*$ is the strip $|y-T(x)|<0.2$, and the active interface consists of
the curves $y=T(x)\pm0.2$ within the domain.
On these curves,
\[
 |\partial_ys^*|=4,\qquad
 |\partial_xs^*|=4T'(x)\geq2.
\]
Since $T'(x)=\tfrac32-x>0$, each fiber of $s^*$ splits into at most two monotone pieces. Thus, the interface is transverse in the sense of \cref{sec:suff_geo}, and $m_*(t)=O(t)$.
For each fixed fiber, the measure of $\{|s^*|\leq t\}$ has a closed form.
On the fiber $\{x\}\times Y$, the condition $|s^*(x,y)|\leq t$ is equivalent to
\[
 0.04-\tfrac{t}{10}\leq\big(y-T(x)\big)^2\leq0.04+\tfrac{t}{10},
\]
so the boundary layer consists of at most two intervals in $y$. Since $T$ is an increasing bijection of $[0,1]$ with $T^{-1}(\tau)=\tfrac32-\sqrt{\tfrac94-2\tau}$, the boundary layer on $X\times\{y\}$ is the preimage under $T$ of at most two intervals.
We approximate $m_*(t)$ by the maximum of these fiber measures over $4001$ equally spaced values of $x$ and of $y$. On $40001$ fibers, the maximum changes by a relative amount below $5\times10^{-4}$.

The cell averages of $u^*$ and $v^*$ are computed exactly by Gauss--Legendre rules, and those of $C$ by a tensor-product Gauss--Legendre rule.
For the cell moments of $k^*$ and $(k^*)^2$, we integrate in $y$ in closed form. The result is piecewise polynomial in $x$, with breakpoints where the curves $y=T(x)\pm0.2$ cross the cell edges in $y$. A Gauss rule on each piece is then exact.
The densities $f$ and $g$ are smooth except where the strip $|y-T(x)|<0.2$ meets the boundary of the square, that is, at $T(x)\in\{0.2,0.8\}$ and $y\in\{0.2,0.8\}$, respectively. There we split the cells and apply a $12$-point Gauss rule on each piece.
In spot checks, these values agree with adaptive quadrature to $10^{-15}$.

We also check the residual decomposition \eqref{eq:residual-decomposition}.
At the averaged potentials, the sum of the interface defect and the covariance term agrees with $F_h(\Lsf^{-1}\Q_h^{\Hh}z^*)$ to $7\times10^{-15}$ on every mesh.
The array roots are computed to residual below $3\times10^{-15}$.

\begin{table}[tb]
\caption{Curved interface \eqref{eq:curved-data}. Here
	$e_{z,h}=\norm{z_h^*-\Lsf^{-1}\Q_h^{\Hh}z^*}_\infty$, and the plan error is measured against the
exact continuum plan. The defect and covariance columns are the maximum
norms of the interface defect and the covariance term in
\eqref{eq:residual-decomposition}. The last column is the first-step error ratio of the full-step
Newton iteration started from the perturbation of amplitude $0.005$.}
\label{tab:curved}
\centering\small
\begin{tabular}{rrrrrrrr}
\toprule
$n$&$e_{z,h}$&$e_{z,h}/h^2$&plan error$/h$&defect&covariance&$L_{\rm cov}h^2/$cov.&Newton\\
\midrule
32  &$1.8210\times10^{-3}$&0.932&3.852&$1.15\times10^{-2}$&$1.26\times10^{-4}$&25.7&0.072\\
64  &$7.2550\times10^{-4}$&1.486&3.904&$3.65\times10^{-3}$&$3.26\times10^{-5}$&24.8&0.048\\
128 &$1.8570\times10^{-4}$&1.522&3.927&$7.72\times10^{-4}$&$8.29\times10^{-6}$&24.4&0.031\\
256 &$3.4960\times10^{-5}$&1.146&3.939&$1.90\times10^{-4}$&$2.09\times10^{-6}$&24.2&0.025\\
512 &$1.0870\times10^{-5}$&1.425&3.945&$5.00\times10^{-5}$&$5.24\times10^{-7}$&24.1&0.013\\
1024&$3.1250\times10^{-6}$&1.638&3.947&$1.42\times10^{-5}$&$1.31\times10^{-7}$&24.0&0.013\\
2048&$8.0670\times10^{-7}$&1.692&3.949&$3.55\times10^{-6}$&$3.29\times10^{-8}$&24.0&0.010\\
\bottomrule
\end{tabular}
\end{table}

\Cref{tab:curved} is consistent with second-order supercloseness.
The ratio $e_{z,h}/h^2$ stays between $0.93$ and $1.69$, and the observed order computed from $n=64$ and $n=2048$ is $1.96$.
The orders between consecutive meshes range from $1.33$ to $2.41$, reflecting the varying positions at which the curves cross the cells.
The plan error is first order, with plan error$/h$ tending to $3.95$.

Both terms in the residual decomposition are of second order.
The covariance term is about one percent of the interface defect, and the ratio of $L_{\rm cov}h^2$ to this term tends to $24.0$.

This factor can be explained as follows. For linear $f$ and $u^*$ on an interval $X_i$ of width $w_i$, the cell covariance equals $f'u^{*\prime}w_i^2/12$. For smooth $f$ and $u^*$, we therefore have
\[
 \frac1{w_i}\int_{X_i}(f-f_i)(u^*-\bar u_i)\dd x
 =\frac{f'u^{*\prime}}{12}w_i^2+o(w_i^2).
\]
Since $w_i=h/\sqrt2$, we have $w_i^2/12=h^2/24$, which gives the factor $24$.

For interval cells, the covariance bound in the proof of \cref{thm:superclose} is $L_{\rm cov}h^2/12$.
Its ratio to the observed covariance term tends to $2$, which is exactly $h^2/w_i^2$.
Thus, apart from measuring the cell width by the product-cell diameter $h$, this part of the estimate is sharp.

The estimate \eqref{eq:superclose} is conservative.
At $n=2048$, its numerical evaluation for interval cells gives $2.66\times10^{-3}$, whereas $e_{z,h}=8.07\times10^{-7}$.
About a factor $3.7\times10^2$ comes from the interface-defect estimate. This estimate uses the global Lipschitz constant $L_s=36.1$ of the slack from \cref{lem:Lip} to set the width of the boundary layer.
A further factor of about $9$ comes from the secant estimate \eqref{eq:secant}.
The Lipschitz constants $L_f=8.68$ and $L_g=8.28$ of the densities used in this evaluation are finite-difference estimates on a fine grid, rather than rigorous upper bounds.

For the transferred boundary-layer bound \eqref{eq:layer-transfer}, we use
\[
 t\in\{0.1,0.05,0.025,0.0125,0.00625\}.
\]
The measured values satisfy $\nu_h(t)\leq m_*(t+\widetilde\lambda_h)$ for all $35$ pairs of a mesh and a value of $t$.
Here $\widetilde\lambda_h\approx L_sh/2$ is the maximum of $|s_h^*-s^*|$ sampled at $25$ points per cell.
This is a sampled diagnostic, not a rigorous bound on the maximum slack error.

Finally, we test the full-step Newton iteration started from $z_h^*+0.005(\cos2\pi x,\sin2\pi y)$, again with dense direct solves of the Newton equations.
The first-step maximum-norm error ratios are reported in the last column of \cref{tab:curved}.
For $n\geq128$, the discrete modulus $\omega_h(r)=2\nu_h(2r)/(\mu_h\eps)$ of \cref{prop:inexact-trials} is about $0.6$. Hence, the sufficient condition $\omega_h(r)<1$ of \cref{thm:local}, applied on the cell space, is satisfied.
The iteration reaches the stopping tolerance after two or three steps.

For amplitude $0.02$ and $n\leq1024$, the modulus is between $2.1$ and $2.4$, so the sufficient condition does not apply.
The iteration nevertheless reaches the stopping tolerance in two or three steps.
\subsection{Near a tangency}\label{sec:num-tangent}
To test the case $\alpha=1/2$, we replace $T$ in \eqref{eq:curved-data} by the non-monotone cubic
\begin{equation}\label{eq:tangent-data}
 T(x)=x+2.5\,x(1-x)(1-2x),
\end{equation}
and keep $u^*$, $v^*$, $\eps=0.1$, $\delta=0$, and the construction of $C,f,g$.
We also keep the slack $s^*(x,y)=10\big(0.04-(y-T(x))^2\big)$ of \eqref{eq:curved-data}.
The density bounds, computed numerically, are $c\approx0.226$ and $E\approx3.856$.

Since $T'(x)=3.5-15x+15x^2$, the map $T$ has a local maximum at $x_c=\tfrac12-\tfrac{\sqrt{15}}{30}\approx0.3709$ and a local minimum at $x_m=\tfrac12+\tfrac{\sqrt{15}}{30}\approx0.6291$.
At the four levels $y=T(x_c)\pm0.2$ and $y=T(x_m)\pm0.2$, the fiber $x\mapsto s^*(x,y)$ has a double zero at $x_c$ or $x_m$. That is, the fiber vanishes there together with its first derivative, while its second derivative is nonzero. Hence, each of these four points is a tangency, and transversality in the sense of \cref{sec:suff_geo} fails there. We call these four values of $y$ the \emph{tangent levels}.
At the levels $y=T(x_c)+0.2$ and $y=T(x_m)-0.2$, the fiber is nonpositive near $x_c$ and $x_m$, respectively, so the interface is approached from the side where $k^*=\delta$. At the levels $y=T(x_c)-0.2$ and $y=T(x_m)+0.2$, the fiber is nonnegative near $x_c$ and $x_m$, respectively, so the interface is approached from the side where $k^*>\delta$.
The fibers $y\mapsto s^*(x,y)$ remain transverse, since $|\partial_ys^*|=4$ on the active interface.

Near the tangent level $y=T(x_c)+0.2$, we have $y-T(x)=0.2+\tfrac12|T''(x_c)|(x-x_c)^2+O(|x-x_c|^3)$. Hence, the section of the boundary layer $\{x:|s^*(x,y)|\le t\}$ at this level has measure
\[
 2\left(\frac{t}{2|T''(x_c)|}\right)^{1/2}+O(t),
\]
where $|T''(x_c)|=\sqrt{15}$. The leading coefficient is $2(2\sqrt{15})^{-1/2}\approx0.719$.
The sampled maximum over $2001$ uniformly spaced fibers in each direction, augmented by the tangent levels, gives $0.721$ after division by $\sqrt t$ at $t=10^{-5}$.
This sampled maximum need not equal $m_*(t)$, since the maximizing fiber can move away from a tangent level.
Thus, the double zeros at $x_c$ and $x_m$ give $m_*(t)=O(\sqrt t)$, that is, $\alpha=1/2$, and \cref{thm:superclose} predicts $e_{z,h}=O(h^{3/2})$.

The two components of the interface defect have different behavior.
In the proof of \cref{thm:superclose}, the $u$-component is controlled by the sections $\{y:|s^*(x,y)|\le t\}$, and the $v$-component by the sections $\{x:|s^*(x,y)|\le t\}$.
With $t=L_sh$, the factor $(L_sh/(4\eps))$ multiplies a section measure of order $h$ for the transverse fibers in $y$, and of order $h^{1/2}$ for the fibers in $x$ at the tangent levels. Thus, the corresponding defect estimates have orders $h^2$ and $h^{3/2}$.
Since the potential equations are coupled, these estimates do not by themselves give separate rates for the potential components.
We examine those rates numerically below.

The integrals are computed as before, except on subintervals that end at a tangent level $y^\dagger$. Near such a level, $k_Y^*$ contains a one-sided term of the form $|y-y^\dagger|^{3/2}$, which is not smooth. On these subintervals, we change the integration variable from $y$ to $\rho\ge0$ by $y=y^\dagger-\rho^2$ or $y=y^\dagger+\rho^2$. Then $|y-y^\dagger|^{3/2}=\rho^3$, and the integrand becomes smooth in $\rho$. Without this change of variables, Gauss rules lose up to seven digits on those cells.

We write $e_u$ and $e_v$ for the two components of $e_{z,h}=\norm{z_h^*-\Lsf^{-1}\Q_h^{\Hh}z^*}_\infty$.
On a single mesh sequence, the observed order of $e_v$ depends strongly on where the tangent levels fall inside the cells. For instance, nine meshes with $n$ between $512$ and $2048$ give an observed order of $e_v$ near one.
To reduce this dependence on position, we replace $T$ by $T+\xi$ with $\xi\in\{0,\tfrac1{4n},\tfrac2{4n},\tfrac3{4n}\}$, which moves the four tangent levels by a quarter cell at a time. We report the maximum and the mean over these four positions.

\begin{table}[tb]
\caption{Interface with tangencies \eqref{eq:tangent-data}. For each $n$, the
potential errors and defect components are maximized (or averaged) over the
four shifts $\xi$. Here, $e_u$ and $e_v$ are the two components of
$e_{z,h}$, and the defect columns are the two components of the interface
defect in \eqref{eq:residual-decomposition}. The plan error is for
$\xi=0$. The last row gives the observed orders over $n=128,\dots,2048$.}
\label{tab:tangent}
\centering\small
\begin{tabular}{rrrrrrr}
\toprule
$n$&$\max e_u/h^2$&$\max e_v/h^{3/2}$&mean $e_v/h^{3/2}$&$\max$ defect$_u$&$\max$ defect$_v$&plan error$/h$\\
\midrule
128 &3.620&0.568&0.528&$1.66\times10^{-3}$&$2.48\times10^{-3}$&4.886\\
256 &3.912&0.521&0.451&$4.07\times10^{-4}$&$9.09\times10^{-4}$&4.907\\
512 &3.905&0.559&0.468&$1.09\times10^{-4}$&$2.91\times10^{-4}$&4.916\\
1024&4.143&0.617&0.500&$2.53\times10^{-5}$&$1.00\times10^{-4}$&4.922\\
2048&3.974&0.628&0.462&$6.79\times10^{-6}$&$3.55\times10^{-5}$&4.924\\
\midrule
order&1.96&1.45&1.52&1.99&1.54&\\
\bottomrule
\end{tabular}
\end{table}

\Cref{tab:tangent} supports the predicted combined rate $O(h^{3/2})$ and shows faster convergence of the $u$-component.
Its observed order is $1.96$, with defect order $1.99$.
The $v$-component is consistent with order $3/2$: the maximum over the four positions has order $1.45$, the mean has order $1.52$, and $\max e_v/h^{3/2}$ stays between $0.52$ and $0.63$.
The plan error remains of first order.

For $\xi=0$, the residual decomposition \eqref{eq:residual-decomposition} holds up to $1.4\times10^{-14}$ on the meshes in \cref{tab:tangent}.
The sampled comparisons with \eqref{eq:layer-transfer} also hold on all meshes and for all values of $t$ tested. Here, $m_*$ is sampled as above, and the slack error $|s_h^*-s^*|$ is sampled at $25$ points per cell.
These comparisons are numerical checks rather than rigorous bounds.

Using the sampled values of $m_*$ and the estimated constants, the right-hand side of \eqref{eq:superclose} for interval cells evaluates to $9.5\times10^{-2}$ at $n=2048$ and $\xi=0$, whereas $e_{z,h}=1.1\times10^{-5}$. This evaluation is not a rigorous upper bound. However, it indicates the large gap between the theoretical estimate and the observed error.

The same example illustrates the effect of $\alpha$ on local Newton convergence. We start the full-step Newton iteration from the perturbations $z_h^*+\varsigma\,(\cos2\pi x,\sin2\pi y)$ of the array root, with
$\varsigma\in\{0.05,0.02,0.01,0.005\}$. Let $e^0$ and $e^1$ denote the maximum-norm errors before and after one Newton step with step size one. We fit a least-squares line to the points $(\log e^0,\log e^1)$ over the four amplitudes. Its slope, which estimates the local $q$-order of the iteration, is $1.53$ for $n=512$ and $1.51$ for $n=2048$.
These finite-amplitude slopes are consistent with the continuum order $1+\alpha=3/2$ in \eqref{eq:local-order}.
For the transverse interface \eqref{eq:curved-data}, the corresponding slopes are $1.84$ and $1.93$.

For $a=0.1$, the ratio $e^1/e^0$ is about $0.15$ in the transverse case, but between $0.98$ and $1.00$ for the interface with tangencies. This is consistent with a smaller neighborhood of rapid convergence when the interface has tangencies.
For $a$ below a few multiples of $1/n$, the local slopes become irregular in both examples.
This behavior is consistent with the influence of the discretization scale described by the term $\lambda_h$ in \eqref{eq:two-scale-newton} and \cref{rem:two-scale}.
\subsection{Curvature criterion, explicit convergence radius, and uniformity in the lower bound}
\label{sec:num-curvature}
We test \cref{thm:curvature} and \cref{cor:explicit-radius} on the data
of \cref{sec:curvature}: $X=Y=[0,1]$, $C=\kappa_C(x-y)^2/2$,
$f(x)=1+0.1x$, $g(y)=1+0.1y$, with exact cell averages and $\kappa_C=16$. Then
$c=1$, $U_{X,e}=U_{Y,e}=\kappa_C$, and $R_X=R_Y=0.2\kappa_C$. Hence, the margin conditions \eqref{eq:fc-margins} hold with $\gamma_X=\gamma_Y=0.8\kappa_C/(1+1/\eps)$,
and the boundary-layer bound \eqref{eq:fc-uniform-layer} holds with
$A_\eps=4\gamma_X^{-1/2}$ for all $\delta\in[0,0.95]$.

We use three
regularizations: A ($\eps=4$), B ($\eps=1$), and C ($\eps=0.1$). With
$L_C=\sqrt2\,\kappa_C$, the Lipschitz constant of $C$ on $X\times Y$, and with
$\lambda_h$ replaced by its bound $L_\lambda h$ here and in
\cref{tab:explicit-radius}, the first mesh condition of
\cref{cor:explicit-radius}(iv) with $\bar\omega=0.9$ is $n\ge24$
for A and $n\ge1193$ for B. For C, the threshold is approximately
$n\ge4.6\times10^6$, in line with the decay of $t_\eps$ like $\eps^3$ (\cref{sec:curvature}).

For A and B, the set where $k^*>\delta$ becomes empty when $\delta$ reaches a critical value $\delta_{\rm crit}$. Then $k^*=\delta$ everywhere, and there is no active interface.
To estimate $\delta_{\rm crit}$, we compute the array root at $n=512$ and the largest discrete slack $S(\delta)=\max_{i,j}s^*_{ij,h}$. The set where $k^*>\delta$ is nonempty exactly when $S(\delta)>0$. We locate the change of sign of $S$ by bisection, starting from the interval $[0,0.95]$ and halving it $24$ times, and obtain the values in \cref{tab:curvature}. If $S(0.95)>0$, as for C, there is no critical value in this range.
We solve each configuration for $\delta\in\{0,0.1,\dots,0.9\}$ and, for A and B, also for $\delta=\delta_{\rm crit}+\{0,\pm10^{-4},\pm10^{-3},\pm10^{-2}\}$.

We test the strong concavity directly. We write $\Delta_r^2u(x)=u(x+r)+u(x-r)-2u(x)$ for the second difference with step $r$. Since $C$ is quadratic with $\partial_x^2C=\kappa_C$, we have $\Delta_r^2C(\cdot,y)=\kappa_Cr^2$. Hence, the conclusion of \cref{thm:curvature} is equivalent to the midpoint inequality
\[
 \Delta_r^2u^*(x)\le(\kappa_C-\gamma_X)r^2,
\]
and similarly for $v^*$.
Cell averaging preserves this inequality for shifts by whole cells.
We therefore test the corresponding midpoint quotients of the array roots at $n=1024$, with $r=20/n$.
These provide a numerical check of the predicted concavity.
We also evaluate the discrete boundary-layer function on the same roots for $t\in\{10^{-1},\dots,10^{-4}\}$.
In this test, we include the cells with zero slack. Hence, we evaluate the inclusive function $m_h$, which bounds $\nu_h$ from above.
At $t=10^{-4}$, one cell contributes $0.098$ to $m_h(t)/\sqrt t$.

We measure supercloseness against reference solutions on $n=2048$, averaged to the coarse cells. This is possible because the partitions are nested.
The reference roots are computed to residual below $2\times10^{-15}$. For $\eps=0.1$, this Newton--CG method stalls near $5\times10^{-9}$ at $n=2048$. We therefore finish with at most eight Newton steps with step size one, whose Newton equations are solved by the conjugate gradient method described at the beginning of \cref{sec:numerics}.
Replacing the $n=2048$ references by $n=1024$ changes the errors by at most $10^{-6}$, about four percent of the $n=256$ values.

\begin{table}[tb]
\caption{Curvature criterion across lower bounds $\delta$, $\kappa_C=16$. Suprema are over
the values of $\delta$ listed in the text. The midpoint test uses $r=20/1024$ and the second difference $\Delta_r^2$.
Here $e_{z,h}^{\rm ref}=\norm{z_h^*-\Lsf^{-1}\Q_h^{\Hh}z_{2048}^*}_\infty$.
The last column reports $\sup_\delta e_{z,h}^{\rm ref}/h^2$ for
$n=64,128,256$. A dash means not computed (B) or no critical value $\delta_{\rm crit}$ in
$[0,0.95]$ (C).}
\label{tab:curvature}
\centering\small
\begin{tabular}{lrrrrrrl}
\toprule
&$\eps$&$\delta_{\rm crit}$&$\max\Delta_r^2u_h^*/r^2$&$\kappa_C-\gamma_X$&$\sup_{\delta,t}m_h(t)/\sqrt t$&$A_\eps$&$\sup_\delta e_{z,h}^{\rm ref}/h^2$\\
\midrule
A&4  &0.3437&1.31&5.76 &1.02&1.25&0.19, 0.16, 0.17\\
B&1  &0.7096&3.15&9.60 &1.00&1.58&--\\
C&0.1&--    &6.64&14.84&1.01&3.71&0.69, 0.70, 0.74\\
\bottomrule
\end{tabular}
\end{table}

\Cref{tab:curvature} tests the conclusions of \cref{thm:curvature} that are uniform in $\delta$, at the sampled values of $\delta$. The midpoint quotients stay below the bound $\kappa_C-\gamma_X$ by a factor of at least $2.2$, and the sampled boundary-layer function stays below $A_\eps\sqrt t$.
Near $\delta_{\rm crit}$, the largest sampled ratio $m_h(t)/\sqrt t$ for each $\delta$ lies between $0.43$ and $1.02$. This is consistent with the bound $A_\eps\sqrt t$, which is uniform in $\delta$.
The sampled reference-error supremum has observed orders $2.09$ (A) and $1.95$ (C).
Its largest sampled values occur away from $\delta_{\rm crit}$, consistent with the faster behavior observed for transverse interfaces.

\begin{table}[tb]
\caption{Explicit convergence radius of \cref{cor:explicit-radius} across
lower bounds $\delta$, $\kappa_C=16$. For every $\delta$ considered, one Newton step with step size one is
taken from six perturbations of the array root with maximum norm
$j\,r_{0,h}$. These are two random directions with entries uniform in $[-1,1]$, two random sign patterns, and the direction $(1,-1)$, used twice. The table
gives the largest error ratio over $\delta$ and the six perturbations.
Condition \eqref{eq:explicit-armijo} holds for $\bar\omega=0.5$ and $\eps=4$.}
\label{tab:explicit-radius}
\centering\small
\begin{tabular}{lrrrrrrr}
\toprule
&$n$&$\bar\omega$&$r_{0,h}$&$j=1$&$j=4$&$j=16$&$\tau=1$ accepted\\
\midrule
A&128 &0.5&0.1315&0.056&0.087&0.154&all\\
A&512 &0.5&0.2729&0.068&0.108&0.117&all\\
A&512 &0.9&0.9897&0.106&0.120&0.057&all\\
B&2048&0.9&0.0169&0.059&0.119&0.267&all\\
\bottomrule
\end{tabular}
\end{table}

\Cref{cor:explicit-radius} guarantees contraction in the ball of radius $r_{0,h}$, uniformly in $\delta$. \Cref{tab:explicit-radius} tests this at sampled values of $\delta$ and for six perturbations.
The observed ratios remain below $\bar\omega$ even at $16r_{0,h}$ for the tested perturbations. This suggests that the guaranteed radius is conservative. The row A, $n=512$, $\bar\omega=0.9$ is an exception to this reading. There, $16r_{0,h}\approx15.8$ lies far outside the range $(-1,1]$ of the potentials given by \eqref{eq:fc-box}, and the ratio decreases with the amplitude.
The Armijo condition \eqref{eq:Armijo} accepted the step size $\tau=1$ in every test, including cases where \eqref{eq:explicit-armijo} does not hold.
For comparison, we also ran the family of \cref{rem:lower-bound-nonuniform} on arrays with $n=64$ and $\theta\in\{0.2,0.1,0.05,0.02,0.01,0.005,0.001\}$.
For the full family $0<\theta<1/2$, we have $\nu_h(t)=1$ whenever $t\ge4\theta$, and hence $\sup_{0<\theta<1/2}\nu_h(t)=1$ for every $t>0$.
For every sampled $\theta$, the Newton step with step size one from $u=v=\frac12+\theta$ with slope one gives the error ratio $0.5000$, in agreement with the exact factor $1/2$.
Thus, the curvature criterion provides a local theory that is uniform in $\delta$, whereas the comparison family has no uniform superlinear error modulus.
\subsection{The globalized inexact semismooth Newton method}\label{sec:num-2d}
We run \cref{alg:globalized} from $z^0=0$ with $\varpi=1/2$ and $\mathrm{tol}=10^{-8}$. In step~(2), we use either the Jacobi iteration or the conjugate gradient method capped at $\bar N_{\rm J}(\eta_m)$ iterations, as in the remarks after \cref{thm:algorithm}. Both inner solvers stop by the same forcing condition \eqref{eq:inf-forcing}.
We consider two families of problems, one in one dimension and one in two dimensions.

The second family uses $X=Y=[0,1]^2$, $\eps=1$, $\delta=10^{-4}$,
$C(x,y)=|x-y|^2/2$, and the functions
\[
\begin{aligned}
 f_0(x)&=0.12+1.3e^{-45|x-(0.25,0.32)|^2}
                   +0.9e^{-65|x-(0.72,0.68)|^2},\\
 g_0(y)&=0.12+e^{-50|y-(0.70,0.28)|^2}
                   +0.75e^{-55|y-(0.35,0.75)|^2}.
\end{aligned}
\]
We use $n\times n$ square partitions with $n=6,12,24,48$. The cellwise data are $f_i=\lambda_Xf_0(x_i)$ and $g_j=\lambda_Yg_0(y_j)$, where $x_i$ and $y_j$ are the cell centers and the constants $\lambda_X,\lambda_Y>0$ are chosen so that
\[
 \sum_{i=1}^{n^2}w_if_i=1,\qquad \sum_{j=1}^{n^2}\omega_jg_j=0.78.
\]
The common bounds $c=0.4$ and $E=5.6$ hold on all four meshes.

\begin{table}[tb]
\caption{\Cref{alg:globalized} with $\varpi=1/2$, $\mathrm{tol}=10^{-8}$.
In the pattern, G denotes an outer iteration with a safeguard phase, and N an accepted Newton trial. The pattern is the same for both inner solvers. The column ``grad.''\ gives the total number of scaled gradient steps.
The Jacobi and CG columns give the total number of inner iteration steps and the maximum per outer iteration. The two ratio columns give the
largest linear residual ratios $\norm{H_hd+F_h}_\infty/\norm{F_h}_\infty$
and $\norm{H_hd+F_h}_h/\norm{F_h}_h$ over both inner solvers. The forcing
bound is $\bar\eta=\varpi/8=0.0625$. The last
column is the Jacobi operation count in the sense of \eqref{eq:total-work}
against that bound evaluated with the observed $R$ ($R=3$ in 1D, $R=0$ in
2D).}
\label{tab:globalized}
\centering\footnotesize\setlength{\tabcolsep}{3.5pt}
\begin{tabular}{lrlrrrrrr}
\toprule
case&$n$&pattern&grad.&Jacobi&CG&$\infty$-ratio&weighted ratio&ops/bound\\
\midrule
1D uniform&64&GGGNNNN&33&406/174&69/19&$6.17\times10^{-2}$&$7.38\times10^{-2}$&454/62216\\
1D uniform&256&GGGNNNN&34&402/152&69/17&$6.12\times10^{-2}$&$7.39\times10^{-2}$&451/62216\\
1D uniform&1024&GGGNNNN&34&434/176&70/19&$6.06\times10^{-2}$&$7.38\times10^{-2}$&483/62216\\
1D graded&64&GGGNNNN&34&405/174&67/20&$6.09\times10^{-2}$&$7.35\times10^{-2}$&454/62216\\
1D graded&256&GGGNNNN&34&399/155&66/16&$6.10\times10^{-2}$&$7.38\times10^{-2}$&448/62216\\
1D graded&1024&GGGNNNN&34&430/172&71/19&$6.06\times10^{-2}$&$7.38\times10^{-2}$&479/62216\\
\midrule
2D&6&NNNNN&0&42/20&35/13&$4.85\times10^{-2}$&$5.00\times10^{-2}$&53/1981\\
2D&12&NNNNN&0&40/17&34/12&$5.85\times10^{-2}$&$4.17\times10^{-2}$&51/1981\\
2D&24&NNNNN&0&40/17&46/17&$5.20\times10^{-2}$&$4.17\times10^{-2}$&51/1981\\
2D&48&NNNNN&0&42/20&42/14&$6.16\times10^{-2}$&$5.63\times10^{-2}$&53/1981\\
\bottomrule
\end{tabular}
\end{table}

In one dimension, three Newton trials are rejected on every mesh, and the safeguard uses $33$ or $34$ scaled gradient steps in total.
Every later Newton trial is accepted.
In two dimensions, every Newton trial is accepted, starting from $z^0=0$ (\cref{tab:globalized}). In these runs, the capped conjugate gradient method always meets the forcing condition, so the Jacobi fallback described after \cref{thm:algorithm} is never used.

The ratio in the forcing condition \eqref{eq:inf-forcing}, measured in the maximum norm, stays below $\bar\eta$ in every run.
In one dimension, the largest ratios measured in the weighted norm exceed this threshold. Thus, the forcing condition in the maximum norm does not imply the corresponding condition in the weighted norm, so the two conditions are genuinely different.

The counts of full-array operations vary little across the tested meshes and remain well below \eqref{eq:total-work}.
With the observed number $R$ of safeguard phases, the bound \eqref{eq:total-work} equals $62216$ in one dimension and $1981$ in two dimensions.
With $R$ replaced by $M_{\rm tol}$, it equals $75741$ and $2311$, respectively.
Here, we have $\bar N_{\rm G}=541$ in one dimension and $\bar N_{\rm G}=11$ in two dimensions.
The larger one-dimensional Jacobi counts are consistent with the dependence of $q_{\rm J}$ on $\eps$.

These experiments test the globalized iteration, its forcing rule, and its operation counts on the specified arrays.
The two-dimensional densities are normalized separately on each mesh. Hence, these runs do not test convergence under mesh refinement. The mesh-independence results of \cref{sec:local,sec:mesh,sec:curvature} are illustrated by the manufactured examples in \cref{sec:num-interface,sec:num-tangent,sec:num-curvature}.

Dense storage and a complete marginal evaluation cost $O(N_XN_Y)$.
An application of the Newton derivative $H_h(\vartheta)$ to a vector, without assembling it as a matrix, has the same worst-case cost.
Thus, iteration counts that are uniform in the mesh do not make the total arithmetic cost independent of the mesh size.
For $\delta=0$, the plan can be sparse. Exploiting its support to reduce the cost of evaluating the marginals and the slopes is a separate algorithmic development.
\section{Conclusions}
The marginal structure of quadratic Pearson transport yields a maximum-norm inverse estimate and a Newton remainder controlled
linearly by the measure of the boundary layer around the active interface.
Together, they provide a function-space local superlinear convergence theorem with local $q$-order at least $1+\alpha$. Moreover, for each fixed sequence of refined partitions whose data converge uniformly, they provide a common convergence neighborhood, whose guaranteed radius may depend on that sequence.
The same interface geometry gives an $O(h^{1+\alpha}+h^2)$ estimate for the discrete potentials relative to the cell averages of the exact potentials.
For transverse interfaces, this gives second-order supercloseness, while the piecewise-constant plan in general has first-order approximation error.

For smooth data on convex domains, a curvature margin in the input data verifies the interface condition with a constant that is uniform
in the lower bound $\delta$. It also gives an explicit convergence radius for the Newton iteration and an explicit mesh threshold.
This provides a local theory that is uniform in $\delta$.
\Cref{rem:lower-bound-nonuniform} shows why this uniformity matters. A common superlinear error modulus can fail even when the problem for each individual $\delta$ is well behaved.
The numerical examples indicate that the guaranteed convergence radius is conservative.

The globalized inexact semismooth Newton method of \cref{alg:globalized} connects global residual reduction with these local results.
Its numbers of outer iterations and of full-array operations admit bounds independent of the mesh.
Under the curvature criterion, the bound on the number of safeguard phases is also explicit and uniform over meshes that meet the threshold with a fixed margin (\cref{thm:algorithm}).
The bound on the number of full-array operations uses only the reduction of the residual by the factor $\varpi$ in each outer iteration, as for first-order methods, and it grows like $\eps^{-1}$ up to logarithms.
On each fixed mesh, the eventual convergence is $q$-quadratic (\cref{thm:algorithm}(iii)).
Each full-array operation has worst-case cost $O(N_XN_Y)$. For $\delta=0$, exploiting the support of the plan may reduce this cost.

As a supporting tool, the residual-square identity \eqref{eq:gap-squares} gives a reliable a posteriori error bound for the reconstructed plan, with the sharp coefficient $1/\sqrt{2\eps}$, which is attained.
The bound also measures the marginal residuals, so its size need not be comparable to the plan error alone.
Exact cell reduction and $L^2$ data stability turn it into a total a posteriori error bound relative to the continuum problem at finite solver tolerance (\cref{prop:reduction,prop:stability,cor:total}).
The data radius keeps the marginal coercivity of the Pearson terms. Hence, it does not involve the plan bound $M$, and it scales like $\eps^{-1/2}$ in the density errors.
In the experiments of the supplementary materials (\cref{S-sec:supp-numerics}), the total a posteriori error bound is within a factor of about $2$ to $3.5$ of the continuum error.

All constants are for fixed $\eps>0$ and $c>0$, and their dependence is recorded in \cref{rem:eps-c}.
Uniform limits as $\eps\downarrow0$ or $c\downarrow0$ remain questions for further analysis.
Extending the present arguments to the degenerate balanced equation would also connect them with methods such as the regularized quasi-Newton iteration of \cite{lorenz2021quadratic}.
The KL penalty and the comparison of the Pearson and KL models are treated in the supplementary materials.

\clearpage

\clearpage
\section*{Supplementary materials: Further Results, Supporting Experiments, and KL Extension}
\markboth{Supplementary materials}{Supplementary materials}
\setcounter{section}{0}
\setcounter{equation}{0}
\setcounter{theorem}{0}
\setcounter{table}{0}
\setcounter{figure}{0}
\renewcommand{\thesection}{S\arabic{section}}
\renewcommand{\thetable}{S\arabic{table}}
\renewcommand{\thefigure}{S\arabic{figure}}
\renewcommand{\theHsection}{S\arabic{section}}
\renewcommand{\theHtable}{S\arabic{table}}
\renewcommand{\theHfigure}{S\arabic{figure}}
\renewcommand{\topfraction}{0.85}
\renewcommand{\bottomfraction}{0.6}
\renewcommand{\textfraction}{0.12}
\renewcommand{\floatpagefraction}{0.75}
\setcounter{topnumber}{3}
\setcounter{totalnumber}{4}

The following supplementary materials accompany the paper \emph{Regularized Unbalanced Optimal Transport with Pearson Marginal Penalties}. It contains results that complement the main paper. It is organized as follows.
 
In \cref{S-sec:aposteriori-limits}, we record a bound for the potentials, show that the coefficient in the a posteriori error bound \eqref{M-eq:sharp} of the main paper is attained, record two limitations of this bound, and quantify the bias introduced by a positive lower bound $\delta$. In \cref{S-sec:data-refinement}, we discuss how the data errors can be accessed in practice and give a computable refinement of the data constant. \Cref{S-sec:global,S-sec:globalized-further} concern the solvers. \Cref{S-sec:global} records the weighted global solvers together with their mesh-uniform rates, and \cref{S-sec:globalized-further} gives further results on the globalized iteration of \cref{M-sec:practical} of the main paper. \Cref{S-sec:supp-numerics} reports supporting numerical experiments. \Cref{S-sec:kl} treats the KL extension, namely duality and a globally controlled continuation, discretization and a posteriori error bounds, and a quantitative comparison of the Pearson and KL models. 
 
Two conventions hold throughout. First, the materials, involving Pearson divergence (\cref{S-sec:aposteriori-limits,S-sec:data-refinement}), allows the case $\delta=0$. In the section involving KL divergence (\cref{S-sec:kl}), by contrast, the lower bound $\delta$ is strictly positive, since the KL potentials are bounded above only by $\log(f/(\delta b))$ and $\log(g/(\delta a))$. Second, numbers carrying the prefix S (sections, results, equations, tables) refer to this supplementary materials, while numbers without it refer to the main paper.

\section{Sharpness and limitations of the a posteriori error bound}\label{S-sec:aposteriori-limits}
In this section, we complement \cref{M-prop:aposteriori} of the main paper. We first record a bound for the potentials and show that the coefficient $1/\sqrt{2\eps}$ in \eqref{M-eq:sharp} is attained. We then record two limitations of the a posteriori error bound and quantify the bias introduced by a positive lower bound $\delta$. The notation is that of \cref{M-sec:duality} of the main paper.

\subsection{A potential bound and sharpness of the coefficient}
By \eqref{M-eq:gap-residual}, we have $\norm{F(z)}_{\Hh}^2\le2E\,\Gap(z)$ for every $z\in\Hh$. Together with the last inequality of \cref{M-prop:global-constants}, this gives the potential bound
\begin{equation}\label{S-eq:dual-gap}
 \norm{z-z^*}_{\Hh}\leq\frac1c\norm{F(z)}_{\Hh}\le\frac{\sqrt{2E}}c\sqrt{\Gap(z)}.
\end{equation}

We now show that the coefficient $1/\sqrt{2\eps}$ in \eqref{M-eq:sharp} cannot be improved. In the proof of \cref{M-prop:aposteriori}, equality in \eqref{M-eq:sharp} holds when $d_X=fe_u$ and $d_Y=ge_v$ and when
\[
	(p(\T z-C)-p(\T z^*-C))\T(z-z^*)=\eps|p(\T z-C)-p(\T z^*-C)|^2,
\]
in \eqref{M-eq:cocoercivity}. The latter holds when $\T z-C-\varepsilon\delta$ and $\T z^*-C-\varepsilon\delta$ are nonnegative.

\begin{proposition}\label{S-prop:sharpness}
Let $a>0$ and $\eps>0$. Let $|X|=|Y|=a$, $C=0$, $f\equiv g\equiv2a/\eps$, and $0\le\delta<1/\eps$. Then these data satisfy \eqref{M-eq:data}, and, for every $t>0$, the potential $z=(\tfrac12+t,\tfrac12+t)$ satisfies equality in \eqref{M-eq:sharp}. If, in addition, $t\le(1-\eps\delta)/2$, then $z$ lies in the cap set \eqref{M-eq:caps}.
\end{proposition}
\begin{proof}
Since $c=2a/\eps$ and $\delta<1/\eps$, we have $\delta\max\{a,b\}=\delta a<2a/\eps=c$, so the data satisfy \eqref{M-eq:data}.
We first show that the minimizers are $u^*=v^*=\tfrac12$ and $k^*=1/\eps$. We compute the slack
\[\T z^*-C-\eps\delta=1-\eps\delta>0.\]
Hence, we have
\[
 k^*=\delta+\eps^{-1}(1-\eps\delta)=\frac1\eps,\qquad
 k_X^*=ak^*=\frac a\eps=\frac{2a}\eps\Big(1-\frac12\Big)=f(1-u^*),
\]
and similarly $k_Y^*=g(1-v^*)$. Thus, \eqref{M-eq:optimality} holds, and by the uniqueness in \cref{M-prop:duality}, we indeed get the minimizers.

Now, let $t>0$ and $z=(\tfrac12+t,\tfrac12+t)$. We have the slack \[s=\T z-C-\varepsilon\delta=1+2t-\eps\delta>0.\]
By \eqref{M-eq:Gap} and \eqref{M-eq:residual-intro}, a direct computation gives
\[
 k_z=\frac{1+2t}\eps,\qquad
 F(z)=\Big(\frac{a(1+2t)}\eps+\frac{2a}\eps\Big(t-\frac12\Big),\ \frac{a(1+2t)}\eps+\frac{2a}\eps\Big(t-\frac12\Big)\Big)=\Big(\frac{4at}\eps,\frac{4at}\eps\Big).
\]
By \eqref{M-eq:gap-squares} and $|X|=|Y|=a$, we obtain
\[
 \Gap(z)=2\cdot\frac12\cdot a\cdot\frac{(4at/\eps)^2}{2a/\eps}=\frac{8a^2t^2}\eps.
\]
Since $|X\times Y|=a^2$, we also have
\[
 \norm{k_z-k^*}_\Zz=\frac{2t}\eps\cdot a=\frac{2at}\eps=\sqrt{\frac{\Gap(z)}{2\eps}}.
\]
This is equality in \eqref{M-eq:sharp}.
Finally, the cap set requires $u,v\le1-\delta b/f=1-\eps\delta/2$. Since $u=v=\tfrac12+t$, this is equivalent to $t\le(1-\eps\delta)/2$.
\end{proof}

By the last assertion of \cref{S-prop:sharpness}, the constant is also sharp for iterates that are projected onto the cap set.
\subsection{The factor two and efficiency}
We next explain where the factor in \eqref{M-eq:sharp} comes from. We then show that the a posteriori error bound, although reliable, is not efficient with respect to the plan error. Here, as in \cref{M-sec:numerics} of the main paper, the effectivity index of a bound is its ratio to the actual error.

\begin{remark}\label{S-rem:factor-two}
For an arbitrary feasible plan $k$, \cref{M-lem:gap-decomp} and the $\eps$-strong convexity of the primal functional $\Pp$ give only
\[
 \norm{k-k^*}_\Zz\leq\sqrt{\frac{2\big(\Pp(k)+\G(z)-\Lambda\big)}{\eps}} .
\]
For the reconstructed plan $k_z$ in \eqref{M-eq:Gap}, one instead bound the primal strong-convexity bound \begin{equation}\label{S-eq:primal-gap}\Pp(k_z)-\min\Pp=\Pp(k_z)-\Pp(k^*)\ge\frac\eps2\norm{k_z-k^*}_\Zz^2.\end{equation}
We also add a \emph{dual Bregman bound}, which is a lower bound for the dual excess $\G(z)-\G(z^*)$. Indeed, since $F(z^*)=0$ and the computation in \cref{M-prop:duality}, we obtain 
\begin{equation}\label{S-eq:dual-bregman}
	\begin{aligned}
		\G(z)-\G(z^*)&=\G(z)-\G(z^*)-\ip{F(z^*)}{z-z^*}\\
		&\geq\int_{X\times Y}\big[\phi(r)-\phi(r^*)-p(r^*)(r-r^*)\big]\dd x\dd y\ge \frac{\varepsilon}{2}\norm{k_z-k^*}_{\Zz}^2,
	\end{aligned}
\end{equation}
where $r=\T z-C$ and $r^*=\T z^*-C$. 
The last estimate comes from the fact that $\phi$ is convex and $\phi'=p$ is $1/\eps$-Lipschitz, so
\[
 \phi(r)-\phi(r^*)-p(r^*)(r-r^*)\geq\frac\eps2|p(r)-p(r^*)|^2
\]
by \cite[Theorem~2.1.5]{nesterov2004introductory}.

By, \eqref{S-eq:primal-gap} and \eqref{S-eq:dual-bregman}, we obtain
\[\norm{k_z-k^*}_\Zz\le\sqrt{\Gap(z)/\eps}.\]
The identity \eqref{M-eq:reflected-squares} improves this bound to $\sqrt{\Gap(z)/(2\eps)}$. The improvement comes from isolating the cross term $2\ip{d}{\T e}$, which \eqref{M-eq:cocoercivity} bounds below by $2\eps\norm{d}_\Zz^2$. For comparison, \cref{S-prop:KLradius} proves the bound $\sqrt{\Gap_{\rm KL}/\eps}$ for reconstructed KL plans, by the argument of \eqref{S-eq:dual-bregman}.

We note that, in the example of \cref{S-prop:sharpness}, the plan exceeds the lower bound everywhere, so there is no active interface. In particular, \cref{M-prop:aposteriori} does not claim that the same coefficient remains sharp under an additional nondegeneracy hypothesis on the active interface.
\end{remark}

\begin{remark}\label{S-rem:no-plan-efficiency}
The estimate \eqref{M-eq:sharp} is reliable, but its effectivity index with respect to the plan error is not bounded, even for potentials arbitrarily close to $z^*$. To see this, we take $X=Y=[0,1]$, $f=g=1$, $\eps=1$, and $C=\delta=0$. Then \eqref{M-eq:optimality} gives $u^*=v^*=1/3$ and $k^*=2/3$. For $0<t<1/4$, we set
\[
 u_t=\tfrac13+t+t^2,\qquad v_t=\tfrac13-t,\qquad z_t=(u_t,v_t).
\]
These constant potentials lie in the cap set $\mathcal A_0$ of \eqref{M-eq:caps}, and $\T z_t-C=\tfrac23+t^2>0$. Hence, we have
\[
 k_{z_t}=\tfrac23+t^2,\qquad \norm{k_{z_t}-k^*}_\Zz=t^2.
\]
On the other hand, the residual components in \eqref{M-eq:gap-squares} are
\[
 F_u(z_t)=k_{z_t}+u_t-1=t+2t^2,\qquad F_v(z_t)=k_{z_t}+v_t-1=-t+t^2.
\]
Therefore, we obtain
\[
 \Gap(z_t)=\tfrac12\big(F_u(z_t)^2+F_v(z_t)^2\big)=t^2+t^3+\tfrac52t^4,
\]
and hence
\[
 \frac{\sqrt{\Gap(z_t)/(2\eps)}}{\norm{k_{z_t}-k^*}_\Zz}
 =\frac{\sqrt{1+t+\tfrac52t^2}}{\sqrt2\,t}\longrightarrow\infty
 \qquad(t\downarrow0).
\]
The first-order perturbations cancel in $u_t+v_t$, so the plan error is of second order, while the residual is of first order. As in \cref{S-prop:sharpness}, the plan exceeds the lower bound everywhere. The two families realize the opposite extremes of \eqref{M-eq:reflected-squares}, namely $d_X=fe_u$ and $\T e=O(t^2)$.

The a posteriori error bound is, however, efficient with respect to the potential error. Indeed, by \eqref{M-eq:gap-residual}, we have $\Gap(z)\le\norm{F(z)}_{\Hh}^2/(2c)$. By \cref{M-prop:global-constants}, we have $\norm{F(z)}_{\Hh}\le L\norm{z-z^*}_{\Hh}$, since $F(z^*)=0$. Hence, for every $z\in\Hh$,
\begin{equation}\label{S-eq:potential-efficiency}
 \norm{k_z-k^*}_{\Zz}\le\sqrt{\frac{\Gap(z)}{2\eps}}
 \le\frac{\norm{F(z)}_{\Hh}}{2\sqrt{c\eps}}
 \le\frac{L}{2\sqrt{c\eps}}\,\norm{z-z^*}_{\Hh}.
\end{equation}
Consequently, the effectivity index with respect to the plan error is at most
\[
 \frac{L}{2\sqrt{c\eps}}\cdot\frac{\norm{z-z^*}_{\Hh}}{\norm{k_z-k^*}_{\Zz}}.
\]
A large effectivity index therefore forces $\norm{k_z-k^*}_{\Zz}\ll\norm{z-z^*}_{\Hh}$, that is, a near cancellation of $\T(z-z^*)$ on the set where $k^*>\delta$. By \cref{M-prop:reduction}, the same statements hold for array potentials, with $c,E,L$ replaced by $\mu_h,E_h^{\max},L_h$.
\end{remark}
\subsection{A quantitative perturbation of the lower bound}\label{S-sec:lower-bound-bias}
Finally, we quantify how much a positive lower bound $\delta$ changes the solution at a fixed regularization parameter. This bias is what one adds to \eqref{M-eq:total-plan}, as noted after \cref{M-cor:total}, to bound the distance to the optimal plan for $\delta=0$.

\begin{proposition}\label{S-prop:lower-bound}
Let $k_\delta^*$ and $k_0^*$ solve \eqref{M-eq:primal} with the same data and with lower bounds $\delta$ and $0$, respectively, and let $z_\delta^*$ and $z_0^*$ be the corresponding potentials. We set
\begin{equation}\label{S-eq:sigma-bound}
 \sigma=\Big(b^2\int_X\frac1f\dd x+a^2\int_Y\frac1g\dd y\Big)^{1/2}
 \le\sqrt{\frac{ab(a+b)}c}.
\end{equation}
Then
\begin{equation}\label{S-eq:lower-bound}
 \norm{k_\delta^*-k_0^*}_\Zz
 \leq\frac\delta2\Big[3\sqrt{ab}+\Big(ab+\frac{\sigma^2}\eps\Big)^{1/2}\Big]
 \le\delta\Big(2\sqrt{ab}+\frac{\sigma}{2\sqrt\eps}\Big),
 \qquad
 \norm{z_\delta^*-z_0^*}_{\Hh}
 \leq\frac{\sqrt{a+b}}c\norm{k_\delta^*-k_0^*}_\Zz.
\end{equation}
\end{proposition}
\begin{proof}
We shift the problem with lower bound $\delta$ to the lower bound $0$. We put $q_\delta=k_\delta^*-\delta\geq0$. In the variable $q=k-\delta\geq0$, the problem with lower bound $\delta$ has gradient $A(q+\delta)$, with $A$ from \eqref{M-eq:VI}. Its variational inequality, tested with $q=k_0^*$, and the variational inequality of the problem with lower bound $0$, tested with $q_\delta$, give
\[
 \ip{A(q_\delta+\delta)}{k_0^*-q_\delta}_\Zz\ge0,\qquad
 \ip{A(k_0^*)}{q_\delta-k_0^*}_\Zz\ge0.
\]
We add the two inequalities and set $d=q_\delta-k_0^*$, which gives
\[
\ip{A(q_\delta)-A(q_\delta+\delta)}{d}_\Zz
\ge\ip{A(q_\delta)-A(k_0^*)}{d}_\Zz
=\eps\norm{d}_\Zz^2+\mathcal J(d),
\]
where
\[
 \mathcal J(d)=\int_X\frac{|d_X|^2}f\dd x+\int_Y\frac{|d_Y|^2}g\dd y.
\]
We observe that, since the constant function $\delta$ has marginals $\delta b$ and $\delta a$,
\[
 A(q+\delta)-A(q)=\delta\Big(\eps+\frac bf+\frac ag\Big)\qquad\text{for every }q\in\Zz.
\]
Since $b/f$ depends only on $x$ and $a/g$ only on $y$, we therefore compute
\[
 \ip{A(q_\delta)-A(q_\delta+\delta)}{d}_\Zz
 =-\delta\Big[\eps\int_{X\times Y}d\dd x\dd y+b\int_X\frac{d_X}{f}\dd x+a\int_Y\frac{d_Y}{g}\dd y\Big].
\]
We estimate these three terms separately. By the Cauchy--Schwarz inequality and $|X\times Y|=ab$, the first term satisfies
\[
 \Big|\eps\delta\int_{X\times Y}d\dd x\dd y\Big|\le\eps\delta\sqrt{ab}\,\norm{d}_\Zz.
\]
For the second term, we also apply the Cauchy--Schwarz inequality and get
\[
 \Big|\delta b\int_X\frac{d_X}{f}\dd x\Big|\le\delta b\Big(\int_X\frac1f\dd x\Big)^{1/2}\Big(\int_X\frac{|d_X|^2}f\dd x\Big)^{1/2}.
\]
In the same way, we bound the third term
\[\Big|\delta a\int_Y\frac{d_Y}{g}\dd y\Big|\le\delta a\Big(\int_Y\frac1g\dd y\Big)^{1/2}\Big(\int_Y\frac{|d_Y|^2}g\dd y\Big)^{1/2}.\]
Next, by the Cauchy--Schwarz inequality in $\mathbb R^2$, the sum of the bounds for the second and third terms is at most $\delta\sigma\sqrt{\mathcal J(d)}$. Combining the three estimates, we obtain
\[
 \eps\norm{d}_\Zz^2+\mathcal J(d)
 \le\eps\delta\sqrt{ab}\,\norm{d}_\Zz+\delta\sigma\sqrt{\mathcal J(d)} .
\]
Since $\mathcal J-\delta\sigma\sqrt{\mathcal J}\ge-\delta^2\sigma^2/4$, it follows that
\[
 \eps\norm{d}_\Zz^2-\eps\delta\sqrt{ab}\,\norm{d}_\Zz-\frac{\delta^2\sigma^2}4\le0.
\]
Solving the quadratic inequality, we obtain that
\[
 \norm{d}_\Zz\le\frac\delta2\Big[\sqrt{ab}+\Big(ab+\frac{\sigma^2}\eps\Big)^{1/2}\Big].
\]
Since $k_\delta^*-k_0^*=d+\delta$ and $\norm{\delta}_\Zz=\delta\sqrt{ab}$, the triangle inequality gives the first estimate in \eqref{S-eq:lower-bound}. The bound $\delta(2\sqrt{ab}+\delta/(2\sqrt{\varepsilon}))$ in \eqref{S-eq:lower-bound} follows from $\sqrt{x+y}\le\sqrt x+\sqrt y$. 

The bound on $\sigma$ in \eqref{S-eq:sigma-bound} follows from two bounds $\int_X1/f\dd x\le a/c$ and $\int_Y1/g\dd y\le b/c$.

For the second estimate in \eqref{S-eq:lower-bound}, \eqref{M-eq:optimality} gives $u_\delta^*-u_0^*=(k_{0,X}^*-k_{\delta,X}^*)/f$. By $f\ge c$ and \eqref{M-eq:Tnorm},
\[|u_\delta^*-u_0^*|\le\frac{\sqrt b\,\norm{k_\delta^*-k_0^*}_\Zz}{c}.\]
In the same way, $|v_\delta^*-v_0^*|\le\sqrt a\,\norm{k_\delta^*-k_0^*}_2/c$.
\end{proof}
Thus, a positive lower bound $\delta$ introduces an $O(\delta)$ bias at fixed $\eps$, and the coefficient of $\delta$ in \eqref{S-eq:lower-bound} is of order $\eps^{-1/2}$ as $\eps\downarrow0$.
The term $\mathcal J(d)$ is responsible for this order. If we drop it, we only obtain
\[
 \eps\norm{d}_\Zz^2\le\delta\Big\|\eps+\frac bf+\frac ag\Big\|_\Zz\norm{d}_\Zz,
 \qquad\text{hence}\qquad
 \norm{k_\delta^*-k_0^*}_\Zz\le\delta\Big(\sqrt{ab}+\frac1\eps\Big\|\eps+\frac bf+\frac ag\Big\|_\Zz\Big).
\]
Since $\norm{\eps+b/f+a/g}_2\to\norm{b/f+a/g}_2>0$ as $\eps\downarrow0$, this coefficient is of order $\eps^{-1}$.

We also note that the feasible sets are nested, that is, $\{k\ge\delta\}\subset\{k\ge\delta'\}$ for $0\le\delta'\le\delta$, and that $k_0^*+\delta$ is feasible for the lower bound $\delta$. Hence, we have
\[
 \min_{k\ge0}\Pp(k)\le\min_{k\ge\delta}\Pp(k)\le\Pp(k_0^*+\delta)\longrightarrow\Pp(k_0^*)=\min_{k\ge0}\Pp(k)
 \qquad(\delta\downarrow0),
\]
by the continuity of $\Pp$. Consequently, the optimal values are nonincreasing as $\delta\downarrow0$ and converge to the optimal value for $\delta=0$.
\section{Data access and a computable refinement of the data constant}\label{S-sec:data-refinement}
In this section, we complement the $L^2$ data stability of \cref{M-prop:stability} and the total a posteriori error bound of \cref{M-cor:total} of the main paper. We first discuss how the data errors entering these results can be bounded in practice. We then give a computable refinement of the data constant.

\subsection{Data access}\label{S-sec:data-access}
The a posteriori error bound of \cref{M-cor:total} is computable only when its data errors are, so we briefly discuss how they can be bounded. Let $\mathcal E_h$ be the total $L^2$ data error of \cref{M-sec:stability} of the main paper.

We first consider data in the Sobolev space $H^1$, that is, $f\in H^1(X)$, $g\in H^1(Y)$, and $C\in H^1(X\times Y)$. We assume that each cell $R$ satisfies a Poincar\'e inequality with constant $C_P\operatorname{diam}(R)$, where $C_P$ does not depend on $h$. Then, for every function $w$ in $H^1$ on the corresponding domain, we have
\[
 \norm{w-\Q_hw}_{L^2(R)}\le C_P\operatorname{diam}(R)\norm{\nabla w}_{L^2(R)}\le C_Ph\norm{\nabla w}_{L^2(R)}.
\]
Summing the squares of these estimates over the cells gives $\norm{w-\Q_hw}_{L^2(\overline R)}\le C_Ph\norm{\nabla w}_{L^2(\overline R)}$, where $\overline R$ is $X,\,Y$ or $X\times Y$. Hence, if the cellwise data are cell averages, that is $f_h=\Q_h^X f,\,g_h=\Q_h^Y g,\,C_h=\Q_hC$, then \[\mathcal E_h=\norm{f-f_h}_{L^2(X)}+\norm{g-g_h}_{L^2(Y)}+\norm{C-C_h}_\Zz=O(h).\]

Unlike the algebraic part \eqref{M-eq:finite-gap}, the data errors $\norm{f-f_h}_{L^2(X)}$, $\norm{g-g_h}_{L^2(Y)}$, and $\norm{C-C_h}_\Zz$ in $D_h$ of \cref{M-prop:stability} cannot be computed from the arrays alone. $D_h$ must be bounded separately, for instance by analytic data formulas, by known regularity constants as in the $H^1$ estimate above, or by validated integration.

For discontinuous data, the rate is lower. We consider a density that is smooth except across a hypersurface, where it jumps. Suppose that the cells that meet this hypersurface have total volume $O(h)$. Since $c\le f,f_h\le E$, we have $|f-f_h|\le E-c$ on these cells, and hence their contribution to $\norm{f-f_h}_{L^2(X)}^2$ is $O(h)$. Thus, the jump contributes $O(h^{1/2})$ to $\mathcal E_h$, and hence to $D_h$.
\Cref{S-sec:num-jump} tests such a jump, placed so that it is not aligned with any mesh in the sequence.

\subsection{A computable refinement of the data constant}\label{S-sec:marginal-bounds}
The data radius of \cref{M-prop:stability} uses the a priori bounds $\norm{1-u^*}_{L^\infty(X)}\le2$ and $\norm{1-v^*}_{L^\infty(Y)}\le2$ from \eqref{M-eq:fc-box}. As noted there, smaller known upper bounds $\theta_X,\theta_Y$ may be used instead. We now show that such bounds can be computed from an array potential whenever uniform data-error bounds are available.

Let $z_h$ be an array potential, and let $r_{\infty,h}=\norm{F_h(z_h)}_\infty/\mu_h$ be the residual radius of \cref{M-sec:uniform-stability} of the main paper, so that $\norm{z_h-z_h^*}_\infty\le r_{\infty,h}$. With $\theta_X^*,\theta_Y^*$ from \eqref{M-eq:theta-star} and $\zeta_h$ from \eqref{M-eq:zeta}, we define the computable marginal bounds
\begin{align*}
 B_{X,h}&=\norm{K_X(z_h)}_\infty+\frac{2b}{\eps}r_{\infty,h}
 +\theta_X^*\eta_{f,\infty}+E\zeta_h,\\
 B_{Y,h}&=\norm{K_Y(z_h)}_\infty+\frac{2a}{\eps}r_{\infty,h}
 +\theta_Y^*\eta_{g,\infty}+E\zeta_h.
\end{align*}
We claim that $\norm{k_X^*}_{L^\infty(X)}\le B_{X,h}$ and $\norm{k_Y^*}_{L^\infty(Y)}\le B_{Y,h}$. Indeed, we split
\[
 k_X^*=\Lsf^X K_X(z_h)+\big(\Lsf^X K_X(z_h^*)-\Lsf^X K_X(z_h)\big)+\big(k_X^*-\Lsf^X K_X(z_h^*)\big).
\]
For the second term, we recall that $K_{ij}(z)=p(u_i+v_j-C_{ij})$ and that $p$ is $1/\eps$-Lipschitz. Hence, for all $i=1,\dots,N_X,\,j=1,\dots,N_Y$,
\[
 |K_{ij}(z_h)-K_{ij}(z_h^*)|
 \le\frac1\eps\big|(u_h-u_h^*)_i+(v_h-v_h^*)_j\big|
 \le\frac2\eps\norm{z_h-z_h^*}_\infty
 \le\frac2\eps r_{\infty,h}.
\]
Thus, we obtain
\[
	\norm{\Lsf^XK_X(z_h)-\Lsf^XK_X(z_h^*)}_{L^\infty(X)}=\norm{K_X(z_h)-K_X(z_h^*)}_\infty\le\frac{2b}{\eps}r_{\infty,h}.
\]
For the third term, \eqref{M-eq:optimality} gives $K_X(z_h^*)=f_h(1-u_h^*)$ and $k_X^*=f(1-u^*)$. Hence, by $0\le1-u^*\le\theta_X^*$ from \eqref{M-eq:theta-star}, $f_h\le E$, and \cref{M-lem:uniform-stability}, we obtain
\[
 |k_X^*-\Lsf^X K_X(z_h^*)|\le|f-f_h|\,|1-u^*|+f_h|u_h^*-u^*|\le\theta_X^*\eta_{f,\infty}+E\zeta_h.
\]
The three bounds together give $\norm{k_X^*}_{L^\infty(X)}\le B_{X,h}$. The bound $\norm{k_Y^*}_{L^\infty(Y)}\le B_{Y,h}$ is analogous.

Since $1-u^*=k_X^*/f\le\norm{k_X^*}_{L^\infty(X)}/c$, the numbers
\[
 \theta_X=\min\{2,B_{X,h}/c\},\qquad \theta_Y=\min\{2,B_{Y,h}/c\}
\]
are known upper bounds for $\norm{1-u^*}_{L^\infty(X)}$ and $\norm{1-v^*}_{L^\infty(Y)}$ in the sense of \cref{M-prop:stability}. With $\Delta_{\theta,h}$ from \cref{M-prop:stability} for these $\theta_X,\theta_Y$, we obtain the refined data bound
\begin{equation}\label{S-eq:refined-data-bound}
 D_h^{\rm ref}=\frac1{2\eps}\left[\norm{C-C_h}_2+\Big(\norm{C-C_h}_2^2
 +\frac{\eps}{c}\Delta_{\theta,h}^2\Big)^{1/2}\right].
\end{equation}
Since $\theta_X,\theta_Y\le2$, the refinement never exceeds the baseline radius $D_h$ at the same data using $\theta_X=\theta_Y=2$. It improves the baseline only when $B_{X,h}<2c$ or $B_{Y,h}<2c$. We note that this refinement is not circular. It uses the maximum-norm data estimate \eqref{M-eq:uniform-stability} and a computable residual, not the $L^2$ radius $D_h$. When only $L^2$ data errors are available, the baseline bound \eqref{M-eq:Dh-intro} remains the applicable a posteriori error bound.
\section{Weighted global solvers}\label{S-sec:global}
In this section, we record globally convergent solvers for the array problem $\min\G_h$, with constants that do not depend on the mesh. The estimates are standard consequences of strong convexity and a Lipschitz gradient \cite{beck2017first,nesterov2004introductory}. The only point specific to this problem is that they are uniform when stated in the weighted inner product \eqref{M-eq:weighted-norm}. In the main paper, these solvers supply initial approximations for the local Newton theory of \cref{M-sec:local,M-sec:mesh}.

We first fix the inner product. By \cref{M-prop:reduction} of the main paper, $F_h$ is the gradient of $\G_h$ in the weighted inner product, and all methods below use this inner product, with the constants $\mu_h$, $E_h^{\max}$, $L_h$ of \eqref{M-eq:finite-constants}. Thus, $\G_h$ is $\mu_h$-strongly convex with $L_h$-Lipschitz gradient, where $\mu_h\ge c$ and $L_h\le L$.
In the Euclidean inner product of the coefficient arrays, that is, without the weights $w_i$ and $\omega_j$, the gradient would instead be $WF_h$ with $W=\operatorname{diag}(w_i,\omega_j)$. Condition-number bounds would then acquire a factor $w_{\max}/w_{\min}$. The weighted inner product avoids this factor, including on graded meshes. Finally, the Newton derivatives $H_h(\vartheta)$, the slopes \eqref{M-eq:slopes}, and the spectral bounds $\mu_hI\le H_h(\vartheta)\le L_hI$ are those of \cref{M-sec:newton-derivatives} of the main paper.

\subsection{Projected gradient and acceleration}\label{S-sec:global-pg}
We denote by $P_h$ the projection onto the cap set \eqref{M-eq:caps} for arrays, that is, onto the set of array potentials with $u_i\leq1-\delta b/f_i$ and $v_j\leq1-\delta a/g_j$ for $i=1,\dots,N_X$ and $j=1,\dots,N_Y$. It acts componentwise by
\[
 (P_hz)_{u,i}=\min\Big\{u_i,\ 1-\frac{\delta b}{f_i}\Big\},\qquad
 (P_hz)_{v,j}=\min\Big\{v_j,\ 1-\frac{\delta a}{g_j}\Big\},\qquad \text{for }i=1,\dots,N_X,\,j=1,\dots, N_Y.
\]
We write $\bar u_i=1-\delta b/f_i$ and $\bar v_j=1-\delta a/g_j$. For an array potential $y$ and every $z$ in the cap set, we have
\[
	\norm{z-y}_h^2=\sum_{i=1}^{N_X}w_i\big(u_i-y_{u,i}\big)^2+\sum_{j=1}^{N_Y}\omega_j\big(v_j-y_{v,j}\big)^2,
\]
where each $u_i$ is constrained only by $u_i\le\bar u_i$ and each $v_j$ only by $v_j\le\bar v_j$. 
Each term of this sum depends on a single entry of $z$, and each entry has its own constraint. Hence, the minimum of the sum is the sum of the minima,
\[
 \min_{z\in\mathcal A_\delta}\norm{z-y}_h^2
 =\sum_{i=1}^{N_X}w_i\min_{u_i\le\bar u_i}\big(u_i-y_{u,i}\big)^2
 +\sum_{j=1}^{N_Y}\omega_j\min_{v_j\le\bar v_j}\big(v_j-y_{v,j}\big)^2,
\]
and it is attained by minimizing each term separately.
Since $w_i>0$, the minimizer of $w_i(u_i-y_{u,i})^2$ over $u_i\le\bar u_i$ is $\min\{y_{u,i},\bar u_i\}$, and similarly for $v_j$. Thus, $P_hy$ is the orthogonal projection of $y$ onto the cap set in the weighted inner product.

The cap set contains $z_h^*$ by \eqref{M-eq:caps}. It also contains $0$, because $\delta b\leq\delta\max\{a,b\}<c\leq f_i$ by \eqref{M-eq:data}, and similarly for $g_j$. Here, we use the strict inequality $\delta\max\{a,b\}<c$ in \eqref{M-eq:data}. 
The minimizer $z_h^*$ of $\G_h$ over all array potentials lies in $\mathcal A_\delta$. Hence, $z_h^*=\operatorname*{arg\,min}_{z\in\mathcal A_\delta}\G_h(z)$, and the minimizer is unique because $\G_h$ is strictly convex. Thus, the projected methods below converge to the same point $z_h^*$, at which $F_h(z_h^*)=0$.
Moreover, $z^0=0$ is an admissible starting point.

For the step size $\tau_h=2/(L_h+\mu_h)$, the projected gradient (PG) method is
\begin{equation}\label{S-eq:PG}
 z^{m+1}=P_h(z^m-\tau_hF_h(z^m)).
\end{equation}
For V-FISTA, we put $\gamma_h=\sqrt{\mu_h/L_h}$ and $\varsigma_h=(1-\gamma_h)/(1+\gamma_h)$, initialize $z^{-1}=z^0\in\mathcal A_\delta$, and iterate
\begin{equation}\label{S-eq:VFISTA}
 y^m=z^m+\varsigma_h(z^m-z^{m-1}),\qquad
 z^{m+1}=P_h(y^m-L_h^{-1}F_h(y^m)).
\end{equation}
With $q_h=(L_h-\mu_h)/(L_h+\mu_h)$, the standard estimates for strongly convex problems \cite{beck2017first,nesterov2004introductory} give
\begin{align}
 \norm{z^m-z_h^*}_h&\leq q_h^m\norm{z^0-z_h^*}_h,\label{S-eq:PG-rate}\\
 \G_h(z^m)-\G_h(z_h^*)&\leq
 (1-\gamma_h)^m\left[\G_h(z^0)-\G_h(z_h^*)
                  +\tfrac{\mu_h}{2}\norm{z^0-z_h^*}_h^2\right]
 \label{S-eq:VFISTA-rate}
\end{align}
for PG and V-FISTA, respectively. Since $z_h^*$ is also the unconstrained minimizer, the a posteriori error bounds of \cref{M-sec:stability} apply at every iterate, with the gradient $F_h$ evaluated without projection, or with the duality gap $\Gap_h$.

\subsection{A globalized Newton--CG method}\label{S-sec:global-armijo}
We now describe the Newton--CG method referred to in \cref{M-sec:newton-derivatives,M-sec:numerics} of the main paper. For brevity, we write $H_m=H_h(\vartheta(z^m))$. An inexact Newton direction $d^m$ satisfies $H_md^m=-F_h(z^m)+r^m$, where the linear residual $r^m$ obeys the weighted forcing condition
\begin{equation}\label{S-eq:forcing}
 \norm{r^m}_h\leq\eta\sqrt{\mu_h/L_h}\,\norm{F_h(z^m)}_h,
 \qquad 0\leq\eta<1.
\end{equation}
We then backtrack from $\tau=1$, replacing $\tau$ by $\beta\tau$ with a fixed $\beta\in(0,1)$, until the Armijo condition \eqref{M-eq:Armijo} of the main paper holds. If the inner solver does not meet \eqref{S-eq:forcing} within a fixed iteration limit, we use the fallback direction $d^m=-F_h(z^m)/L_h$ instead. In contrast to the first-order methods above, the Newton iterates are not projected onto the cap set.

\begin{proposition}\label{S-prop:global-Newton}
Let $\beta\in(0,1)$, $0<\sigma<1/2$, and $0\le\eta<1$ be the backtracking factor, the Armijo parameter in \eqref{M-eq:Armijo}, and the forcing parameter in \eqref{S-eq:forcing}. Let $(z^m)$ be generated from an array potential $z^0$ by the Newton--CG method of this subsection, with Newton directions satisfying \eqref{S-eq:forcing} or the fallback direction $d^m=-F_h(z^m)/L_h$. We set
\begin{equation}\label{S-eq:global-newton-constants}
 \underline\tau_h=\beta\min\left\{1,
 \frac{2(1-\sigma)\mu_h(1-\eta)}{L_h(1+\eta)^2}\right\},\qquad
 q_{N,h}=1-\frac{2\sigma\underline\tau_h\mu_h(1-\eta)}{L_h}.
\end{equation}
Then every accepted step size satisfies $\tau_m\geq\underline\tau_h$, and
\begin{equation}\label{S-eq:global-newton-rate}
 \G_h(z^m)-\G_h(z_h^*)\leq q_{N,h}^m[\G_h(z^0)-\G_h(z_h^*)].
\end{equation}
Moreover, let $\kappa_h=L_h/\mu_h$ and $0<\eta<1$. Then the conjugate gradient (CG) method, started from zero in the coordinates $W^{1/2}d$, meets \eqref{S-eq:forcing} in at most
\begin{equation}\label{S-eq:CG}
 J_h(\eta)=\left\lceil
 \frac{\log(2\kappa_h/\eta)}
 {\log((\sqrt{\kappa_h}+1)/(\sqrt{\kappa_h}-1))}\right\rceil
\end{equation}
iterations in exact arithmetic. Since $\kappa_h\leq L/c$, the factors $q_{N,h}$ and, at fixed $\eta$, the counts $J_h(\eta)$ are uniform over all meshes satisfying \eqref{M-eq:data}.
\end{proposition}
\begin{proof}
For brevity, we write $F=F_h(z^m)$, $d=d^m$, $r=r^m$, and $H=H_m$. We also use the energy norms
\[
 \norm{d}_H^2=\ip{Hd}{d}_h,\qquad \norm{r}_{H^{-1}}^2=\ip{H^{-1}r}{r}_h.
\]
Since $\mu_hI\le H\le L_hI$, we have \[L_h^{-1/2}\norm{w}_h\le\norm{w}_{H^{-1}}\le\mu_h^{-1/2}\norm{w}_h\] for every array $w$.

Now, we compare the direction with the residual. By the above bounds and \eqref{S-eq:forcing}, we estimate
\[
 \norm{r}_{H^{-1}}\leq\mu_h^{-1/2}\norm{r}_h
 \leq\eta L_h^{-1/2}\norm{F}_h\leq\eta\norm{F}_{H^{-1}}.
\]
Since $H$ is self-adjoint and positive definite in \eqref{M-eq:H-cont}, \eqref{M-eq:H-form-cont} and $Hd=r-F$, we have
\[
 \norm{d}_H^2=\ip{Hd}{d}_h=\ip{r-F}{H^{-1}(r-F)}_h=\norm{r-F}_{H^{-1}}^2,
\]
and $\ip{F}{d}_h=-\norm{F}_{H^{-1}}^2+\ip{H^{-1}r}{F}_h$. Hence, it follows that
\[
 \norm{d}_H\leq(1+\eta)\norm{F}_{H^{-1}},\qquad
 \ip{F}{d}_h\leq-\norm{F}_{H^{-1}}^2+\norm{r}_{H^{-1}}\norm{F}_{H^{-1}}\le-(1-\eta)\norm{F}_{H^{-1}}^2.
\]

Next, we bound the accepted step size. By the spectral bounds $\mu_hI\le H\le L_hI$, we get $\norm{d}_h^2\leq\norm{d}_H^2/\mu_h$. 
Since $F_h$ is the gradient of $\G_h$ in $\ip{\cdot}{\cdot}_h$ and is $L_h$-Lipschitz in $\norm{\cdot}_h$, we have, for all array potentials $z$ and $e$,
\begin{equation}\label{S-eq:descent}
 \G_h(z+e)-\G_h(z)-\ip{F_h(z)}{e}_h
 =\int_0^1\ip{F_h(z+te)-F_h(z)}{e}_h\dd t
 \leq\int_0^1L_ht\norm{e}_h^2\dd t
 =\frac{L_h}{2}\norm{e}_h^2 .
\end{equation}
Applying \eqref{S-eq:descent} with $z=z^m$ and $e=\tau d$, we obtain
Using $\ip{F}{d}_h\le-(1-\eta)\norm{F}_{H^{-1}}^2$, we see that the Armijo condition \eqref{M-eq:Armijo} holds whenever
\[
 -(1-\sigma)(1-\eta)+\frac{L_h}{2\mu_h}\tau(1+\eta)^2\le0,
 \qquad\text{that is,}\qquad
 \tau\leq\frac{2(1-\sigma)\mu_h(1-\eta)}{L_h(1+\eta)^2}.
\]
Since the backtracking starts at $\tau=1$ and reduces $\tau$ by the factor $\beta$, the accepted step size satisfies $\tau_m\geq\underline\tau_h$.

We now derive the rate. Let $z^{m+1}=z^m+\tau_md$ be the accepted iterate. First, let $d$ be a Newton direction satisfying \eqref{S-eq:forcing}. By the Armijo condition \eqref{M-eq:Armijo}, the bound $-\ip{F}{d}_h\geq(1-\eta)\norm{F}_{H^{-1}}^2$, and $\tau_m\geq\underline\tau_h$, we have
\[
 \G_h(z^m)-\G_h(z^{m+1})
 \geq-\sigma\tau_m\ip{F}{d}_h
 \geq\sigma\underline\tau_h(1-\eta)\norm{F}_{H^{-1}}^2
 \geq\sigma\underline\tau_h(1-\eta)\frac{\norm{F}_h^2}{L_h},
\]
where the last step uses $\norm{F}_{H^{-1}}^2\geq\norm{F}_h^2/L_h$.

Next, let $d=-F/L_h$ be the fallback direction. By \eqref{S-eq:descent} with $e=-F/L_h$, and since $\sigma<1/2$, we obtain
\[
 \G_h(z^m+d)\le\G_h(z^m)-\frac{\norm{F}_h^2}{2L_h}\le\G_h(z^m)-\sigma\frac{\norm{F}_h^2}{L_h}=\G_h(z^m)+\sigma\ip{F}{d}_h .
\]
Hence, $\tau=1$ satisfies \eqref{M-eq:Armijo}, so $\tau_m=1\geq\underline\tau_h$. Since $\underline\tau_h\leq\beta<1$ and $0<1-\eta\leq1$, the previous display gives
\[
 \G_h(z^m)-\G_h(z^{m+1})
 \geq\sigma\frac{\norm{F}_h^2}{L_h}
 \geq\sigma\underline\tau_h(1-\eta)\frac{\norm{F}_h^2}{L_h}.
\]

In both cases, we use $\norm{F}_h^2\geq2\mu_h[\G_h(z^m)-\G_h(z_h^*)]$, which follows from the $\mu_h$-strong convexity of $\G_h$, and obtain
\[
 \G_h(z^{m+1})-\G_h(z_h^*)
 \leq\Big(1-\frac{2\sigma\underline\tau_h\mu_h(1-\eta)}{L_h}\Big)\big[\G_h(z^m)-\G_h(z_h^*)\big]
 =q_{N,h}\big[\G_h(z^m)-\G_h(z_h^*)\big].
\]
By induction, this gives \eqref{S-eq:global-newton-rate}.

Finally, we count CG iterations. Let $W=\operatorname{diag}(w_i,\omega_j)$. In the coordinates $y=W^{1/2}d$, the Newton equation $Hd=-F$ reads $\widehat Hy=\hat b$ with
\[
 \widehat H=W^{1/2}HW^{-1/2},\qquad \hat b=-W^{1/2}F .
\]
Since $H$ is self-adjoint in $\ip{\cdot}{\cdot}_h$, the matrix $\widehat H$ is symmetric, and its spectrum lies in $[\mu_h,L_h]$. Moreover, $\norm{y}_2=\norm{d}_h$ for every $d$. In particular, $\norm{\hat b}_2=\norm{F}_h$, and the residual $\hat r=\hat b-\widehat Hy$ satisfies $\norm{\hat r}_2=\norm{r}_h$.
Let $y^*=\widehat H^{-1}\hat b$, let $y_j$ be the $j$th CG iterate with $y_0=0$, and set $\chi_h=(\sqrt{\kappa_h}-1)/(\sqrt{\kappa_h}+1)$. The standard CG estimate gives
\[
 \norm{y^*-y_j}_{\widehat H}\le2\chi_h^j\norm{y^*}_{\widehat H}.
\]
Since $\widehat H\le L_hI$ and $\widehat H\ge\mu_hI$, we have
\[
 \norm{\hat b-\widehat Hy_j}_2\le\sqrt{L_h}\,\norm{y^*-y_j}_{\widehat H},\qquad
 \norm{y^*}_{\widehat H}=\big(\hat b^{\mathsf T}\widehat H^{-1}\hat b\big)^{1/2}\le\frac{\norm{\hat b}_2}{\sqrt{\mu_h}} .
\]
Hence,
\[
 \norm{r}_h=\norm{\hat b-\widehat Hy_j}_2\le2\sqrt{\kappa_h}\,\chi_h^j\norm{F}_h .
\]
The forcing condition \eqref{S-eq:forcing} requires $\norm{r}_h\le\eta\kappa_h^{-1/2}\norm{F}_h$. This holds whenever $2\kappa_h\chi_h^j\le\eta$, that is, whenever
\[
 j\geq\frac{\log(2\kappa_h/\eta)}{\log(1/\chi_h)} ,
\]
and in particular for $j=J_h(\eta)$.
\end{proof}

\subsection{Total error of the global methods}\label{S-sec:global-total}
We now combine the algebraic rates with the data-stability estimates. For the initial value $z^0=0$, \eqref{M-eq:fc-box} applied to the cellwise data gives $|u_{h,i}^*|<1$ and $|v_{h,j}^*|<1$. Since $\sum_{i=1}^{N_X}w_i=a$ and $\sum_{j=1}^{N_Y}\omega_j=b$, we have
\[
 \norm{z^0-z_h^*}_h=\norm{z_h^*}_h\leq R_0,\qquad R_0=\sqrt{a+b},
\]
and, by the $L_h$-Lipschitz gradient, $\G_h(0)-\G_h(z_h^*)\leq L_hR_0^2/2$. We set
\[
 \Delta_m=
 \begin{cases}
 q_h^m&\text{(PG)},\\
 \sqrt{2\kappa_h}\,(1-\gamma_h)^{m/2}&\text{(V-FISTA)},\\
 \sqrt{\kappa_h}\,q_{N,h}^{m/2}&\text{(Newton--CG)}.
 \end{cases}
\]
For PG, \eqref{S-eq:PG-rate} gives $\norm{z^m-z_h^*}_h\le R_0\Delta_m$ directly. For V-FISTA and Newton--CG, the strong convexity bound $\frac{\mu_h}2\norm{z^m-z_h^*}_h^2\le\G_h(z^m)-\G_h(z_h^*)$, together with \eqref{S-eq:VFISTA-rate} or \eqref{S-eq:global-newton-rate}, gives the same bound $\norm{z^m-z_h^*}_h\le R_0\Delta_m$.
Since $p$ is $1/\eps$-Lipschitz and $\norm{\T w}\le R_0\norm{w}_h$ by \eqref{M-eq:Tnorm}, we have
\[
 \norm{K(z)-K(z')}_{2,h}\leq\eps^{-1}R_0\,\norm{z-z'}_h.
\]
We combine these bounds with \cref{M-prop:stability} and the triangle inequality, and obtain
\begin{equation}\label{S-eq:global-total}
 \norm{z^m-z^*}_{\Hh}\leq D_{z,h}+R_0\Delta_m,\qquad
 \norm{K(z^m)-k^*}_2\leq D_h+\eps^{-1}\,R_0^2\Delta_m.
\end{equation}
All factors are uniform in the mesh. Consequently, when $\mathcal E_h=O(h)$, the algebraic error $R_0\Delta_m$ reaches the size of the data error after $O(\log h^{-1})$ outer iterations, and with a fixed forcing parameter $\eta$, the total number of CG iterations has the same order.

We stress that this concerns the global phase only. The local superlinear theory of \cref{M-sec:local} concerns the outer iteration, and an exact inner solve may require up to $N_X+N_Y$ CG iterations.
\section{Further results on the globalized iteration}\label{S-sec:globalized-further}
In this section, we complement \cref{M-sec:practical} of the main paper. We first analyze the globalized iteration with exact Newton trials, in which the gradient steps of the safeguard may be plain or diagonally scaled. We then add remarks on the forcing terms and on the choice of the inner solver. The notation is that of \cref{M-sec:practical}.

\Cref{M-lem:inf-gradient} of the main paper gives maximum-norm contractions for plain and scaled gradient steps. We now use these steps as a safeguard for Newton trials computed with exact linear solves.

\medskip
\noindent\textit{The globalized iteration with exact Newton trials.}
We fix $\varpi\in(0,1)$ and $z^0$, and let $P$ be either $\tau I$ with
$0<\tau\le\tau_{\infty,h}$ or $\mathsf D_h^{-1}$. Given $z^m$ with $F_h(z^m)\ne0$, the iteration proceeds as follows.
\begin{enumerate}
\item[(a)] We compute the Newton trial
$z_{\rm N}^m=z^m-H_h(\vartheta(z^m))^{-1}F_h(z^m)$ with an exact linear solve.
\item[(b)] If $\norm{F_h(z_{\rm N}^m)}_\infty\le\varpi\norm{F_h(z^m)}_\infty$,
we set $z^{m+1}=z_{\rm N}^m$.
\item[(c)] Otherwise, we compute the gradient steps
$z^{m,j+1}=z^{m,j}-PF_h(z^{m,j})$ from $z^{m,0}=z^m$, and set
$z^{m+1}=z^{m,j}$ for the first $j\ge1$ with
$\norm{F_h(z^{m,j})}_\infty\le\varpi\norm{F_h(z^m)}_\infty$.
\end{enumerate}

Before stating the result, we fix some notation. We use the discrete Newton modulus $\omega_h(r)=2\nu_h(2r)/(\mu_h\eps)$ of \cref{M-prop:inexact-trials}, the step size $\tau_\infty$ of \cref{M-sec:safeguard}, which does not depend on the mesh, and $\log_+=\max\{\log,0\}$.
We write $q_{\rm G}$ for the contraction factor of \cref{M-lem:inf-gradient}(iv) that corresponds to $P$. For a forcing bound $0\le\bar\eta<\varpi$, we set
\[
 \bar\omega_{\bar\eta}=\frac{\varpi-\bar\eta}{1+\bar\eta}.
\]
When the explicit radius $r_{0,h}$ of \cref{M-cor:explicit-radius} with
$\bar\omega=\bar\omega_{\bar\eta}$ is available, we also set
\begin{equation}\label{S-eq:m0}
 m_0=\left\lceil\frac{\log_+\big((2+\bar\eta)E/(c\,r_{0,h})\big)}{\log(1/\varpi)}\right\rceil .
\end{equation}
For exact Newton trials, we use $\bar\eta=0$, so that $\bar\omega_0=\varpi$.
For the forcing bound $\bar\eta$ of \cref{M-alg:globalized}, \eqref{S-eq:m0} is
the number $m_0$ of \cref{M-thm:algorithm}(v).

\begin{proposition}\label{S-prop:exact-trials}
Assume \eqref{M-eq:data} for the cellwise data.
\begin{enumerate}
\item[(i)] The iteration is well defined, and each safeguard phase in step (c) uses at most $N_{\rm G}$ gradient steps, with $N_{\rm G}$ from \eqref{M-eq:gradient-block}. For every $m$,
\begin{equation}\label{S-eq:exact-trials-decrease}
 \norm{F_h(z^{m+1})}_\infty\le\varpi\norm{F_h(z^m)}_\infty,
 \qquad
 \norm{z^m-z_h^*}_\infty\le\frac{\varpi^m}{\mu_h}\norm{F_h(z^0)}_\infty .
\end{equation}
\item[(ii)] On each fixed mesh, the Newton trial in step (a) is accepted whenever
\[
	\omega_h(\norm{z^m-z_h^*}_\infty)\le\frac{\varpi\mu_h}{L_{\infty,h}},\quad\text{or}\quad\frac{2m_h\big(4r_{\infty,h}(z^m)\big)}{\mu_h\eps}\le\varpi.
\]
Only finitely many Newton trials are rejected, and after the
last rejection, the iteration consists of Newton steps with step size $\tau=1$.
\item[(iii)] Along a sequence of refined partitions, we assume the hypotheses of \cref{M-thm:mesh-local},
and let $\widehat r>0$ satisfy $\widehat\omega(\widehat r)\le\varpi c/L_\infty$.
Then, on every mesh of the sequence, every Newton trial with
$\varpi^m\norm{F_h(z^0)}_\infty\le c\,\widehat r$ is accepted.
\item[(iv)] Assume the hypotheses of
\cref{M-cor:explicit-radius} with $\bar\omega=\varpi$ and
$\lambda_h<\varpi^2t_\eps$. Then every Newton trial with
$r_{\infty,h}(z^m)\le r_{0,h}/2$ is accepted.
For $z^0=0$, Newton trials can be rejected only for $m<m_0$, with $m_0$ from
\eqref{S-eq:m0} for $\bar\eta=0$.
\end{enumerate}
\end{proposition}
For $P=\tau_\infty I$ or $P=\mathsf D_h^{-1}$, the number $N_{\rm G}$ in (i) is bounded independently of the mesh, as shown after \cref{M-lem:inf-gradient}. For $z^0=0$, we have $\norm{F_h(z^0)}_\infty\le E$, as in the proof of \cref{M-thm:algorithm}(i).
In (iv), at most $m_0N_{\rm G}$ gradient steps are used in total. Under the stronger mesh condition $\lambda_h\le\tfrac12\varpi^2t_\eps$ of
\cref{M-cor:explicit-radius}(iv), we have $r_{0,h}\ge\tfrac14\varpi^2t_\eps$, and hence
\[
 m_0\le\left\lceil\frac{\log_+\big(8E/(c\,\varpi^2t_\eps)\big)}
 {\log(1/\varpi)}\right\rceil .
\]
This bound is explicit and uniform over all meshes satisfying the margin
condition and over $\delta\in[0,\bar\delta]$.

\begin{proof}
(i) The Newton trials are well defined, because every $H_h(\vartheta)$ is invertible by \cref{M-lem:inverse}. If the trial is accepted in step (b), then $\norm{F_h(z^{m+1})}_\infty\le\varpi\norm{F_h(z^m)}_\infty$. Otherwise, we apply \cref{M-lem:inf-gradient}(iv) to the gradient steps of step (c), with initial value $z^m$. It gives
\[
 \norm{F_h(z^{m,j})}_\infty\le\varpi\norm{F_h(z^m)}_\infty\qquad\text{for every }j\ge N_{\rm G}.
\]
Since $L_{\infty,h}\ge\mu_h$ and $\varpi<1$, we have $N_{\rm G}\ge1$. Hence, the stopping test in step (c) holds for some $1\le j\le N_{\rm G}$, and again $\norm{F_h(z^{m+1})}_\infty\le\varpi\norm{F_h(z^m)}_\infty$. By induction and the lower bound in \eqref{M-eq:inf-residual}, we obtain
\[
 \norm{z^m-z_h^*}_\infty\le\frac{\norm{F_h(z^m)}_\infty}{\mu_h}\le\frac{\varpi^m}{\mu_h}\norm{F_h(z^0)}_\infty ,
\]
which proves \eqref{S-eq:exact-trials-decrease}.

(ii) We write $e=z^m-z_h^*$. The exact Newton trial satisfies \eqref{M-eq:inf-forcing} with $\eta_m=0$. For $\eta_m=0$, the condition in \cref{M-prop:inexact-trials}(iii) reads
\[
 \omega_h(\norm{e}_\infty)\le\frac{\varpi\mu_h}{L_{\infty,h}},
\]
and the condition \eqref{M-eq:direct-acceptance} in \cref{M-prop:inexact-trials}(i) reads
\[
 \frac{2}{\mu_h\eps}\,m_h\big(4r_{\infty,h}(z^m)\big)\le\varpi .
\]
These are the two acceptance criteria in (ii).
On a fixed mesh, $\nu_h(t)=0$ for all $t$ below the smallest nonzero value of
$|s_{ij,h}^*|$, and for all $t$ if there is none. Since
$\norm{e}_\infty\to0$ by (i), we have $\omega_h(\norm{e}_\infty)=0$ for all large $m$. Hence, the first criterion holds for all large $m$, and only finitely many Newton trials are rejected.

(iii) Let $h=h_\ell$ for some $\ell$. By the definition \eqref{M-eq:uniform-layer}, we have $\nu_h\le\widehat\nu$. Since $\mu_h\ge c$ and $L_{\infty,h}\le L_\infty$, it follows that
\[
 \omega_h(r)=\frac{2\nu_h(2r)}{\mu_h\eps}\le\frac{2\widehat\nu(2r)}{c\eps}=\widehat\omega(r),
 \qquad
 \frac{\varpi c}{L_\infty}\le\frac{\varpi\mu_h}{L_{\infty,h}} .
\]
If $\varpi^m\norm{F_h(z^0)}_\infty\le c\,\widehat r$, then (i) and $\mu_h\ge c$ give $\norm{e}_\infty\le\widehat r$. Since $\widehat\omega$ is nondecreasing, we obtain
\[
 \omega_h(\norm{e}_\infty)\le\widehat\omega(\widehat r)\le\frac{\varpi c}{L_\infty}\le\frac{\varpi\mu_h}{L_{\infty,h}},
\]
which is the first criterion in (ii).

(iv) The exact Newton trial satisfies \eqref{M-eq:inf-forcing} with $\eta_m=0$. Hence, by \cref{M-prop:inexact-trials}(i), it is accepted whenever
\[
 \frac{2}{\mu_h\eps}\,m_h\big(4r_{\infty,h}(z^m)\big)\le\varpi .
\]
We set $\widehat t=\varpi^2t_\eps$. By \cref{M-thm:curvature} and the definition of $t_\eps$, we have
\[
 m_*(\widehat t\,)\le A_\eps\varpi\sqrt{t_\eps}=\frac{\varpi c\eps}{2}.
\]
Let $r_{\infty,h}(z^m)\le r_{0,h}/2$. By \eqref{M-eq:explicit-radius} with $\bar\omega=\varpi$, we have $4r_{\infty,h}(z^m)+\lambda_h\le2r_{0,h}+\lambda_h=\widehat t$. By the inclusive transfer $m_h(t)\le m_*(t+\lambda_h)$ in the proof of \cref{M-thm:mesh-local}, the monotonicity of $m_*$, and $\mu_h\ge c$, we obtain
\[
 \frac{2}{\mu_h\eps}\,m_h\big(4r_{\infty,h}(z^m)\big)
 \le\frac{2}{c\eps}\,m_*(\widehat t\,)\le\varpi .
\]
This proves the first assertion.
For $z^0=0$, we have $\norm{F_h(z^0)}_\infty\le E$, so (i) gives
\[
 r_{\infty,h}(z^m)=\frac{\norm{F_h(z^m)}_\infty}{\mu_h}\le\frac{\varpi^mE}{c}.
\]
By the definition \eqref{S-eq:m0} of $m_0$ with $\bar\eta=0$, we have $\varpi^mE/c\le r_{0,h}/2$ for every $m\ge m_0$. Hence, every Newton trial with $m\ge m_0$ is accepted.
\end{proof}

We add some comments on \cref{S-prop:exact-trials}.
First, the iteration itself uses no convergence radius. Parts (ii) and (iii) show that only finitely many Newton trials are rejected, but for general data they give no explicit count, because (ii) concerns a fixed mesh and the radius $\widehat r$ in (iii) is not explicit. Part (iv) gives an explicit count. At most $m_0$ trials are rejected, so the rejections cost at most $m_0$ additional linear solves and residual evaluations, and the safeguard uses at most $m_0N_{\rm G}$ gradient steps. Without the stronger mesh condition $\lambda_h\le\tfrac12\varpi^2t_\eps$ of \cref{M-cor:explicit-radius}(iv), the radius $r_{0,h}=\tfrac12(\varpi^2t_\eps-\lambda_h)$ can be arbitrarily small, since $\lambda_h$ may be close to $\varpi^2t_\eps$, and $m_0$ grows like $\log(1/r_{0,h})/\log(1/\varpi)$.

Second, for both choices of $P$, we have $-\log q_{\rm G}=O(\eps)$ and $\log\big(L_{\infty,h}/(\varpi\mu_h)\big)=O(\log(1/\eps))$ as $\eps\downarrow0$, so $N_{\rm G}$ is of order $\eps^{-1}\log(1/\eps)$. This is a worst-case bound. In the one-dimensional experiments of \cref{M-sec:num-2d}, which use the same scaled gradient steps as safeguard, the bound $541$ per safeguard phase compares with $33$ or $34$ scaled gradient steps in total over three safeguard phases.

Third, when $r_{0,h}$ is available, one may instead apply gradient steps until the computable test $\norm{F_h(z)}_\infty\le\mu_hr_{0,h}$ holds, and then use Newton steps with step size one. By \cref{M-cor:explicit-radius}(iii) and (i), these Newton steps converge. By \cref{M-lem:inf-gradient} and \eqref{M-eq:inf-residual}, the test holds after at most
\[
 \left\lceil\frac{\log_+\big(L_{\infty,h}\norm{F_h(z^0)}_\infty/(\mu_h^2r_{0,h})\big)}
 {-\log q_{\rm G}}\right\rceil
\]
gradient steps.

Fourth, the weighted Newton--CG method of \cref{S-prop:global-Newton} may replace the gradient steps in step (c). On each fixed mesh, it meets the residual test after finitely many steps, since $F_h(z)\to0$. However, its rate \eqref{S-eq:global-newton-rate} controls the weighted norm, and $\norm{F}_\infty\le w_{\min}^{-1/2}\norm{F}_h$ with the smallest weight $w_{\min}=\min\{\min_{i\in\N_X}w_i,\min_{j\in\N_Y}\omega_j\}$. Hence, it gives no mesh-uniform bound on the number of steps, unlike $N_{\rm G}$.

\begin{remark}\label{S-rem:general-forcing}
The finiteness of rejections in \cref{M-thm:algorithm}(iii) holds for more general forcing terms than the choice $\eta_m=\min\{\bar\eta,\norm{F_h(z^m)}_\infty\}$ of \cref{M-alg:globalized}. Indeed, for inexact Newton trials satisfying \eqref{M-eq:inf-forcing}, part (i) of \cref{S-prop:exact-trials} holds unchanged, so $\norm{z^m-z_h^*}_\infty\to0$ and $\omega_h(\norm{z^m-z_h^*}_\infty)=0$ for all large $m$. If, in addition,
\[
 \eta_m\le\frac{\varpi\mu_h^2}{L_{\infty,h}^2}\qquad\text{for all sufficiently large }m,
\]
then the condition of \cref{M-prop:inexact-trials}(iii) holds for all large $m$. Hence, only finitely many Newton trials are rejected. Since $\mu_h\ge c$ and $L_{\infty,h}\le L_\infty$, this holds, for instance, whenever $\eta_m\le\varpi c^2/L_\infty^2$ eventually. After the last rejection, \eqref{M-eq:inexact-tail} gives
\[
 \norm{z^{m+1}-z_h^*}_\infty\le\eta_m\frac{L_{\infty,h}}{\mu_h}\norm{z^m-z_h^*}_\infty
\]
for all large $m$. Hence, the convergence is superlinear if $\eta_m\to0$.
\end{remark}

We close this section with remarks on the choice of the inner solver. Since
the forcing condition \eqref{M-eq:inf-forcing} is computable, any inner solver may be used in step
(2) of \cref{M-alg:globalized}. The role of \cref{M-lem:jacobi} is to provide a provable,
mesh-independent bound on the number of Jacobi iteration steps.

The conjugate gradient method, by contrast, is naturally controlled in the weighted
norm. In the coordinates $W^{1/2}d$ of the weighted Newton--CG method
(\cref{S-prop:global-Newton}), it reduces the weighted residual. The
conversion $\norm{r}_\infty\le\min_{i,j}\{w_i,\omega_j\}^{-1/2}\norm{r}_h$ then adds a term of order $\sqrt{\kappa_h}\log(1/\min_{i,j}\{w_i,\omega_j\})$ to its
iteration bound. For this reason, we cap the conjugate gradient method at
$N_{\rm J}(\eta_m)$ iterations. If the forcing condition is not met, we restart the Jacobi iteration from zero. This keeps the guaranteed bound $2N_{\rm J}(\eta_m)$.

The two solvers also scale differently in $\eps$. As $\eps\downarrow0$, we have
$-\log q_{\rm J}\approx c\eps/\max\{a,b\}$, so $N_{\rm J}$ grows like
$\eps^{-1}$, whereas the factor of the conjugate gradient method depends on
$\sqrt{\kappa_h}$.

Finally, the weighted global solvers of \cref{S-sec:global} remain
alternatives with objective-based guarantees. However, the maximum-norm results of
\cref{M-sec:practical} and of this section are the ones that apply to \cref{M-alg:globalized}.
\section{Supporting numerical experiments}\label{S-sec:supp-numerics}
In this section, we report experiments that support \cref{M-sec:numerics} of the main
paper. They concern the a posteriori error bounds, the sensitivity to the parameters and to the choice of inner product, and the work needed to reach a prescribed value of the a posteriori error bound. We begin with two remarks on how the numbers should be read.

First, as in the main paper, the inequalities being tested are bounds in exact real arithmetic, and all reported values of the bounds are binary64 evaluations of those formulas. For duality gaps, we use the sum of nonnegative squares in \eqref{M-eq:finite-gap}. These values measure the effectivity of the proved bounds, but they are not machine-verified interval enclosures.

Second, the reference quantities are computed as accurately as the data
allow. Continuum plan errors in the manufactured tests use analytic
polynomial moments, and density errors use adaptive one-dimensional
quadrature. The uncertainty of each discrete reference solution is estimated
by the residual radius $r_{\infty,h}$ of \cref{M-sec:uniform-stability} or by the plan bound \eqref{M-eq:sharp} with the discrete duality gap \eqref{M-eq:finite-gap}.

\subsection{The total a posteriori error bound}\label{S-sec:num-complete}
We first test the total a posteriori error bound \eqref{M-eq:total-plan} on smooth data
with an exact solution. On $X=Y=[0,1]$, we choose $\eps=1$, $\delta=0$,
$C=0$, and
\begin{align}\label{S-eq:smooth-data}
 u^*(x)&=0.18+0.07x,&v^*(y)&=0.23+0.06y^2,\notag\\
 f(x)&=\frac{u^*(x)+0.25}{1-u^*(x)},&
 g(y)&=\frac{v^*(y)+0.215}{1-v^*(y)}.
\end{align}
Since $\int_0^1v^*\dd y=0.25$ and $\int_0^1u^*\dd x=0.215$, the optimality conditions \eqref{M-eq:optimality} hold with the exact plan $k^*(x,y)=u^*(x)+v^*(y)$. We use
cell averages of the rational densities, computed by adaptive quadrature.
The lower density bound is $c=f(0)=0.43/0.82$, and $M=\max\{\delta,2/\eps\}=2$.

We run the projected gradient method \eqref{S-eq:PG} from $z^0=0$ and stop at the first iterate satisfying
$\rho_{\rm sh}\leq0.1h$, where
\begin{equation}\label{S-eq:plan-radii}
 \rho_{\rm sh}=\sqrt{\Gap_h/(2\eps)},\qquad
 \rho_F=\frac{\sqrt{a+b}}{\eps\mu_h}\norm{F_h}_h
\end{equation}
are the \emph{sharp gap radius} from \eqref{M-eq:sharp} and the \emph{residual plan radius}. The latter follows from $\norm{z_h-z_h^*}_h\le\norm{F_h}_h/\mu_h$ (\cref{M-prop:global-constants}) and the bound $\norm{K(z)-K(z')}_{2,h}\le\eps^{-1}\sqrt{a+b}\,\norm{z-z'}_h$ for the reconstruction in \eqref{M-eq:finite-reconstruct}. The upper panel of \cref{S-tab:complete} reports every term in \eqref{M-eq:total-plan}. The actual continuum error is integrated against the analytic polynomial plan, not against a refined-grid reference.

\begin{table}[tb]
\caption{Total a posteriori error bound for \eqref{S-eq:smooth-data}.
Upper panel: solve to a finite tolerance, with effectivity index
$(D_h+\rho_{\rm sh})/\norm{\Lsf K-k^*}_\Zz$. Lower panel: the same discrete problems
solved to convergence, showing the effectivity index of the data part $D_h$ and
of the refined data bound \eqref{S-eq:refined-data-bound}. Here $h=\sqrt2/n$. The values of the bounds are binary64 evaluations of bounds in exact real arithmetic.}
\label{S-tab:complete}
\centering\small
\begin{tabular}{rrrrrr}
\toprule
$n$ & actual error & $D_h$ & $\rho_{\rm sh}$ & $D_h+\rho_{\rm sh}$ & effectivity index\\
\midrule
16  & $6.1312\times10^{-3}$ & $5.2731\times10^{-3}$ & $6.9439\times10^{-3}$ & $1.2217\times10^{-2}$ &1.993\\
32  & $3.6460\times10^{-3}$ & $2.6371\times10^{-3}$ & $4.1880\times10^{-3}$ & $6.8252\times10^{-3}$ &1.872\\
64  & $1.3898\times10^{-3}$ & $1.3186\times10^{-3}$ & $1.5596\times10^{-3}$ & $2.8782\times10^{-3}$ &2.071\\
128 & $8.3359\times10^{-4}$ & $6.5933\times10^{-4}$ & $9.5115\times10^{-4}$ & $1.6105\times10^{-3}$ &1.932\\
256 & $3.2055\times10^{-4}$ & $3.2967\times10^{-4}$ & $3.5598\times10^{-4}$ & $6.8564\times10^{-4}$ &2.139\\
512 & $1.9184\times10^{-4}$ & $1.6483\times10^{-4}$ & $2.1747\times10^{-4}$ & $3.8230\times10^{-4}$ &1.993\\
\bottomrule
\end{tabular}

\medskip
\emph{Data part after convergence, and the computable refinement.}

\smallskip
\begin{tabular}{rrrrr}
\toprule
$n$ & converged error & $D_h/$error & $D_h^{\rm ref}$ & $D_h^{\rm ref}/$error\\
\midrule
16  & $1.7766\times10^{-3}$ & 2.968 & $2.7031\times10^{-3}$ & 1.521\\
32  & $8.8843\times10^{-4}$ & 2.968 & $1.3086\times10^{-3}$ & 1.473\\
64  & $4.4423\times10^{-4}$ & 2.968 & $6.4328\times10^{-4}$ & 1.448\\
128 & $2.2212\times10^{-4}$ & 2.968 & $3.1886\times10^{-4}$ & 1.436\\
256 & $1.1106\times10^{-4}$ & 2.968 & $1.5873\times10^{-4}$ & 1.429\\
512 & $5.5530\times10^{-5}$ & 2.968 & $7.9190\times10^{-5}$ & 1.426\\
\bottomrule
\end{tabular}
\end{table}

\noindent\textit{Effectivity at a finite tolerance.}
The computations in the upper panel stop after 9, 10, 12, 13, 15, and 16 PG
steps, respectively. On every mesh, the bound exceeds the actual error, with
effectivity indices between $1.87$ and $2.14$. We note that these are the
effectivity indices of the total bound $D_h+\rho_{\rm sh}$ at the prescribed stopping
rule, not of the data part $D_h$ alone. Indeed, since $\rho_{\rm sh}$ is deliberately
left at order $h$, a non-negligible part of the reported actual error is
the algebraic error. In fact, with this stopping rule, $D_h$ and $\rho_{\rm sh}$ are of
comparable size.

\medskip
\noindent\textit{Effectivity of the data part and its refinement.}
To separate the two effects, the lower panel solves the same discrete problems with cell-averaged data to a residual below $7\times10^{-14}$, and recomputes the continuum plan error. The ratio $D_h/\norm{\Lsf K(z_h^*)-k^*}_\Zz$ is then essentially constant at $2.968$ on all meshes.
This is the effectivity index of the data part, which the stopping rule of the upper panel hides.

The lower panel of \cref{S-tab:complete} also evaluates the computable refinement of \cref{S-sec:marginal-bounds}. For the monotone densities of \eqref{S-eq:smooth-data}, the uniform errors of the cell averages are attained at the cell endpoints. The converged array solution supplies the residual and marginal terms in the marginal bounds $B_{X,h}$ and $B_{Y,h}$. The resulting factors $\theta_X,\theta_Y$ lie between $0.95$ and $1.04$, whereas the a priori value is $2$. Consequently, $D_h^{\rm ref}$ reduces the effectivity index of the data part from $2.97$ to $1.52$ on the coarsest mesh and to $1.43$ on the finest. Thus, the gap between $D_h$ and $D_h^{\rm ref}$ is essentially the price of the a priori bound $1-u^*\le2$.

\medskip
\noindent\textit{Effectivity for the algebraic error.}
To isolate the algebraic error, we also run V-FISTA \eqref{S-eq:VFISTA} with the values of the densities of
\eqref{S-eq:smooth-data} at the cell midpoints and with the lower bound $\delta=0.05$, for $n=32,128,512$. We compare the iterates with a converged solution of the same discrete problem, whose plan bound \eqref{M-eq:sharp} is below $10^{-14}$. At steps $2$
through $8$, \cref{S-tab:algebraic} compares two estimators of the algebraic plan error, namely $\rho_{\rm sh}$ and
$\rho_F$ from \eqref{S-eq:plan-radii}. 

\begin{table}[tb]
\caption{Effectivity indices for the algebraic plan error at selected V-FISTA steps, given as ranges over
$n=32,128,512$.
The estimators are $\rho_F$ and $\rho_{\rm sh}$ from \eqref{S-eq:plan-radii}. The values of the bounds are binary64 evaluations of bounds in exact real arithmetic.}
\label{S-tab:algebraic}
\centering\small
\begin{tabular}{rrr}
\toprule
V-FISTA step & effectivity index of $\rho_F$ & effectivity index of $\rho_{\rm sh}$\\
\midrule
2 &6.656--6.735 &1.612--1.638\\
3 &11.266--11.633 &2.700--2.798\\
4 &21.531--21.615 &5.137--5.176\\
5 &32.269--32.399 &7.701--7.761\\
6 &43.953--44.158 &10.543--10.554\\
7 &36.937--36.989 &8.851--8.870\\
8 &33.336--33.427 &7.982--8.033\\
\bottomrule
\end{tabular}
\end{table}

Two observations follow from \cref{S-tab:algebraic}. First, $\rho_{\rm sh}$
is about four times tighter than $\rho_F$ at every reported step.
\Cref{S-rem:effectivity-bound} below explains this factor. Second, the
effectivity index of each estimator with respect to the true plan error is a
different quantity. It varies along the iteration, because the direction of the error changes. For $\rho_{\rm sh}$, it grows from near $1.6$ at step~2
to near $10.5$ at step~6.

We emphasize that \eqref{M-eq:sharp} is a guaranteed bound, not an
asymptotically exact estimator. In fact, its effectivity index with respect to the plan error admits no uniform bound. Indeed, \cref{S-rem:no-plan-efficiency} gives potentials
arbitrarily close to the optimizer for which it is unbounded. The values in
\cref{S-tab:algebraic} are therefore observations for this trajectory, not
instances of a general bound.

\begin{remark}[Sharp gap versus residual plan radii]\label{S-rem:effectivity-bound}
By \eqref{M-eq:gap-residual} for the cellwise data, applied with $c,E$ replaced by $\mu_h,E_h^{\max}$, we have for every array potential $z$
\[
 \frac{2\sqrt{(a+b)\mu_h}}{\sqrt\eps\,\mu_h}
 \leq
 \frac{\sqrt{a+b}\,\norm{F_h(z)}_h/(\eps\mu_h)}
 {\sqrt{\Gap_h(z)/(2\eps)}}
 \leq\frac{2\sqrt{(a+b)E_h^{\max}}}{\sqrt\eps\,\mu_h}.
\]
Thus, the ratio of the two estimators used in \cref{S-tab:algebraic} is bounded
above and below independently of $z$. Using
$c\le\mu_h\le E_h^{\max}\le E$, we obtain bounds that are uniform over every
admissible mesh, with constants depending only on $a,b,c,E$, and $\eps$. For
the data of \eqref{S-eq:smooth-data} with $\delta=0.05$, this interval is
$[3.91,4.55]$, which accounts for the observed factor near $4$ in
\cref{S-tab:algebraic}.

We note that this bound concerns only the ratio of the two estimators. It
does not bound either estimator divided by the true plan error, because the
common factor $\norm{F_h(z)}_h$ cancels in the ratio but not in the
effectivity. The effectivity index depends on the direction of $z-z_h^*$ and can be
unbounded (\cref{S-rem:no-plan-efficiency}). The variation along the iteration
in \cref{S-tab:algebraic} reflects this dependence on the direction along the
V-FISTA trajectory.
\end{remark}
\subsection{A discontinuous positive input density}\label{S-sec:num-jump}
We next test the total a posteriori error bound \eqref{M-eq:total-plan} on data that are not continuous, as discussed
in \cref{S-sec:data-access}. Let $\xi=\sqrt2/3$, and choose
\[
 u^*(x)=
 \begin{cases}0.2,&x<\xi,\\0.4,&x>\xi,\end{cases}
 \qquad v^*(y)=0.2,\qquad \eps=1,\quad\delta=0,\quad C=0.
\]
Then $k^*=u^*+0.2$, so that $k_X^*=k^*$ and $k_Y^*=0.4\xi+0.6(1-\xi)=0.6-0.2\xi$. Hence, the optimality conditions \eqref{M-eq:optimality} hold for the data
\begin{equation}\label{S-eq:jump-data}
 f(x)=\begin{cases}1/2,&x<\xi,\\1,&x>\xi,\end{cases}
 \qquad g(y)=\frac{0.6-0.2\xi}{0.8}.
\end{equation}
Thus, the lower density bound is $c=1/2$. Since $\xi$ is irrational, the jump point $\xi$ is not a cell boundary
on any uniform mesh. Let $\theta_n$ be the fraction of the cell containing $\xi$ that lies to
the left of $\xi$. On this cell, the cell average of $f$ is $\frac12\theta_n+1-\theta_n$, so that $f-f_h$ equals $-(1-\theta_n)/2$ on the left part and $\theta_n/2$ on the right part. Since $f_h=f$ on all other cells, we obtain
\[
	\norm{f-f_h}_{L^2(X)}^2=\frac1n\Big(\theta_n\frac{(1-\theta_n)^2}4+(1-\theta_n)\frac{\theta_n^2}4\Big)=\frac{\theta_n(1-\theta_n)}{4n}.
\]
Since $g$ is constant, we have $g_h=g$, and $C_h=C=0$. Hence, \eqref{M-eq:Dh-intro} with $\eps=1$ and $c=1/2$ gives
\[
	D_h=\frac{\norm{f-f_h}_{L^2(X)}}{\sqrt{c\eps}}=\sqrt2\,\norm{f-f_h}_{L^2(X)}=\frac{\sqrt2}2\sqrt{\frac{\theta_n(1-\theta_n)}n}.
\]
We compute the plan error by
integrating the two constant exact values on the two parts of the cell containing $\xi$,
and V-FISTA solves each array problem to residual $10^{-12}$.

\begin{table}[tb]
\caption{Discontinuous positive density, exact cell averages.
The sharp gap radius $\rho_{\rm sh}$ is below $3.5\times10^{-13}$.
Here, the total bound is $D_h+\rho_{\rm sh}$ and $h=\sqrt2/n$. The values of the bounds are binary64 evaluations of bounds in exact real arithmetic.}
\label{S-tab:jump}
\centering\small
\begin{tabular}{rrrrr}
\toprule
$n$&$\theta_n$&actual continuum error&total bound&effectivity index\\
\midrule
16&0.5425&$2.5157\times10^{-2}$&$8.8069\times10^{-2}$&3.501\\
32&0.0849&$9.8814\times10^{-3}$&$3.4850\times10^{-2}$&3.527\\
64&0.1699&$9.4329\times10^{-3}$&$3.3193\times10^{-2}$&3.519\\
128&0.3398&$8.4422\times10^{-3}$&$2.9602\times10^{-2}$&3.506\\
256&0.6796&$5.8903\times10^{-3}$&$2.0623\times10^{-2}$&3.501\\
512&0.3591&$4.2770\times10^{-3}$&$1.4992\times10^{-2}$&3.505\\
1024&0.7182&$2.8379\times10^{-3}$&$9.9406\times10^{-3}$&3.503\\
\bottomrule
\end{tabular}
\end{table}

\Cref{S-tab:jump} shows that the bound decays like $h^{1/2}$, as the regularity of the data permits. The ratios of consecutive values oscillate with $\theta_n$. Moreover, the ratio of the actual error to $\norm{f-f_h}_{L^2(X)}$ lies between $0.401$ and $0.404$ on all seven meshes. Hence, the bound and the error scale identically, and the effectivity index is nearly constant, close to $3.5$.

This effectivity index is dominated by the constant $\sqrt2=1/\sqrt{\eps c}$ multiplying $\norm{f-f_h}_{L^2(X)}$ in $D_h$, not by the algebraic part of the bound, which is below $3.5\times10^{-13}$ throughout. We also note that this example uses only the $L^2$ data stability of \cref{M-prop:stability}. The hypothesis of \cref{M-sec:mesh} that the data converge uniformly does not hold for the jump.
\subsection{Straight-interface refinement data}\label{S-sec:num-interface-data}
For completeness, \cref{S-tab:interface} lists the refinement data for the
straight transverse interface \eqref{M-eq:interface-data}, which are
summarized in \cref{M-sec:num-interface} of the main paper.

\begin{table}[tb]
\caption{Refinement data for the straight transverse interface \eqref{M-eq:interface-data}. The column ``full steps'' gives the number of Newton steps with step size one taken from \eqref{M-eq:local-initial}. The potential error is $e_{z,h}=\norm{z_h^*-\Lsf^{-1}\Q_h^{\Hh}z^*}_\infty$, with the uncertainty of the discrete reference solution bounded as described in \cref{M-sec:num-interface}. The plan error is measured against
the exact continuum plan. The zero-plan fraction is the fraction of product cells on which the discrete plan $K(z_h^*)$ vanishes. For the continuum plan $k^*=(1.3-x-y)_+$, the corresponding fraction of $X\times Y$ is $0.7^2/2=0.245$.}
\label{S-tab:interface}
\centering\small
\begin{tabular}{rrrrrr}
\toprule
$n$&full steps&plan error&$e_{z,h}$&$e_{z,h}/h^2$&zero-plan fraction\\
\midrule
16 &2&$2.2103\times10^{-2}$&$5.9567\times10^{-4}$&0.07625&0.25781\\
32 &2&$1.1066\times10^{-2}$&$8.3690\times10^{-5}$&0.04285&0.24707\\
64 &2&$5.5382\times10^{-3}$&$3.9626\times10^{-5}$&0.08115&0.24170\\
128&2&$2.7701\times10^{-3}$&$5.4819\times10^{-6}$&0.04491&0.24445\\
256&3&$1.3854\times10^{-3}$&$2.5814\times10^{-6}$&0.08459&0.24582\\
512&3&$6.9275\times10^{-4}$&$3.4842\times10^{-7}$&0.04567&0.24514\\
\bottomrule
\end{tabular}
\end{table}

\subsection{Weighted Newton--CG in two dimensions}\label{S-sec:num-newton-cg}
We now test the weighted Newton--CG method of \cref{S-prop:global-Newton} on
the two-dimensional data of \cref{M-sec:num-2d} of the main paper. We run it
from $z^0=0$ with $\eta=10^{-6}$, $\sigma=10^{-4}$, $\beta=1/2$, a limit of $200$ CG iterations per Newton direction, and the stopping tolerance $10^{-8}$ for the weighted residual $\norm{F_h}_h$. As in \cref{M-sec:num-2d}, the data are normalized separately on each mesh. Hence, this experiment measures the work of the solver only, and it does not measure continuum errors.

\begin{table}[tbp]
\caption{Work of the Newton--CG method in two dimensions.
Here, $N=N_X=N_Y=n^2$ is the number of cells of each marginal space, and $h=2/n$.
The last column gives the bound \eqref{S-eq:CG} on the number of CG iterations per Newton direction. The fallback direction was never used.}
\label{S-tab:Newton2D}
\centering\small
\begin{tabular}{rrrrrrr}
\toprule
$n$&$N$&outer iterations&total CG&max CG per outer iteration&final residual&CG bound\\
\midrule
6 &36 &4&65&20&$4.92\times10^{-12}$&34\\
12&144&4&76&25&$1.12\times10^{-12}$&36\\
24&576&4&84&28&$7.77\times10^{-9}$&37\\
48&2304&4&87&30&$9.51\times10^{-10}$&37\\
\bottomrule
\end{tabular}
\end{table}

As \cref{S-tab:Newton2D} shows, the Newton--CG method takes four outer iterations on every
mesh, and the smallest accepted step size is $1/2$ on every mesh.
Moreover, the number of CG iterations in each outer iteration stays below the bound \eqref{S-eq:CG}.

\subsection{Sensitivity to the regularization parameter, the lower bound, and the inner product}\label{S-sec:num-sensitivity}
The experiments of the main paper keep most parameters that affect the conditioning fixed.
Here, we report three short tests.
Each test checks a dependence stated earlier, rather than introducing
a new comparison of solvers.
\begin{itemize}
\item The first test varies $\eps$ for the cost and densities of \cref{S-sec:model} on $n=512$ cells, with $\delta=10^{-4}$. The data and the mesh are the same in every row.
\item The second test compares the optimal plans for $\delta>0$ and for $\delta=0$ on the same problem with $n=512$ and $\eps=1$.
\item The third test uses quadratically graded meshes for the data \eqref{S-eq:smooth-data} with $\delta=0.05$. It compares the projected gradient method (PG) in the weighted inner product
with the same method in the Euclidean inner product of the coefficient arrays.
\end{itemize}
In all runs, the final residuals are at or below the prescribed stopping tolerances.

\begin{table}[tbp]
\caption{Sensitivity to the parameters and to the inner product. Top: decreasing $\eps$
increases the mesh-uniform bound $L_h/\mu_h$ on the condition number and the work of V-FISTA. Here, $q_h=(L_h-\mu_h)/(L_h+\mu_h)$ is the factor in \eqref{S-eq:PG-rate}, and the last column is the fraction of $X\times Y$, measured by volume, on which the discrete plan equals $\delta$. Middle: the change $\norm{k_\delta^*-k_0^*}_\Zz$ of the plan and the bound of \cref{S-prop:lower-bound} are both linear in $\delta$. The effectivity index is the bound divided by $\norm{k_\delta^*-k_0^*}_\Zz$. Bottom: on graded meshes, the bound $L_h/\mu_h$ in the weighted inner product and the number of PG iterations stay essentially constant. In contrast, the condition number $\kappa_{\mathrm{E,full}}$ of the Newton derivative with all slopes equal to one, in the Euclidean inner product, and the number of PG iterations in the Euclidean inner product grow. The values of the bounds are binary64 evaluations of bounds in exact real arithmetic.}
\label{S-tab:sensitivity}
\centering\small
\emph{Sensitivity to $\eps$}\par\smallskip
\begin{tabular}{rrrrr}
\toprule
$\eps$ & $L_h/\mu_h$ & $q_h$ & V-FISTA steps & fraction at $\delta$\\
\midrule
2.0 & 5.182  & 0.67647 & 31  & 0.20635\\
1.0 & 7.000  & 0.75000 & 38  & 0.28339\\
0.5 & 10.636 & 0.82812 & 50  & 0.37252\\
0.2 & 21.545 & 0.91129 & 82  & 0.48761\\
0.1 & 39.727 & 0.95089 & 115 & 0.56805\\
\bottomrule
\end{tabular}

\medskip
\emph{Perturbation of the lower bound}\par\smallskip
\begin{tabular}{rrrr}
\toprule
$\delta$ & $\norm{k_\delta^*-k_0^*}_\Zz$ & bound \eqref{S-eq:lower-bound} & effectivity index\\
\midrule
$10^{-1}$ & $6.0780\times10^{-2}$ & $2.4384\times10^{-1}$ & 4.012\\
$10^{-2}$ & $5.9512\times10^{-3}$ & $2.4384\times10^{-2}$ & 4.097\\
$10^{-3}$ & $5.9388\times10^{-4}$ & $2.4384\times10^{-3}$ & 4.106\\
$10^{-4}$ & $5.9375\times10^{-5}$ & $2.4384\times10^{-4}$ & 4.107\\
$10^{-6}$ & $5.9374\times10^{-7}$ & $2.4384\times10^{-6}$ & 4.107\\
\bottomrule
\end{tabular}

\medskip
\emph{Weighted versus Euclidean inner product}\par\smallskip
\begin{tabular}{rrrrrr}
\toprule
$n$ & $w_{\max}/w_{\min}$ & $L_h/\mu_h$ & $\kappa_{\mathrm{E,full}}$ & weighted PG & Euclidean PG\\
\midrule
16  & 31  & 5.136 & 42.438  & 38 & 394\\
32  & 63  & 5.153 & 84.789  & 38 & 772\\
64  & 127 & 5.162 & 169.562 & 38 & 1487\\
128 & 255 & 5.166 & 339.119 & 37 & 2839\\
\bottomrule
\end{tabular}
\end{table}

We discuss the three panels of \cref{S-tab:sensitivity} in turn. The first
panel shows that the constants of the theory are stated for fixed $\eps$. As $\eps$
decreases from $2$ to $0.1$, the bound $L_h/\mu_h$ grows from $5.18$ to
$39.73$, and the number of V-FISTA iterations grows from $31$ to $115$, consistent
with \cref{M-rem:eps-c}.

The second panel is consistent with the $O(\delta)$ bias of \cref{S-prop:lower-bound} caused by a positive lower bound. For $\delta$ between $10^{-6}$ and $10^{-1}$, the effectivity index lies between $4.01$ and $4.11$, and it equals $4.107$ for $\delta\le10^{-4}$. Thus, the change of the plan and the bound are both linear in $\delta$.

The final panel illustrates the role of the weighted inner product in
\cref{S-sec:global}. The bound $L_h/\mu_h$ stays near $5.16$, and
PG in the weighted inner product takes $37$--$38$ iterations. In contrast, the condition number $\kappa_{\mathrm{E,full}}$ in the Euclidean inner product grows from $42.4$ to $339.1$,
and the number of PG iterations in the Euclidean inner product grows from $394$ to $2839$. This experiment does not
change the mesh-uniform bounds in the weighted inner product. It isolates the
loss of conditioning that the weighted inner product avoids.
\subsection{Work at a prescribed value of the a posteriori error bound}
\label{S-sec:num-matched}
The preceding tests fix the mesh and then measure the effectivity index.
We now reverse the viewpoint and use the same bounds as \emph{stopping and
mesh-selection rules}. Thus, the work is compared at a common prescribed tolerance $\mathrm{TOL}$ for the total a posteriori error bound. We revisit the smooth data \eqref{S-eq:smooth-data}
and the discontinuous data \eqref{S-eq:jump-data}. Every rule uses the same
projected gradient method \eqref{S-eq:PG} in the weighted inner product, started from $z^0=0$, so there is no
inner linear solve. This separates the effect of the estimator and of the choice of mesh
from the comparison of solvers in \cref{M-sec:num-2d}.

\medskip
\noindent\textit{The five rules.}
With $\rho_{\rm sh}$ and $\rho_F$ from \eqref{S-eq:plan-radii}, we write
\[
 \rho_{\rm B}=\sqrt{\Gap_h/\eps}=\sqrt2\,\rho_{\rm sh},
\]
and compare the following rules:
\begin{itemize}
\item sharp: $D_h+\rho_{\rm sh}\le\mathrm{TOL}$,
\item refined: $D_h^{\rm ref}+\rho_{\rm sh}\le\mathrm{TOL}$,
\item residual: $D_h+\rho_F\le\mathrm{TOL}$,
\item Bregman, from \eqref{S-eq:dual-bregman}: $D_h+\rho_{\rm B}\le\mathrm{TOL}$,
\item fixed residual: we choose the first mesh with $D_h<\mathrm{TOL}$, and then solve
to $\norm{F_h}_h\le10^{-8}$, independently of $h$ and $\mathrm{TOL}$.
\end{itemize}
The first four rules guarantee that the total bound is at most $\mathrm{TOL}$. The last rule is a heuristic. For the refined
rule, $D_h^{\rm ref}$ is recomputed from the current residual and the marginal
bounds of \cref{S-sec:marginal-bounds} after every step. Since this quantity is not known before a discrete
solve, we try meshes in increasing order and include the work spent on
rejected coarser meshes.

\medskip
\noindent\textit{Work measure.}
At each step of the projected gradient method, one reconstruction and its marginals supply
$F_h$, the sum of nonnegative squares \eqref{M-eq:finite-gap}, and the marginal bounds. We therefore measure the work by the number of evaluated iterates, including the initial one, times the number $n^2$ of product cells, summed over all meshes on which the discrete problem is solved. For the refined rule, these include the rejected coarser meshes. Unlike wall times, this count is reproducible. \Cref{S-tab:matched-work} reports it relative to the sharp rule at the same tolerance. We also recorded the numbers of steps on the selected mesh and in total, the operation counts, and the wall times. We do not use the wall times for comparison, since the
absolute one-dimensional times are subsecond and hence more sensitive to
Python overhead than this count. Exact continuum plan errors are evaluated by the
same analytic moments and split-cell formulas as in
\cref{S-sec:num-complete,S-sec:num-jump}.

\begin{table}[tbp]
\caption{Stopping and mesh selection for a prescribed tolerance $\mathrm{TOL}$ of the total a posteriori error bound.
For the smooth problem, the tolerances are $\mathrm{TOL}=10^{-2},5\times10^{-3},2.5\times10^{-3}$.
For the jump problem, they are $0.30,0.20,0.15$. The entries in the $n$ and
PG columns correspond to these tolerances in order. The PG column counts the projected gradient steps, including those on rejected coarser meshes. The column ``work/sharp'' is the work of the rule divided by the work of the sharp rule at the same tolerance. The columns ``bound/$\mathrm{TOL}$'' and ``error/$\mathrm{TOL}$'' are the reported bound and the exact continuum error divided by $\mathrm{TOL}$.}
\label{S-tab:matched-work}
\centering\scriptsize
\begin{tabular}{llrrrrr}
\toprule
problem & rule & $n$ & cumulative PG & work/sharp & bound/$\mathrm{TOL}$ & error/$\mathrm{TOL}$\\
\midrule
smooth & sharp & 16, 32, 64 & 10, 12, 13 & 1.000 & 0.843--0.954 & 0.320--0.401\\
       & refined & 16, 16, 32 & 9, 12, 73 & 0.250--0.909 & 0.869--0.983 & 0.447--0.613\\
       & residual & 16, 32, 64 & 13, 15, 16 & 1.214--1.273 & 0.833--0.941 & 0.188--0.197\\
       & Bregman & 16, 32, 64 & 11, 12, 14 & 1.000--1.091 & 0.859--0.974 & 0.266--0.320\\
       & fixed residual & 16, 32, 64 & 38, 38, 38 & 2.786--3.545 & 0.527 & 0.178\\
\midrule
jump   & sharp & 16, 16, 16 & 2, 3, 4 & 1.000 & 0.923--0.984 & 0.379--0.550\\
       & refined & 16, 16, 16 & 2, 3, 4 & 1.000 & 0.923--0.984 & 0.379--0.550\\
       & residual & 16, 16, 16 & 5, 6, 7 & 1.600--2.000 & 0.829--0.925 & 0.130--0.180\\
       & Bregman & 16, 16, 16 & 3, 4, 5 & 1.200--1.333 & 0.796--0.904 & 0.260--0.322\\
       & fixed residual & 16, 16, 16 & 33, 33, 33 & 6.800--11.333 & 0.294--0.587 & 0.084--0.168\\
\bottomrule
\end{tabular}
\end{table}

\noindent\textit{Smooth data.}
For the smooth data, all three tolerances are already met on the three
coarsest candidate meshes $n=16,32,64$. The refined rule selects a mesh
coarser by a factor of two at the two smaller tolerances, and the same mesh
$n=16$ at the largest one. Even when the work on all rejected
coarser meshes is included, the work of the refined rule is $0.25$--$0.91$ times that of the sharp rule. The saving is modest, because $D_h$ uses
the bound $1-u^*\le2$, and the marginal bounds improve it only to
$\theta_X,\theta_Y\approx1$ (\cref{S-tab:complete}). By contrast, replacing
$\rho_{\rm sh}$ by $\rho_{\rm B}$ changes the work by at most $9.1\%$, and
$\rho_F$ costs $21$--$27\%$ more. The fixed residual rule uses
$2.79$--$3.55$ times the work of the sharp rule.

\medskip
\noindent\textit{Discontinuous data.}
For the jump data, the radius $D_h=\sqrt2\,\norm{f-f_h}_{L^2(X)}$ is below
all three tolerances already on the coarsest candidate mesh $n=16$. Hence, every
rule selects $n=16$, and the comparison reduces to the stopping rules. Here,
the $L^2$ error of the cell averages decreases, but the uniform error entering the
marginal bounds does not. Consequently, the marginal bounds give $\theta_X=2$ and
$D_h^{\rm ref}=D_h$, so the refined and sharp rules coincide. We note that,
because of the minimum with $2$ in \eqref{S-eq:refined-data-bound}, the
refinement never selects a finer mesh than the sharp rule, even when the
uniform data error stagnates.

Among the remaining rules, the Bregman rule costs $20$--$33\%$ more, the
residual rule $60$--$100\%$ more, and the fixed residual rule $6.8$--$11.3$
times more. The fixed residual rule happens to satisfy the total bound of the sharp rule in all six reported runs, but only after solving
substantially past the stopping point of the sharp rule.

\medskip
\noindent\textit{Conservatism of the data radii.}
Finally, the exact errors in \cref{S-tab:matched-work} reach $32$--$61\%$ of
the prescribed tolerance for the sharp and refined rules. Accordingly, this
experiment prescribes the value of the guaranteed bound, not the actual error, which is unknown in practice. Reporting error$/\mathrm{TOL}$ makes the remaining
conservatism visible.
\section{KL extension and model comparison}\label{S-sec:kl}
In this section, we treat the KL marginal penalty. We first establish
duality and a globally controlled continuation of the dual
(\cref{S-sec:KLduality}). We then discuss discretization, gap bounds, and
localization (\cref{S-sec:KLdisc}). Finally, we compare the KL and Pearson
models quantitatively (\cref{S-sec:model}). Throughout this section, the lower bound
$\delta$ is strictly positive.

\subsection{KL duality and a globally controlled continuation}\label{S-sec:KLduality}
Let $X,Y$ have finite positive volumes $a,b$. We assume
$0<c\leq f,g\leq E$, $0\leq C\in L^\infty(X\times Y)$, $\eps>0$, and
$0<\delta\max\{a,b\}<c$, and all norms use the Lebesgue measure. We set
\[
 \psi(t)=t\log t-t+1,\quad
 p(r)=\max\{\delta,r/\eps\},\quad
 \phi(r)=\sup_{q\geq\delta}(rq-\eps q^2/2).
\]
The KL primal and dual problems are then
\begin{align}
	\Pp_{\rm KL}(k)&=\int_{X\times Y} Ck\dd x\dd y+\frac\eps2\norm{k}_\Zz^2
       +\int_Xf\psi(k_X/f)\dd x+\int_Yg\psi(k_Y/g)\dd y,\quad k\geq\delta,\label{S-eq:KLprimal}\\
       \G_{\rm KL}(u,v)&=\int_{X\times Y}\phi(u+v-C)\dd x\dd y+\int_Xfe^{-u}\dd x+\int_Yge^{-v}\dd y.\label{S-eq:KLdual}
\end{align}
In contrast to the Pearson case, the dual can take the value $+\infty$ on
$L^2$.

\begin{proposition}\label{S-prop:KLduality}
There is a unique primal minimizer $k^*$ and a unique dual minimizer
$z^*=(u^*,v^*)$. They satisfy the optimality conditions
\begin{equation}\label{S-eq:KL-optimality}
 k^*=p(u^*+v^*-C),\qquad k_X^*=fe^{-u^*},\qquad k_Y^*=ge^{-v^*},
\end{equation}
and
\[
 \min\Pp_{\rm KL}=\int_Xf\dd x+\int_Yg\dd y-\min\G_{\rm KL}.
\]
\end{proposition}
\begin{proof}
We first prove existence and uniqueness of the primal minimizer. The primal
functional is proper, lower semicontinuous, and $\eps$-strongly convex on the
closed convex set $\{k\ge\delta\}$. Since $\int_{X\times Y} Ck\dd x\dd y\ge0$ and the KL terms are nonnegative, we have
\[
	\Pp_{\rm KL}(k)\ge\frac\eps2\norm{k}_{\Zz}^2,
\]
so $\Pp_{\rm KL}$ is coercive.
Therefore, the direct method of the calculus of variations gives a unique minimizer $k^*$.

Next, we compute the gradient of $\Pp_{\rm KL}$. On the set $\{k\ge\delta\}$, we have
$k_X/f\geq\delta b/E$ and $k_Y/g\geq\delta a/E$. For $t\ge t_0>0$, we have
$|\log t|\leq\varkappa(1+t)$, where the constant $\varkappa$ depends only on $t_0$. Hence, by the dominated convergence theorem, $\Pp_{\rm KL}$ is differentiable along feasible segments. Its gradient is
\[
 A_{\rm KL}(k)=C+\eps k+\log(k_X/f)+\log(k_Y/g).
\]

We now derive the optimality conditions. We define
\[
 u^*=-\log(k_X^*/f),\qquad v^*=-\log(k_Y^*/g).
\]
By the same logarithmic bound, we have $u^*\in L^2(X)$ and $v^*\in L^2(Y)$. With these potentials, the gradient at $k^*$ becomes
\[
 A_{\rm KL}(k^*)=C+\eps k^*-u^*-v^*.
\]
The primal variational inequality states that $\ip{A_{\rm KL}(k^*)}{\ell-k^*}\ge0$ for all $\ell\ge\delta$. Since the constraint is pointwise, this holds if and only if, for almost every $(x,y)$, the value $k^*(x,y)$ minimizes
\[
 q\mapsto\frac\eps2q^2-(u^*+v^*-C)(x,y)\,q\qquad\text{over }q\ge\delta.
\]
This minimizer is $p(u^*+v^*-C)$. Together with the definitions of $u^*$ and $v^*$, this gives \eqref{S-eq:KL-optimality}.

Next, we show that $z^*$ is a dual minimizer. We write $r=u+v-C$ and $r^*=u^*+v^*-C$. Since $\phi$ is convex with $\phi'=p$, and $s\mapsto e^{-s}$ is convex, we have, pointwise,
\[
 \phi(r)\ge\phi(r^*)+p(r^*)(r-r^*),\qquad
 e^{-u}\ge e^{-u^*}-e^{-u^*}(u-u^*),\qquad
 e^{-v}\ge e^{-v^*}-e^{-v^*}(v-v^*).
\]
We multiply the second inequality by $f$ and the third by $g$, integrate all three, and add them. By \eqref{S-eq:KL-optimality}, we have $p(r^*)=k^*$, $fe^{-u^*}=k_X^*$, and $ge^{-v^*}=k_Y^*$. Moreover, $r-r^*=(u-u^*)+(v-v^*)$. Hence, we obtain
\[
	\G_{\rm KL}(z)\ge\G_{\rm KL}(z^*)+\int_{X\times Y} k^*(u-u^*)\dd x\dd y+\int_{X\times Y} k^*(v-v^*)\dd x\dd y-\int_Xk_X^*(u-u^*)\dd x-\int_Yk_Y^*(v-v^*)\dd y.
\]
Since $u-u^*$ depends only on $x$, Fubini's theorem and the definition of the marginal $k_X^*$ give
\[
 \int_{X\times Y}k^*(x,y)\,(u-u^*)(x)\,\dd x\dd y=\int_X(u-u^*)(x)\int_Yk^*(x,y)\,\dd y\,\dd x=\int_Xk_X^*(u-u^*)\dd x,
\]
and similarly $\int_{X\times Y} k^*(v-v^*)\dd x\dd y=\int_Yk_Y^*(v-v^*)\dd y$. Hence, the last four terms cancel, and $\G_{\rm KL}(z)\geq\G_{\rm KL}(z^*)$ for all $z\in \Hh$. Thus, $z^*$ is a dual minimizer.
It is unique,
because the exponential terms are strictly convex wherever they are finite.

Finally, we prove the value identity. Since $k^*$ and $z^*$ are minimizers, it suffices to show that $\Pp_{\rm KL}(k^*)+\G_{\rm KL}(z^*)=\int_Xf\dd x+\int_Yg\dd y$. By \eqref{S-eq:KLprimal} and \eqref{S-eq:KLdual}, we have
\[
	\begin{aligned}
	 \Pp_{\rm KL}(k^*)+\G_{\rm KL}(z^*)
 &=\int_{X\times Y}\Big[Ck^*+\frac\eps2(k^*)^2+\phi(r^*)\Big]\dd x\dd y\\
 &\quad+\int_X\big[f\psi(k_X^*/f)+fe^{-u^*}\big]\dd x
 +\int_Y\big[g\psi(k_Y^*/g)+ge^{-v^*}\big]\dd y.	
	\end{aligned}
\]
We simplify the three brackets. Since $k^*=p(r^*)$ attains the supremum in the definition of $\phi$, and $r^*=u^*+v^*-C$, we have
\[
 \phi(r^*)=r^*k^*-\frac\eps2(k^*)^2,
 \qquad\text{so}\qquad
 Ck^*+\frac\eps2(k^*)^2+\phi(r^*)=(u^*+v^*)k^*.
\]
Since $k_X^*=fe^{-u^*}$ by \eqref{S-eq:KL-optimality}, we have
\[
 \psi\Big(\frac{k_X^*}f\Big)=\psi(e^{-u^*})=-u^*e^{-u^*}-e^{-u^*}+1,
 \qquad\text{so}\qquad
 f\psi(k_X^*/f)+fe^{-u^*}=f-u^*k_X^*.
\]
In the same way, $g\psi(k_Y^*/g)+ge^{-v^*}=g-v^*k_Y^*$. Hence,
\[
 \Pp_{\rm KL}(k^*)+\G_{\rm KL}(z^*)
 =\int_Xf\dd x+\int_Yg\dd y+\int(u^*+v^*)k^*\dd x\dd y-\int_Xu^*k_X^*\dd x-\int_Yv^*k_Y^*\dd y.
\]
By Fubini's theorem, as in the proof of minimality, \[\int_{X\times Y} u^*k^*\dd x\dd y=\int_Xu^*k_X^*\dd x \quad\text{and}\quad\int_{X\times Y} v^*k^*\dd x\dd y=\int_Yv^*k_Y^*\dd y.\] Thus, the last three terms cancel, and we obtain $\Pp_{\rm KL}(k^*)+\G_{\rm KL}(z^*)=\int_Xf\dd x+\int_Yg\dd y$.
\end{proof}

The dual \eqref{S-eq:KLdual} has no global curvature bound, since the
exponential terms are neither strongly convex nor smooth on all of $L^2$. We
therefore localize the minimizer in an explicit box and continue the
exponential terms quadratically outside it. To this end, we define the
explicit endpoints
\begin{align}
 U_+&=\log\frac{E}{\delta b},&V_+&=\log\frac{E}{\delta a},\notag\\
 B_X&=\max\left\{1,\frac bc\max(\delta,V_+/\eps)\right\},&
 B_Y&=\max\left\{1,\frac ac\max(\delta,U_+/\eps)\right\},\label{S-eq:box-constants}\\
 \ell_X&=-\log B_X,&\ell_Y&=-\log B_Y,\notag\\
 r_X(x)&=\log\frac{f(x)}{\delta b},&
 r_Y(y)&=\log\frac{g(y)}{\delta a}.\notag
\end{align}
For a scalar $C^2$ function $q$ and endpoints $\ell<r$, we define the
quadratic continuation of $q$ at the endpoints by
\[
 \mathfrak Q_{\ell,r}[q](s)=
 \begin{cases}
 q(\ell)+q'(\ell)(s-\ell)+q''(\ell)(s-\ell)^2/2,&s<\ell,\\
 q(s),&\ell\leq s\leq r,\\
 q(r)+q'(r)(s-r)+q''(r)(s-r)^2/2,&s>r.
 \end{cases}
\]

\begin{proposition}\label{S-prop:box}
The dual minimizer belongs to the box
\[
 \mathcal B_{\rm KL}=\{\ell_X\leq u\leq r_X,\quad\ell_Y\leq v\leq r_Y\}.
\]
For $x\in X$ and $y\in Y$, let $q_{X,x}(s)=f(x)e^{-s}$ and $q_{Y,y}(s)=g(y)e^{-s}$ for $s\in\mathbb R$. Let
\[
	\widetilde\G_{\rm KL}(u,v)=\int_{X\times Y}\phi(u+v-C)\dd x\dd y
 +\int_X\mathfrak Q_{\ell_X,r_X(x)}[q_{X,x}]\big(u(x)\big)\dd x
 +\int_Y\mathfrak Q_{\ell_Y,r_Y(y)}[q_{Y,y}]\big(v(y)\big)\dd y .
\]
Then $\widetilde\G_{\rm KL}$ has the same minimizer, and it is globally strongly
convex and smooth on $L^2(X)\times L^2(Y)$, with constants
\begin{equation}\label{S-eq:KLconstants}
 \mu_{\rm KL}=\delta\min\{a,b\},\qquad
 L_{\rm KL}=(a+b)/\eps+E\max\{B_X,B_Y\}.
\end{equation}
These constants do not depend on the mesh.
\end{proposition}
\begin{proof}
We first prove the box bounds. By \eqref{S-eq:KL-optimality}, we have $k^*\ge\delta$, and hence $fe^{-u^*}=k_X^*\ge\delta b$. Therefore,
\[
 u^*\le\log\frac{f}{\delta b}=r_X\le\log\frac E{\delta b}=U_+,
\]
and similarly $v^*\leq r_Y\leq V_+$. For the lower bound, we suppose
that $u^*(x)\leq0$. Since $C\ge0$ and $v^*\le V_+$, we have $u^*(x)+v^*(y)-C(x,y)\le V_+$, and hence
\[
 ce^{-u^*(x)}\le f(x)e^{-u^*(x)}=k_X^*(x)
 =\int_Yp(u^*(x)+v^*(y)-C(x,y))\dd y
 \leq b\max\{\delta,V_+/\eps\}.
\]
Thus, $e^{-u^*(x)}\le B_X$, that is, $u^*(x)\geq-\log B_X=\ell_X$. When $u^*(x)>0$, this bound holds trivially, since
$B_X\geq1$. The bounds for $v^*$ follow in the same way, with $a,b$ and $X,Y$
interchanged.

Next, we compute the constants. 
We first consider the continued source term. By the definition of $\mathfrak Q$, and since $q_{X,x}''(s)=f(x)e^{-s}$, $e^{-\ell_X}=B_X$, and $f(x)e^{-r_X(x)}=\delta b$, we have
\[
 \frac{\dd^2}{\dd s^2}\,\mathfrak Q_{\ell_X,r_X(x)}[q_{X,x}](s)=
 \begin{cases}
 f(x)B_X,& s<\ell_X,\\
 f(x)e^{-s},& \ell_X\le s\le r_X(x),\\
 \delta b,& s>r_X(x).
 \end{cases}
\]
For $\ell_X\le s\le r_X(x)$, we have $\delta b=f(x)e^{-r_X(x)}\le f(x)e^{-s}\le f(x)e^{-\ell_X}=f(x)B_X$. Since $\delta b<c\le f(x)\le E$ and $B_X\ge1$, we obtain, for all real $s$,
\[
 \delta b\le\frac{\dd^2}{\dd s^2}\,\mathfrak Q_{\ell_X,r_X(x)}[q_{X,x}](s)\le EB_X .
\]
Hence, these bounds hold for all real $u$. In the same way, using $q_{Y,y}''(s)=g(y)e^{-s}$, $e^{-\ell_Y}=B_Y$, and $g(y)e^{-r_Y(y)}=\delta a$, we obtain, for all real $s$,
\[
 \delta a\le\frac{\dd^2}{\dd s^2}\,\mathfrak Q_{\ell_Y,r_Y(y)}[q_{Y,y}](s)\le EB_Y .
\]

For the term $z\mapsto\int_{X\times Y}\phi(\T z-C)\dd x\dd y$, the gradient is $z\mapsto\T^*p(\T z-C)$. Since $0\le p'\le1/\eps$ and $\norm{\T w}_2\le\sqrt{a+b}\,\norm{w}$, this gradient is monotone and Lipschitz with constant $(a+b)/\eps$.

We now add these bounds. Since $z\mapsto\int_{X\times Y}\phi(\T z-C)\dd x\dd y$ is convex, the strong convexity constant of $\widetilde\G_{\rm KL}$ comes from the terms
\[
 u\mapsto\int_X\mathfrak Q_{\ell_X,r_X(x)}[q_{X,x}]\big(u(x)\big)\dd x,\qquad
 v\mapsto\int_Y\mathfrak Q_{\ell_Y,r_Y(y)}[q_{Y,y}]\big(v(y)\big)\dd y
\]
alone. By the lower bounds on their second derivatives, it is $\min\{\delta b,\delta a\}=\delta\min\{a,b\}$. Adding the upper bounds, the gradient of $\widetilde\G_{\rm KL}$ is Lipschitz with constant $(a+b)/\eps+E\max\{B_X,B_Y\}$. This gives \eqref{S-eq:KLconstants}.

Finally, we show that $z^*$ minimizes $\widetilde\G_{\rm KL}$. Its gradient is
\[
 \nabla\widetilde\G_{\rm KL}(z)=\T^*p(\T z-C)
 +\Big(\mathfrak Q_{\ell_X,r_X(x)}[q_{X,x}]'\big(u(x)\big),\ \mathfrak Q_{\ell_Y,r_Y(y)}[q_{Y,y}]'\big(v(y)\big)\Big).
\]
Since $z^*\in\mathcal B_{\rm KL}$, we have $\ell_X\le u^*(x)\le r_X(x)$, so \[\mathfrak Q_{\ell_X,r_X(x)}[q_{X,x}]'(u^*(x))=q_{X,x}'(u^*(x))=-f(x)e^{-u^*(x)},\] and similarly for $v^*$. Moreover, $\T^*k=(k_X,k_Y)$. Hence, by \eqref{S-eq:KL-optimality},
\[
 \nabla\widetilde\G_{\rm KL}(z^*)
 =\big(k_X^*-fe^{-u^*},\ k_Y^*-ge^{-v^*}\big)=0 .
\]
Since $\widetilde\G_{\rm KL}$ is strongly convex, $z^*$ is its unique minimizer.
\end{proof}

As a consequence, we put $\gamma_{\rm KL}=\sqrt{\mu_{\rm KL}/L_{\rm KL}}$. Then the projected gradient method and V-FISTA applied to
$\widetilde\G_{\rm KL}+I_{\mathcal B_{\rm KL}}$ have the distance factor and the objective factor
\[
 q_{\rm KL}=\frac{L_{\rm KL}-\mu_{\rm KL}}{L_{\rm KL}+\mu_{\rm KL}}
 \qquad\text{and}\qquad 1-\gamma_{\rm KL},
\]
respectively. These factors are obtained as in \cref{S-sec:global-pg}, from the standard estimates for strongly convex problems \cite{beck2017first,nesterov2004introductory}.
We note that the continuation is what makes V-FISTA well defined. Indeed, its
extrapolated points may leave $\mathcal B_{\rm KL}$, where the exponential terms without continuation have no global curvature bound.

For small $\delta$, the small value of $\mu_{\rm KL}$ is intrinsic to this global
box, since at the upper endpoints $r_X$ and $r_Y$, the curvatures of the exponential terms are exactly
$\delta b$ and $\delta a$. \Cref{S-prop:tighten} below removes this defect by
localization.

\begin{remark}[Dependence of the constants of the KL problem on the lower bound]\label{S-rem:KL-scaling}
As for the constants of the Pearson problem in \cref{M-rem:eps-c} of the main paper, the
bound $L_{\rm KL}/\mu_{\rm KL}$ on the condition number of the KL problem is not bounded uniformly in $\delta$ and $\eps$. It grows as $\delta\downarrow0$ or $\eps\downarrow0$. By \eqref{S-eq:box-constants} and \eqref{S-eq:KLconstants}, it is of order
\[
 \frac{E\max\{a,b\}\log\big(E/(\delta\min\{a,b\})\big)}
 {c\eps\delta\min\{a,b\}},
\]
In particular, it blows up as $\delta\downarrow0$. For this reason, the a posteriori tightening of the box in \cref{S-prop:tighten} is needed to obtain a
practically useful a posteriori error bound for the KL model when $\delta$ is small.
\end{remark}
\subsection{Discretization, gap bounds, and localization}\label{S-sec:KLdisc}
We now pass to cellwise data, following the Pearson analysis of
\cref{M-sec:duality} of the main paper. Let $f_h,g_h,C_h$ be cellwise data on partitions
with volumes $w_i,\omega_j$, and let $\norm{\cdot}_h$ denote the weighted
norm. The reconstruction is $K_{ij}(z)=p(u_i+v_j-C_{ij})$.

\begin{proposition}\label{S-prop:KLreduction}
The minimizers of $\Pp_{\rm KL}$ and $\G_{\rm KL}$ with data $f_h,g_h,C_h$
are cellwise constant. They are the lifts of the minimizers of the
corresponding array functionals $\Pp_{{\rm KL},h}$ and $\G_{{\rm KL},h}$, that
is, of \eqref{S-eq:KLprimal} and \eqref{S-eq:KLdual} restricted to arrays.
Consequently, the statements of
\cref{M-prop:reduction} of the main paper hold for the KL primal and
dual problems.
\end{proposition}
\begin{proof}
As in the proof of \cref{M-prop:reduction}, we show that cell averaging
decreases the primal objective with data $f_h,g_h,C_h$. We write $k=\Q_hk+(k-\Q_hk)$ and recall two
facts from that proof. First, cell averaging commutes with marginalization, that is,
$(\Q_hk)_X=\Q_h^Xk_X$. Second, since $C_h$ is constant on product cells, we have
\[\int_{X\times Y} C_hk\dd x\dd y=\int_{X\times Y} C_h\Q_hk\dd x\dd y.\]
Thus, the linear cost term is unchanged. Moreover,
the quadratic term satisfies the Pythagorean identity
\[
 \norm{k}_\Zz^2=\norm{\Q_hk}_\Zz^2+\norm{k-\Q_hk}_\Zz^2,
\]
so it strictly decreases under $\Q_h$ unless $k=\Q_hk$.

It remains to treat the KL marginal term. Since $f_h$ equals the constant $f_i$ on each cell $X_i$, and $\psi$ is convex, Jensen's inequality for the
probability measure $w_i^{-1}\dd x$ on $X_i$ gives
\[
 \psi\!\left(\frac{(\Q_h^Xk_X)_i}{f_i}\right)
 =\psi\!\left(\frac1{w_i}\int_{X_i}\frac{k_X}{f_i}\dd x\right)
 \leq\frac1{w_i}\int_{X_i}\psi\!\left(\frac{k_X(x)}{f_i}\right)\dd x.
\]
We multiply by $w_if_i$, sum over $i$, and argue in the same way over $j$. This
shows that averaging does not increase the KL marginal terms, that is,
\[
 \int_Xf_h\psi\Big(\frac{(\Q_hk)_X}{f_h}\Big)+\int_Yg_h\psi\Big(\frac{(\Q_hk)_Y}{g_h}\Big)
 \le\int_Xf_h\psi\Big(\frac{k_X}{f_h}\Big)\dd x+\int_Yg_h\psi\Big(\frac{k_Y}{g_h}\Big)\dd y.
\]

Combining the three terms, we see that averaging strictly decreases
$\Pp_{\rm KL}$ unless $k$ is already cellwise constant. Therefore, the
continuum minimizer is cellwise constant, and on cellwise constant functions,
$\Pp_{\rm KL}$ coincides with $\Pp_{{\rm KL},h}$. The dual statement follows
as in the proof of \cref{M-prop:reduction}, with the optimality conditions
\eqref{S-eq:KL-optimality} in place of the Pearson ones.
\end{proof}

For an array potential $z$, we define the discrete duality gap
\[
 \Gap_{{\rm KL},h}(z)=\Pp_{{\rm KL},h}(K(z))+\G_{{\rm KL},h}(z)
 -\sum_{i=1}^{N_X}w_if_i-\sum_{j=1}^{N_Y}\omega_jg_j.
\]
We write $r_{ij}=u_i+v_j-C_{ij}$. Since $K_{ij}(z)=p(r_{ij})$ attains the supremum in the definition of $\phi$, we have $\phi(r_{ij})=r_{ij}K_{ij}-\frac\eps2K_{ij}^2$, as in the proof of \cref{S-prop:KLduality}. Hence, the coupling terms reduce to
\[
	\sum_{(i,j)\in\N}w_i\omega_j(u_i+v_j)K_{ij}=\sum_{i=1}^{N_X}w_iu_iK_{X,i}+\sum_{j=1}^{N_Y}\omega_jv_jK_{Y,j}.
\]
Moreover, a direct computation gives
\[
 f_i\psi\Big(\frac{K_{X,i}}{f_i}\Big)+f_ie^{-u_i}-f_i+u_iK_{X,i}
 =f_ie^{-u_i}\psi\Big(\frac{K_{X,i}}{f_i}e^{u_i}\Big),
\]
and similarly on $Y$. Therefore, the gap reduces to the marginal terms,
\begin{equation}\label{S-eq:KLgap}
 \Gap_{{\rm KL},h}(z)=
 \sum_{i=1}^{N_X}w_if_i e^{-u_i}\psi\!\left(\frac{K_{X,i}}{f_i}e^{u_i}\right)
 +\sum_{j=1}^{N_Y}\omega_jg_j e^{-v_j}\psi\!\left(\frac{K_{Y,j}}{g_j}e^{v_j}\right).
\end{equation}
For a general feasible plan $k$ in place of $K(z)$, one adds the nonnegative term
\[
	\sum_{(i,j)\in\N}w_i\omega_j\Big[\phi(r_{ij})-r_{ij}k_{ij}+\frac\eps2k_{ij}^2\Big].
\]

\begin{proposition}\label{S-prop:KLradius}
For every array potential $z$,
\[
 \norm{ K(z)-k_{{\rm KL},h}^*}_h
 \leq\sqrt{\Gap_{{\rm KL},h}(z)/\eps}.
\]
\end{proposition}
\begin{proof}
We write $k=K(z)$ and $k^*=k^*_{{\rm KL},h}$, and we denote by $z^*=(u^*,v^*)$ the array minimizer of $\G_{{\rm KL},h}$. By the discrete version of \cref{S-prop:KLduality}, we have
\[
	\min\Pp_{{\rm KL},h}+\min\G_{{\rm KL},h}=\sum_{i=1}^{N_X}w_if_i+\sum_{j=1}^{N_Y}\omega_jg_j.
\]
Hence, the gap splits into the primal excess and the dual excess,
\[
 \Gap_{{\rm KL},h}(z)=\big[\Pp_{{\rm KL},h}(k)-\Pp_{{\rm KL},h}(k^*)\big]+\big[\G_{{\rm KL},h}(z)-\G_{{\rm KL},h}(z^*)\big].
\]
We show that each excess is at least $\frac\eps2\norm{k-k^*}_h^2$.

For the primal excess, we use that $\Pp_{{\rm KL},h}$ is $\eps$-strongly convex, and that $k^*$ satisfies the variational inequality $\ip{A_{{\rm KL},h}(k^*)}{k-k^*}_h\ge0$, since $k\ge\delta$. Hence, we obtain
\[
 \Pp_{{\rm KL},h}(k)-\Pp_{{\rm KL},h}(k^*)
 \ge\ip{A_{{\rm KL},h}(k^*)}{k-k^*}_h+\frac\eps2\norm{k-k^*}_h^2
 \ge\frac\eps2\norm{k-k^*}_h^2.
\]

Since $r_{ij}-r^*_{ij}=(u_i-u_i^*)+(v_j-v_j^*)$, $k^*_{X,i}=\sum_{j=1}^{N_Y}\omega_jk^*_{ij}$, and $k^*_{Y,j}=\sum_{i=1}^{N_X}w_ik^*_{ij}$, we have
\[
	\begin{aligned}
		\sum_{(i,j)\in\N}w_i\omega_jp(r^*_{ij})(r_{ij}-r^*_{ij})
		&=\sum_{i=1}^{N_X}w_i(u_i-u_i^*)\sum_{j=1}^{N_Y}\omega_jk^*_{ij}+\sum_{j=1}^{N_Y}\omega_j(v_j-v_j^*)\sum_{i=1}^{N_X}w_ik^*_{ij}\\
		&=\sum_{i=1}^{N_X}w_ik^*_{X,i}(u_i-u_i^*)+\sum_{j=1}^{N_Y}\omega_jk^*_{Y,j}(v_j-v_j^*).
	\end{aligned}
\]
Hence, by the optimality conditions,
\[
	\sum_{(i,j)\in\N}w_i\omega_jp(r^*_{ij})(r_{ij}-r^*_{ij})
	-\sum_{i=1}^{N_X}w_if_ie^{-u_i^*}(u_i-u_i^*)
	-\sum_{j=1}^{N_Y}\omega_jg_je^{-v_j^*}(v_j-v_j^*)=0 .
\]
We can write the dual excess as in \cref{S-rem:factor-two}, that is,
\begin{align*}
 \G_{{\rm KL},h}(z)-\G_{{\rm KL},h}(z^*)
 &=\sum_{(i,j)\in\N}w_i\omega_j\big[\phi(r_{ij})-\phi(r^*_{ij})-p(r^*_{ij})(r_{ij}-r^*_{ij})\big]\\
 &\quad+\sum_{i=1}^{N_X}w_if_i\big[e^{-u_i}-e^{-u_i^*}+e^{-u_i^*}(u_i-u_i^*)\big]\\
 &\quad+\sum_{j=1}^{N_Y}\omega_jg_j\big[e^{-v_j}-e^{-v_j^*}+e^{-v_j^*}(v_j-v_j^*)\big].
\end{align*}
We first estimate the second and third sums in this decomposition. The function $s\mapsto e^{-s}$ is convex with derivative $-e^{-s}$, so its tangent line at $u_i^*$ lies below it, that is,
\[
 e^{-u_i}\ge e^{-u_i^*}-e^{-u_i^*}(u_i-u_i^*).
\]
Hence, we have
\[
 e^{-u_i}-e^{-u_i^*}+e^{-u_i^*}(u_i-u_i^*)\ge0,
 \qquad
 e^{-v_j}-e^{-v_j^*}+e^{-v_j^*}(v_j-v_j^*)\ge0,
\]
where the second inequality follows in the same way. Since $f_i,g_j>0$, the second and third sums are nonnegative.
Next, we estimate the first sum. Since $p$ is nondecreasing and $1/\eps$-Lipschitz, the pointwise bound of \cref{S-rem:factor-two} gives
\[
 \phi(r_{ij})-\phi(r^*_{ij})-p(r^*_{ij})(r_{ij}-r^*_{ij})\geq\frac\eps2|p(r_{ij})-p(r^*_{ij})|^2=\frac\eps2|k_{ij}-k^*_{ij}|^2.
\]
Combining these two estimates with the decomposition above, we obtain
\[
	\G_{{\rm KL},h}(z)-\G_{{\rm KL},h}(z^*)\ge\frac\eps2\sum_{(i,j)\in\N}w_i\omega_j|k_{ij}-k^*_{ij}|^2=\frac\eps2\norm{k-k^*}_h^2.
\]

Adding the two lower bounds gives $\Gap_{{\rm KL},h}(z)\ge\eps\norm{k-k^*}_h^2$, which is the claim.
\end{proof}

\Cref{S-prop:KLradius} bounds the algebraic error of the reconstructed plan, in the sense of \cref{M-sec:contributions} of the main paper. We now add a bound for the data-approximation error, as in \cref{M-cor:total}.

\begin{proposition}\label{S-prop:KL-stability}
Suppose that both the data $f,g,C$ and the cellwise data $f_h,g_h,C_h$ satisfy the assumptions of \cref{S-sec:KLduality}. We define
\[
 D_{{\rm KL},h}=\frac1\eps\left[
	 \norm{C-C_h}_\Zz+\frac{\sqrt b}{c}\norm{f-f_h}_{L^2(X)}
 +\frac{\sqrt a}{c}\norm{g-g_h}_{L^2(Y)}\right].
\]
Then
\begin{align*}
 \norm{\Lsf k_{{\rm KL},h}^*-k_{\rm KL}^*}_\Zz&\leq D_{{\rm KL},h},\\
 \norm{\Lsf^Xu_{{\rm KL},h}^*-u_{\rm KL}^*}_{L^2(X)}
 &\leq D_{{\rm KL},h}/(\delta\sqrt b)+\norm{f-f_h}_{L^2(X)}/c,\\
 \norm{\Lsf^Yv_{{\rm KL},h}^*-v_{\rm KL}^*}_{L^2(Y)}
 &\leq D_{{\rm KL},h}/(\delta\sqrt a)+\norm{g-g_h}_{L^2(Y)}/c.
\end{align*}
Moreover, for every array potential $z_h$,
\begin{equation}\label{S-eq:KLtotal}
 \norm{\Lsf K(z_h)-k_{\rm KL}^*}_\Zz
 \leq D_{{\rm KL},h}+\sqrt{\Gap_{{\rm KL},h}(z_h)/\eps}.
\end{equation}
\end{proposition}
\begin{proof}
We follow the proof of \cref{M-prop:stability}. Let $d=\Lsf k_{{\rm KL},h}^*-k_{\rm KL}^*$, and let $A_{{\rm KL},h}$ be the gradient with the cellwise data. Both minimizers lie in the common feasible set $\{k\ge\delta\}$. Both minimizers lie in the common feasible set $\{k\ge\delta\}$. Testing the variational inequality for $k_{{\rm KL},h}^*$ with $k_{\rm KL}^*$, and the one for $k_{\rm KL}^*$ with $k_{{\rm KL},h}^*$, we obtain
\[
 \ip{A_{{\rm KL},h}(k_{{\rm KL},h}^*)}{d}_\Hh\le0\le\ip{A_{\rm KL}(k_{\rm KL}^*)}{d}_\Hh.
\]
Subtracting $\ip{A_{{\rm KL},h}(k_{\rm KL}^*)}{d}_\Zz$ from both sides, we obtain
\[
 \ip{A_{{\rm KL},h}(k_{{\rm KL},h}^*)-A_{{\rm KL},h}(k_{\rm KL}^*)}{d}_\Zz
 \le\ip{A_{\rm KL}(k_{\rm KL}^*)-A_{{\rm KL},h}(k_{\rm KL}^*)}{d}_\Zz.
\]
Since the data terms cancel in the difference, and $\log$ is increasing, we have
\[
\begin{aligned}
 \ip{A_{{\rm KL},h}(k_{{\rm KL},h}^*)-A_{{\rm KL},h}(k_{\rm KL}^*)}{d}
 &=\eps\norm{d}_\Zz^2
 +\int_X\big(\log k_{{\rm KL},h,X}^*-\log k_{{\rm KL},X}^*\big)\big(k_{{\rm KL},h,X}^*-k_{{\rm KL},X}^*\big)\dd x\\
 &\quad+\int_Y\big(\log k_{{\rm KL},h,Y}^*-\log k_{{\rm KL},Y}^*\big)\big(k_{{\rm KL},h,Y}^*-k_{{\rm KL},Y}^*\big)\dd y\\
 &\ge\eps\norm{d}_\Zz^2 .
\end{aligned}
\]
For every feasible $k$, the terms $\eps k$, $\log k_X$, and $\log k_Y$ cancel in the difference of the gradients, and we obtain
\[
 A_{\rm KL}(k)-A_{{\rm KL},h}(k)
 =C-C_h+\log(f_h/f)+\log(g_h/g).
\]
In particular, this holds for $k=k_{\rm KL}^*$ on the right-hand side of the variational-inequality estimate.
Since $\log$ is $1/c$-Lipschitz on $[c,\infty)$, we have
\[
	\norm{\log(f_h/f)}_{L^2(X\times Y)}=\sqrt b\,\norm{\log f_h-\log f}_{L^2(X)}\le\frac{\sqrt b}c\norm{f-f_h}_{L^2(X)},
\]
and similarly on $Y$. Hence, by the triangle inequality and the definition of $D_{{\rm KL},h}$,
\[
 \norm{A_{\rm KL}(k_{\rm KL}^*)-A_{{\rm KL},h}(k_{\rm KL}^*)}_\Zz
 \le\norm{C-C_h}_\Zz+\frac{\sqrt b}{c}\norm{f-f_h}_{L^2(X)}+\frac{\sqrt a}{c}\norm{g-g_h}_{L^2(Y)}
 =\eps D_{{\rm KL},h}.
\] By the Cauchy--Schwarz inequality, we obtain $\eps\norm{d}_\Zz^2\le\eps D_{{\rm KL},h}\norm{d}_\Zz$, which proves the plan estimate.

For the potentials, \eqref{S-eq:KL-optimality} for both problems gives
\[
 u_{{\rm KL},h}^*-u_{\rm KL}^*=\big(\log k_{{\rm KL},X}^*-\log k_{{\rm KL},h,X}^*\big)+\big(\log f_h-\log f\big).
\]
Since $k_{\rm KL}^*\ge\delta$ and $k_{{\rm KL},h}^*\ge\delta$, we observe that $k_{{\rm KL},X}^*\ge\delta b$ and $k_{{\rm KL},h,X}^*\ge\delta b$. On $[\delta b,\infty)$, the logarithm is $1/(\delta b)$-Lipschitz. Moreover, we have $\norm{k_X-\widehat k_X}_{L^2(X)}\leq\sqrt b\norm{k-\widehat k}_\Zz$. Hence, the first term satisfies
\[
 \norm{\log k_{{\rm KL},X}^*-\log k_{{\rm KL},h,X}^*}_{L^2(X)}\le\frac{\sqrt b\,D_{{\rm KL},h}}{\delta b}=\frac{D_{{\rm KL},h}}{\delta\sqrt b}.
\] The second term is at most $\norm{f-f_h}_{L^2(X)}/c$. Hence,
\[
 \norm{\Lsf^Xu_{{\rm KL},h}^*-u_{\rm KL}^*}_{L^2(X)}
 \leq D_{{\rm KL},h}/(\delta\sqrt b)+\norm{f-f_h}_{L^2(X)}/c.
\]
We bound $\norm{\Lsf^Yv_{{\rm KL},h}^*-v_{\rm KL}^*}_{L^2(Y)}$ in the same way.
Finally, \cref{S-prop:KLreduction},
\cref{S-prop:KLradius}, and the triangle inequality give \eqref{S-eq:KLtotal}.
\end{proof}

The potential estimates combine with the factors of the global projected gradient method and of V-FISTA,
exactly as in \cref{S-sec:global-total}. This gives bounds with explicit constants for the total error, consisting of the data-approximation error and the algebraic error. To state them, we let
\[
	D_{z,{\rm KL},h}=\Big(\big(D_{{\rm KL},h}/(\delta\sqrt b)+\norm{f-f_h}_{L^2(X)}/c\big)^2+\big(D_{{\rm KL},h}/(\delta\sqrt a)+\norm{g-g_h}_{L^2(Y)}/c\big)^2\Big)^{1/2}
\]
be the combination of the two potential bounds in \cref{S-prop:KL-stability}, and
\[
 R_{\rm KL}^2=a\max\{\log B_X,U_+\}^2+b\max\{\log B_Y,V_+\}^2.
\]
Then $R_{\rm KL}$ bounds the weighted distance from zero to every point of
the box $\mathcal B_{\rm KL}$. Starting from $z^0=0$, the projected gradient method and V-FISTA have potential errors at most
\[
 D_{z,{\rm KL},h}+R_{\rm KL}q_{\rm KL}^m
 \qquad\text{and}\qquad
 D_{z,{\rm KL},h}+\sqrt{2L_{\rm KL}/\mu_{\rm KL}}\,R_{\rm KL}(1-\gamma_{\rm KL})^{m/2},
\]
respectively.

These global factors inherit the poor conditioning described in
\cref{S-rem:KL-scaling}. We now show how a computed potential can be used to
shrink the box, and hence to improve the constants a posteriori.

\begin{proposition}\label{S-prop:tighten}
Let $z^0$ lie in the box $\mathcal B_{\rm KL}$ for arrays, and let $D_0$ be a known upper bound for $\norm{z^0-z_h^*}_h$, where $z_h^*$ is the array minimizer. One admissible choice is
$D_0=\norm{\nabla_h\widetilde\G_{{\rm KL},h}(z^0)}_h/\mu_{\rm KL}$.
We intersect the coordinate intervals of $\mathcal B_{\rm KL}$ with
\[
 [u_i^0-D_0/\sqrt{w_i},u_i^0+D_0/\sqrt{w_i}],\qquad
 [v_j^0-D_0/\sqrt{\omega_j},v_j^0+D_0/\sqrt{\omega_j}],
\]
and denote the new endpoints by
$\underline u_i,\overline u_i,\underline v_j,\overline v_j$ for $i=1,\dots,N_X$ and $j=1,\dots,N_Y$. Then the
minimizer lies in the new box. Moreover, the objective continued quadratically at the new endpoints has the strong convexity and smoothness constants
\begin{align*}
	\mu_h^0&=\min\{\min_{i\in\N_X} f_i e^{-\overline u_i},
	\min_{j\in\N_Y}g_j e^{-\overline v_j}\},\\
 L_h^0&=(a+b)/\eps+
 \max\{\max_{i\in\N_X} f_i e^{-\underline u_i},
 \max_{j\in\N_Y}g_j e^{-\underline v_j}\}.
\end{align*}
They satisfy $\mu_h^0\geq\mu_{\rm KL}$ and $L_h^0\leq L_{\rm KL}$.
\end{proposition}
\begin{proof}
	We first show that the minimizer lies in the new box. For $i=1,\dots,N_X$, we drop all other terms in the squared weighted norm, and obtain
\[
 w_i|u_i^0-u_i^*|^2\le\norm{z^0-z_h^*}_h^2\le D_0^2.
\]
Hence, $|u_i^0-u_i^*|\le D_0/\sqrt{w_i}$, and similarly $|v_j^0-v_j^*|\le D_0/\sqrt{\omega_j}$ for $j=1,\dots,N_Y$. By \cref{S-prop:box}, $z_h^*$ also lies in $\mathcal B_{\rm KL}$. Therefore, $z_h^*$ lies in the intersection.

Next, we verify the admissible choice of $D_0$. Since $\widetilde\G_{{\rm KL},h}$ is $\mu_{\rm KL}$-strongly convex and $\nabla_h\widetilde\G_{{\rm KL},h}(z_h^*)=0$, we have
\[
 \mu_{\rm KL}\norm{z^0-z_h^*}_h^2
 \le\ip{\nabla_h\widetilde\G_{{\rm KL},h}(z^0)}{z^0-z_h^*}_h
 \le\norm{\nabla_h\widetilde\G_{{\rm KL},h}(z^0)}_h\norm{z^0-z_h^*}_h.
\]
Dividing by $\mu_{\rm KL}\norm{z^0-z_h^*}_h$ gives $\norm{z^0-z_h^*}_h\le D_0$.

We now compute the constants. For $u\in[\underline u_i,\overline u_i]$ and $v\in[\underline v_j,\overline v_j]$, the second derivatives $f_ie^{-u}$ and $g_je^{-v}$ of the source and target terms satisfy
\[
 f_ie^{-\overline u_i}\le f_ie^{-u}\le f_ie^{-\underline u_i},\qquad
 g_je^{-\overline v_j}\le g_je^{-v}\le g_je^{-\underline v_j}.
\]
As in the proof of \cref{S-prop:box}, the quadratic continuation at the new endpoints keeps these bounds for all real $u$ and $v$.
The gradient of the coupling term is monotone and Lipschitz with constant $(a+b)/\eps$. Hence, the continued objective is $\mu_h^0$-strongly convex and has an $L_h^0$-Lipschitz gradient. It has the same minimizer, by the argument at the end of the proof of \cref{S-prop:box}.

Finally, we compare the constants. The new box is contained in $\mathcal B_{\rm KL}$, so $\overline u_i\le\log(f_i/(\delta b))$, $\underline u_i\ge\ell_X$, $\overline v_j\le\log(g_j/(\delta a))$, and $\underline v_j\ge\ell_Y$. Hence, we have
\[
\begin{aligned}
 f_ie^{-\overline u_i}&\ge\delta b\ge\delta\min\{a,b\},&\qquad
 f_ie^{-\underline u_i}&\le Ee^{-\ell_X}=EB_X,\\
 g_je^{-\overline v_j}&\ge\delta a\ge\delta\min\{a,b\},&\qquad
 g_je^{-\underline v_j}&\le Ee^{-\ell_Y}=EB_Y.
\end{aligned}
\]
Taking the minimum and the maximum over $i$ and $j$ gives $\mu_h^0\ge\mu_{\rm KL}$ and $L_h^0\le L_{\rm KL}$.
\end{proof}

In practice, a computable bound $D_0$ can also be obtained from one
step of the projected gradient method. Let $\mathcal S$ be the projected gradient map with step size
$\tau=2/(L+\mu)$, where $\mu$ and $L$ are strong convexity and smoothness constants of the continued functional, for
instance $\mu_{\rm KL}$ and $L_{\rm KL}$. Then $\mathcal S$ is a contraction with factor
$q=(L-\mu)/(L+\mu)$, and $z_h^*=\mathcal Sz_h^*$. Hence, we have
\[
 \norm{z-z_h^*}_h\le\norm{z-\mathcal Sz}_h+\norm{\mathcal Sz-\mathcal Sz_h^*}_h\le\norm{z-\mathcal Sz}_h+q\norm{z-z_h^*}_h,
\]
that is, $(1-q)\norm{z-z_h^*}_h\leq\norm{z-\mathcal Sz}_h$.
Since $\tau/(1-q)=1/\mu$, the quantity $\norm{(z-\mathcal Sz)/\tau}_h/\mu$ is an admissible choice of $D_0$. Our implementation uses this bound. We note that the coordinate
factors $1/\sqrt{w_i}$ make the tightening a localization on a fixed mesh.
It is unrelated to the maximum-norm Newton theory of the main paper, which is uniform in the mesh.
\subsection{A quantitative comparison of the Pearson and KL models}\label{S-sec:model}
We now compare the two models directly. Let $\Pp_\chi$ denote the primal functional
\eqref{S-eq:KLprimal} with $\psi_\chi(t)=(t-1)^2/2$ in place of $\psi(t)$.
Thus, both problems have the same cost, the same regularization parameter, and the same lower bound $\delta>0$. We write their minimizers as $k_\chi^*$ and $k_{\rm KL}^*$. For a
feasible plan $k$, we denote its marginal ratios by $t_X=k_X/f$ and $t_Y=k_Y/g$. We also set $r(t)=t-1-\log t$ for $t>0$, and
\[
 \mathfrak R(k)=
 \sqrt b\,\norm{r(t_X)}_{L^2(X)}
 +\sqrt a\,\norm{r(t_Y)}_{L^2(Y)}.
\]
\begin{proposition}\label{S-prop:model-comparison}
The two minimizers satisfy
\begin{equation}\label{S-eq:model}
 \norm{k_\chi^*-k_{\rm KL}^*}_\Zz
 \leq\eps^{-1}\min\{\mathfrak R(k_\chi^*),\mathfrak R(k_{\rm KL}^*)\}.
\end{equation}
Moreover, let $t_X,t_Y$ be the marginal ratios of the KL minimizer
$k_{\rm KL}^*$. Then
\begin{equation}\label{S-eq:model-energy}
 \norm{k_\chi^*-k_{\rm KL}^*}_\Zz
 \leq\frac{1}{2\sqrt\eps}\Big(\int_Xf\,r(t_X)^2\dd x+\int_Yg\,r(t_Y)^2\dd y\Big)^{1/2}.
\end{equation}
Suppose, in addition, that the two marginal ratios of one of the two minimizers lie in
$[1-\theta,1+\theta]$ for some $0<\theta<1$, and let $t_X,t_Y$ denote the ratios of
that minimizer. Then
\begin{equation}\label{S-eq:model-L4}
 \norm{k_\chi^*-k_{\rm KL}^*}_\Zz
 \leq\frac{\sqrt b\,\norm{t_X-1}_{L^4(X)}^2+
 \sqrt a\,\norm{t_Y-1}_{L^4(Y)}^2}
              {2\eps(1-\theta)^2}.
\end{equation}
If that minimizer is $k_{\rm KL}^*$, then, in addition,
\begin{equation}\label{S-eq:model-energy-L4}
 \norm{k_\chi^*-k_{\rm KL}^*}_\Zz
 \leq\frac{1}{4\sqrt\eps(1-\theta)^2}\Big(\int_Xf(t_X-1)^4\dd x+\int_Yg(t_Y-1)^4\dd y\Big)^{1/2}.
\end{equation}
\end{proposition}
\begin{proof}
The gradients of the two primal functionals are
\[
 A_\chi(k)=C+\eps k+(t_X-1)+(t_Y-1),\qquad
 A_{\rm KL}(k)=C+\eps k+\log t_X+\log t_Y,
\]
where $t_X,t_Y$ are the marginal ratios of $k$. The marginal terms of both gradients are gradients of convex functionals, and hence monotone. Therefore, both gradients are $\eps$-strongly monotone on the set $\{k\ge\delta\}$.

We first prove \eqref{S-eq:model}. Let $d=k_\chi^*-k_{\rm KL}^*$. The variational inequalities of the two problems, tested with the other minimizer, give
\[
 \ip{A_\chi(k_\chi^*)}{-d}_\Zz\ge0,\qquad
 \ip{A_{\rm KL}(k_{\rm KL}^*)}{d}_\Zz\ge0.
\]
Adding these two inequalities and subtracting $\ip{A_\chi(k_{\rm KL}^*)}{d}_\Zz$ on both sides, we obtain
\begin{equation}\label{S-eq:model-core}
	\ip{A_\chi(k_\chi^*)-A_\chi(k_{\rm KL}^*)}{d}_{\Zz}
 \leq\ip{A_{\rm KL}(k_{\rm KL}^*)-A_\chi(k_{\rm KL}^*)}{d}_\Zz.
\end{equation}
Since $A_\chi$ is affine with linear part $k\mapsto\eps k+k_X/f+k_Y/g$, and the pairing of a function of $x$ alone with $d$ equals its pairing with $d_X$, we have
\begin{equation}\label{S-eq:model-lhs}
 \ip{A_\chi(k_\chi^*)-A_\chi(k_{\rm KL}^*)}{d}_\Zz
 =\eps\norm{d}_\Zz^2+\mathcal J(d)\ge\eps\norm{d}_\Zz^2,
 \qquad
 \mathcal J(d)=\int_X\frac{|d_X|^2}f\dd x+\int_Y\frac{|d_Y|^2}g\dd y.
\end{equation}
For the right-hand side of \eqref{S-eq:model-core}, let $t_X=k^*_{ {\rm KL},X}/f,t_Y=k^*_{ {\rm KL},Y}/g$. Since $\log t-(t-1)=-r(t)$, we have
\[
 A_{\rm KL}(k_{\rm KL}^*)-A_\chi(k_{\rm KL}^*)=-r(t_X)-r(t_Y).
\]
Since $r(t_X)$ depends only on $x$ and $r(t_Y)$ only on $y$, Fubini's theorem gives
\begin{equation}\label{S-eq:r_t_X_Y}
 \ip{A_{\rm KL}(k_{\rm KL}^*)-A_\chi(k_{\rm KL}^*)}{d}
 =-\int_{X\times Y}\big(r(t_X)+r(t_Y)\big)d\dd x\dd y
 =-\int_Xr(t_X)\,d_X\dd x-\int_Yr(t_Y)\,d_Y\dd y.
\end{equation}
By the Cauchy--Schwarz inequality and $\norm{d_X}_{L^2(X)}\le\sqrt b\,\norm{d}_\Zz$, we obtain
\[
 \Big|\int_Xr(t_X)d_X\dd x\Big|\le\sqrt b\,\norm{r(t_X)}_{L^2(X)}\norm{d}_\Zz,
\]
and similarly on $Y$. Hence, the right-hand side of \eqref{S-eq:model-core} is at most $\mathfrak R(k_{\rm KL}^*)\norm{d}_2$, and we obtain
\[
 \eps\norm{d}_\Zz^2\le\mathfrak R(k_{\rm KL}^*)\norm{d}_\Zz.
\]
If we interchange the roles of the two problems, the same argument with the strong monotonicity of $A_{\rm KL}$ gives $\eps\norm{d}_\Zz^2\le\mathfrak R(k_\chi^*)\norm{d}_\Zz$. Dividing by $\eps\norm{d}_\Zz$ proves \eqref{S-eq:model}.

We now prove \eqref{S-eq:model-energy}. By \eqref{S-eq:model-lhs}, the left-hand side of \eqref{S-eq:model-core} equals $\eps\norm{d}_\Zz^2+\mathcal J(d)$.
For the right-hand side, we write $r(t_X)d_X=\big(\sqrt f\,r(t_X)\big)\big(d_X/\sqrt f\big)$, and similarly on $Y$. By the Cauchy--Schwarz inequality in $L^2$ and in $\mathbb R^2$, the right-hand side of \eqref{S-eq:r_t_X_Y} is at most $\varrho\sqrt{\mathcal J(d)}$, where
\[
 \varrho=\Big(\int_Xf\,r(t_X)^2\dd x+\int_Yg\,r(t_Y)^2\dd y\Big)^{1/2}.
\]
Hence, we obtain
\[
 \eps\norm{d}_\Zz^2\le\varrho\sqrt{\mathcal J(d)}-\mathcal J(d)\le\frac{\varrho^2}4,
\]
which is \eqref{S-eq:model-energy}.
The interchanged argument would require the same coercivity for $A_{\rm KL}$. The derivative of $\log t_X$ with respect to $k_X$ is $1/k_X$, which is bounded below only through an upper bound on $k_X$. For this reason, \eqref{S-eq:model-energy} is stated only with the ratios of the KL minimizer.

Finally, we prove \eqref{S-eq:model-L4} and \eqref{S-eq:model-energy-L4}. We have $r(1)=r'(1)=0$ and $r''(t)=t^{-2}$. Hence, by Taylor's theorem with the Lagrange form of the remainder, for every $t\in[1-\theta,1+\theta]$, there is $\xi$ between $1$ and $t$ such that
\[
 0\le r(t)=\frac{(t-1)^2}{2\xi^2}\leq\frac{(t-1)^2}{2(1-\theta)^2}.
\]
Since $\norm{(t_X-1)^2}_{L^2(X)}=\norm{t_X-1}_{L^4(X)}^2$, and similarly on $Y$, inserting this bound into $\mathfrak R$ in \eqref{S-eq:model} gives \eqref{S-eq:model-L4}. Inserting it into $\varrho$ in \eqref{S-eq:model-energy} gives \eqref{S-eq:model-energy-L4}.
\end{proof}

\subsubsection{Finite solver tolerances}
In computations, both models are solved only approximately. We now
account for this algebraic error. Let $K_\chi$ and $K_{\rm KL}$ be the
reconstructed plans of array potentials for the Pearson and KL problems, and set
\[
 e_\chi=\sqrt{\Gap_{\chi,h}/(2\eps)},\qquad
 e_{\rm KL}=\sqrt{\Gap_{{\rm KL},h}/\eps},
\]
where $\Gap_{\chi,h}$ is the discrete duality gap \eqref{M-eq:finite-gap} of the Pearson problem. By \cref{M-prop:aposteriori} of the main paper and \cref{S-prop:KLradius}, these are upper bounds for $\norm{K_\chi-k_{\chi,h}^*}_h$ and $\norm{K_{\rm KL}-k_{{\rm KL},h}^*}_h$. By the triangle and reverse triangle inequalities, the
difference $\norm{k_{\chi,h}^*-k_{{\rm KL},h}^*}_h$ of the exact discrete models lies in the interval
\begin{equation}\label{S-eq:enclosure}
 \left[
  (\norm{K_\chi-K_{\rm KL}}_h-e_\chi-e_{\rm KL})_+,\,
  \norm{K_\chi-K_{\rm KL}}_h+e_\chi+e_{\rm KL}
 \right].
\end{equation}

This interval requires both solves. An upper bound from one solve follows from
\cref{S-prop:model-comparison} applied to the cellwise data. Indeed, by \cref{M-prop:reduction} and \cref{S-prop:KLreduction}, the discrete minimizers are the continuum minimizers for the cellwise data, so \cref{S-prop:model-comparison} holds for them.

Let $K$ be a reconstructed plan, and suppose that $\norm{K-K_h^*}_h\leq e$, where $K_h^*$ is the corresponding discrete minimizer. For $i=1,\dots,N_X$, the Cauchy--Schwarz inequality with the weights $\omega_j$ gives
\[
 |K_{X,i}-K_{X,i}^*|
 \le\Big(\sum_{j=1}^{N_Y}\omega_j\Big)^{1/2}\Big(\sum_{j=1}^{N_Y}\omega_j|K_{ij}-K_{ij}^*|^2\Big)^{1/2}
 \le\sqrt b\,\frac{e}{\sqrt{w_i}},
\]
since \[\sum_{j=1}^{N_Y}\omega_j=b\quad \text{and}\quad w_i\sum_{j=1}^{N_Y}\omega_j|K_{ij}-K^*_{ij}|^2\le e^2.\] Similarly, we have $|K_{Y,j}-K_{Y,j}^*|\leq\sqrt{a/\omega_j}\,e$.
Moreover, every marginal of a feasible plan satisfies $K_{X,i}^*\ge\delta b$. Hence, the exact source ratio $K_{X,i}^*/f_i$ lies in the interval
\[
 I_{X,i}=\left[
 \max\left\{\frac{\delta b}{f_i},
            \frac{K_{X,i}-\sqrt{b/w_i}\,e}{f_i}\right\},\
 \frac{K_{X,i}+\sqrt{b/w_i}\,e}{f_i}\right],
\]
and the exact target ratio lies in the interval $I_{Y,j}$ defined analogously. The function $r$ is convex and nonnegative, so its maximum over an interval is attained
at an endpoint. We denote these maxima by
\[
 \widehat r_{X,i}=\max_{t\in I_{X,i}}r(t),\qquad
 \widehat r_{Y,j}=\max_{t\in I_{Y,j}}r(t).
\]
Then \eqref{S-eq:model} gives
\[
 \norm{k_{\chi,h}^*-k_{{\rm KL},h}^*}_h
 \leq\frac1\eps\left[
	 \sqrt b\left(\sum_{i=1}^{N_X}w_i\widehat r_{X,i}^2\right)^{1/2}
 +\sqrt a\left(\sum_{j=1}^{N_Y}\omega_j\widehat r_{Y,j}^2\right)^{1/2}\right].
\]
Since the intervals $I_{X,i}$ and $I_{Y,j}$ contain the exact ratios whenever $\norm{K-K_h^*}_h\le e$, this bound accounts for the algebraic error of the computed plan, and it requires only one solve. One may use either $K_\chi$ and $e_\chi$, or $K_{\rm KL}$ and $e_{\rm KL}$. If both solves are available, one takes the smaller bound.
With the KL solve, \eqref{S-eq:model-energy} gives in addition
\[
 \norm{k_{\chi,h}^*-k_{{\rm KL},h}^*}_h
 \leq\frac1{2\sqrt\eps}\left(
 \sum_{i=1}^{N_X}w_if_i\widehat r_{X,i}^2+\sum_{j=1}^{N_Y}\omega_jg_j\widehat r_{Y,j}^2\right)^{1/2},
\]
where the intervals are built from $K_{\rm KL}$ and $e_{\rm KL}$. 

\subsubsection{Numerical comparison of the models and localization for the KL problem}
We finally illustrate these estimates numerically. The comparison uses
$n=256$ uniform cells on each copy of $[0,1]$, $\eps=1$, $\delta=10^{-4}$, and
$C(x,y)=4(x-y)^2$. The cellwise densities are the values at the cell midpoints of
\begin{align*}
 f(x)&=0.75+1.1e^{-70(x-0.28)^2}+0.35e^{-120(x-0.72)^2},\\
 g(y)&=0.55+0.85e^{-55(y-0.68)^2}.
\end{align*}

\noindent\textit{Solvers.}
The Pearson problem is solved by V-FISTA. The KL problem is solved in two
ways. First, we apply V-FISTA to
$\widetilde\G_{\rm KL}+I_{\mathcal B_{\rm KL}}$. Second, we use the L-BFGS-B method of SciPy to minimize the objective $\G_{\rm KL}$ without continuation, with
the box $\mathcal B_{\rm KL}$ of \cref{S-prop:box} passed as bound constraints on each coordinate. Since L-BFGS-B never evaluates the
objective outside these bounds, it does not need the quadratic
continuation. V-FISTA needs it, because its extrapolation step may leave the box.

\begin{table}[tb]
	\caption{Comparison of the models. The last column gives the weighted residual at which each solver stopped. The transported mass is $\int_{X\times Y} k\dd x\dd y$.}
\label{S-tab:model}
\centering\small
\begin{tabular}{lrrrr}
\toprule
method&transported mass&iterations&gradient evaluations&residual\\
\midrule
Pearson V-FISTA&0.460175&47&95&$1.74\times10^{-12}$\\
KL V-FISTA&0.526141&1314&2629&$1.23\times10^{-9}$\\
KL L-BFGS-B&0.526141&15&32&$5.22\times10^{-9}$\\
\bottomrule
\end{tabular}
\end{table}

\medskip
\noindent\textit{Difference of the models.}
The computed difference $\norm{K_\chi-K_{\rm KL}}_h$ of the plans is $0.091222202$, that is, $13.669930\%$ of the norm of the KL
plan. Of the two KL solves, the V-FISTA solve has the smaller
computed radius $e_{\rm KL}$, so we evaluate \eqref{S-eq:enclosure}
with that solve. Its weighted residual is $1.23\times10^{-9}$ in
\cref{S-tab:model}, compared with $5.22\times10^{-9}$ for L-BFGS-B. The sum $e_\chi+e_{\rm KL}$ of the
two computed radii is $1.349\times10^{-9}$, and it is dominated by
$e_{\rm KL}$.

The difference and the radii are evaluated in binary64
arithmetic rather than with directed rounding. Hence, we do not claim a rigorous
enclosure in floating-point arithmetic. For orientation, however, we subtract and add
the displayed sum of the radii, and round the endpoints outward. This gives
the interval $[0.091222200,0.091222204]$ for the given
algebraic tolerances. We stress that the difference of about $14\%$ is a
property of the two models, not of the solvers, since the Pearson and KL problems have different
objectives.

The upper bound from one solve, evaluated at the KL reconstruction, is $0.260043521$,
with effectivity index about $2.85$. Its energy form \eqref{S-eq:model-energy}, evaluated at the same KL
reconstruction, is $0.111847751$, with effectivity index about $1.23$.

\medskip
\noindent\textit{Choice of the KL solver.}
We do not compare the two KL solvers for efficiency. The continued
V-FISTA iteration is included because it comes with an explicit bound on the condition number and an a priori rate. In contrast, L-BFGS-B is substantially
faster on this example (\cref{S-tab:model}). The large number of V-FISTA iterations
is therefore the price of the conservative global bounds. It does not indicate that V-FISTA is a competitive default solver for these
data.

\medskip
\noindent\textit{Localization.}
We now use the L-BFGS-B solve in \cref{S-prop:tighten}. The bound
$L_{\rm KL}/\mu_{\rm KL}$ on the condition number is $3.50\times10^5$. We evaluate the admissible choice of $D_0$ in \cref{S-prop:tighten} with the reported L-BFGS-B residual
and $\mu_{\rm KL}=10^{-4}$, and obtain the potential radius $D_0=5.22\times10^{-5}$.
Tightening the box with this radius yields $\mu_h^0=0.3390$ and
$L_h^0=2.7956$, and hence the bound $L_h^0/\mu_h^0=8.25$. Evaluating the same estimate with $\mu_h^0$ in place of $\mu_{\rm KL}$ then gives the potential radius $1.54\times10^{-8}$. Thus,
localization improves the reported radius by more than three
orders of magnitude. As before, these values are binary64 evaluations
of bounds in exact real arithmetic, not machine-verified interval enclosures.

\ifSubfilesClassLoaded{\bibliographystyle{amsplain}\bibliography{references}}{}
\end{document}

\bibliographystyle{amsplain}
\bibliography{references}

@book{beck2017first,
  author    = {Beck, Amir},
  title     = {First-Order Methods in Optimization},
  series    = {MOS-SIAM Series on Optimization},
  volume    = {25},
  publisher = {SIAM},
  address   = {Philadelphia, PA},
  year      = {2017},
  doi       = {10.1137/1.9781611974997}
}

@article{alphonse2026globalized,
  author  = {Alphonse, Amal and Christof, Constantin and Hinterm\"uller, Michael and Papadopoulos, Ioannis P. A.},
  title   = {A globalized inexact semismooth {N}ewton method for nonsmooth fixed-point equations involving variational inequalities},
  journal = {Comput. Optim. Appl.},
  volume  = {93},
  number  = {1},
  pages   = {1--55},
  year    = {2026},
  doi     = {10.1007/s10589-025-00722-8}
}

@inproceedings{blondel2018smooth,
  author    = {Blondel, Mathieu and Seguy, Vivien and Rolet, Antoine},
  title     = {Smooth and Sparse Optimal Transport},
  booktitle = {Proceedings of the Twenty-First International Conference on Artificial Intelligence and Statistics},
  series    = {Proceedings of Machine Learning Research},
  volume    = {84},
  pages     = {880--889},
  year      = {2018}
}

@inproceedings{genans2026adaptive,
  author    = {Genans, Ferdinand},
  title     = {Fast and Large-Scale Unbalanced Optimal Transport via its Semi-Dual and Adaptive Gradient Methods},
  booktitle = {Proceedings of the Thirty-Ninth Conference on Learning Theory},
  series    = {Proceedings of Machine Learning Research},
  volume    = {336},
  pages     = {2561--2600},
  year      = {2026},
  url       = {https://proceedings.mlr.press/v336/genans26a.html}
}

@article{chizat2018scaling,
  author  = {Chizat, L\'ena\"ic and Peyr\'e, Gabriel and Schmitzer, Bernhard and Vialard, Fran\c{c}ois-Xavier},
  title   = {Scaling algorithms for unbalanced optimal transport problems},
  journal = {Math. Comp.},
  volume  = {87},
  number  = {314},
  pages   = {2563--2609},
  year    = {2018},
  doi     = {10.1090/mcom/3303}
}

@article{chizat2018unbalanced,
  author  = {Chizat, L\'ena\"ic and Peyr\'e, Gabriel and Schmitzer, Bernhard and Vialard, Fran\c{c}ois-Xavier},
  title   = {Unbalanced optimal transport: Dynamic and {Kantorovich} formulations},
  journal = {J. Funct. Anal.},
  volume  = {274},
  number  = {11},
  pages   = {3090--3123},
  year    = {2018},
  doi     = {10.1016/j.jfa.2018.03.008}
}

@misc{gonzalez2025linear,
  author  = {Gonz\'alez-Sanz, Alberto and Nutz, Marcel and Riveros Valdevenito, Andr\'es},
  title   = {Linear convergence of gradient descent for quadratically regularized optimal transport},
  howpublished = {arXiv preprint arXiv:2509.08547},
  year    = {2025},
  note    = {Version 3, May 2026},
  doi     = {10.48550/arXiv.2509.08547}
}

@misc{gonzalez2026stability,
  author  = {Gonz\'alez-Sanz, Alberto and Nutz, Marcel},
  title   = {Stability of Quadratically Regularized Optimal Transport},
  howpublished = {arXiv preprint arXiv:2605.27883},
  year    = {2026},
  doi     = {10.48550/arXiv.2605.27883}
}

@article{gonzalez2026scalar,
  author  = {Gonz\'alez-Sanz, Alberto and Nutz, Marcel},
  title   = {Sparsity of quadratically regularized optimal transport: Scalar case},
  journal = {SIAM J. Math. Anal.},
  volume  = {58},
  number  = {4},
  pages   = {4162--4183},
  year    = {2026},
  doi     = {10.1137/25M1754480}
}

@article{hintermuller2004mesh,
  author  = {Hinterm\"uller, Michael and Ulbrich, Michael},
  title   = {A Mesh-Independence Result for Semismooth {Newton} Methods},
  journal = {Math. Program.},
  volume  = {101},
  number  = {1},
  pages   = {151--184},
  year    = {2004},
  doi     = {10.1007/s10107-004-0540-9}
}

@article{liero2018optimal,
  author  = {Liero, Matthias and Mielke, Alexander and Savar\'e, Giuseppe},
  title   = {Optimal entropy-transport problems and a new {Hellinger--Kantorovich} distance between positive measures},
  journal = {Invent. Math.},
  volume  = {211},
  number  = {3},
  pages   = {969--1117},
  year    = {2018},
  doi     = {10.1007/s00222-017-0759-8}
}

@article{lorenz2021quadratic,
  author  = {Lorenz, Dirk A. and Manns, Paul and Meyer, Christian},
  title   = {Quadratically regularized optimal transport},
  journal = {Appl. Math. Optim.},
  volume  = {83},
  number  = {3},
  pages   = {1919--1949},
  year    = {2021},
  doi     = {10.1007/s00245-019-09614-w}
}

@book{nesterov2004introductory,
  author    = {Nesterov, Yurii},
  title     = {Introductory Lectures on Convex Optimization: A Basic Course},
  series    = {Applied Optimization},
  volume    = {87},
  publisher = {Kluwer Academic Publishers},
  address   = {Boston, MA},
  year      = {2004},
  doi       = {10.1007/978-1-4419-8853-9}
}

@article{nguyen2023uot,
  author  = {Nguyen, Quang Minh and Nguyen, Hoang H. and Zhou, Yi and Nguyen, Lam M.},
  title   = {On Unbalanced Optimal Transport: Gradient Methods, Sparsity and Approximation Error},
  journal = {J. Mach. Learn. Res.},
  volume  = {24},
  number  = {384},
  pages   = {1--41},
  year    = {2023},
  url     = {https://www.jmlr.org/papers/v24/22-1158.html}
}

@article{nutz2025potentials,
  author  = {Nutz, Marcel},
  title   = {Quadratically regularized optimal transport: Existence and multiplicity of potentials},
  journal = {SIAM J. Math. Anal.},
  volume  = {57},
  number  = {3},
  pages   = {2622--2649},
  year    = {2025},
  doi     = {10.1137/24M1638136}
}

@article{peyre2019computational,
  author    = {Peyr\'e, Gabriel and Cuturi, Marco},
  title     = {Computational Optimal Transport: With Applications to Data Science},
  journal   = {Found. Trends Mach. Learn.},
  volume    = {11},
  number    = {5--6},
  pages     = {355--607},
  year      = {2019},
  doi       = {10.1561/2200000073}
}

@article{ulbrich2002semismooth,
  author  = {Ulbrich, Michael},
  title   = {Semismooth {N}ewton methods for operator equations in function spaces},
  journal = {SIAM J. Optim.},
  volume  = {13},
  number  = {3},
  pages   = {805--841},
  year    = {2002}
}

@article{hintermuller2002primal,
  author  = {Hinterm{\"u}ller, M. and Ito, K. and Kunisch, K.},
  title   = {The primal-dual active set strategy as a semismooth {N}ewton method},
  journal = {SIAM J. Optim.},
  volume  = {13},
  number  = {3},
  pages   = {865--888},
  year    = {2002}
}

@book{repin2008aposteriori,
  author    = {Repin, Sergey},
  title     = {A posteriori estimates for partial differential equations},
  series    = {Radon Series on Computational and Applied Mathematics},
  volume    = {4},
  publisher = {Walter de Gruyter},
  address   = {Berlin},
  year      = {2008},
}

@book{neittaanmaki2004reliable,
  author    = {Neittaanm\"aki, Pekka and Repin, Sergey},
  title     = {Reliable methods for computer simulation: {E}rror control and a posteriori estimates},
  series    = {Studies in Mathematics and its Applications},
  volume    = {33},
  publisher = {Elsevier},
  address   = {Amsterdam},
  year      = {2004},
}

@article{bartels2020gap,
  author    = {Bartels, S\"oren and Milicevic, Marijo},
  title     = {Primal-dual gap estimators for a posteriori error analysis of nonsmooth minimization problems},
  journal   = {ESAIM Math. Model. Numer. Anal.},
  volume    = {54},
  year      = {2020},
  number    = {5},
  pages     = {1635--1660},
}

@article{ern2013adaptive,
  author    = {Ern, Alexandre and Vohral\'ik, Martin},
  title     = {Adaptive inexact {N}ewton methods with a posteriori stopping criteria for nonlinear diffusion {PDE}s},
  journal   = {SIAM J. Sci. Comput.},
  volume    = {35},
  year      = {2013},
  number    = {4},
  pages     = {A1761--A1791},
}

@article{papez2018estimating,
  author    = {Pape\v{z}, Jan and Strako\v{s}, Zden\v{e}k and Vohral\'ik, Martin},
  title     = {Estimating and localizing the algebraic and total numerical errors using flux reconstructions},
  journal   = {Numer. Math.},
  volume    = {138},
  year      = {2018},
  number    = {3},
  pages     = {681--721},
}

@article{qi1993nonsmooth,
  author    = {Qi, Liqun and Sun, Jie},
  title     = {A nonsmooth version of {N}ewton's method},
  journal   = {Math. Program.},
  volume    = {58},
  year      = {1993},
  number    = {1--3},
  pages     = {353--367},
}

@book{ulbrich2011semismooth,
  author    = {Ulbrich, Michael},
  title     = {Semismooth {N}ewton methods for variational inequalities and constrained optimization problems in function spaces},
  series    = {MOS-SIAM Series on Optimization},
  volume    = {11},
  publisher = {SIAM},
  address   = {Philadelphia, PA},
  year      = {2011},
}

@article{weiser2005asymptotic,
  author    = {Weiser, Martin and Schiela, Anton and Deuflhard, Peter},
  title     = {Asymptotic mesh independence of {N}ewton's method revisited},
  journal   = {SIAM J. Numer. Anal.},
  volume    = {42},
  year      = {2005},
  number    = {5},
  pages     = {1830--1845},
}

@article{WW2011sparsity,
  author = {Wachsmuth, Gerd and Wachsmuth, Daniel},
  title = {Convergence and Regularization Results for Optimal Control Problems with Sparsity Functional},
  journal = {ESAIM: Control, Optimisation and Calculus of Variations},
  volume = {17},
  number = {3},
  pages = {858--886},
  year = {2011},
  doi = {10.1051/cocv/2010027},
  note = {\url{https://doi.org/10.1051/cocv/2010027}},
  url = {https://www.numdam.org/articles/10.1051/cocv/2010027/}
}

@incollection{sejourne2023unbalanced,
  author    = {S\'ejourn\'e, Thibault and Peyr\'e, Gabriel and Vialard, Fran\c{c}ois-Xavier},
  title     = {Unbalanced optimal transport, from theory to numerics},
  booktitle = {Handbook of Numerical Analysis},
  volume    = {24},
  pages     = {407--471},
  publisher = {Elsevier},
  year      = {2023}
}

@article{schiebinger2019optimal,
  author  = {Schiebinger, Geoffrey and Shu, Jian and Tabaka, Marcin and Cleary, Brian and Subramanian, Vidya and Solomon, Aryeh and Gould, Joshua and Liu, Siyan and Lin, Stacie and Berube, Peter and Lee, Lia and Chen, Jenny and Brumbaugh, Justin and Rigollet, Philippe and Hochedlinger, Konrad and Jaenisch, Rudolf and Regev, Aviv and Lander, Eric S.},
  title   = {Optimal-transport analysis of single-cell gene expression identifies developmental trajectories in reprogramming},
  journal = {Cell},
  volume  = {176},
  number  = {4},
  pages   = {928--943},
  year    = {2019},
  doi     = {10.1016/j.cell.2019.01.006}
}

@inproceedings{fatras2021unbalanced,
  author    = {Fatras, Kilian and S\'ejourn\'e, Thibault and Flamary, R\'emi and Courty, Nicolas},
  title     = {Unbalanced minibatch optimal transport; applications to domain adaptation},
  booktitle = {Proceedings of the 38th International Conference on Machine Learning},
  series    = {Proceedings of Machine Learning Research},
  volume    = {139},
  pages     = {3186--3197},
  year      = {2021}
}

@article{wiesel2025sparsity,
  author  = {Wiesel, Johannes and Xu, Xingyu},
  title   = {Sparsity of quadratically regularized optimal transport: Bounds on concentration and bias},
  journal = {SIAM J. Math. Anal.},
  volume  = {57},
  number  = {6},
  pages   = {6498--6521},
  year    = {2025},
  doi     = {10.1137/25M1723633}
}

@inproceedings{chapel2021unbalanced,
  author    = {Chapel, Laetitia and Flamary, R\'emi and Wu, Haoran and F\'evotte, C\'edric and Gasso, Gilles},
  title     = {Unbalanced optimal transport through non-negative penalized linear regression},
  booktitle = {Advances in Neural Information Processing Systems},
  volume    = {34},
  year      = {2021}
}

@misc{gonzalez2026geometry,
  author       = {Gonz\'alez-Sanz, Alberto and Nutz, Marcel},
  title        = {Geometry and Convergence of Quadratically Regularized Optimal Transport {I}},
  howpublished = {arXiv preprint arXiv:2609.20400},
  year         = {2026},
  doi          = {10.48550/arXiv.2609.20400}
}

@misc{lorenz2019orlicz,
  author       = {Lorenz, Dirk A. and Mahler, Hinrich},
  title        = {Orlicz-space regularization for optimal transport and algorithms for quadratic regularization},
  howpublished = {arXiv preprint arXiv:1909.06082},
  year         = {2019},
  doi          = {10.48550/arXiv.1909.06082}
}

@article{tartavel2016wasserstein,
  author  = {Tartavel, Guillaume and Peyr{\'e}, Gabriel and Gousseau, Yann},
  title   = {{Wasserstein} Loss for Image Synthesis and Restoration},
  journal = {SIAM Journal on Imaging Sciences},
  volume  = {9},
  number  = {4},
  pages   = {1726--1755},
  year    = {2016},
  doi     = {10.1137/16M1067494}
}

@inproceedings{arjovsky2017wasserstein,
  author    = {Arjovsky, Martin and Chintala, Soumith and Bottou, L{\'e}on},
  title     = {{Wasserstein} Generative Adversarial Networks},
  booktitle = {Proceedings of the 34th International Conference on Machine Learning},
  series    = {Proceedings of Machine Learning Research},
  volume    = {70},
  pages     = {214--223},
  publisher = {PMLR},
  year      = {2017},
  url       = {https://proceedings.mlr.press/v70/arjovsky17a.html}
}

@inproceedings{kusner2015word,
  author    = {Kusner, Matt and Sun, Yu and Kolkin, Nicholas and Weinberger, Kilian},
  title     = {From Word Embeddings to Document Distances},
  booktitle = {Proceedings of the 32nd International Conference on Machine Learning},
  series    = {Proceedings of Machine Learning Research},
  volume    = {37},
  pages     = {957--966},
  publisher = {PMLR},
  year      = {2015},
  url       = {https://proceedings.mlr.press/v37/kusnerb15.html}
}

@article{chiappori2010hedonic,
  author  = {Chiappori, Pierre-Andr{\'e} and McCann, Robert J.
             and Nesheim, Lars},
  title   = {Hedonic Price Equilibria, Stable Matching, and Optimal
             Transport: Equivalence, Topology, and Uniqueness},
  journal = {Economic Theory},
  volume  = {42},
  number  = {2},
  pages   = {317--354},
  year    = {2010},
  doi     = {10.1007/s00199-009-0455-z}
}

@article{carlier2010matching,
  author  = {Carlier, Guillaume and Ekeland, Ivar},
  title   = {Matching for Teams},
  journal = {Economic Theory},
  volume  = {42},
  number  = {2},
  pages   = {397--418},
  year    = {2010},
  doi     = {10.1007/s00199-008-0415-z}
}

@book{galichon2016optimal,
  author    = {Galichon, Alfred},
  title     = {Optimal Transport Methods in Economics},
  publisher = {Princeton University Press},
  address   = {Princeton, NJ},
  year      = {2016},
  doi       = {10.1515/9781400883592}
}

@book{villani2009optimal,
  author    = {Villani, C{\'e}dric},
  title     = {Optimal Transport: Old and New},
  series    = {Grundlehren der mathematischen Wissenschaften},
  volume    = {338},
  publisher = {Springer},
  address   = {Berlin, Heidelberg},
  year      = {2009},
  doi       = {10.1007/978-3-540-71050-9}
}

@article{mccann1997convexity,
  author  = {McCann, Robert J.},
  title   = {A Convexity Principle for Interacting Gases},
  journal = {Advances in Mathematics},
  volume  = {128},
  number  = {1},
  pages   = {153--179},
  year    = {1997},
  doi     = {10.1006/aima.1997.1634}
}

@article{caffarelli1996boundary,
  author  = {Caffarelli, Luis A.},
  title   = {Boundary Regularity of Maps with Convex Potentials. {II}},
  journal = {Annals of Mathematics},
  volume  = {144},
  number  = {3},
  pages   = {453--496},
  year    = {1996},
  doi     = {10.2307/2118564}
}

@article{jordan1998variational,
  author  = {Jordan, Richard and Kinderlehrer, David and Otto, Felix},
  title   = {The Variational Formulation of the {Fokker--Planck} Equation},
  journal = {SIAM Journal on Mathematical Analysis},
  volume  = {29},
  number  = {1},
  pages   = {1--17},
  year    = {1998},
  doi     = {10.1137/S0036141096303359}
}

@inproceedings{nowozin2016fgan,
  author    = {Nowozin, Sebastian and Cseke, Botond and Tomioka, Ryota},
  title     = {{f-GAN}: Training Generative Neural Samplers using
               Variational Divergence Minimization},
  booktitle = {Advances in Neural Information Processing Systems},
  volume    = {29},
  publisher = {Curran Associates, Inc.},
  year      = {2016},
  url       = {https://proceedings.neurips.cc/paper/2016/hash/cedebb6e872f539bef8c3f919874e9d7-Abstract.html}
}

@article{si2024multivariate,
  author  = {Si, Tong and Wang, Yunge and Zhang, Lingling
             and Richmond, Evan and Ahn, Tae-Hyuk and Gong, Haijun},
  title   = {Multivariate Time Series Change-Point Detection with
             a Novel {Pearson}-like Scaled {Bregman} Divergence},
  journal = {Stats},
  volume  = {7},
  number  = {2},
  pages   = {462--480},
  year    = {2024},
  doi     = {10.3390/stats7020028}
}

@article{duchi2021learning,
  author  = {Duchi, John C. and Namkoong, Hongseok},
  title   = {Learning Models with Uniform Performance via
             Distributionally Robust Optimization},
  journal = {The Annals of Statistics},
  volume  = {49},
  number  = {3},
  pages   = {1378--1406},
  year    = {2021},
  doi     = {10.1214/20-AOS2004}
}

@inproceedings{zimmermann2024visa,
  author    = {Zimmermann, Heiko and Naesseth, Christian A.
               and van de Meent, Jan-Willem},
  title     = {{VISA}: Variational Inference with Sequential
               Sample-Average Approximations},
  booktitle = {Advances in Neural Information Processing Systems},
  volume    = {37},
  publisher = {Curran Associates, Inc.},
  year      = {2024},
  url       = {https://proceedings.neurips.cc/paper_files/paper/2024/hash/fa948624dfde013671e72c1a7ca4aebc-Abstract-Conference.html}
}

@book{wahlbin1995superconvergence,
  author    = {Wahlbin, Lars B.},
  title     = {Superconvergence in {G}alerkin Finite Element Methods},
  series    = {Lecture Notes in Mathematics},
  volume    = {1605},
  publisher = {Springer-Verlag},
  address   = {Berlin},
  year      = {1995},
  doi       = {10.1007/BFb0096835}
}

@article{hinze2012semismooth,
  author  = {Hinze, Michael and Vierling, Morten},
  title   = {The semi-smooth {N}ewton method for variationally discretized control constrained elliptic optimal control problems; implementation, convergence and globalization},
  journal = {Optim. Methods Softw.},
  volume  = {27},
  number  = {6},
  pages   = {933--950},
  year    = {2012},
  doi     = {10.1080/10556788.2012.676046}
}

@misc{wachsmuth2025globalized,
  author       = {Wachsmuth, Daniel},
  title        = {A globalized inexact semismooth {N}ewton method for strongly convex optimal control problems},
  howpublished = {arXiv preprint arXiv:2503.21612},
  year         = {2025},
  doi          = {10.48550/arXiv.2503.21612}
}

@article{armijo1966minimization,
  author  = {Armijo, Larry},
  title   = {Minimization of functions having {L}ipschitz continuous first partial derivatives},
  journal = {Pacific J. Math.},
  volume  = {16},
  number  = {1},
  pages   = {1--3},
  year    = {1966}
}

@book{nocedal2006numerical,
  author    = {Nocedal, Jorge and Wright, Stephen J.},
  title     = {Numerical Optimization},
  edition   = {2},
  series    = {Springer Series in Operations Research and Financial Engineering},
  publisher = {Springer},
  address   = {New York},
  year      = {2006}
}

@book{clarke1983optimization,
  author    = {Clarke, Frank H.},
  title     = {Optimization and Nonsmooth Analysis},
  publisher = {Wiley-Interscience},
  address   = {New York},
  year      = {1983}
}
\end{document}